%% file: BGORS_Stein_Wishart_2026-06-03__as_submitted_.tex
\documentclass[12pt,authoryear]{elsarticle}
\journal{}

\makeatletter
\def\ps@pprintTitle{%
 \let\@oddhead\@empty
 \let\@evenhead\@empty
 \def\@oddfoot{\hfill\thepage}%
 \let\@evenfoot\@oddfoot}
\makeatother

\usepackage{amsmath,amssymb,amsfonts,amsthm,graphicx}
\usepackage[bookmarksnumbered,colorlinks=true]{hyperref}
\usepackage[labelfont=bf]{caption}
\usepackage{natbib}

\providecommand{\doi}[1]{\href{https://doi.org/#1}{doi:#1}}
\usepackage{xurl} 
\renewcommand{\doi}[1]{%
 \href{https://doi.org/#1}{\nolinkurl{doi:#1}}%
}

\usepackage{appendix} 
\usepackage{dsfont} 
\usepackage{enumitem} 
\usepackage{mathtools} 
\usepackage{float} 
\usepackage{xcolor} 
\usepackage{booktabs}
\usepackage{longtable} 
\usepackage{multirow} 
\usepackage{geometry} 
\newcounter{pstep} 
\newcommand{\step}[1]{%
\par\medskip\noindent
\refstepcounter{pstep}
\textbf{Step \thepstep: #1.}
}

\numberwithin{equation}{section}
\numberwithin{table}{section}
\numberwithin{figure}{section}

\theoremstyle{plain}
\newtheorem{theorem}{Theorem}[section]
\newtheorem{proposition}[theorem]{Proposition}
\newtheorem{lemma}[theorem]{Lemma}
\newtheorem{corollary}[theorem]{Corollary}

\theoremstyle{definition}

\newtheorem{remark}{Remark}[section]
\newtheorem*{open}{Open problem}

\newcommand{\N}{\mathbb{N}}
\newcommand{\R}{\mathbb{R}}
\newcommand{\C}{\mathbb{C}}

\newcommand{\EE}{\mathsf{E}} 
\newcommand{\Var}{\mathsf{Var}} 
\newcommand{\bb}[1]{\boldsymbol{#1}}
\newcommand{\rd}{\mathrm{d}}
\newcommand{\tr}{\mathrm{tr}}
\newcommand{\etr}{\mathrm{etr}}
\newcommand{\vecc}{\mathrm{vec}}
\newcommand{\vecp}{\mathrm{vecp}}
\newcommand{\OO}{\mathcal{O}}

\newcommand{\ii}{\mathrm{i}\hspace{0.2mm}}
\newcommand{\ind}{\mathds{1}}
\newcommand{\leqdef}{\vcentcolon=}

\allowdisplaybreaks

\begin{document}

\begin{frontmatter}

\title{Stein's method for the Wishart distribution}

\author[a1]{Gabriel Bailly}
\author[a2]{Robert E.\ Gaunt}
\author[a3]{Fr\'ed\'eric Ouimet\corref{mycorrespondingauthor}}
\author[a4]{Donald Richards}
\author[a1]{Rainer von Sachs}

\address[a1]{ISBA/LIDAM, UCLouvain, Louvain-La-Neuve, Belgium}
\address[a2]{The University of Manchester, Manchester, M13 9PL, UK}
\address[a3]{Universit\'e du Qu\'ebec \`a Trois-Rivi\`eres, Trois-Rivi\`eres, QC G8Z 4M3, Canada}
\address[a4]{Penn State University, University Park, PA 16802, USA}

\cortext[mycorrespondingauthor]{Corresponding author. Email address: frederic.ouimet2@uqtr.ca}

\begin{abstract}
In this work, we develop Stein's method for the Wishart distribution on the cone of positive definite matrices. We establish the basic ingredients of a Wishart Stein framework: we derive an extended-generator-based Stein characterization from the Wishart diffusion process, identify the corresponding transition semigroup through the noncentral Wishart law, provide an explicit semigroup representation for the solution of the Stein equation, and obtain regularity estimates for the solution. The new methodology is demonstrated in four applications: (i) an order $n^{-1}$ bound, for smooth test functions, for the Wishart approximation of uncentered group-mean scatter matrices in MANOVA; (ii) a quantitative multivariate Satterthwaite approximation; (iii) local/integrated De Bruijn identities and logarithmic Sobolev inequalities for the Wishart measure; and (iv) Stein's method of moments for the shape and scale parameters, including structured scale estimation.
\end{abstract}


\end{frontmatter}


\section{Introduction}\label{sec:intro}

Stein's method is a powerful technique that has classically been applied to bound the distance between two probability distributions with respect to a probability metric. The method was introduced in Charles Stein's seminal paper \citep{s72} in the context of normal approximation. Shortly afterwards, the method was extended to Poisson approximation \citep{c75}, and the method has since been adapted to many of the most important univariate probability distributions, including the exponential \citep{cfr11,pr11}, gamma \citep{Luk1994PhD,gpr17}, variance-gamma \citep{g14} and $\alpha$-stable distributions \citep{ah19book,x19}. Stein's method has also been developed for multivariate distributions, most notably the multivariate normal \citep{MR1035659,g91} and multivariate $\alpha$-stable distributions \citep{ah19,stable24}, and a general theory of Stein's method for multivariate distributions is beginning to emerge \citep{mrrs23}. Introductions to Stein's method and some of its numerous applications throughout the mathematical sciences are given in the monographs \citep{MR2732624,np12} and surveys \citep{ross11,survey23}.

Whilst Stein's method has reached an impressive level of maturity for univariate and multivariate distributions, a conspicuous absence in the literature is a systematic development for matrix-variate distributions. The first step towards the development of Stein's method for matrix-variate distributions was taken in the recent work of \citet{GauntOuimetRichards2026}, which develops the basic framework of Stein's method for matrix normal approximation. In this paper, we make a considerably more substantial step by extending Stein's method to the Wishart distribution, arguably the most important matrix-variate distribution, with numerous applications in multivariate statistics, Bayesian analysis and random matrix theory.

The development of Stein's method for the Wishart distribution is of particular methodological interest on the following grounds. Firstly, whilst there are instances of matrix normal approximations in which the problem must be dealt with at the matrix-variate level (see Section~4.2 of \citet{GauntOuimetRichards2026} for an example), there have been instances in the literature in which a vectorization argument has been used to reduce a matrix normal approximation problem to a multivariate normal approximation, for which the powerful existing theory on Stein's method for multivariate normal approximation can then be applied; see, e.g., \citet{m22,nz22,tudor}. Such vectorization arguments are not possible for the Wishart distribution, which underlines the importance of a theory of Stein's method for Wishart approximation.

Secondly, the basic Stein's method framework for matrix normal approximation that was established by \citet{GauntOuimetRichards2026} extended the framework of Stein's method for multivariate normal approximation in a very natural manner with the proofs following similar lines to those employed in Stein's method for multivariate normal approximation. However, the relatively clean extension of Stein's method for multivariate normal approximation to matrix normal approximation is perhaps atypical. In contrast, extending Stein's method for gamma approximation to the Wishart distribution is highly non-trivial, and the techniques we employ in setting up Stein's method for the Wishart distribution may prove to be crucial in enabling researchers to extend Stein's method to other matrix-variate distributions.

\subsection{Review of Stein's method for the gamma distribution}

Before giving a summary of our contributions on Stein's method for the Wishart distribution, it will be helpful to recall the essential ingredients of Stein's method for gamma approximation. The starting point for this is the following \emph{Stein characterization} of the gamma distribution; see \citet{dz91,Luk1994PhD,gpr17}. For $r, \lambda > 0$, if $\Gamma(r, \lambda)$ denotes the gamma distribution with density $f_{r, \lambda}(x) = \lambda^{r} x^{r - 1} e^{-\lambda x}/\Gamma(r)$, $x > 0$, then a real-valued random variable $X$ satisfies
\begin{equation}\label{Steinchargamma}
X\sim \Gamma(r, \lambda) \qquad \Leftrightarrow \qquad \EE[X f''(X) + (r - \lambda X) f'(X)] = 0 ~~\forall f\in C_{\mathcal{A}^{\Gamma}}^2(0, \infty),
\end{equation}
where $C_{\mathcal{A}^{\Gamma}}^2(0, \infty)$ denotes the class of all twice differentiable functions $f:(0,\infty)\to \R$ for which the expectations $\EE[|f'(X_{\infty})|]$, $\EE[|X_{\infty} f'(X_{\infty})|]$, and $\EE[|X_{\infty} f''(X_{\infty})|]$ are finite for $X_{\infty}\sim \Gamma(r, \lambda)$. The Stein characterization \eqref{Steinchargamma} leads to the following \emph{Stein equation} for the gamma distribution:
\begin{equation}\label{gammasteineqn}
\mathcal{A}^{\Gamma} f(x) = h(x) - \EE[h(X)],
\end{equation}
where the \emph{Stein operator} $\mathcal{A}^{\Gamma}$ is defined by
\begin{align}\label{gammasteinop}
\mathcal{A}^{\Gamma} f(x) = x f''(x) + (r - \lambda x) f'(x),
\end{align}
$h: (0,\infty) \to \R$ is a real-valued test function, and $X_{\infty}\sim \Gamma(r, \lambda)$. On evaluating both sides of the Stein equation \eqref{gammasteineqn} at a random variable of interest $X$ and taking expectations, we get that
\begin{align}\label{transfer}
\EE[h(X)] - \EE[h(X_{\infty})] = \EE[Xf_h''(X) + (r - \lambda X)f_h'(X)],
\end{align}
where $f_h$ solves the Stein equation \eqref{gammasteineqn}. Thus, the problem of bounding the quantity $\EE[h(X)] - \EE[h(X_{\infty})]$ is reduced to solving the Stein equation \eqref{gammasteineqn} and then bounding the right-hand side of \eqref{transfer}, which is generally more tractable on account of the fact that this expectation only involves the random variable $X$. In order for this procedure to be effective, suitable bounds on the derivatives of the solution $f_h$ are required.

The Stein equation \eqref{gammasteineqn} was obtained by \citet{Luk1994PhD} using the generator approach to Stein's method \citep{MR1035659,g91} in which the Stein operator for the $\Gamma(r, \lambda)$ distribution is recognized as the generator of the well-known Cox--Ingersoll--Ross process, which is defined through the stochastic differential equation (SDE)
\[
\rd X_t = (r - \lambda X_t) \, \rd t + \sqrt{2X_t} \, \rd B_t, \qquad X_0 = x,
\]
where $(B_t)_{t\geq0}$ is a standard Brownian motion. By standard theory on Markov processes, it then follows that the solution to the Stein equation \eqref{gammasteineqn} is given by
\begin{align}\label{soln}
f_h(x) = -\int_0^{\infty} \big\{\mathcal{P}_t h(x) - \EE[h(X)]\big\} \, \rd t,
\end{align}
where the transition semigroup operator $\mathcal{P}_t$ is given by $\mathcal{P}_t h(x) \leqdef \EE[h(X_t) \mid X_0 = x]$, for $t\geq 0$. Provided that the test function $h$ has sufficient regularity such that an interchange in the order of differentiation and integration is permissible, we get that
\begin{align}\label{solnderiv}
f_h^{(m)}(x) = -\int_0^{\infty} (\mathcal{P}_t h)^{(m)}(x) \, \rd t, \qquad m\geq1.
\end{align}
In order to obtain bounds on the derivatives of $f_h$, \citet{Luk1994PhD} derived an explicit formula for the transition semigroup $\mathcal{P}_th(x)$ and its derivatives. It was shown in Lemma~2.4 of \citet{Luk1994PhD} that $X_t \mid \{X_0 = x\}\sim \sigma^2(t)\cdot\Gamma(r, \lambda, \theta_tx)$, where $\sigma^2(t) = 1 - e^{-\lambda t}$, $\theta_t = 2\lambda e^{-\lambda t}/\sigma^2(t)$, and $\Gamma(r, \lambda, \theta)$ denotes the noncentral gamma distribution with noncentrality parameter $\theta$. It thus follows that
\begin{equation}\label{semifor}
\mathcal{P}_th(x) = \EE[h(\sigma^2(t) Y)],
\end{equation}
where $Y\sim \Gamma(r, \lambda, \theta_tx)$. An expression for the $m$-th order derivative of the transition semigroup $\mathcal{P}_th(x)$ can then be obtained via direct calculation, with the result, found in Lemma~2.5 of \citet{Luk1994PhD}, given by the simple formula:
\begin{equation}\label{eq:Lemma.2.5.Luk}
(\mathcal{P}_t h)^{(m)}(x) = e^{-m\lambda t} \, \EE[h^{(m)}(\sigma^2(t) Y_m)], \qquad m\geq1,
\end{equation}
where $Y_m\sim \Gamma(r + m, \lambda, \theta_tx)$. If the $m$-th derivative of $h$ exists and is bounded, then we have the bound $|(\mathcal{P}_t h)^{(m)}(x)|\leq e^{-m\lambda t} \, \|h^{(m)}\|_{\infty}$, where $\|\cdot\|_{\infty}$ is the supremum norm. When this estimate is applied to \eqref{solnderiv}, it then follows from dominated convergence (see Theorem~2.6 of \citet{Luk1994PhD}) that
\begin{equation}\label{lukbound}
\|f_h^{(m)}\|_{\infty} \leq\frac{1}{m\lambda}\|h^{(m)}\|_{\infty}, \qquad m\geq1.
\end{equation}
Under the weaker assumption that the $(m - 1)$-th derivative of $h$ exists and is bounded, \citet[Theorem~2.19]{Gaunt2013PhD} obtained the following bound:
\begin{align}\label{gbound}
\|f_h^{(m)}\|_{\infty} \leq\bigg\{\frac{\sqrt{2\pi} + e^{-1}}{\sqrt{r + m - 1}} + \frac{2}{r + m - 1}\bigg\}\|h^{(m - 1)}\|_{\infty}, \qquad m\geq1,
\end{align}
where $h^{(0)}\equiv h$. One cannot hope for a bound of the form $\|f_h^{(m)}\|\leq C_{r, \lambda,m}\|h^{(m - 2)}\|_{\infty}$, $m\geq2$, for some constant $C_{r, \lambda,m} > 0$ due to a universal counterexample of \citet{ev15}.

\subsection{Summary of our contributions}

We begin our development of Stein's method for the Wishart distribution by identifying the Wishart diffusion as the matrix-valued analog of the Cox--Ingersoll--Ross process. In Proposition~\ref{prop:generator}, we give the extended generator of the Wishart process, introduced by \citet{Bru1991}, and this generator provides the natural candidate for the Wishart Stein operator. We then provide the transition law of the Wishart process in Proposition~\ref{prop:extension.Luk.1994.Lemma.2.4}. In particular, for positive definite matrices $\Sigma$ and $W$, write $\Sigma_t \leqdef (1 - e^{-2t})\Sigma$ and, for $t > 0$, let $\mathfrak{S}_t\sim \smash{\mathcal{W}_d(\alpha, I_d, e^{-2t}\Sigma_t^{-1/2} W \Sigma_t^{-1/2})}$. Then the transition semigroup admits the explicit representation
\[
\mathcal{P}_t^{\mathcal{W}} h(W) = \EE\big[h(\Sigma_t^{1/2} \mathfrak{S}_t \Sigma_t^{1/2})\big],
\]
which is the Wishart analog of the gamma semigroup representation \eqref{semifor} used by \citet{Luk1994PhD}. This transition law also identifies the invariant law of the Wishart process as $\mathfrak{W}_{\infty}\sim \mathcal{W}_d(\alpha,\Sigma)$. The generator formulation is then used to establish the full Stein characterization in Corollary~\ref{cor:Stein.Wishart.generator.step.1}. As noted in Remark~\ref{rem:Wishart.Stein.operator.d.1}, when $d = 1$ the Wishart Stein operator reduces, up to the positive multiplicative factor $4\Sigma$, to the classical gamma Stein operator \eqref{gammasteinop}.

The main technical obstacle in obtaining the semigroup solution of the Wishart Stein equation and its regularity estimates is the need to control the transition semigroup as a function of its starting point $W$. The difficulty is already visible in the scalar gamma case. If $f(u;r, \lambda, \theta_t x)$ denotes the noncentral gamma density appearing in \eqref{semifor}, then the key transfer identity underlying Lemma~2.5 of \citet{Luk1994PhD} is
\begin{equation}\label{eq:transfer}
\frac{\partial}{\partial x} f(u;r, \lambda, \theta_t x) = \frac{\theta_t}{2}\{f(u;r + 1, \lambda, \theta_t x)-f(u;r, \lambda, \theta_t x)\} = -\frac{\theta_t}{2\lambda}\frac{\partial}{\partial u} f(u;r + 1, \lambda, \theta_t x),
\end{equation}
where the two equalities follow from the standard Bessel identities obtained by specializing Eqs.~(16.3.1)~and~(16.3.4) of \citet{dlmf16} to $\gamma = r$, $n=1$ and $p=q=0$. We recover \eqref{eq:Lemma.2.5.Luk} by iterating the above transfer identity and integrating by parts: for any $m\geq 1$,
\[
(\mathcal{P}_t h)^{(m)}(x) = \left(\frac{\theta_t\sigma^2(t)}{2\lambda}\right)^m \int_0^{\infty} h^{(m)}(\sigma^2(t)u) f(u;r + m, \lambda, \theta_t x) \, \rd u = e^{-m \lambda t} \, \EE[h^{(m)}(\sigma^2(t)Y_m)],
\]
where $Y_m\sim \Gamma(r + m, \lambda, \theta_t x)$. The actual proof of Lemma~2.5 by \citet{Luk1994PhD} exploits the Poisson mixture representation of the noncentral gamma density to justify differentiating under the integral sign and integrating by parts, and to apply the argument described above on each summand of the mixture.

The Wishart case is substantially more difficult. There is no comparably simple Poisson mixture representation for the noncentral Wishart density, and the scalar Bessel derivative identities are replaced by determinantal derivative formulae for the Bessel function of matrix argument; see Eqs.~(2.3) and~(2.4) of \citet{Herz1955}. A direct matrix analog of Luk's proof is therefore not tractable. The key new observation of the present paper is that the transfer from derivatives in the noncentrality parameter to derivatives in the state variable, as in \eqref{eq:transfer}, can instead be obtained through Laplace transforms. This is carried out in two steps in Section~\ref{sec:tech.lemmas}. First, recall that for a sufficiently integrable function $g:\mathcal{S}_{++}^d\to \R$, its Laplace transform is defined by
\begin{equation}\label{eq:Laplace.transform}
\mathcal{L}[g](T) \leqdef \int_{\mathcal{S}_{++}^d} \etr(-T X) \, g(X) \, \rd X, \qquad T\in \mathcal{S}_{+}^d,
\end{equation}
where $\mathcal{S}_{+}^d$ denotes the cone of $d\times d$ real nonnegative definite matrices. Using this definition, Lemma~\ref{lem:Laplace.polynomial.calculus} proves the following integration-by-parts Laplace calculus for polynomial differential operators $P(\nabla_{\!X})$ built from the symmetric gradient $\nabla_{\!X}$:
\[
\mathcal{L}[P(\nabla_{\!X}) f_{\alpha, I_d, \Theta}^{\mathcal{W}}](T) = P(T) \, \mathcal{L}[f_{\alpha, I_d, \Theta}^{\mathcal{W}}](T), \qquad T\in \mathcal{S}_{+}^d,
\]
where $f_{\alpha, \Sigma, \Theta}^{\mathcal{W}}$ denotes the noncentral Wishart density defined in \eqref{eq:noncentral.Wishart.density}.
Second, differentiating the noncentral Wishart Laplace transform with respect to the noncentrality parameter $\Theta$ produces the factor $T(I_d + 2T)^{-1}$, and the identity
\[
T(I_d + 2T)^{-1} = \frac{T \, \mathrm{adj}(I_d + 2T)}{|I_d + 2T|} = T \, \mathrm{adj}(I_d + 2T) \mathcal{L}[f_{2, I_d, 0_{d\times d}}^{\mathcal{W}}](T)
\]
converts this factor into a polynomial in $T$ at the cost of shifting the shape parameter of the Wishart density by~$2$. Combining this with Lemma~\ref{lem:Laplace.polynomial.calculus} gives Lemma~\ref{lem:Laplace.derivative.Wishart}, namely, for any pairs of indices $(i_1, j_1), \ldots, (i_k, j_k)\in \{1, \ldots,d\}^2$,
\[
\left(\prod_{\ell=1}^k \nabla_{\Theta, i_{\ell} j_{\ell}}\right) f_{\alpha, I_d, \Theta}^{\mathcal{W}}(X) = (-1)^k \left(\prod_{\ell=1}^k \Big\{\nabla_{\!X} \, \mathrm{adj}(I_d + 2\nabla_{\!X})\Big\}_{i_{\ell} j_{\ell}}\right) f_{\alpha + 2k, I_d, \Theta}^{\mathcal{W}}(X).
\]
For $k=1$, this is the matrix transfer formula which generalizes \eqref{eq:transfer}. Notice that, unlike in the gamma case where each derivative in the noncentrality parameter is transferred to a first-order derivative in the state variable, each factor $\{\nabla_{\!X} \, \mathrm{adj}(I_d + 2\nabla_{\!X})\}_{ij}$ is a differential operator of order $d$, since $\mathrm{adj}(I_d + 2\nabla_{\!X})$ is a matrix polynomial in $\nabla_{\!X}$ of degree $d - 1$.

After an integration by parts on the cone $\mathcal{S}_{++}^d$, the operator in Lemma~\ref{lem:Laplace.derivative.Wishart} is moved from the noncentral Wishart density onto the test function $h$, and we can prove Lemma~\ref{lem:Luk.Wishart.Lemma.2.5}, namely
\[
\left(\prod_{\ell=1}^k \nabla_{\Lambda, i_{\ell} j_{\ell}}\right) (\mathcal{Q}_t h)(\Lambda) = e^{-2kt} \, \EE\left[\left(\prod_{\ell=1}^k \mathcal{D}_{i_{\ell} j_{\ell}}\right) h(\mathfrak{Y}_{k,t})\right], \qquad \mathcal{D} = \nabla_{\!X} \, \mathrm{adj}(I_d - 2\nabla_{\!X}),
\]
where $\Lambda\in \mathcal{S}_{++}^d$, $\mathfrak{Y}_{k,t}\sim \mathcal{W}_d(\alpha + 2k, I_d, e^{-2t}\Lambda)$, and $\mathcal{Q}_t$ is a scale-normalized analog of the transition semigroup operator $\mathcal{P}_t^{\mathcal{W}}$; see the definition in \eqref{eq:Wishart.Q.operator.def}. This lemma is the technical core of the paper: it shows how derivatives of $(\mathcal{Q}_t h)(\Lambda)$ with respect to the entries of the noncentrality matrix $\Lambda$ can be pushed through the noncentral Wishart transition kernel and transferred onto the test function~$h$, thereby generalizing Lemma~2.5 of \citet{Luk1994PhD}, or equivalently, Eq.~\eqref{eq:Lemma.2.5.Luk}.

This transfer mechanism first enters the proofs of the main results through Lemma~\ref{lem:contraction}. In Lemma~\ref{lem:contraction}, the transition semigroup is written as
\[
(\mathcal{P}_t^{\mathcal{W}}h)(W) = (\mathcal{Q}_t h_t)(\Sigma_t^{-1/2}W\Sigma_t^{-1/2}), \qquad h_t(X) \leqdef h(\Sigma_t^{1/2}X\Sigma_t^{1/2}),
\]
and the difference between $(\mathcal{Q}_t h_t)(\Sigma_t^{-1/2}W\Sigma_t^{-1/2})$ and the central value $(\mathcal{Q}_t h_t)(0)$ is controlled by differentiating along the line segment in the $\Lambda$ variable. The case $k = 1$ of Lemma~\ref{lem:Luk.Wishart.Lemma.2.5} gives the required derivative bound in terms of the $\mathcal{D}$-seminorm of $h_t$. This decay estimate is the key input in Theorem~\ref{thm:Stein.solutions.Wishart}, where we prove that the semigroup integral
\[
f_h(W) \leqdef -\int_0^{\infty}\Big\{\mathcal{P}^{\mathcal{W}}_t h(W)-\EE[h(\mathfrak{W}_{\infty})]\Big\} \, \rd t
\]
is well defined and solves the Wishart Stein equation. This representation is the Wishart analog of the gamma solution formula \eqref{soln}.

Finally, Theorem~\ref{thm:smoothness.estimates.wishart} completes the analytic part of the Wishart Stein framework by deriving regularity estimates for the semigroup solution. Its proof applies Lemma~\ref{lem:Luk.Wishart.Lemma.2.5} with $k = m$ to the scale-normalized transition semigroup, transfers the resulting derivatives onto the test function through products of the operators $\mathcal{D}_{t, \Sigma} \leqdef \nabla_{\!W} \, \mathrm{adj}(I_d - 2\Sigma_t\nabla_{\!W})$, and then integrates the resulting semigroup bounds in the representation of $f_h$. This yields the indexed partial derivative estimate \eqref{eq:Wishart.Stein.bound.1} and the coordinate-free directional estimate \eqref{eq:Wishart.Stein.bound.4}. As noted in Remark~\ref{rem:gamma.smoothness.Wishart.Stein}, when $d = 1$ the bound \eqref{eq:Wishart.Stein.bound.1} reduces to the classical bound \eqref{lukbound} of \citet{Luk1994PhD}, up to the $4\Sigma$ factor inherited from the relation between the one-dimensional Wishart and gamma Stein operators.

The form of the estimates in Theorem~\ref{thm:smoothness.estimates.wishart} is dictated by the transfer operator in Lemma~\ref{lem:Luk.Wishart.Lemma.2.5}. Since each entry of $\mathcal{D}_{t, \Sigma}$ is a differential operator of order at most $d$, the product of $m$ such operators naturally leads to seminorms involving derivatives of $h$ up to order $md$. Equivalently, the appearance of the $md$-th order regularity assumption is the cost of pushing $m$ derivatives of the Wishart transition semigroup through the noncentral Wishart kernel and onto the test function. This derivative count is sharp for the transfer identity and for the $\mathcal{D}$-seminorm bounds used in Theorem~\ref{thm:smoothness.estimates.wishart}; see Remark~\ref{rem:Wishart.derivative.loss.md}. It does not, however, rule out the possibility that a different argument could save derivatives at the level of the solution of the Stein equation itself. In the gamma case, such a phenomenon is reflected in the bound \eqref{gbound}, where one derivative of $h$ is saved relative to \eqref{lukbound}. It will be the subject of future work to determine whether a Wishart generalization of \eqref{gbound}, involving derivatives of $h$ up to order $m(d - 1)$, can be attained. We expect that one cannot in general obtain bounds for the $m$-th order partial derivatives of the solution of the Wishart Stein equation under substantially weaker differentiability assumptions on $h$.

This phenomenon of bounds on derivatives of the solution of a Stein equation requiring existence of derivatives whose order increases as a function of the dimension $d$ appears to be novel, although there is a precedent in the literature that stronger regularity conditions may need to be imposed in order to guarantee existence of derivatives of solutions of Stein equations as one moves from a Stein equation for a univariate distribution to its multivariate analog. Indeed, the classic bound $\|f_h^{(3)}\|_{\infty} \leq 2\|h'\|_{\infty}$ of \citet{steinmonograph} for the solution of the standard normal Stein equation $f''(x) - xf'(x) = h(x) - \EE[h(Z)]$, for $Z\sim \mathcal{N}(0,1)$, was shown by \citet{raic04} to not carry over to the solution of the multivariate normal Stein equation, in that there exist Lipschitz test functions $h$ for which the third-order partial derivatives of the solution of the multivariate normal Stein equation do not exist.

In Section~\ref{sec:applications}, we provide several applications of our framework for Stein's method for the Wishart distribution, with examples given that concern distributional approximation, information theory and parameter estimation. We begin by considering the Wishart approximation of uncentered group-mean scatter matrices in Multivariate Analysis of Variance (MANOVA). In Section~\ref{sec:Wishart.approximation.MANOVA}, we derive an explicit order $n^{-1}$ bound to quantify this classical distributional approximation. Our approach involves Wishart generalizations of two techniques that have proven to be highly effective in deriving chi-square approximations via Stein's method (see, for example, \citet{gpr17,gr23}), these being a connection between the Wishart and matrix normal Stein operators and symmetry considerations to attain the optimal $n^{-1}$ rate. In Section~\ref{sec:Wishart.approximation.Satterthwaite}, we derive explicit bounds to quantify the multivariate Satterthwaite approximation proposed by \citet{TanGupta1983}, in which the distribution of a positive linear combination of independent Wishart random matrices is approximated by a single Wishart distribution with matching first two moments, a multivariate analog of the classical Satterthwaite approximation introduced by \citet{Satterthwaite1941,Satterthwaite1946}.

Whilst Stein's method has traditionally been applied to derive distributional approximations, there has been a growing trend in recent years in which the method has found exciting applications in other areas such as information-theoretic inequalities \citep{lnp15} and computational statistics \citep{survey23}. In our applications of Sections~\ref{sec:Wishart.deBruijn.LSI} and \ref{sec:Wishart.parameter.estimation}, we exhibit the utility of our Wishart Stein framework in such domains. In Section~\ref{sec:Wishart.deBruijn.LSI}, we apply aspects of our generator approach formulation of Stein's method for the Wishart distribution developed in Section~\ref{sec:main.results} to derive local and integrated De Bruijn identities for the Wishart semigroup, and complement these results by deriving a logarithmic Sobolev inequality for the Wishart measure, thereby generalizing some of the main results of \citet{ArrasSwan2017} given in the gamma setting. In Section~\ref{sec:Wishart.parameter.estimation}, we utilize our Stein characterization of the Wishart distribution together with the recently introduced Stein's method of moments of \citet{EbnerFischerGauntPickerSwan2025} to derive new closed-form estimators for the Wishart distribution. Our simulations suggest that our estimators improve on the classical moment estimators in terms of the relative Frobenius error and offer an attractive alternative to the maximum likelihood estimator, with the closed-form of the estimators offering benefits in terms of computational cost and amenability to theoretical analysis.

\subsection{Outline of the paper}

Section~\ref{sec:definitions} introduces the definitions and notation regarding matrix spaces, the noncentral Wishart distribution, and infinitesimal and extended generators. Sections~\ref{sec:main.results}~and~\ref{sec:applications} present the main results and the applications, respectively. Section~\ref{sec:tech.lemmas} establishes preliminary lemmas regarding the noncentral Wishart distribution that are needed in the proofs of the main results. Sections~\ref{sec:proofs.main.results}~and~\ref{sec:proofs.application.results} provide the detailed proofs for the main results and the applications, respectively. Supplementary material for some of the applications is collected at the end of the paper.

\section{Definitions and notation}\label{sec:definitions}

Throughout, $[d] \leqdef \{1, \ldots, d\}$ for $d\in \N \equiv \{1, 2, \ldots\}$. Let $\mathcal{S}^d$, $\mathcal{S}_{+}^d$, and $\mathcal{S}_{++}^d$ denote the sets of real symmetric, nonnegative definite, and positive definite $d\times d$ matrices, respectively. Unless mentioned otherwise, these spaces are equipped with the Frobenius norm $\|\cdot\|_F$. Let $O(d)$ denote the orthogonal group of $d\times d$ matrices, consisting of all real matrices $H$ such that $H^{\top} H = I_d$. For any square matrix $A$, let $\tr(A)$ be its trace, $\etr(A) \leqdef \exp\{\tr(A)\}$, and $|A|$ its determinant. For $S\in \mathcal{S}_{+}^d$, the matrix $S^{1/2}$ denotes the unique nonnegative definite square root, $\|S\|_2$ the spectral norm, and $\vecp(S) = (S_{11}, S_{12}, S_{22}, \ldots, S_{1d}, \ldots, S_{dd})^{\top}$ its half-vectorization. The symbols $\bb{0}_d$, $0_{d\times d}$, and $I_d$ denote the $d$-dimensional zero vector, the $d\times d$ zero matrix, and the $d\times d$ identity, respectively.

If $B$ is an open subset of a finite-dimensional real Euclidean space, in particular if $B\subseteq \mathcal{S}^d$ is open in the relative topology, and $m\in \N_0$, let $C^{m}(B)$ be the class of real-valued functions $f:B\to\R$ that are $m$ times continuously differentiable on $B$ (all partial derivatives up to total order $m$ exist and are continuous), and let $C_b^{m}(B)$ be the subclass for which all partial derivatives up to total order $m$, including the order-$0$ derivative $f$ itself, are bounded on $B$. For a map $F:B\to E$, where $E$ is a finite-dimensional real Euclidean space, $D^kF(x)$ denotes the $k$-th Fr\'echet derivative of $F$ at $x$, viewed as a $k$-linear map on the ambient real vector space of $B$ with values in $E$; its evaluation in directions $u_1, \ldots,u_k$ is written $D^kF(x)[u_1, \ldots,u_k]$, with the convention $D^0F = F$, and, when $B\subseteq\mathcal{S}^d$, the directions $u_i$ lie in $\mathcal{S}^d$.

For any $d\in \N$, the multivariate gamma function~$\Gamma_d$ is defined by
\begin{equation}\label{eq:gamma.integral}
\Gamma_d(\nu) \leqdef \int_{\mathcal{S}_{++}^d} \etr(-X) |X|^{\nu - (d + 1)/2} \, \rd X, \qquad \textrm{Re}(\nu) > (d-1)/2,
\end{equation}
which is a natural generalization to the cone $\mathcal{S}_{++}^d$ of the classical gamma function.

For any shape parameter $\alpha \in (d-1, \infty)$, any scale matrix $\Sigma \in \mathcal{S}_{++}^d$, and any noncentrality parameter $\Theta\in \R^{d\times d}$ such that $\Theta \Sigma^{-1}\in \mathcal{S}_{+}^d$, the density of the noncentral Wishart distribution, henceforth denoted $\mathcal{W}_d(\alpha, \Sigma, \Theta)$, is given, for all $X \in \mathcal{S}_{++}^d$, by
\begin{equation}\label{eq:noncentral.Wishart.density}
f_{\alpha, \Sigma, \Theta}^{\mathcal{W}}(X) \leqdef \frac{|X|^{\alpha/2 - (d + 1)/2} \etr(-\Sigma^{-1}X/2)}{|2 \Sigma|^{\alpha/2} \Gamma_d(\alpha/2)} \etr(-\Theta/2) \, {}_0F_1(\alpha/2; \Theta\Sigma^{-1}X/4),
\end{equation}
where ${}_0F_1$ denotes the Bessel function of matrix argument introduced by \citet[Section~2]{Herz1955}. Whenever a random matrix $\mathfrak{X}$ of size $d \times d$ follows this distribution, one writes $\mathfrak{X}\sim \mathcal{W}_d(\alpha, \Sigma, \Theta)$. When $\Theta = 0_{d \times d}$, one recovers the (central) Wishart distribution, and one writes $\mathfrak{X}\sim \mathcal{W}_d(\alpha, \Sigma)$, omitting the noncentrality parameter.

\begin{remark}\label{rem:Wishart.mgf}
For any shape parameter $\alpha\in (d-1, \infty)$, any scale matrix $\Sigma \in \mathcal{S}_{++}^d$, and any noncentrality parameter $\Theta\in \R^{d\times d}$ such that $\Theta \Sigma^{-1}\in \mathcal{S}_{+}^d$, the Laplace transform of $\mathfrak{W}\sim \mathcal{W}_d(\alpha, \Sigma, \Theta)$ is
\[
\EE[\etr(-T \mathfrak{W})]
= \frac{\etr\{-T \Sigma(I_d + 2 T \Sigma)^{-1} \Theta\}}{|I_d + 2 T \Sigma|^{\alpha/2}}, \qquad T\in \mathcal{S}_{+}^d;
\]
see, e.g., Theorem~10.3.3 of \citet{Muirhead1982}.
\end{remark}

For a matrix-variate Markov process $(\mathfrak{M}_t)_{t\geq 0}$ taking values in $\mathcal{S}_{+}^d$, the transition semigroup of operators $(\mathcal{P}_t)_{t\geq 0}$ is defined, for every measurable function $f$ for which the expectation below is finite, by
\[
\mathcal{P}_t f(M) \leqdef \EE[f(\mathfrak{M}_t) \mid \mathfrak{M}_0 = M], \qquad t\geq 0.
\]
The corresponding infinitesimal generator of $(\mathfrak{M}_t)_{t\geq 0}$ is defined on its domain by
\[
\mathcal{A} f(M) \leqdef \lim_{s\downarrow 0} \frac{\mathcal{P}_s f(M) - f(M)}{s},
\]
provided that the limit exists. More generally, for a diffusion, we use the same notation for the extended generator: if $f$ is sufficiently smooth and there exists a measurable function $g$ such that $(f(\mathfrak{M}_t) - f(\mathfrak{M}_0)-\int_0^t g(\mathfrak{M}_s) \, \rd s)_{t\geq 0}$ is a local martingale, then we write $\mathcal{A}f = g$. When $f$ belongs to the domain of the infinitesimal generator, the two notions agree. In this sense, It\^o's formula identifies $\mathcal{A} f(M)$ with the drift term in the stochastic differential of $f(\mathfrak{M}_t)$. Such a characterization is central both to describing the dynamics of the Markov process $(\mathfrak{M}_t)_{t\geq 0}$ and to deriving Stein-type identities.

\section{Main results}\label{sec:main.results}

Following \citet[Section~5.2]{Bru1991}, we define the Wishart process $(\mathfrak{W}_t)_{t\geq 0}$ as a $\mathcal{S}_{++}^d$-valued generalization of the well-known Cox--Ingersoll--Ross process through the SDE
\begin{equation}\label{eq:Wishart.process}
\begin{aligned}
 \, \rd \mathfrak{W}_t
& \leqdef 2 \, (\alpha \Sigma - \mathfrak{W}_t) \, \rd t + \sqrt{2} \, (\mathfrak{W}_t^{1/2} \, \rd \mathfrak{B}_t \Sigma^{1/2} + \Sigma^{1/2} \, \rd \mathfrak{B}_t^{\top} \mathfrak{W}_t^{1/2}), \qquad \mathfrak{W}_0 \leqdef W, \\
\end{aligned}
\end{equation}
where $W, \Sigma\in \mathcal{S}_{++}^d$ are constant matrices, $\alpha > d-1$ is a scalar, and $(\mathfrak{B}_t)_{t\geq 0}$ is a $d\times d$ matrix of independent standard Brownian motions.

The explicit expression for the extended generator of $(\mathfrak{W}_t)_{t\geq 0}$, denoted $\mathcal{A}^{\mathcal{W}}$, is derived in Proposition~\ref{prop:generator} below. This result appears without proof in Eq.~(5.12) of \citet{Bru1991}, under a different parametrization. It also follows from the much more general result of \citet[Theorems~2.4 and~2.6]{CuchieroFilipovicMayerhoferTeichmann2011}, by taking $a = 2\Sigma$, $b = 2\alpha\Sigma$, $B(S) = -2S$, $c = 0$, $\gamma = 0_{d\times d}$, and $m(\cdot) = \mu(\cdot) = 0$, with the diffusion matrix parameter denoted here by $a$ to avoid confusion with the shape parameter $\alpha$. A short and self-contained proof is included in Section~\ref{sec:proofs.main.results} for completeness.

\begin{proposition}\label{prop:generator}
For any $f\in C^2(\mathcal{S}_{++}^d)$, the extended generator of $(\mathfrak{W}_t)_{t\geq 0}$ is given by
\begin{equation}\label{eq:Wishart.process.generator}
\mathcal{A}^{\mathcal{W}}f(S) = 2 \, \tr\{(\alpha \Sigma - S) \nabla f(S)\} + 4 \, \tr\{S \nabla \Sigma \nabla f(S)\}, \qquad S\in \mathcal{S}_{++}^d,
\end{equation}
where $\nabla \leqdef (\frac{1}{2} (1 + \delta_{ij}) \partial / \partial S_{ij})_{1\leq i,j \leq d}$ denotes the $d\times d$ symmetric gradient, and $\delta_{ij}$ denotes the Kronecker delta.
\end{proposition}

According to Theorem~4.20 of \citet{Pfaffel2012}, under the assumption that $\alpha\in [d + 1, \infty)\cap \N$,
\begin{equation}\label{eq:Theorem.4.20.Pfaffel.2012}
\mathfrak{W}_t \mid \{\mathfrak{W}_0 = W\}\sim \mathcal{W}_d(\alpha, \Sigma_t, e^{-2t} \Sigma_t^{-1} W), \qquad \Sigma_t \leqdef (1 - e^{-2t}) \Sigma.
\end{equation}
The strengthened assumption on $\alpha$ ensures positive definiteness of the Wishart process.

The next proposition extends \eqref{eq:Theorem.4.20.Pfaffel.2012} to the full range $\alpha\in (d - 1, \infty)$. It identifies the transition law of the matrix Wishart process started from any $W\in \mathcal{S}_{+}^d$, gives its limiting law, and records the semigroup representation and invariance of $\mathcal{W}_d(\alpha, \Sigma)$ needed below.

\begin{proposition}\label{prop:extension.Luk.1994.Lemma.2.4}
Let $\alpha\in (d - 1, \infty)$ and $\Sigma\in \mathcal{S}_{++}^d$ be given, and set $\Sigma_t \leqdef (1 - e^{-2t})\Sigma$. For the matrix Wishart process defined in \eqref{eq:Wishart.process}, with transition kernel on $\mathcal{S}_{+}^d$, for every initial state $W\in \mathcal{S}_{+}^d$ we have
\begin{align}
\mathfrak{W}_t \mid \{\mathfrak{W}_0 = W\} &\sim \mathcal{W}_d(\alpha, \Sigma_t, e^{-2t} \Sigma_t^{-1} W), \qquad t > 0, \label{eq:W.t.conditional} \\
\mathfrak{W}_t \mid \{\mathfrak{W}_0 = W\} &\stackrel{\mathrm{law}}{\longrightarrow} \mathfrak{W}_{\infty} \sim \mathcal{W}_d(\alpha, \Sigma), \qquad t\to \infty. \label{eq:W.infty}
\end{align}
In particular, let $(\mathcal{P}^{\mathcal{W}}_t)_{t\geq 0}$ be the transition semigroup with kernel $P_t(W, \, \rd Y)$ on $\mathcal{S}_{+}^d$, so that, for every bounded Borel measurable function $h:\mathcal{S}_{+}^d\to\R$,
\begin{equation}\label{semiformula}
(\mathcal{P}^{\mathcal{W}}_t h)(W) \leqdef \int_{\mathcal{S}_{+}^d} h(Y) \, P_t(W, \, \rd Y) = \EE\big[h(\mathfrak{W}_t)\mid \mathfrak{W}_0 = W\big].
\end{equation}
When $h$ is only defined on $\mathcal{S}_{++}^d$, we use the same notation for $t > 0$ and $W\in \mathcal{S}_{+}^d$ by choosing any bounded Borel extension of $h$ to $\mathcal{S}_{+}^d$ since $P_t(W, \partial\mathcal{S}_{+}^d) = 0$. For $\mathfrak{S}_t\sim \smash{\mathcal{W}_d(\alpha, I_d, e^{-2t} \Sigma_t^{-1/2} W \Sigma_t^{-1/2})}$, the above shows
\begin{equation}\label{eq:rep}
\mathcal{P}_t^{\mathcal{W}} h(W) = \EE\big[h(\Sigma_t^{1/2} \mathfrak{S}_t \Sigma_t^{1/2})\big].
\end{equation}
Moreover, given a probability measure $\mu$ on $\mathcal{S}_{+}^d$, the push-forward measure $\mu \mathcal{P}^{\mathcal{W}}_t$ is defined by
\[
(\mu \mathcal{P}^{\mathcal{W}}_t)(A) \leqdef \int_{\mathcal{S}_{+}^d} P_t(W, A) \, \mu(\rd W), \qquad A\subseteq \mathcal{S}_{+}^d \text{ Borel}.
\]
Hence, for $\xi \leqdef \mathcal{W}_d(\alpha, \Sigma)$, we have the invariance
\begin{equation}\label{eq:invariance}
\xi \mathcal{P}^{\mathcal{W}}_t = \xi.
\end{equation}
\end{proposition}

This leads to the following Stein characterization for the Wishart distribution.

\begin{corollary}[Stein characterization]\label{cor:Stein.Wishart.generator.step.1}
Let $\alpha\in (d-1,\infty)$ and $\Sigma\in \mathcal{S}_{++}^d$ be given. Then
\[
\mathfrak{W}\sim \mathcal{W}_d(\alpha, \Sigma) \qquad \Leftrightarrow \qquad \EE\big[\mathcal{A}^{\mathcal{W}} f(\mathfrak{W})\big] = 0 ~~\forall f\in C_{\mathcal{A}^{\mathcal{W}}}^2(\mathcal{S}_{++}^d),
\]
where, for $\mathfrak{W}_{\infty}\sim \mathcal{W}_d(\alpha, \Sigma)$,
\begin{equation}\label{eq:C.2.A.W}
\begin{aligned}
C_{\mathcal{A}^{\mathcal{W}}}^2(\mathcal{S}_{++}^d)
&\leqdef \Big\{f\in C^2(\mathcal{S}_{++}^d) :
\EE[|\tr\{\Sigma \nabla f(\mathfrak{W}_{\infty})\}|] < \infty, ~\EE[|\tr\{\mathfrak{W}_{\infty} \nabla f(\mathfrak{W}_{\infty})\}|] < \infty, \Big. \\[-2mm]
&\hspace{43.7mm} \Big. \EE[|\tr\{\mathfrak{W}_{\infty} \nabla \Sigma \nabla f(\mathfrak{W}_{\infty})\}|] < \infty, ~\EE[|f(\mathfrak{W}_{\infty})|] < \infty, \Big. \\[-2mm]
&\hspace{68.4mm} \Big.\text{and } \EE\big[\|\mathfrak{W}_{\infty}^{1/2}\nabla f(\mathfrak{W}_{\infty})\Sigma^{1/2}\|_F^2\big] < \infty~\Big\}.
\end{aligned}
\end{equation}
\end{corollary}

\begin{remark}\label{rem:Wishart.Stein.operator.d.1}
When $d = 1$, the space $\mathcal{S}_{++}^1$ is naturally identified with $(0, \infty)$ and the scale matrix $\Sigma$ is a positive scalar. Under this identification, $\mathfrak{W}\sim \mathcal{W}_1(\alpha, \Sigma) \equiv \Gamma(r = \alpha/2, \lambda = 1/(2\Sigma))$, where the gamma distribution is parametrized as in \eqref{Steinchargamma}. Moreover, \eqref{eq:Wishart.process.generator} becomes
\[
\mathcal{A}^{\mathcal{W}} f(x)
= 2(\alpha \Sigma - x) f'(x) + 4 \, \Sigma x f''(x)
= 4 \, \Sigma \left\{x f''(x) + (r - \lambda x) f'(x)\right\}.
\]
Thus Corollary~\ref{cor:Stein.Wishart.generator.step.1} reduces, up to the positive multiplicative factor $4 \, \Sigma$ and a slightly more restrictive class of test functions, to the gamma Stein characterization \eqref{Steinchargamma}.
\end{remark}

\begin{remark}\label{rem:Stein.Haff.identity}
Aside from the different class of test functions, the forward implication in Corollary~\ref{cor:Stein.Wishart.generator.step.1} is a special case of the classical Stein--Haff identity for the Wishart distribution \citep[Theorem~2.1]{Haff1979WishartIdentity}. To compare notations, set $p = d$ and $k = \alpha$. If $\mathfrak{W}\sim \mathcal{W}_d(\alpha, \Sigma)$, then, under the regularity, integrability and boundary assumptions in that theorem, for a scalar function $h:\mathcal{S}_{++}^d\to\R$ and a matrix-valued function $F:\mathcal{S}_{++}^d\to\R^{d\times d}$,
\begin{equation}\label{eq:Stein.Haff.identity}
\begin{aligned}
\EE[h(\mathfrak{W})\tr\{F(\mathfrak{W})\Sigma^{-1}\}]
&= 2 \, \EE[h(\mathfrak{W})(\mathrm{div}^*F_{(1/2)})(\mathfrak{W})] + 2 \, \EE\left[\tr\left\{\frac{\partial h}{\partial S}(\mathfrak{W})F_{(1/2)}(\mathfrak{W})\right\}\right] \\
&\qquad + \{\alpha - (d + 1)\} \, \EE[h(\mathfrak{W})\tr\{\mathfrak{W}^{-1}F(\mathfrak{W})\}].
\end{aligned}
\end{equation}
Here, $F_{(1/2)}(S)$ denotes the matrix obtained by multiplying the off-diagonal entries of $F(S)$ by $1/2$, and $\mathrm{div}^*$ denotes the matrix divergence, namely
\[
(\mathrm{div}^*F)(S) \leqdef \sum_{i,j = 1}^d \frac{\partial F_{ij}(S)}{\partial S_{ij}}.
\]
Taking $h = 1$ in \eqref{eq:Stein.Haff.identity} gives
\begin{equation}\label{eq:Stein.Haff.vector.identity}
\EE\left[2 \, (\mathrm{div}^*F_{(1/2)})(\mathfrak{W}) + \{\alpha - (d + 1)\} \, \tr\{\mathfrak{W}^{-1}F(\mathfrak{W})\} - \tr\{F(\mathfrak{W})\Sigma^{-1}\}\right] = 0.
\end{equation}
For a scalar function $f\in C^2(\mathcal{S}_{++}^d)$ for which this specialization is justified, take
\[
F(S) = 2 S \nabla f(S) \Sigma.
\]
Using the symmetric-gradient convention from Proposition~\ref{prop:generator}, so that $\nabla_{ij}S_{k\ell} = (\delta_{ik}\delta_{j\ell} + \delta_{i\ell}\delta_{jk})/2$, and using the symmetries of $\Sigma$ and $\nabla$, one has
\[
\begin{aligned}
4 \, \mathrm{div}^*\{(S \nabla f(S) \Sigma)_{(1/2)}\}
&= 4 \, \sum_{i,j,k = 1}^d \nabla_{ij}\{S_{ik}(\nabla f(S)\Sigma)_{kj}\} \\
&= 4 \, \sum_{i,j,k = 1}^d \frac{\delta_{jk} + \delta_{ij}\delta_{ik}}{2}(\nabla f(S)\Sigma)_{kj} + 4 \, \sum_{i,j,k, \ell = 1}^d S_{ik}\Sigma_{\ell j}\nabla_{ij}\nabla_{k\ell} f(S) \\
&= 2 \, (d + 1) \, \tr\{\Sigma \nabla f(S)\} + 4 \, \tr\{S \nabla \Sigma \nabla f(S)\}.
\end{aligned}
\]
Substituting $F(S) = 2S\nabla f(S)\Sigma$ into \eqref{eq:Stein.Haff.vector.identity} therefore yields
\[
\EE\left[2 \, \tr\{(\alpha \Sigma - \mathfrak{W})\nabla f(\mathfrak{W})\} + 4 \, \tr\{\mathfrak{W}\nabla\Sigma\nabla f(\mathfrak{W})\}\right] = 0,
\]
that is,
\[
\EE\big[\mathcal{A}^{\mathcal{W}}f(\mathfrak{W})\big] = 0.
\]
Thus the forward implication in Corollary~\ref{cor:Stein.Wishart.generator.step.1} is recovered from the Stein--Haff identity by restricting Haff's arbitrary matrix field to the gradient-type field $F(S) = 2S\nabla f(S)\Sigma$. Conversely, the full Stein--Haff identity is not a formal consequence of the forward implication in Corollary~\ref{cor:Stein.Wishart.generator.step.1}, since it applies to arbitrary matrix fields $F(S)$ and scalar multipliers $h(S)$, whereas the generator identity tests only fields generated by scalar potentials $f$.
\end{remark}

The fundamental differential operator matrix appearing in the sequel is
\[
\mathcal{D} \leqdef \nabla_{\!X} \, \mathrm{adj}(I_d - 2\nabla_{\!X}),
\]
where $\nabla_{\!X}$ denotes the symmetric gradient with respect to the matrix variable $X$. This is the transfer operator behind Lemma~\ref{lem:Luk.Wishart.Lemma.2.5}: derivatives in the noncentrality parameter of the Wishart kernel are rewritten as applications of $\mathcal{D}$ to the test function. The seminorms built from $\mathcal{D}$, and later from its time-dependent analogue $\mathcal{D}_{t, \Sigma}$, are therefore the quantities that control the semigroup contraction estimate (Lemma~\ref{lem:contraction}), the absolute convergence of the Stein-solution representation (Theorem~\ref{thm:Stein.solutions.Wishart}), and the derivative bounds for the Stein solution (Theorem~\ref{thm:smoothness.estimates.wishart}).

The next lemma bounds the Wishart semigroup's deviation from its stationary limit by leveraging the test function's Lipschitz and $\mathcal{D}$-regularity. This decay rate is crucial to obtain well-defined solutions to the Wishart Stein equation, as it ensures the absolute convergence of the integral in the semigroup representation presented in Theorem~\ref{thm:Stein.solutions.Wishart} below.

\begin{lemma}\label{lem:contraction}
Let $\alpha > 3d-3$ and $\Sigma\in \mathcal{S}_{++}^d$ be given. Let $(\mathfrak{W}_t)_{t\geq 0}$ be the Wishart process defined in \eqref{eq:Wishart.process}, with transition semigroup $(\mathcal{P}^{\mathcal{W}}_t)_{t\geq 0}$ and stationary limiting distribution $\xi = \mathcal{W}_d(\alpha, \Sigma)$. For any given $t > 0$, let $\Sigma_t = (1 - e^{-2t})\Sigma$ and define
\[
h_t(X) \leqdef h(\Sigma_t^{1/2}X\Sigma_t^{1/2}), \qquad X\in \mathcal{S}_{++}^d,
\]
with $h:\mathcal{S}_{++}^d\to\R$ being bounded and Lipschitz. Assume that $h\in C^d(\mathcal{S}_{++}^d)$ and that for every differential monomial $\overline{\nabla}$ of order at most $d$ in the entries of $\nabla_{\!X}$, there exist constants $C_{\overline{\nabla}}, N_{\overline{\nabla}}\geq 0$ such that
\begin{equation}\label{eq:contraction.asump.h.t}
|\overline{\nabla}h(X)| \leq C_{\overline{\nabla}} (1 + \|X\|_F^{N_{\overline{\nabla}}}), \qquad X\in \mathcal{S}_{++}^d,
\end{equation}
and assume that
\begin{equation}\label{eq:def.M.t.D.h}
M_t^{\mathcal{D}}(h) \leqdef \sup_{X\in \mathcal{S}_{++}^d} \sup_{\substack{U\in \mathcal{S}^d\\ \|U\|_F = 1}}\big|\langle U, \mathcal{D}\rangle_F \, h_t(X)\big| < \infty.
\end{equation}
Then, for every $W\in \mathcal{S}_{++}^d$, we have
\begin{equation}\label{eq:Wishart.Holder.contraction}
\begin{aligned}
\big|(\mathcal{P}^{\mathcal{W}}_t h)(W)-\EE[h(\mathfrak{W}_{\infty})]\big|
&\leq \frac{e^{-2t}}{1 - e^{-2t}} \, M_t^{\mathcal{D}}(h) \, \|\Sigma^{-1/2}W\Sigma^{-1/2}\|_F \\
&\qquad + e^{-2t} [h]_1 \, \|\Sigma\|_2 \, \EE[\|\mathfrak{V}\|_F],
\end{aligned}
\end{equation}
where $\mathfrak{V}\sim \mathcal{W}_d(\alpha, I_d)$.
\end{lemma}

Define the smoothing operator
\begin{equation}\label{eq:Wishart.Q.operator.def}
(\mathcal{Q}_t h)(\Lambda) \leqdef \int_{\mathcal{S}_{++}^d} h(X) \, f_{\alpha, I_d, e^{-2t}\Lambda}^{\mathcal{W}}(X) \, \rd X, \qquad \Lambda\in \mathcal{S}_{+}^d, ~~t\geq 0.
\end{equation}
By Proposition~\ref{prop:extension.Luk.1994.Lemma.2.4}, for every bounded Borel measurable function $h:\mathcal{S}_{++}^d\to\R$ and every $W\in \mathcal{S}_{+}^d$, we have
\[
\mathcal{P}_t^{\mathcal{W}} h(W) = (\mathcal{Q}_t h_t)(\Lambda_t(W)), \qquad t > 0,
\]
where $\Lambda_t(W) \leqdef \Sigma_t^{-1/2}W\Sigma_t^{-1/2}$, $h_t(X) = h(\Sigma_t^{1/2}X\Sigma_t^{1/2})$, and $\Sigma_t = (1 - e^{-2t})\Sigma$. Theorem~\ref{thm:Stein.solutions.Wishart} below provides an explicit semigroup solution of the Wishart Stein equation for bounded Lipschitz test functions whose rescaled family $(h_t)_{t > 0}$ has uniformly bounded first $\mathcal{D}$-seminorm away from $t = 0$.

\begin{theorem}[Solution of the Wishart Stein equation for $C_b^d$ test functions]\label{thm:Stein.solutions.Wishart}
Let $\alpha > 3d-3$ and $\Sigma\in \mathcal{S}_{++}^d$ be given. Let $(\mathfrak{W}_t)_{t\geq 0}$ be the Wishart process in \eqref{eq:Wishart.process} with transition semigroup $(\mathcal{P}^{\mathcal{W}}_t)_{t\geq 0}$, extended generator $\mathcal{A}^{\mathcal{W}}$ given by \eqref{eq:Wishart.process.generator}, and stationary limiting distribution $\xi = \mathcal{W}_d(\alpha, \Sigma)$ from Proposition~\ref{prop:extension.Luk.1994.Lemma.2.4}. Let $h\in C_b^d(\mathcal{S}_{++}^d)$ be real-valued. Since $\mathcal{S}_{++}^d$ is convex, $h$ is bounded and Lipschitz; let $[h]_1$ denote its minimum Lipschitz constant with respect to the Frobenius norm. Set
\[
M_{1, \Sigma}^{\mathcal{D}}(h)
 \leqdef \sup_{t\geq 1} \sup_{X\in \mathcal{S}_{++}^d} \sup_{\substack{U\in \mathcal{S}^d\\ \|U\|_F = 1}} \big|\langle U, \mathcal{D}\rangle_F \, h_t(X)\big|,
\qquad \langle U, \mathcal{D}\rangle_F \leqdef \sum_{i,j = 1}^d U_{ij} \, \mathcal{D}_{ij}.
\]
Then $M_{1, \Sigma}^{\mathcal{D}}(h) < \infty$, and the following hold. The function
\begin{equation}\label{eq:fh.def.Wishart.beta}
f_h(W) \leqdef -\int_0^{\infty}\Big\{\mathcal{P}^{\mathcal{W}}_t h(W)-\EE[h(\mathfrak{W}_{\infty})]\Big\} \, \rd t
\end{equation}
is well defined for every $W\in \mathcal{S}_{++}^d$ and solves the Wishart Stein equation
\begin{equation}\label{eq:Stein.equation.Wishart}
\mathcal{A}^{\mathcal{W}} f_h(W) = h(W) - \EE[h(\mathfrak{W}_{\infty})], \qquad W\in \mathcal{S}_{++}^d,
\end{equation}
where $\mathfrak{W}_{\infty}\sim \mathcal{W}_d(\alpha, \Sigma)$. Moreover, if $\mathfrak{V}\sim \mathcal{W}_d(\alpha, I_d)$, then
\begin{equation}\label{eq:fh.bound.Wishart.beta}
|f_h(W)| \leq 2\|h\|_{\infty} + \frac{1}{2} \left\{\frac{M_{1, \Sigma}^{\mathcal{D}}(h)}{1 - e^{-2}} \, \|\Sigma^{-1/2}W\Sigma^{-1/2}\|_F + [h]_1 \, \|\Sigma\|_2 \, \EE[\|\mathfrak{V}\|_F]\right\}.
\end{equation}
\end{theorem}

\begin{remark}\label{rem:alpha.threshold.Stein.solutions.Wishart}
The condition $\alpha > 3d-3$ is a sufficient condition imposed by the differentiation argument used in the proof of Theorem~\ref{thm:Stein.solutions.Wishart} and explained below, and is not a sharp existence condition for the Wishart process or for its transition law. Indeed, Proposition~\ref{prop:extension.Luk.1994.Lemma.2.4} works already in the larger range $\alpha > d-1$. The stronger lower bound enters through Lemma~\ref{lem:Luk.Wishart.Lemma.2.5}, whose proof transfers derivatives in the noncentrality parameter to $X$-derivatives using Lemma~\ref{lem:Laplace.derivative.Wishart} and the integration-by-parts estimate of Lemma~\ref{lem:Laplace.polynomial.calculus}. It is not required for the existence of the expectation on the right-hand side of \eqref{eq:claim}; as shown in Remark~\ref{rem:eq.claim.expectation.exists}, that expectation is finite under the natural condition $\alpha > d-1$.

A single differentiation with respect to an entry of the noncentrality parameter $\Theta$ gives, at the level of Laplace transforms and up to sign, an entry of the matrix $T(I_d + 2T)^{-1}$. We then write
\[
T(I_d + 2T)^{-1} = \frac{T \, \mathrm{adj}(I_d + 2T)}{|I_d + 2T|}.
\]
This identity has two effects. The denominator $|I_d + 2T|^{-1}$ shifts the Wishart shape parameter from $\alpha$ to $\alpha + 2$, since it changes the determinant factor $|I_d + 2T|^{-\alpha/2}$ into $|I_d + 2T|^{-(\alpha + 2)/2}$. The numerator $T \, \mathrm{adj}(I_d + 2T)$ is a polynomial of degree at most $d$ in the entries of $T$, so Lemma~\ref{lem:Laplace.polynomial.calculus} has to be applied to $f_{\alpha + 2, I_d, \Theta}^{\mathcal{W}}$ with $m = d$. Its condition is therefore $\alpha + 2 > d-1 + 2d$. Equivalently, after dividing by $2$, this is $(\alpha + 2)/2 > (d-1)/2 + d$.

The term $(d-1)/2$ is the usual base integrability threshold in the multivariate gamma integral on $\mathcal{S}_{++}^d$. For a shape parameter $\gamma$, the determinant factor in the Wishart density is $|X|^{\gamma/2-(d + 1)/2}$, and the matrix-gamma integrability condition near $\partial\mathcal{S}_{+}^d$ is $\gamma/2 > (d-1)/2$. In the one-step argument of Lemma~\ref{lem:Laplace.derivative.Wishart}, the relevant shifted shape is $\gamma = \alpha + 2$, not $\alpha$, because of the denominator above. The additional $ + d$ in $(\alpha + 2)/2 > (d-1)/2 + d$ accounts for the worst possible loss caused by the differential operator of order at most $d$ coming from $T \, \mathrm{adj}(I_d + 2T)$. Heuristically, each $X$-derivative may lower by one the exponent of the determinant factor controlling the behavior near $\partial\mathcal{S}_{+}^d$. Thus, after up to $d$ derivatives, the determinant factor behaves as if its exponent were $\gamma/2-(d + 1)/2-d$. Requiring this exponent to remain larger than $-1$ gives $\gamma/2 > (d-1)/2 + d$, and with $\gamma = \alpha + 2$ this is exactly $\alpha > 3d-3$.

Thus $\alpha > 3d-3$ guarantees enough boundary decay to justify the integration by parts and to avoid boundary contributions in the proof. It should be viewed as a regularity assumption for the method used here, rather than as a proved necessary condition for the existence of the expectation in \eqref{eq:claim}, for the existence of the solution of the Stein equation, or for derivative estimates obtained by a different method. It remains an open problem to extend Lemma~\ref{lem:Luk.Wishart.Lemma.2.5}, and the resulting Stein solution estimates, under the natural condition $\alpha > d-1$.
\end{remark}

Next, we derive bounds on the derivatives of the semigroup solution of the Wishart Stein equation. Define the following time-dependent differential operator matrix:
\[
\mathcal{D}_{t, \Sigma} \leqdef \nabla_{\!W} \, \mathrm{adj}(I_d - 2 \Sigma_t \nabla_{\!W}), \qquad \Sigma_t = (1 - e^{-2t})\Sigma.
\]
For any integer $p\geq 0$, let $\mathcal{B}_{p}^{\mathcal{D}, \Sigma}(\mathcal{S}_{++}^d)$ denote the class of functions $h:\mathcal{S}_{++}^d\to\R$ such that, for every choice of index pairs $(i_1, j_1), \ldots, (i_p, j_p)\in [d]^2$, the iterated derivatives $\smash{\left(\prod_{\ell=1}^p (\mathcal{D}_{t, \Sigma})_{i_{\ell} j_{\ell}}\right) h}$ exist for every $t > 0$ and satisfy
\[
\sup_{t > 0} \left\|\left(\prod_{\ell=1}^p (\mathcal{D}_{t, \Sigma})_{i_{\ell} j_{\ell}}\right) h\right\|_{\infty} < \infty,
\]
with the convention that the order-$0$ iterate is $h$ itself.

\begin{theorem}[Regularity of the semigroup solution of the Wishart Stein equation]\label{thm:smoothness.estimates.wishart}
Let $m\in \N$, let $\alpha > 3d-3$ and $\Sigma\in \mathcal{S}_{++}^d$. Let $h\in C_b^{md}(\mathcal{S}_{++}^d)$ be real-valued, and let $f_h$ denote the semigroup solution of the Wishart Stein equation defined in \eqref{eq:fh.def.Wishart.beta}. Then the following hold:
\begin{itemize}
\item[(i)] If $h\in \mathcal{B}^{\mathcal{D}, \Sigma}_{m}(\mathcal{S}_{++}^d)$, then for any index pairs $(i_1,j_1), \ldots,(i_m,j_m)\in [d]^2$,
\begin{equation}\label{eq:Wishart.Stein.bound.1}
\left\|\left(\prod_{\ell=1}^m \nabla_{\!W, i_{\ell} j_{\ell}}\right) f_h\right\|_{\infty}
\leq \frac{1}{2m} \sup_{t > 0} \left\|\left(\prod_{\ell=1}^m (\mathcal{D}_{t, \Sigma})_{i_{\ell} j_{\ell}}\right) h\right\|_{\infty}.
\end{equation}

\item[(ii)] For any sufficiently smooth $g:\mathcal{S}_{++}^d\to\R$, define
\begin{equation}\label{eq:M.m}
\mathcal{M}_m(g)
 \leqdef \sup_{W\in \mathcal{S}_{++}^d} \, \sup_{\substack{U_1, \ldots,U_m\in \mathcal{S}^d\\ \|U_1\|_F = \cdots = \|U_m\|_F = 1}} \big|D^m g(W)[U_1, \ldots,U_m]\big|.
\end{equation}
Define also
\begin{equation}\label{eq:M.m.D.Sigma}
\mathcal{M}_m^{\mathcal{D}, \Sigma}(h)
 \leqdef \sup_{t > 0} \, \sup_{W\in \mathcal{S}_{++}^d} \, \sup_{\substack{U_1, \ldots,U_m\in \mathcal{S}^d\\ \|U_1\|_F = \cdots = \|U_m\|_F = 1}}
\left|\Big(\prod_{\ell=1}^m \langle U_{\ell}, \mathcal{D}_{t, \Sigma}\rangle_F\Big) h(W)\right|,
\end{equation}
where
\[
\langle U, \mathcal{D}_{t, \Sigma}\rangle_F \leqdef \sum_{i,j = 1}^d U_{ij} \, (\mathcal{D}_{t, \Sigma})_{ij}.
\]
If $\mathcal{M}_m^{\mathcal{D}, \Sigma}(h) < \infty$, then
\begin{equation}\label{eq:Wishart.Stein.bound.4}
\mathcal{M}_m(f_h) \leq \frac{1}{2m} \, \mathcal{M}_m^{\mathcal{D}, \Sigma}(h).
\end{equation}
\end{itemize}
\end{theorem}

\begin{remark}\label{rem:explicit.h.norms.Wishart.Stein}
The seminorms of $h$ appearing in \eqref{eq:Wishart.Stein.bound.1} and \eqref{eq:M.m.D.Sigma} can be bounded explicitly in terms of ordinary derivatives of $h$, at the cost of rather large constants depending on $d$ and $\Sigma$. For $r\in \N$, set
\[
\mathcal{N}_r(h) \leqdef \max_{1\leq q\leq r}\mathcal{M}_q(h), \qquad
A_{d, \Sigma} \leqdef d!\big(1 + 2\sqrt{d}\|\Sigma\|_2\big)^{d-1}, \qquad
B_{d, \Sigma} \leqdef d A_{d, \Sigma} \, ,
\]
where $\mathcal{M}_q$ is defined as in \eqref{eq:M.m}. Then, for any index pairs $(i_{\ell},j_{\ell})\in [d]^2$,
\begin{equation}\label{eq:explicit.bound.indexed.D.h}
\sup_{t > 0}\left\|\left(\prod_{\ell=1}^m(\mathcal{D}_{t, \Sigma})_{i_{\ell}j_{\ell}}\right)h\right\|_{\infty} \leq A_{d, \Sigma}^m \, \mathcal{N}_{md}(h),
\end{equation}
and
\begin{equation}\label{eq:explicit.bound.directional.D.h}
\mathcal{M}_m^{\mathcal{D}, \Sigma}(h) \leq B_{d, \Sigma}^m \, \mathcal{N}_{md}(h).
\end{equation}
Consequently, \eqref{eq:Wishart.Stein.bound.1} gives
\[
\left\|\left(\prod_{\ell=1}^m \nabla_{\!W,i_{\ell}j_{\ell}}\right)f_h\right\|_{\infty} \leq \frac{A_{d, \Sigma}^m}{2m} \, \mathcal{N}_{md}(h),
\]
and \eqref{eq:Wishart.Stein.bound.4} gives
\[
\mathcal{M}_m(f_h) \leq \frac{B_{d, \Sigma}^m}{2m} \, \mathcal{N}_{md}(h).
\]

Indeed, the entries of $I_d-2\Sigma_t\nabla_{\!W}$ have coefficient $\ell^1$-norm at most $1 + 2\sqrt{d}\|\Sigma\|_2$, uniformly in $t > 0$, because $\|\Sigma_t\|_2\leq\|\Sigma\|_2$. Each entry of $\mathrm{adj}(I_d-2\Sigma_t\nabla_{\!W})$ is a signed sum of $(d-1)!$ products of $d-1$ such entries. Multiplication by the outer matrix $\nabla_{\!W}$ and summation over one index give the factor $d$, and therefore each entry $(\mathcal{D}_{t, \Sigma})_{ij}$ is a constant-coefficient differential operator of order at most $d$ whose total coefficient size is at most $A_{d, \Sigma}$. Iterating $m$ such operators gives \eqref{eq:explicit.bound.indexed.D.h}. For the directional seminorm, if $\|U\|_F = 1$, then $\sum_{i,j = 1}^d |U_{ij}|\leq d$, so each directional operator $\langle U, \mathcal{D}_{t, \Sigma}\rangle_F$ has total coefficient size at most $B_{d, \Sigma}$, which gives \eqref{eq:explicit.bound.directional.D.h}.

The constants in \eqref{eq:explicit.bound.indexed.D.h} and \eqref{eq:explicit.bound.directional.D.h} are included to make the dependence on $d$ and $\Sigma$ explicit and to remove the supremum over $t$ from the test-function seminorms. The sharper seminorm $\mathcal{M}_m^{\mathcal{D}, \Sigma}(h)$ is kept in the main statement because it tracks the actual differential operator generated by the Wishart kernel and can be substantially smaller than the right-hand side of \eqref{eq:explicit.bound.directional.D.h}.
\end{remark}

\begin{remark}\label{rem:Wishart.derivative.loss.md}
The appearance of derivatives of $h$ up to order $md$ in Theorem~\ref{thm:smoothness.estimates.wishart} comes from the way derivatives of the Wishart transition semigroup are transferred from the initial state to the integration variable in the noncentral Wishart kernel. By Proposition~\ref{prop:extension.Luk.1994.Lemma.2.4}, the initial state enters $(\mathcal{P}^{\mathcal{W}}_t h)(W)$ through the noncentrality parameter $e^{-2t}\Lambda_t(W)$, where $\Lambda_t(W) = \Sigma_t^{-1/2}W\Sigma_t^{-1/2}$. Since $W\mapsto \Lambda_t(W)$ is linear, each derivative with respect to $W$ gives one derivative with respect to the noncentrality parameter. Lemma~\ref{lem:Luk.Wishart.Lemma.2.5} then transfers such a derivative to an $X$-differential operator:
\begin{equation}\label{eq:Wishart.one.step.derivative.loss}
\nabla_{\Lambda,ij}(\mathcal{Q}_t^{(\gamma)}g)(\Lambda) = e^{-2t} \, \mathcal{Q}_t^{(\gamma + 2)}(\mathcal{D}_{ij}g)(\Lambda), \qquad \mathcal{D} = \nabla_{\!X} \, \mathrm{adj}(I_d - 2\nabla_{\!X}),
\end{equation}
where $\mathcal{Q}_t^{(\gamma)}$ denotes the operator in \eqref{eq:Wishart.Q.operator.def} with shape parameter $\gamma$ in place of $\alpha$. The operator $\mathrm{adj}(I_d - 2\nabla_{\!X})$ is a polynomial of degree at most $d-1$ in the entries of $\nabla_{\!X}$, and the additional factor $\nabla_{\!X}$ makes $\mathcal{D}$ an operator of order at most $d$. Thus one derivative in the initial state may cost up to $d$ derivatives of the test function. Iterating \eqref{eq:Wishart.one.step.derivative.loss} $m$ times gives a product of $m$ such operators. The resulting operator has order at most $md$, and the proof therefore requires enough regularity of $h$ to control all the derivatives that can arise from this product. This is the reason for the hypotheses involving $C_b^{md}(\mathcal{S}_{++}^d)$, $\mathcal{B}^{\mathcal{D}, \Sigma}_{m}(\mathcal{S}_{++}^d)$, and $\mathcal{M}_m^{\mathcal{D}, \Sigma}(h)$.

This derivative loss is not merely an artifact of estimating too many terms separately. The product of $m$ transfer operators obtained by iterating \eqref{eq:Wishart.one.step.derivative.loss} has genuine order $md$ in general: although there are cancellations in some entries of $\mathcal{D}$, these cancellations do not remove the highest-order part of the operator in general. For instance, in the scale-normalized case the diagonal entries of $\mathcal{D}$ retain order $d$, and iterating such entries gives operators of order $md$. Thus the $md$ derivative count is sharp for the transfer identity and for the $\mathcal{D}$-seminorm bounds used in Theorem~\ref{thm:smoothness.estimates.wishart}. This does not by itself prove that $md$ derivatives of $h$ are necessary for the mere existence of the $m$-th derivative of the solution of the Stein equation, since a different argument might avoid transferring all derivatives onto $h$ through \eqref{eq:Wishart.one.step.derivative.loss}.
\end{remark}

\begin{remark}\label{rem:gamma.smoothness.Wishart.Stein}
When $d = 1$, the adjugate in the definition of $\mathcal{D}_{t, \Sigma}$ is the adjugate of a $1\times 1$ matrix and hence is equal to $1$. Since the symmetric gradient is then the ordinary derivative, one has $\mathcal{D}_{t, \Sigma} = \, \rd/\rd x, ~ t > 0$. Thus \eqref{eq:Wishart.Stein.bound.1} becomes
\[
\|f_h^{(m)}\|_{\infty} \leq \frac{1}{2m}\|h^{(m)}\|_{\infty}, \qquad m\geq1.
\]
As in Remark~\ref{rem:Wishart.Stein.operator.d.1}, $\mathcal{W}_1(\alpha, \Sigma) \equiv \Gamma(r = \alpha/2, \lambda = 1/(2\Sigma))$ and $\mathcal{A}^{\mathcal{W}} = 4\Sigma \, \mathcal{A}^{\Gamma}$. Thus, if $f_h^{\Gamma} \leqdef 4\Sigma f_h$ denotes the corresponding solution of the gamma Stein equation, then
\[
\|(f_h^{\Gamma})^{(m)}\|_{\infty} \leq \frac{2\Sigma}{m}\|h^{(m)}\|_{\infty} = \frac{1}{m\lambda}\|h^{(m)}\|_{\infty}, \qquad m\geq1,
\]
which is \eqref{lukbound}.
\end{remark}

\section{Applications}\label{sec:applications}

\subsection{Quantitative Wishart approximation of uncentered group-mean scatter matrices}\label{sec:Wishart.approximation.MANOVA}

The Wishart distribution naturally arises in multivariate analysis of variance (MANOVA) as the exact finite-sample distribution of the uncentered group-mean scatter matrix when the underlying populations are Gaussian. Indeed, consider a balanced design with $\nu$ independent groups, where each group $k \in [\nu]$ consists of $n$ independent observations $\bb{Z}_{k,i} \in \R^d$ drawn from a normal distribution $\mathcal{N}_d(\bb{0}_d, \Sigma)$ with $\nu\in \N \cap [d, \infty)$ and $\Sigma\in \mathcal{S}_{++}^d$. If we define the standardized group means as $\smash{\bb{X}_k^{(n)} \leqdef n^{-1/2} \sum_{i=1}^n \bb{Z}_{k,i}}$, the matrix composed of these vectors,
\[
\mathfrak{X}^{(n)} \leqdef [\bb{X}_1^{(n)}, \ldots, \bb{X}_{\nu}^{(n)}]^{\top},
\]
follows a matrix normal distribution $\mathcal{N}_{\nu\times d}(0_{\nu\times d}, I_{\nu} \otimes \Sigma)$. Its density with respect to the Lebesgue measure is given, for all $X \in \R^{\nu\times d}$, by $\smash{\phi_{0_{\nu\times d}, I_{\nu}, \Sigma}(X) = (2\pi)^{-\nu d/2} |\Sigma|^{-\nu/2} \etr\{-\frac{1}{2} X \Sigma^{-1} X^{\top}\}}$; see, e.g., \citet[Theorem~2.2.1]{GuptaNagar2000}. The uncentered group-mean scatter matrix $(\mathfrak{X}^{(n)})^{\top} \mathfrak{X}^{(n)}$ then has a Wishart distribution, namely
\[
(\mathfrak{X}^{(n)})^{\top}\mathfrak{X}^{(n)}\sim \mathcal{W}_d(\nu, \Sigma);
\]
see, e.g., \citet[Theorem~10.3.2]{Muirhead1982}.

When the underlying raw observations $\bb{Z}_{k,i}$ have mean zero and covariance $\Sigma$ but are not exactly Gaussian, the multivariate central limit theorem implies that the standardized group means $\smash{\bb{X}_k^{(n)}}$ converge to a Gaussian distribution as the group size $n \to\infty$. By the continuous mapping theorem, the scatter matrix converges in law to a Wishart distribution for any fixed number of groups $\nu$. Instead of applying the Wishart Stein equation directly, we quantify this convergence by applying Stein's method for matrix normal approximation \citep{GauntOuimetRichards2026} to the matrix of standardized group means and then by pushing the result forward through the quadratic map $X\mapsto X^{\top}X$.

The connection between the two Stein operators is explicit and is isolated in the following lemma. In the mean-zero case considered here, the matrix normal Stein operator of \citet{GauntOuimetRichards2026}, specialized to row covariance $\Psi = I_{\nu}$ and column covariance $\Sigma$, is
\begin{equation}\label{eq:MN.OU.generator.mean.zero.application}
\mathcal{A}^{\mathrm{OU}}g(X) = -\tr\{X^{\top}\nabla_{\!X} g(X)\} + \tr\{\Sigma \nabla_{\!X}^{\top}\nabla_{\!X} g(X)\}, \qquad X\in \R^{\nu\times d},
\end{equation}
where $\nabla_{\!X} = (\nabla_{\!X,ar})_{a\in [\nu],r\in [d]}$ denotes the ordinary rectangular gradient, with $\nabla_{\!X,ar} = \partial/\partial X_{ar}$.

\begin{lemma}[Wishart pullback of the matrix normal Stein operator]\label{lem:matrix.normal.Wishart.operator.connection}
Let $\nu\in [d, \infty)\cap \N$, $\Sigma\in \mathcal{S}_{++}^d$, and let $f:\mathcal{S}_{++}^d\to\R$ be smooth. Set
\[
\mathcal{R}_{\nu,d} \leqdef \{X\in \R^{\nu\times d}: \mathrm{rank}(X)=d\}.
\]
The set $\mathcal{R}_{\nu,d}$ is open, and its complement has Lebesgue measure zero. For $X\in \mathcal{R}_{\nu,d}$, define $g_f(X) = f(X^{\top}X)$. Then, pointwise on $\mathcal{R}_{\nu,d}$,
\begin{equation}\label{eq:MN.Wishart.pullback.identity}
\left(\mathcal{A}^{\mathrm{OU}}g_f\right)(X) = \left(\mathcal{A}^{\mathcal{W}}_{\nu, \Sigma} \, f\right)(X^{\top}X), \qquad X\in \mathcal{R}_{\nu,d},
\end{equation}
where $\nabla f$ denotes the symmetric gradient in the variable $S$ and
\[
\left(\mathcal{A}^{\mathcal{W}}_{\nu, \Sigma} \, f\right)(S) = 2 \, \tr\{(\nu\Sigma - S)\nabla f(S)\} + 4 \, \tr\{S \nabla \Sigma \nabla f(S)\}, \qquad S\in \mathcal{S}_{++}^d.
\]
\end{lemma}

For the proof of Proposition~\ref{prop:Wishart.MANOVA} below, we use a standard positive-definite regularization of the quadratic pullback. This regularization is not meant to extend the unregularized Wishart generator to the boundary of $\mathcal{S}_{++}^d$; rather, it gives a rectangular test function whose derivatives are defined on all of $\R^{\nu\times d}$, which is needed for the Taylor expansions along line segments in $\R^{\nu\times d}$. With $f$ as in Lemma~\ref{lem:matrix.normal.Wishart.operator.connection}, for $\varepsilon>0$, set
\[
g_{f, \varepsilon}(X) \leqdef f(X^{\top}X+\varepsilon I_d), \qquad S_{\varepsilon}(X) \leqdef X^{\top}X+\varepsilon I_d, \qquad X\in \R^{\nu\times d}.
\]
Then the regularized identity corresponding to \eqref{eq:MN.Wishart.pullback.identity} is, for every $X\in \R^{\nu\times d}$,
\begin{equation}\label{eq:MN.Wishart.regularized.identity}
\begin{aligned}
\left(\mathcal{A}^{\mathrm{OU}}g_{f, \varepsilon}\right)(X)
&= 2 \, \tr\{(\nu\Sigma - X^{\top}X)\nabla f(S_{\varepsilon}(X))\} + 4 \, \tr\{X^{\top}X \nabla \Sigma \nabla f(S_{\varepsilon}(X))\} \\
&= \left(\mathcal{A}^{\mathcal{W}}_{\nu, \Sigma} \, f\right)(S_{\varepsilon}(X)) + 2\varepsilon \, \tr\{\nabla f(S_{\varepsilon}(X))\} - 4\varepsilon \, \tr\{I_d \nabla \Sigma \nabla f(S_{\varepsilon}(X))\}.
\end{aligned}
\end{equation}

Since $\mathcal{S}_{++}^d$ is convex, every $h\in C_b^1(\mathcal{S}_{++}^d)$ is Lipschitz on $\mathcal{S}_{++}^d$ and admits a unique bounded continuous extension to $\smash{\mathcal{S}_{+}^d}$. In the statement of the proposition below, $\smash{h(\mathfrak{W}_{\nu}^{(n)})}$ is understood through this extension whenever $\smash{\mathfrak{W}_{\nu}^{(n)}}$ is possibly singular. We make use of the connection \eqref{eq:MN.Wishart.regularized.identity} to generalize Theorem~3.1 of \citet{gpr17} to the Wishart setting.

\begin{proposition}[Quantitative Wishart approximation for MANOVA]\label{prop:Wishart.MANOVA}
Let $\nu\in [d, \infty)\cap \N$ satisfy $\nu > 3d-3$, and $\Sigma\in \mathcal{S}_{++}^d$. Let $\mathfrak{W}\sim \mathcal{W}_d(\nu, \Sigma)$, and let $\mathfrak{Z}_1, \ldots, \mathfrak{Z}_n$ be independent and identically distributed (iid) random matrices in $\R^{\nu\times d}$ satisfying $\EE[\mathfrak{Z}_1] = 0_{\nu\times d}$, $\EE[\vecc(\mathfrak{Z}_1^{\top})\vecc(\mathfrak{Z}_1^{\top})^{\top}] = I_{\nu}\otimes \Sigma$, and assume that $\EE[\|\mathfrak{Z}_1\|_F^8] < \infty$. For $n\in \N$, set
\[
\mathfrak{X}^{(n)} \leqdef \frac{1}{\sqrt{n}}\sum_{i=1}^n \mathfrak{Z}_i, \qquad \mathfrak{W}_{\nu}^{(n)} \leqdef (\mathfrak{X}^{(n)})^{\top}\mathfrak{X}^{(n)}.
\]
Let $h\in C_b^{5d}(\mathcal{S}_{++}^d)$ be real-valued and assume that $\smash{\sum_{m=1}^5 \mathcal{M}_m^{\mathcal{D}, \Sigma}(h) < \infty}$, where $\mathcal{M}_m^{\mathcal{D}, \Sigma}(h)$ is defined in \eqref{eq:M.m.D.Sigma}. Define
\[
\alpha_j \leqdef r_j + s_j\sum_{(a,r),(b,s),(c,t)\in [\nu]\times [d]}|\EE[(\mathfrak{Z}_1)_{ar} (\mathfrak{Z}_1)_{bs} (\mathfrak{Z}_1)_{ct}]|, \qquad j\in\{2,3,4,5\},
\]
where
\[
(r_2,r_3,r_4,r_5) \leqdef (2, 54, 1547,0), \qquad (s_2,s_3,s_4,s_5) \leqdef (2, 85, 6891, 629\,799).
\]
Then, for all $n\geq2$,
\[
\begin{aligned}
|\EE[h(\mathfrak{W}_{\nu}^{(n)})] - \EE[h(\mathfrak{W})]|
&\leq \frac{1}{n} \times (\nu d)^8(1 + K) \left(1\vee \max_{(a,r)\in [\nu]\times [d]}\EE[|(\mathfrak{Z}_1)_{ar}|^8]\right) \\
&\qquad\times \left\{\alpha_2\mathcal{M}_2^{\mathcal{D}, \Sigma}(h)+\alpha_3\mathcal{M}_3^{\mathcal{D}, \Sigma}(h)+\alpha_4\mathcal{M}_4^{\mathcal{D}, \Sigma}(h)+\alpha_5\mathcal{M}_5^{\mathcal{D}, \Sigma}(h)\right\},
\end{aligned}
\]
where $K \leqdef 8\pi \|\Sigma^{-1}\|_{\infty}\left\{1+\|\Sigma\|_2^3 \nu d(\nu d+2)(\nu d+4)\right\}$ and $\|\Sigma^{-1}\|_{\infty}\leqdef \max_{r\in[d]}\sum_{s=1}^d |(\Sigma^{-1})_{rs}|$.
\end{proposition}

\subsection{Quantitative multivariate Satterthwaite approximation}\label{sec:Wishart.approximation.Satterthwaite}

We now turn to the multivariate Satterthwaite approximation, in which the distribution of a positive linear combination of independent Wishart random matrices is approximated by a single Wishart distribution with matching first two moments. Such linear combinations of Wishart random matrices arise in many problems in multivariate statistics: MANOVA when the usual homoscedastic, independent-error covariance structure is relaxed but can be expressed as a Kronecker product \citep{TanGupta1983,NaikRao2001,Mortarino2005}, matrix quadratic forms \citep{Khatri1966,SingullKoski2012}, and robustness studies involving multivariate normal mixture distributions \citep{Tan1980}. This is the multivariate analog of the classical Satterthwaite approximation introduced by \citet{Satterthwaite1941,Satterthwaite1946}, in which the distribution of a linear combination of chi-square random variables is approximated by a single gamma distribution with matching first two moments, providing a tractable sidestep to an otherwise difficult to handle distribution.

\citet{Khatri1971} developed exact series representations, whose truncations yield approximations, for the density function of positive definite quadratic forms in normal vectors. In the central case, these quadratic forms are equivalent to positive linear combinations of independent rank-one Wishart random matrices with a common scale matrix and, by grouping equal coefficients, to sums of independent Wishart random matrices with integer shape parameters and proportional scale matrices. While such truncated series approximations appear to be more accurate than a Satterthwaite approximation, they do not directly apply to the heterogeneous setting considered below, where $\mathfrak{G}_j\sim \mathcal{W}_d(\alpha_j, \Sigma_j)$ may have arbitrary integer shape parameters and non-proportional scale matrices. Furthermore, theoretically assessing the closeness between the true distribution and the truncated series representation remains an open question. On the other hand, Stein's method is tailored to treating this problem in the case of approximation by a single Wishart distribution.

Let $N\in \N$, $\alpha_1, \ldots, \alpha_N\in \N$, and let $\mathfrak{G}_1, \ldots, \mathfrak{G}_N$ be independent random matrices such that $\mathfrak{G}_j\sim \mathcal{W}_d(\alpha_j, \Sigma_j)$ for some fixed $\Sigma_1, \ldots, \Sigma_N\in \mathcal{S}_{++}^d$. Here, for integer $\alpha_j$, the law $\mathcal{W}_d(\alpha_j, \Sigma_j)$ is understood through the standard Gaussian representation, possibly singular when $\alpha_j<d$. Define
\[
\mathfrak{T} \leqdef \sum_{j=1}^N \mathfrak{G}_j,
\]
such that
\[
\EE[\mathfrak{T}] = \sum_{j=1}^N \alpha_j\Sigma_j \equiv \overline{\Sigma}, \qquad \Var\{\vecp(\mathfrak{T})\} = 2\sum_{j=1}^N \alpha_j\mathcal{V}(\Sigma_j),
\]
where we use the notation
\[
\mathcal{V}(W) \leqdef B_d^{\top} (W\otimes W) B_d,
\]
and $B_d$ denotes the $d^2 \times d(d + 1)/2$ transition matrix such that $\mathrm{vecp}(X) = B_d^{\top} \mathrm{vec}(X), ~X\in \mathcal{S}^d$; see, e.g., \citet[Definition~2.5.1~and~Eq.~(1.2.11)]{GuptaNagar2000}. Introduce the Wishart random matrix $\smash{\mathfrak{W}_{\mathrm{Sat}}\sim \mathcal{W}_d(\nu, \widetilde{\Sigma})}$ matching the first two moments of $\mathfrak{T}$:
\begin{align}
\EE[\mathfrak{W}_{\mathrm{Sat}}] &= \nu\widetilde{\Sigma} = \overline{\Sigma}, \notag \\
\Var\{\vecp(\mathfrak{W}_{\mathrm{Sat}})\} &= \frac{2}{\nu}\mathcal{V}(\overline{\Sigma}) = 2\sum_{j=1}^N \alpha_j\mathcal{V}(\Sigma_j). \label{eq:moment-matching-2}
\end{align}
Matching the first moment yields $\widetilde{\Sigma} = \nu^{-1}\overline{\Sigma}$, and it remains to find $\nu$ by matching the second moment. However, in general, Eq.~\eqref{eq:moment-matching-2} is an overdetermined matrix equation for the single scalar $\nu$ and need not admit an exact solution. An exact scalar solution exists in the proportional-scale case $\Sigma_j = \lambda_j\Sigma$ with $\lambda_j > 0$ for all $j\in [N]$; then $\nu$ is uniquely determined by
\[
m_1 \leqdef \sum_{j=1}^N \lambda_j\alpha_j, \qquad m_2 \leqdef \sum_{j=1}^N \lambda_j^2\alpha_j, \qquad \nu = \frac{m_1^2}{m_2}, \qquad \widetilde{\Sigma} = \frac{m_2}{m_1} \Sigma.
\]

On the other hand, when the condition $\Sigma_j = \lambda_j\Sigma$ does not hold, alternative methods are required to circumvent solving~\eqref{eq:moment-matching-2} for $\nu$. Using the Stein equation associated with the approximating law $\smash{\mathcal{W}_d(\nu, \widetilde{\Sigma}_{\nu})}$, where $\smash{\widetilde{\Sigma}_{\nu} = \nu^{-1} \overline{\Sigma}}$ and $\nu$ is the proposed effective degrees of freedom, we can motivate a particular choice of $\nu$ by comparing the upper bounds on the discrepancy in~\eqref{eq:heterogeneous.Satterthwaite.bound.nu}. Some of these methods are presented and numerically compared in Section~\ref{sec:supp.Satterthwaite-details} of the \hyperref[supp]{Supplementary material}.

\begin{proposition}[Quantitative multivariate Satterthwaite approximation]\label{prop:Wishart.Satterthwaite}
In the setting described above, assume that we have $\sum_{j=1}^N \alpha_j \geq d + 1$. Also, fix $\nu > 3d-3$ and set $\smash{\mathfrak{W}_{\nu} \sim \mathcal{W}_d(\nu, \widetilde{\Sigma}_{\nu})}$ with $\smash{\widetilde{\Sigma}_{\nu} \leqdef \nu^{-1}\overline{\Sigma}}$, so that, in particular, $\smash{\EE[\mathfrak{W}_{\nu}] = \overline{\Sigma} = \EE[\mathfrak{T}]}$. Let $h\in C_b^{2d}(\mathcal{S}_{++}^d)$ and assume that $\smash{\mathcal{M}_1^{\mathcal{D}, \widetilde{\Sigma}_{\nu}}(h) + \mathcal{M}_2^{\mathcal{D}, \widetilde{\Sigma}_{\nu}}(h) < \infty}$ in Theorem~\ref{thm:smoothness.estimates.wishart}. Then,
\begin{equation}\label{eq:heterogeneous.Satterthwaite.bound.nu}
\big|\EE[h(\mathfrak{T})]-\EE[h(\mathfrak{W}_{\nu})]\big| \leq \frac{d + 1}{2}\mathcal{M}_2^{\mathcal{D}, \widetilde{\Sigma}_{\nu}}(h)\sum_{j=1}^N \alpha_j\tr(\Sigma_j)\|\Sigma_j-\widetilde{\Sigma}_{\nu}\|_F.
\end{equation}
\end{proposition}

\begin{remark}
While the proof of this result relies on Gaussian integration by parts for each $\mathfrak{G}_j$, the results from Section~\ref{sec:main.results} are essential in the approximation step to derive the upper bound. Indeed, in general, the selected $\nu$ is not an integer.
\end{remark}

\begin{remark}\label{rem:formal.convergence-lower.order.distance}
Since $\widetilde{\Sigma}_{\nu} = \nu^{-1}\overline{\Sigma}\to 0_{d\times d}$ as $\nu\to\infty$, the operator $\mathcal{D}_{t, \widetilde{\Sigma}_{\nu}} = \nabla_{\!W}\mathrm{adj}(I_d-2\widetilde{\Sigma}_{\nu,t}\nabla_{\!W})$ converges formally to the ordinary symmetric gradient $\nabla_{\!W}$. Consequently, although the bound in Proposition~\ref{prop:Wishart.Satterthwaite} is stated for test functions in $C_b^{2d}(\mathcal{S}_{++}^d)$ through the seminorm $\smash{\mathcal{M}_2^{\mathcal{D}, \widetilde{\Sigma}_{\nu}}(h)}$, in the regime $\nu\to\infty$ this seminorm becomes asymptotically equivalent to the usual second-order seminorm $\mathcal{M}_2(h)$. Thus the bound asymptotically behaves like a control in a smooth distance based on bounded second-order derivatives.
\end{remark}

\subsection{De Bruijn identities and logarithmic Sobolev inequalities for the Wishart measure}\label{sec:Wishart.deBruijn.LSI}

We begin this section by obtaining a local De Bruijn identity for the Wishart measure, which we state in Proposition~\ref{prop:Wishart.deBruijn}. Our result provides a natural generalization of the local De Bruijn identity for the gamma measure of Theorem~1 of \citet{ArrasSwan2017} that was proved under minimal conditions. Previously, in the context of the Gamma calculus of \citet{BakryEmery1985} and \citet{BakryGentilLedoux2014}, a local De Bruijn identity had been known to to hold under a density assumption; see Proposition~5.2.2 of \citet{bakry96}.

Fix $\alpha\in (d-1, \infty)$ and $\Sigma\in \mathcal{S}_{++}^d$, and write $\xi = \mathcal{W}_d(\alpha, \Sigma)$ for the stationary law from Proposition~\ref{prop:extension.Luk.1994.Lemma.2.4}. Let $\mu$ be a probability measure on $\mathcal{S}_{++}^d$, which we suppose is absolutely continuous with respect to the Wishart measure $\xi$ with Radon--Nikodym density $g\in C_b^2(\mathcal{S}_{++}^d)$. That is,
\[
\mu(\rd X) = g(X) \, \xi(\rd X).
\]
For $t\geq 0$, define the evolved law and its density with respect to $\xi$ by
\[
\mu_t \leqdef \mu \mathcal{P}_t^{\mathcal{W}}, \qquad g_t \leqdef \frac{\rd \mu_t}{\rd \xi}.
\]
Here the measure $\mu_t$ is the Wishart generalisation of the law of the random variable $X_{\tau}$ given in Definition 1 of \citet{ArrasSwan2017}. It should, however, be noted that Theorem~1 of \citet{ArrasSwan2017} is stated without a density assumption. We also define the relative entropy and the Wishart Fisher information by
\[
\begin{aligned}
\mathrm{Ent}(\mu_t \, \| \, \xi) & \leqdef \int_{\mathcal{S}_{++}^d} g_t(X)\log(g_t(X)) \, \xi(\rd X), \\
J_{\mathcal{W}, \Sigma}(\mu_t \mid \xi) & \leqdef \int_{\mathcal{S}_{++}^d} \Gamma_{\mathcal{W}, \Sigma}(\log(g_t))(X) \, \mu_t(\rd X),
\end{aligned}
\]
\citep[see, e.g.,][Eq.~(5.1.5) and (5.1.6)]{BakryGentilLedoux2014} where the Wishart {\it carr\'e du champ} operator $\Gamma_{\mathcal{W}, \Sigma}$ is formally defined, for every $X\in \mathcal{S}_{++}^d$, by
\[
\begin{aligned}
\Gamma_{\mathcal{W}, \Sigma}(\varphi)(X) & \leqdef \Gamma_{\mathcal{W}, \Sigma}(\varphi, \varphi)(X), \\
\Gamma_{\mathcal{W}, \Sigma}(\varphi, \psi)(X) & \leqdef \frac{1}{2} \left\{\mathcal{A}^{\mathcal{W}}(\varphi\psi)(X) - \varphi(X)\mathcal{A}^{\mathcal{W}}\psi(X) - \psi(X)\mathcal{A}^{\mathcal{W}}\varphi(X)\right\}.
\end{aligned}
\]

Let $\varphi, \psi\in C_b^2(\mathcal{S}_{++}^d)$. We now use the extended generator identity \eqref{eq:Wishart.process.generator} to obtain a useful representation of the Wishart carr\'e du champ operator that we will employ in our proof. From the extended generator identity \eqref{eq:Wishart.process.generator} and the Leibniz rule, the drift term for $f = \varphi \psi$ is
\[
2 \, \tr\{(\alpha \Sigma - X)\nabla(\varphi\psi)(X)\} = 2 \, \varphi(X) \, \tr\{(\alpha \Sigma - X)\nabla \psi(X)\} + 2 \, \psi(X) \, \tr\{(\alpha \Sigma - X)\nabla \varphi(X)\},
\]
and the diffusion term is
\[
\begin{aligned}
4 \, \tr\{X \nabla \Sigma \nabla(\varphi\psi)(X)\}
&= 4 \, \varphi(X) \, \tr\{X \nabla \Sigma \nabla \psi(X)\} + 4 \, \psi(X) \, \tr\{X \nabla \Sigma \nabla \varphi(X)\} \\
&\qquad + 8 \, \tr\{X \nabla \varphi(X)\Sigma \nabla \psi(X)\}.
\end{aligned}
\]
Here the two cross terms in the second-order Leibniz rule are equal by the symmetry of $\nabla\varphi(X)$ and $\nabla\psi(X)$. Consequently,
\begin{equation}\label{eq:Wishart.product.rule}
\mathcal{A}^{\mathcal{W}}(\varphi\psi)(X) = \varphi(X)\mathcal{A}^{\mathcal{W}}\psi(X) + \psi(X)\mathcal{A}^{\mathcal{W}}\varphi(X) + 8 \, \tr\{X \nabla \varphi(X)\Sigma \nabla \psi(X)\},
\end{equation}
and thus,
\[
\Gamma_{\mathcal{W}, \Sigma}(\varphi, \psi)(X) = 4 \, \tr\{X \nabla \varphi(X)\Sigma \nabla \psi(X)\}.
\]
In particular, for all $X\in \mathcal{S}_{++}^d$, $\Gamma_{\mathcal{W}, \Sigma}(\varphi)(X) = 4 \, \|X^{1/2}\nabla\varphi(X)\Sigma^{1/2}\|_F^2 \geq 0$, and
\[
J_{\mathcal{W}, \Sigma}(\mu_t \mid \xi) = \int_{\mathcal{S}_{++}^d} 4 \, \tr\{X \nabla \log(g_t)(X)\Sigma \nabla \log(g_t)(X)\} \, \mu_t(\rd X) \geq 0.
\]

\begin{proposition}[Local De Bruijn identity for the Wishart measure]\label{prop:Wishart.deBruijn}
Assume that there exist constants $m_g,M_g\in (0, \infty)$ such that $m_g \leq g(X)\leq M_g$ for all $X\in \mathcal{S}_{++}^d$. Then, for every $t > 0$, we have $g_t = \mathcal{P}_t^{\mathcal{W}} g$ and
\begin{equation}\label{eq:Wishart.deBruijn}
\frac{\rd}{\rd t}\mathrm{Ent}(\mu_t \, \| \, \xi) = - J_{\mathcal{W}, \Sigma}(\mu_t \mid \xi).
\end{equation}
In particular, the map $t\mapsto \mathrm{Ent}(\mu_t \, \| \, \xi)$ is nonincreasing. Equivalently, after the change of variable $\tau = e^{-2t}$ and the definition $\mu_{\tau} \leqdef \mu \mathcal{P}^{\mathcal{W}}_{-(1/2)\log(\tau)}$, we have
\begin{equation}\label{eq:Wishart.deBruijn.tau}
\frac{\rd}{\rd \tau}\mathrm{Ent}(\mu_{\tau} \, \| \, \xi) = \frac{1}{2\tau} J_{\mathcal{W}, \Sigma}(\mu_{\tau} \mid \xi), \qquad \tau\in (0, 1).
\end{equation}
\end{proposition}

\begin{remark}[On the assumptions in Proposition~\ref{prop:Wishart.deBruijn}]\label{rem:Wishart.deBruijn.assumptions}
The boundedness assumption in Proposition~\ref{prop:Wishart.deBruijn} is a simple sufficient condition for the two analytic justifications needed in the proof. More precisely, the proof only requires the following two assumptions.
\begin{itemize}\setlength\itemsep{0em}
\item[\textbf{(A)}]\phantomsection\label{ass:Wishart.deBruijn.differentiation} For every $t > 0$, all integrals below are finite and differentiation under the integral sign in the entropy functional is valid, namely
\[
\frac{\rd}{\rd t}\mathrm{Ent}(\mu_t \, \| \, \xi) = \int_{\mathcal{S}_{++}^d} (1 + \log(g_t(X))) \, \mathcal{A}^{\mathcal{W}}g_t(X) \, \xi(\rd X).
\]
\item[\textbf{(B)}]\phantomsection\label{ass:Wishart.deBruijn.integration.by.parts} For every $t > 0$, the functions $g_t$ and $\log(g_t)$ satisfy the two integration-by-parts identities:
\[
\int_{\mathcal{S}_{++}^d} \mathcal{A}^{\mathcal{W}}g_t(X) \, \xi(\rd X) = 0
\]
and
\[
\int_{\mathcal{S}_{++}^d} \log(g_t(X)) \, \mathcal{A}^{\mathcal{W}}g_t(X) \, \xi(\rd X) = -4 \int_{\mathcal{S}_{++}^d} \tr\{X \nabla g_t(X)\Sigma \nabla \log(g_t)(X)\} \, \xi(\rd X).
\]
\end{itemize}
A simple sufficient condition for Assumptions~\hyperref[ass:Wishart.deBruijn.differentiation]{\textup{(A)}} and~\hyperref[ass:Wishart.deBruijn.integration.by.parts]{\textup{(B)}} is precisely the boundedness condition used in Proposition~\ref{prop:Wishart.deBruijn}, namely that there exist constants $m_g,M_g\in (0, \infty)$ such that $m_g \leq g(X)\leq M_g$ for all $X\in \mathcal{S}_{++}^d$. Indeed, since $\mathcal{P}_t^{\mathcal{W}}$ is Markov, this condition gives $m_g \leq g_t(X)\leq M_g$ for all $X\in \mathcal{S}_{++}^d$ and $t\geq 0$, so $\log(g_t)$ is bounded. Moreover, the regularity of the Wishart semigroup on $C_b^2(\mathcal{S}_{++}^d)$ gives bounded first and second derivatives of $g_t$ on compact time intervals contained in $(0, \infty)$. Since the coefficients of $\mathcal{A}^{\mathcal{W}}$ grow at most linearly and $\xi$ has finite first moment, dominated convergence justifies Assumption~\hyperref[ass:Wishart.deBruijn.differentiation]{\textup{(A)}}, while Assumption~\hyperref[ass:Wishart.deBruijn.integration.by.parts]{\textup{(B)}} follows from the invariant Wishart integration-by-parts formula. Thus Proposition~\ref{prop:Wishart.deBruijn} remains valid under Assumptions~\hyperref[ass:Wishart.deBruijn.differentiation]{\textup{(A)}} and~\hyperref[ass:Wishart.deBruijn.integration.by.parts]{\textup{(B)}} alone. The same analytic conditions are also the ones needed for the forthcoming Proposition~\ref{prop:Wishart.integrated.deBruijn}, since the proof of that result integrates \eqref{eq:Wishart.deBruijn} from Proposition~\ref{prop:Wishart.deBruijn}.
\end{remark}

The next proposition considers the integrated counterpart of the local identity. The first part of the proposition generalizes the integrated De Bruijn identity for the gamma measure given by Theorem~18 of \citet{ArrasSwan2017}. The second part, which turns this relative entropy formula into an entropy jump statement when the initial law matches the stationary Wishart law in mean and log-determinant, generalizes the differential entropy gap identity for the gamma case that was treated in Remark~19 of \citet{ArrasSwan2017}, which also required a matching mean and logarithmic moment condition.

\begin{proposition}[Integrated De Bruijn identity for the Wishart measure and entropy jump]\label{prop:Wishart.integrated.deBruijn}
Under the assumptions of Proposition~\ref{prop:Wishart.deBruijn}, we have
\begin{equation}\label{eq:Wishart.integrated.deBruijn}
\mathrm{Ent}(\mu \, \| \, \xi)
= \int_0^{\infty} J_{\mathcal{W}, \Sigma}(\mu_t \mid \xi) \, \rd t
= \int_0^1 \frac{1}{2\tau} J_{\mathcal{W}, \Sigma}(\mu_{\tau} \mid \xi) \, \rd \tau.
\end{equation}
Assume furthermore that $\mu$ is the law of a random matrix $\mathfrak{X}$ and that this law has Lebesgue density $f_{\mathfrak{X}}$, that is, $\mu(\rd X) = f_{\mathfrak{X}}(X) \, \rd X$. Equivalently, if $f_{\xi}$ denotes the Lebesgue density of $\xi$, then the initial relative density is $g(X) = f_{\mathfrak{X}}(X)/f_{\xi}(X)$. Assume also that $\mathfrak{X}$ has finite differential entropy
\[
H(\mathfrak{X}) \leqdef -\int_{\mathcal{S}_{++}^d} f_{\mathfrak{X}}(X)\log(f_{\mathfrak{X}}(X)) \, \rd X.
\]
If $\EE[\mathfrak{X}] = \EE[\mathfrak{W}_{\infty}] = \alpha \Sigma$ and $\EE[\log(|\mathfrak{X}|)] = \EE[\log(|\mathfrak{W}_{\infty}|)]\in \R$ with $\mathfrak{W}_{\infty}\sim \xi$, then
\begin{equation}\label{eq:Wishart.entropy.jump}
H(\mathfrak{W}_{\infty}) - H(\mathfrak{X})
= \int_0^{\infty} J_{\mathcal{W}, \Sigma}(\mu_t \mid \xi) \, \rd t
= \int_0^1 \frac{1}{2\tau} J_{\mathcal{W}, \Sigma}(\mu_{\tau} \mid \xi) \, \rd \tau.
\end{equation}
\end{proposition}

We now derive a logarithmic Sobolev inequality for the Wishart measure with $\alpha\geq d$, thereby generalizing the logarithmic Sobolev inequality for the gamma case; see \citet{bakry96} and Proposition~2 of \citet{ArrasSwan2017} for an alternative proof. The gamma logarithmic Sobolev inequality is stated for $\alpha\geq1$ (if one uses the same parametrization for the gamma distribution as we do for the Wishart distribution), and so our condition $\alpha\geq d$ reduces exactly to this condition in the univariate gamma setting.

For a probability measure $\mu\ll \xi$ with density $g = \rd\mu/\rd\xi$, we define the Wishart Fisher information in the weak sense by
\[
J_{\mathcal{W}, \Sigma}(\mu \mid \xi) \leqdef 16 \int_{\mathcal{S}_{++}^d} \tr\{X \nabla \sqrt{g}(X)\Sigma \nabla \sqrt{g}(X)\} \, \xi(\rd X),
\]
whenever $\sqrt{g}$ belongs to the corresponding weighted Sobolev domain, and we set $J_{\mathcal{W}, \Sigma}(\mu \mid \xi) = +\infty$ otherwise. For smooth positive $g$, this is equivalent to
\[
J_{\mathcal{W}, \Sigma}(\mu \mid \xi) = \int_{\mathcal{S}_{++}^d} 4 \, \tr\{X \nabla \log(g)(X)\Sigma \nabla \log(g)(X)\} \, \mu(\rd X).
\]

\begin{proposition}[Logarithmic Sobolev inequality for the Wishart measure]\label{prop:Wishart.LSI}
Let $\alpha\in [d, \infty)$, $\Sigma\in \mathcal{S}_{++}^d$ and $\xi = \mathcal{W}_d(\alpha, \Sigma)$. Let $\mu\ll \xi$ with density $g = \rd\mu/\rd\xi$. If $J_{\mathcal{W}, \Sigma}(\mu \mid \xi) < \infty$, then
\begin{equation}\label{eq:Wishart.LSI}
\mathrm{Ent}(\mu \, \| \, \xi) \leq \frac{1}{2} J_{\mathcal{W}, \Sigma}(\mu \mid \xi).
\end{equation}
\end{proposition}

In the gamma case, Theorem~3 of \citet{ArrasSwan2017} provides an HSI inequality (that connects entropy H, Stein discrepancy S and Fisher information I), which improves on the classical logarithmic Sobolev inequality (their Proposition~2), and provides a gamma analogue of the Gaussian HSI inequality of \citet{lnp15} which improved upon the classical Gaussian logarithmic Sobolev inequality of \citet{Gross1975}. Given that we have established natural Wishart generalizations of the local De Bruijn identity and logarithmic Sobolev inequality for the gamma case, it is natural to ask for a Wishart HSI inequality, which we leave as an open problem.

\begin{open}
Generalize the HSI inequality of \citet[Theorem~3]{ArrasSwan2017} to our Wishart setting.
\end{open}

\subsection{Estimation of the parameters of a Wishart distribution}\label{sec:Wishart.parameter.estimation}

Let $\alpha\in (d-1, \infty)$ and $\smash{\Sigma\in \mathcal{S}_{++}^d}$. Given a random sample $\mathfrak{W}^{(1)}, \ldots, \mathfrak{W}^{(n)} \stackrel{\mathrm{iid}}{\sim} \mathcal{W}_d(\alpha, \Sigma)$ and any function $g:\mathcal{S}_{++}^d\to \R$, define the averaging operator
\[
\overline{g(\mathfrak{W})} = \frac{1}{n} \sum_{k=1}^n g(\mathfrak{W}^{(k)}).
\]
We apply a method of moments (MOM) to the Stein characterization found in Corollary~\ref{cor:Stein.Wishart.generator.step.1} by replacing the expectation with the empirical average.

Suppose first that $\alpha$ is known. This suggests estimating $\Sigma$ by imposing, for a chosen collection of test functions $f\in C_{\mathcal{A}^{\mathcal{W}}}^2(\mathcal{S}_{++}^d)$, the empirical Stein equations
\begin{equation}\label{eq:Wishart.scalar.identity}
\overline{\mathcal{A}^{\mathcal{W}} f(\mathfrak{W})}
= 2 \, \overline{\tr\{(\alpha \widehat{\Sigma} - \mathfrak{W}) \nabla f(\mathfrak{W})\}} + 4 \, \overline{\tr\{\mathfrak{W} \nabla \widehat{\Sigma} \nabla f(\mathfrak{W})\}} = 0.
\end{equation}
For symmetric $U\in \mathcal{S}^d$, define the linear probe
\begin{equation}\label{eq:Wishart.linear.probe}
f_U(S) = \tr(U S), \qquad S\in \mathcal{S}_{++}^d.
\end{equation}

\begin{proposition}[Known shape parameter]\label{prop:Wishart.known.alpha.estimator}
The Stein's method-of-moments estimator of $\Sigma$ based on the linear probes in \eqref{eq:Wishart.linear.probe} is
\[
\widehat{\Sigma} = \frac{1}{\alpha n} \sum_{k=1}^n \mathfrak{W}^{(k)}.
\]
\end{proposition}

\begin{remark}
The estimator $\widehat{\Sigma}$ coincides with the classical method-of-moments estimator obtained from the first moment identity $\EE[\mathfrak{W}^{(1)}] = \alpha \Sigma$. Nonlinear probes would produce other Stein's method-of-moments estimators, but they are not expected to improve on $\smash{\widehat{\Sigma}}$ when $\alpha$ is known.
\end{remark}

Suppose now that $\alpha$ is unknown. Then linear probes only identify the product $M \leqdef \alpha \Sigma$. Indeed, rewriting \eqref{eq:Wishart.process.generator} in terms of $(M, \alpha)$ gives
\[
\mathcal{A}^{\mathcal{W}}_{M, \alpha} f(S) = 2 \, \tr\{(M - S) \nabla f(S)\} + \frac{4}{\alpha} \, \tr\{S \nabla M \nabla f(S)\}, \qquad S\in \mathcal{S}_{++}^d.
\]
Accordingly, one seeks a pair $(\widehat{M}, \widehat{\alpha})$ satisfying, for suitable test functions $f$,
\begin{equation}\label{eq:Wishart.scalar.identity.unknown.alpha}
\overline{\mathcal{A}^{\mathcal{W}}_{\widehat{M}, \widehat{\alpha}} f(\mathfrak{W})} = 2 \, \overline{\tr\{(\widehat{M} - \mathfrak{W}) \nabla f(\mathfrak{W})\}} + \frac{4}{\widehat{\alpha}} \, \overline{\tr\{\mathfrak{W} \nabla \widehat{M} \nabla f(\mathfrak{W})\}} = 0.
\end{equation}

Following Example~2.4 of \citet{EbnerFischerGauntPickerSwan2025}, which treats both the classical moment choice and a logarithmic Stein's method-of-moments choice for the gamma distribution, we consider the following matrix analogs. Since the Wishart Stein operator is written in extended-generator form, these scalar gamma test functions are implemented through scalar potentials whose gradients are proportional to $S$ and equal to $\log(S)$, respectively:
\begin{equation}\label{eq:Wishart.quad.log.probe}
q(S) = \tr(S^2), \qquad \ell(S) = \tr(S \log(S) - S), \qquad S\in \mathcal{S}_{++}^d,
\end{equation}
where $\log(S)$ denotes the symmetric matrix logarithm. For $x,y > 0$, let
\[
\log^{[1]}(x,y) \leqdef
\begin{cases}
\displaystyle \frac{\log(x) - \log(y)}{x-y}, & x\neq y, \\
1/x, & x = y.
\end{cases}
\]
If $S = H\Lambda H^{\top}$, with $H\in O(d)$ and $\Lambda = \mathrm{diag}(\lambda_1, \ldots, \lambda_d)$, let $M^{(S)} \leqdef H^{\top} M H$ denote the representation of $M$ in the eigenbasis of $S$, and define
\[
\mathcal{J}_S(M) \leqdef \frac{1}{2}\tr(M) + \frac{1}{2}\sum_{i,j = 1}^d \lambda_i M^{(S)}_{jj} \log^{[1]}(\lambda_j, \lambda_i).
\]

\begin{proposition}[Unknown shape parameter]\label{prop:Wishart.unknown.alpha.estimator}
Let $n\geq 2$ and $\widehat{M} = \frac{1}{n} \sum_{k=1}^n \mathfrak{W}^{(k)}$. The Stein's method-of-moments estimators based on the linear probes in \eqref{eq:Wishart.linear.probe} together with $q$ and $\ell$ in \eqref{eq:Wishart.quad.log.probe}, respectively, are given by
\begin{equation}\label{eq:Wishart.unknown.alpha.estimator.quad}
\widehat{\alpha}_{\mathrm{quad}} = \frac{\{\tr(\widehat{M})\}^2 + \tr(\widehat{M}^2)}{\overline{\tr(\mathfrak{W}^{ \, 2})} - \tr(\widehat{M}^2)}, \qquad
\widehat{\Sigma}_{\mathrm{quad}} = \frac{\widehat{M}}{\widehat{\alpha}_{\mathrm{quad}}},
\end{equation}
\begin{equation}\label{eq:Wishart.unknown.alpha.estimator.log}
\widehat{\alpha}_{\mathrm{log}} = \frac{2 \, \overline{\mathcal{J}_{\mathfrak{W}}(\widehat{M})}}{\overline{\tr\{\mathfrak{W}\log(\mathfrak{W})\}} - \tr\{\widehat{M} \, \overline{\log(\mathfrak{W})}\}}, \qquad
\widehat{\Sigma}_{\mathrm{log}} = \frac{\widehat{M}}{\widehat{\alpha}_{\mathrm{log}}},
\end{equation}
provided that the denominators of $\widehat{\alpha}_{\mathrm{quad}}$ and $\widehat{\alpha}_{\mathrm{log}}$ are nonzero.
\end{proposition}

\begin{remark}
The quadratic estimator $\smash{(\widehat{\alpha}_{\mathrm{quad}}, \widehat{\Sigma}_{\mathrm{quad}})}$ coincides with the classical method-of-moments estimator obtained by matching the first matrix moment $\EE[\mathfrak{W}^{(1)}] = \alpha\Sigma$ and the scalar second moment
\[
\EE[\tr\{(\mathfrak{W}^{(1)})^2\}] = \alpha(\alpha + 1)\tr(\Sigma^2) + \alpha \tr(\Sigma)^2 = \tr(M^2) + \frac{1}{\alpha} \, \big[\{\tr(M)\}^2 + \tr(M^2)\big].
\]
The logarithmic estimator $\smash{(\widehat{\alpha}_{\mathrm{log}}, \widehat{\Sigma}_{\mathrm{log}})}$ is included as a matrix analog of the logarithmic Stein's method-of-moments choice considered for the gamma distribution in \citet{EbnerFischerGauntPickerSwan2025}.
\end{remark}

\begin{remark}
The estimators $\smash{(\widehat{\alpha}_{\mathrm{quad}}, \widehat{\Sigma}_{\mathrm{quad}})}$ and $\smash{(\widehat{\alpha}_{\mathrm{log}}, \widehat{\Sigma}_{\mathrm{log}})}$ are unconstrained moment estimators. When their denominators are positive, the corresponding estimates of $\alpha$ are positive and the estimates of $\Sigma$ belong to $\mathcal{S}_{++}^d$, but $\widehat{\alpha}_{\mathrm{quad}}$ and $\widehat{\alpha}_{\mathrm{log}}$ need not be larger than $d-1$ for every finite sample. If an estimator taking values in the Wishart parameter space is required, one may impose the constraint $\alpha>d-1$ by truncation or by constrained moment estimation.
\end{remark}

\begin{remark}\label{rem:numerical.illustration}
A numerical comparison of $\smash{(\widehat{\alpha}_{\mathrm{log}}, \widehat{\Sigma}_{\mathrm{log}})}$, $\smash{(\widehat{\alpha}_{\mathrm{quad}}, \widehat{\Sigma}_{\mathrm{quad}})}$ and $\smash{(\widehat{\alpha}_{\mathrm{MLE}}, \widehat{\Sigma}_{\mathrm{MLE}})}$ is summarized in Table~\ref{tab:wishart-smom-mle-winner-counts}. A detailed description of the experiment, together with the full median and interquartile range values for the relative Frobenius error $\smash{\|\widehat{\Sigma}-\Sigma_0\|_F/\|\Sigma_0\|_F}$ of each estimator with respect to a target $\Sigma_0$, is relegated to Section~\ref{sec:supp.log.MOM.vs.quad.MOM.vs.MLE} of the \hyperref[supp]{Supplementary material}. The maximum likelihood estimator (MLE) attains the smallest median error most often, the logarithmic Stein's method-of-moments estimator is the second strongest competitor, and the quadratic, classical moment estimator is least often the best. The same qualitative ranking is observed for the interquartile ranges. Thus, while the MLE is the strongest competitor overall, the logarithmic Stein estimator often comes reasonably close and improves on the quadratic moment estimator. It also has the advantage of the closed-form expression in \eqref{eq:Wishart.unknown.alpha.estimator.log}, which makes it simpler to compute and more directly amenable to theoretical analysis than the MLE.
\end{remark}

\begingroup
\footnotesize
\input{wishart_smom_mle_winner_count_table.tex}
\endgroup

Returning to the case where $\alpha$ is known, by choosing $M$ linear probes
\[
f_{U_1}(S) = \tr(U_1 S), \ldots, f_{U_M}(S) = \tr(U_M S), \qquad S\in \mathcal{S}_{++}^d,
\]
and feeding them into \eqref{eq:Wishart.scalar.identity}, we obtain a family of Stein-type moment equations that produce multiple unbiased estimators of linear functionals of $\Sigma$. When we restrict $\Sigma$ to a linear structured subspace, these equations can be projected by least squares onto that subspace, as detailed next.

\begin{proposition}[Projection onto a structured subspace]\label{prop:projected.Stein.Wishart}
Suppose that the true scale matrix $\Sigma\in \mathcal{S}_{++}^d$ admits the linear representation
\[
\Sigma = \Sigma(\bb{\beta}^{\star}) = \sum_{j=1}^p \beta_j^{\star} B_j
\]
for some $\bb{\beta}^{\star}\in \R^p$ and fixed symmetric templates $\smash{\{B_j\}_{j=1}^p \subseteq \mathcal{S}^d}$. Let $\{U_m\}_{m = 1}^{M} \subseteq \mathcal{S}^{d}$ be given, and for the corresponding linear probes $f_{U_m}(S) \leqdef \tr(U_m S)$, form the structured Stein moment equations
\begin{equation}\label{eq:structured.Wishart.moments}
\tr\{\Sigma(\bb{\beta}) U_m\} = y_{n,m}, \qquad y_{n,m} \leqdef \frac{1}{\alpha n} \sum_{k=1}^n \tr(\mathfrak{W}^{(k)} U_m), \qquad m\in \{1, \ldots,M\}.
\end{equation}
Let $C\in \R^{M\times p}$ be the design matrix with entries $C_{mj} \leqdef \tr(B_j U_m)$, and let $\bb{y}_n \leqdef (y_{n,1}, \ldots,y_{n,M})^{\top}$. If $C$ has full column rank, then the least-squares solution
\[
\widehat{\bb{\beta}}_n = \arg\min_{\bb{\beta}\in \R^p} \|C \bb{\beta} - \bb{y}_n\|_2^2 = (C^{\top} C)^{-1} C^{\top} \bb{y}_n
\]
is strongly consistent for the true coefficient vector $\bb{\beta}^{\star}$, as $n\to\infty$. In particular, the structured estimator $\widehat{\Sigma}_n \leqdef \Sigma(\widehat{\bb{\beta}}_n) = \sum_{j=1}^p \hat{\beta}_{n,j} B_j$ converges almost surely to $\Sigma$ and solves the projected Stein moment equations \eqref{eq:structured.Wishart.moments} in the least-squares sense.
\end{proposition}

\begin{remark}
The estimator $\widehat{\Sigma}_n$ in Proposition~\ref{prop:projected.Stein.Wishart} is the unconstrained least-squares projection onto the linear span of the templates $\{B_j\}_{j=1}^p$. It need not belong to $\mathcal{S}_{++}^d$ for every finite sample unless positive definiteness is imposed separately. Since $\mathcal{S}_{++}^d$ is open and $\smash{\widehat{\Sigma}_n\to\Sigma\in\mathcal{S}_{++}^d}$ almost surely, it is nevertheless positive definite eventually almost surely.
\end{remark}

\begin{remark}\label{rem:numerical.illustration:prop:projected.Stein.Wishart}
A numerical comparison of the projected Stein's method-of-moments estimator $\smash{\widehat{\Sigma}_n}$ and the naive (unstructured) Stein's method-of-moments estimator $\smash{\widetilde{\Sigma}_n}$ is summarized in Table~\ref{tab:wishart-structured-smom-winner-counts}. A detailed description of the experiment, together with the full median and interquartile range values for the relative Frobenius error of $\widehat{\Sigma}$ and the corresponding error ratios, is relegated to Section~\ref{sec:supp.projected.vs.naive} of the \hyperref[supp]{Supplementary material}. For the randomly generated compound-symmetry configurations considered there, the projected Stein estimator yields the smaller median relative Frobenius error for every structured configuration and every sample size. The naive estimator more often yields the smaller interquartile range, but the median-error comparison clearly supports the gain from using the structural information in the Stein moment equations. The naive estimator is the natural unstructured benchmark. The unstructured MLE is not displayed separately, since $\alpha_0$ is known and therefore coincides with $\smash{\alpha_0^{-1}\widehat{M}} = \widetilde{\Sigma}_n$.
\end{remark}

\begingroup
\footnotesize
\input{wishart_structured_smom_winner_count_table.tex}
\endgroup

\section{Preliminary lemmas}\label{sec:tech.lemmas}

This section shows how to differentiate the scale-normalized version of the Wishart transition semigroup ($\mathcal{Q}_t$ in \eqref{eq:Wishart.Q.operator.def}) with respect to the entries of the matrix $\Lambda = (\Lambda_{ij})$ appearing in the noncentrality parameter of the noncentral Wishart kernel $f_{\alpha, I_d, e^{-2t}\Lambda}^{\mathcal{W}}$. The derivatives of $(\mathcal{Q}_t h)(\Lambda)$ with respect to $\Lambda_{ij}$ can be ``pushed through'' the transition kernel and transferred onto $h$ using an appropriate differential operator; see Lemma~\ref{lem:Luk.Wishart.Lemma.2.5}. This lemma is key in proving the existence of the semigroup solution of the Wishart Stein equation in Theorem~\ref{thm:Stein.solutions.Wishart} and its regularity in Theorem~\ref{thm:smoothness.estimates.wishart}. Before we prove Lemma~\ref{lem:Luk.Wishart.Lemma.2.5}, two preliminary lemmas are needed; see Lemmas~\ref{lem:Laplace.polynomial.calculus} and \ref{lem:Laplace.derivative.Wishart} below.

Lemma~\ref{lem:Laplace.polynomial.calculus} provides an expression for Laplace transforms of products of partial derivatives of the noncentral Wishart density (recall that the Laplace transform $\mathcal{L}[g]$ of a sufficiently integrable function $g:\mathcal{S}_{++}^d\to \R$ is given by \eqref{eq:Laplace.transform}). Lemma~\ref{lem:Laplace.derivative.Wishart} shows how partial derivatives of the noncentral Wishart density with respect to the noncentrality parameter transfer to partial derivatives with respect to the state variable.

\begin{lemma}\label{lem:Laplace.polynomial.calculus}
Let $m\in \N$, $\alpha > d-1 + 2m$, and $\Theta\in \mathcal{S}_{+}^d$. Then, for all $T\in \mathcal{S}_{+}^d$ and all index pairs $(i_{\ell}, j_{\ell})\in [d]^2$,
\begin{equation}\label{eq:repeated.IBP}
\mathcal{L}\left[\left(\prod_{\ell=1}^m \nabla_{i_{\ell} j_{\ell}}\right) f_{\alpha, I_d, \Theta}^{\mathcal{W}}\right](T)
= \left(\prod_{\ell=1}^m T_{i_{\ell} j_{\ell}}\right)\mathcal{L}[f_{\alpha, I_d, \Theta}^{\mathcal{W}}](T).
\end{equation}
More generally, if $P(U)$ is a scalar polynomial of degree at most $m$ in the entries of a symmetric $d\times d$ matrix $U$, and $P(\nabla)$ denotes the differential operator obtained by replacing $U_{ij}$ by $\nabla_{ij}$, then
\begin{equation}\label{eq:poly.claim}
\mathcal{L}[P(\nabla) f_{\alpha, I_d, \Theta}^{\mathcal{W}}](T) = P(T) \, \mathcal{L}[f_{\alpha, I_d, \Theta}^{\mathcal{W}}](T).
\end{equation}
\end{lemma}

\begin{proof}[Proof of Lemma~\ref{lem:Laplace.polynomial.calculus}]
Let $T\in \mathcal{S}_{+}^d$ be given. Throughout the proof, $\nabla \equiv \nabla_{\!X}$. For any $i,j\in [d]$, note that
\[
\nabla_{ij} \, \etr(-TX) = -T_{ij} \, \etr(-TX), \qquad T,X\in \mathcal{S}^d.
\]
By repeated integration by parts for which all boundary terms vanish, we have
\begin{equation}\label{eq:IBP}
\int_{\mathcal{S}_{++}^d} \etr(-TX) \left(\prod_{\ell=1}^m \nabla_{i_{\ell} j_{\ell}}\right) f_{\alpha, I_d, \Theta}^{\mathcal{W}}(X) \, \rd X
= \left(\prod_{\ell=1}^m T_{i_{\ell} j_{\ell}}\right) \int_{\mathcal{S}_{++}^d} \etr(-TX) \, f_{\alpha, I_d, \Theta}^{\mathcal{W}}(X) \, \rd X,
\end{equation}
which proves \eqref{eq:repeated.IBP}. Since $\alpha > d - 1 + 2m \geq d - 1 + 2r$ for any $r \leq m$, this identity holds for any product of $r$ derivatives. The claim \eqref{eq:poly.claim} follows by linearity and the fact that any polynomial of degree at most $m$ is a finite linear combination of monomials of degree at most $m$.

It remains to justify \eqref{eq:IBP} rigorously. First, the heuristic is the following. Set
\[
a \leqdef \frac{\alpha}{2}, \qquad \beta \leqdef \frac{\alpha-d-1}{2}, \qquad B_{\Theta}(X) \leqdef {}_0F_1\left(a;\frac{\Theta^{1/2}X\Theta^{1/2}}{4}\right).
\]
The density $f_{\alpha, I_d, \Theta}^{\mathcal{W}}$ can be written as
\begin{equation}\label{eq:expression.density.NCW}
f_{\alpha, I_d, \Theta}^{\mathcal{W}}(X) = c_{\alpha, \Theta} \, |X|^{\beta} \, \etr(-X/2) \, B_{\Theta}(X), \qquad X\in \mathcal{S}_{++}^d,
\end{equation}
for an appropriate constant $c_{\alpha, \Theta} > 0$. Each differentiation of $f_{\alpha, I_d, \Theta}^{\mathcal{W}}(X)$ can lower the power of $|X|$ by at most one, so since $\beta > m-1$, derivatives of order at most $m-1$ should vanish on $\partial\mathcal{S}_{+}^d$, while derivatives of order $m$ should still be locally integrable there. The only additional issue is the noncentral factor $B_{\Theta}(X)$: one must show that its derivatives do not grow too quickly at infinity. We prove below that every differential monomial $\overline{\nabla} = \prod_{\ell=1}^r \nabla_{i_{\ell} j_{\ell}}$ of order at most $m$ (i.e., $r \leq m$) in the entries of $\nabla$ satisfies, for appropriate constants $c_{\Theta},C_{\overline{\nabla}, \Theta} > 0$,
\begin{equation}\label{eq:exponential.growth.derivatives.Bessel}
|\overline{\nabla}B_{\Theta}(X)| \leq C_{\overline{\nabla}, \Theta} \exp\{c_{\Theta}\sqrt{1 + \tr(X)}\}, \qquad X\in \mathcal{S}_{+}^d,
\end{equation}
and this stretched-exponential growth is dominated by the factor $\etr(-X/2)$. It is enough to control $B_{\Theta}(Z)$ for complex symmetric matrices $Z$.

For a partition $\bb{\kappa} = (\kappa_1, \ldots, \kappa_d)\vdash k$ (meaning $\kappa_1 \geq \dots \geq \kappa_d \geq 0$ and $\kappa_1 + \dots + \kappa_d = k$), set $\kappa_{d + 1} \leqdef 0$. By James' integral formula for zonal polynomials \citep[see, e.g.,][Eq.~35.4.3]{Richards2010}, we have for real symmetric $Y$,
\[
C_{\bb{\kappa}}(Y) = C_{\bb{\kappa}}(I_d) \int_{O(d)} \prod_{j=1}^d |(HYH^{\top})_j|^{\kappa_j-\kappa_{j + 1}} \, \rd H,
\]
where $\rd H$ is the normalized Haar probability measure on $O(d)$, and $A_j$ denotes the $j\times j$ top-left corner of $A$. Both sides are polynomials in the entries of $Y$, so the identity extends to complex symmetric $Z$ by polynomial identity. Therefore, for any complex symmetric $Z$,
\[
|C_{\bb{\kappa}}(Z)|
\leq C_{\bb{\kappa}}(I_d) \prod_{j=1}^d \|Z\|_2^{ \, j(\kappa_j-\kappa_{j + 1})}
= C_{\bb{\kappa}}(I_d)\|Z\|_2^k.
\]

Next, since $\alpha > d-1 + 2m$ and $m\geq 1$, we have $a > (d + 1)/2$, hence
\[
(a)_{\bb{\kappa}} = \prod_{j=1}^d \left(a-\frac{j-1}{2}\right)_{\kappa_j} \geq \prod_{j=1}^d \kappa_j! = \frac{k!}{\binom{k}{\kappa_1, \ldots, \kappa_d}} \geq \frac{k!}{\sum_{x_1 + \ldots + x_d = k} \binom{k}{x_1, \ldots,x_d}} = \frac{k!}{d^k}.
\]
Combining the last two inequalities yields
\[
\left|\frac{1}{(a)_{\bb{\kappa}} \, k!} \, C_{\bb{\kappa}}\left(\frac{\Theta^{1/2}Z\Theta^{1/2}}{4}\right)\right|
\leq \frac{d^k C_{\bb{\kappa}}(I_d)}{(k!)^2} \left(\frac{\|\Theta\|_2 \, \|Z\|_2}{4}\right)^k
= \frac{d^k C_{\bb{\kappa}}(I_d)}{(k!)^2} \left(\frac{\sqrt{\|\Theta\|_2 \, \|Z\|_2}}{2}\right)^{2k}.
\]
Therefore, summing over $k\in \N_0$ and $\bb{\kappa}\vdash k$, and using $\sum_{\bb{\kappa}\vdash k} C_{\bb{\kappa}}(I_d) = \{\tr(I_d)\}^k = d^k$, we get
\begin{equation}\label{eq:Bessel.bound}
|B_{\Theta}(Z)|
\leq \sum_{k = 0}^{\infty} \frac{1}{(k!)^2}\left(\frac{d \sqrt{\|\Theta\|_2 \, \|Z\|_2}}{2}\right)^{2k}
\leq \exp\left\{d \sqrt{\|\Theta\|_2 \, \|Z\|_2}\right\}
\leq \exp\left\{c_{\Theta}\sqrt{\|Z\|_F}\right\},
\end{equation}
for some constant $c_{\Theta}\in (0, \infty)$ that depends only on $d$ and $\Theta$. In particular, the Bessel function $B_{\Theta}$ is entire on $\mathcal{S}^d \oplus \ii \mathcal{S}^d$ (the complexification of $\mathcal{S}^d$); cf.\ \citet[p.~486]{Herz1955}.

Let $N \leqdef d(d + 1)/2$, choose the standard orthonormal basis of $\mathcal{S}^d$ for the Frobenius inner product (namely, $\{\bb{e}_i \bb{e}_i^{\top}\}_{i=1}^d \cup \{(\bb{e}_i \bb{e}_j^{\top} + \bb{e}_j \bb{e}_i^{\top})/\sqrt{2}\}_{1 \leq i < j \leq d}$, where $\bb{e}_i$ is the $i$-th standard basis vector of $\R^d$), and identify $\mathcal{S}^d \oplus \ii \mathcal{S}^d$ with $\C^N$. Under this identification, $B_{\Theta}$ becomes an entire function on $\C^N$. Fix any multi-index $\bb{\eta}\in \N_0^N$ with $|\bb{\eta}|\leq m$. By Cauchy's integral formula on the unit polydisc \citep[see, e.g.,][Theorem~1.2.2]{Krantz2001},
\[
\frac{\partial^{|\bb{\eta}|}}{\partial w_1^{\eta_1} \cdots \partial w_N^{\eta_N}} B_{\Theta}(\bb{w}) = \frac{\eta_1!\dots\eta_N!}{(2\pi i)^N} \int_{|z_1-w_1| = 1} \cdots \int_{|z_N-w_N| = 1} \frac{B_{\Theta}(\bb{z})}{\prod_{\nu = 1}^N (z_{\nu}-w_{\nu})^{\eta_{\nu} + 1}} \, \rd z_1 \cdots \, \rd z_N,
\]
hence
\[
\left|\frac{\partial^{|\bb{\eta}|}}{\partial w_1^{\eta_1} \cdots \partial w_N^{\eta_N}} B_{\Theta}(\bb{w})\right| \leq \eta_1!\dots\eta_N! \sup_{|z_1-w_1|\leq 1, \ldots,|z_N-w_N|\leq 1} |B_{\Theta}(\bb{z})|.
\]
On this polydisc, $\|\bb{z}\|_2\leq \|\bb{w}\|_2 + \sqrt{N}$, so the previous bound and the subadditivity of the square root yield
\[
\left|\frac{\partial^{|\bb{\eta}|}}{\partial w_1^{\eta_1} \cdots \partial w_N^{\eta_N}} B_{\Theta}(\bb{w})\right| \leq C_{\bb{\eta}, \Theta} \exp\{c_{\Theta}\sqrt{1 + \|\bb{w}\|_2}\},
\]
with $C_{\bb{\eta}, \Theta} \leqdef \eta_1!\dots\eta_N! \exp\{c_{\Theta} N^{1/4}\}$. Since each entry of the symmetric gradient $\nabla$ is a constant multiple of a coordinate derivative in these variables, it follows, after changing the constant $C_{\bb{\eta}, \Theta}$, that every differential monomial $\overline{\nabla}$ of order at most $m$ in the entries of $\nabla$ satisfies
\begin{equation}\label{eq:bound.derivative.Phi}
|\overline{\nabla}B_{\Theta}(X)| \leq C_{\overline{\nabla}, \Theta}\exp\{c_{\Theta}\sqrt{1 + \|X\|_F}\} \leq C_{\overline{\nabla}, \Theta}\exp\{c_{\Theta}\sqrt{1 + \tr(X)}\}, \qquad X\in \mathcal{S}_{+}^d,
\end{equation}
for some other constant $C_{\overline{\nabla}, \Theta}\in (0, \infty)$, as claimed in \eqref{eq:exponential.growth.derivatives.Bessel}.

Next, we record the cone-integrability estimate that will be used repeatedly: for every $\gamma > -1$, $M\in [0, \infty)$ and $c\in (0, \infty)$,
\begin{equation}\label{eq:integrability}
\int_{\mathcal{S}_{++}^d} (1 + \tr(X))^M |X|^{\gamma} \, \etr(-cX) \, \rd X < \infty.
\end{equation}
This is an immediate consequence of the convergence of the multivariate gamma integral \eqref{eq:gamma.integral}.

Now let $\overline{\nabla}$ be any differential monomial of order $r\leq m$ in the entries of $\nabla$. Using Leibniz' rule together with $\nabla |X|^q = q \, \mathrm{adj}(X) \, |X|^{q-1}$, an induction on $r$ shows that
\[
\overline{\nabla}|X|^{\beta} = \sum_{s = 0}^r p_{\overline{\nabla},s}(X) \, |X|^{\beta-s},
\]
where $p_{\overline{\nabla},s}$ are scalar-valued polynomials. Therefore, every derivative $\overline{\nabla}f_{\alpha, I_d, \Theta}^{\mathcal{W}}(X)$ of order $r\leq m$ is a finite linear combination of terms of the form
\[
p(X) \, |X|^{\beta-s} \, \etr(-X/2) \, \psi(X), \qquad s\leq r,
\]
where $p$ is a scalar-valued polynomial and $\psi$ is a derivative of $B_{\Theta}$ of order at most $r$. Given that $X\in \mathcal{S}_{+}^d$ implies $|X_{ij}| \leq \sqrt{X_{ii}X_{jj}} \leq \tr(X)$ for all $i,j\in [d]$, every generic polynomial of degree~$M$,
\[
p(X) = \sum_{|\alpha| \leq M} c_{\alpha} \prod_{1 \leq i \leq j \leq d} X_{ij}^{\alpha_{ij}},
\]
satisfies
\[
|p(X)| \leq \bigg(\sum_{|\alpha| \leq M} |c_{\alpha}|\bigg) (1 + \tr(X))^M, \qquad X\in \mathcal{S}_{+}^d.
\]
Combining this with the bound \eqref{eq:bound.derivative.Phi} on the derivatives of $B_{\Theta}$ yields, for every $X\in \mathcal{S}_{++}^d$,
\begin{equation}\label{eq:bounds.derivative.noncentral.Wishart}
\begin{aligned}
|\overline{\nabla}f_{\alpha, I_d, \Theta}^{\mathcal{W}}(X)|
&\leq C_{\overline{\nabla}} (1 + \tr(X))^{M_{\overline{\nabla}}} \sum_{s = 0}^r |X|^{\beta-s} \exp\left\{-\frac{1}{2}\tr(X) + c_{\Theta}\sqrt{1 + \tr(X)}\right\} \\
&\leq C_{\overline{\nabla}}' (1 + \tr(X))^{M_{\overline{\nabla}}'} |X|^{\beta-r} \etr(-X/4),
\end{aligned}
\end{equation}
for appropriate constants $C_{\overline{\nabla}}, C_{\overline{\nabla}}' > 0$ and $M_{\overline{\nabla}}, M_{\overline{\nabla}}'\in \N$.

This estimate gives exactly the required boundary behavior. If $r\leq m-1$, then $\beta-r > 0$, so $\beta-s\geq \beta-r > 0$ for all $s\leq r$. Since the derivatives of $B_{\Theta}$ are continuous on all of $\mathcal{S}^d$, every derivative of order at most $m-1$ extends continuously to $\mathcal{S}_{+}^d$ and vanishes on $\partial\mathcal{S}_{+}^d$. If $r = m$, then $\beta-m > -1$, so derivatives of order $m$ are locally integrable near $\partial\mathcal{S}_{+}^d$ by the cone-integrability estimate \eqref{eq:integrability}. In particular, $\overline{\nabla}f_{\alpha, I_d, \Theta}^{\mathcal{W}}\in L^1(\mathcal{S}_{++}^d)$.

We now justify the repeated integration by parts \eqref{eq:IBP} by a localization argument. For the fixed index pairs $(i_{\ell},j_{\ell})\in [d]^2$ and any $A\subseteq [m]$, set
\[
\overline{\nabla}_{A} \leqdef \prod_{\ell\in A} \nabla_{i_{\ell}j_{\ell}},
\]
with the convention that $\overline{\nabla}_{\varnothing}$ is the identity operator. Choose $\chi, \zeta\in C^{\infty}((0, \infty))$ such that
\begin{alignat*}{2}
\chi(t) &= 0 \text{ for } t\in (0,1], \qquad \chi(t) &&= 1 \text{ for } t\in [2, \infty), \\
\zeta(t) &= 1 \text{ for } t\in (0,1], \qquad \, \zeta(t) &&= 0 \text{ for } t\in [2, \infty).
\end{alignat*}
For $\varepsilon\in (0,1)$ and $R > 1$, define
\[
\omega_{\varepsilon,R}(X) \leqdef \chi\left(\frac{|X|}{\varepsilon}\right)\zeta\left(\frac{\tr(X)}{R}\right), \qquad \phi_T(X) \leqdef \etr(-TX), \qquad X\in \mathcal{S}_{++}^d.
\]
Then $\omega_{\varepsilon,R}\in C^{\infty}(\mathcal{S}_{++}^d)$, $\omega_{\varepsilon,R}(X)\to 1$ pointwise on $\mathcal{S}_{++}^d$ as $\varepsilon\downarrow 0$ and $R\uparrow\infty$, and $\omega_{\varepsilon,R} \, \phi_T$ has compact support in $\mathcal{S}_{++}^d$. Indeed, on the support of $\omega_{\varepsilon,R}$, we have $|X| \geq \varepsilon$ and $\tr(X)\leq 2R$, so if $\lambda_1(X)\geq \cdots \geq \lambda_d(X) > 0$ are the eigenvalues of $X$, then
\[
[\lambda_d(X), \lambda_1(X)] = \left[\frac{|X|}{\prod_{i=1}^{d-1} \lambda_i(X)}, \lambda_1(X)\right] \subseteq \left[\frac{\varepsilon}{(2R)^{d-1}}, 2R\right] \subseteq (0, \infty).
\]

Since $\omega_{\varepsilon,R} \, \phi_T$ is compactly supported in the interior of $\mathcal{S}^d \cong \R^{d(d + 1)/2}$, ordinary Euclidean integration by parts gives
\begin{equation}\label{eq:IBP.localization}
\int_{\mathcal{S}_{++}^d} \omega_{\varepsilon,R}(X)\phi_T(X) \, \overline{\nabla}_{[m]} f_{\alpha, I_d, \Theta}^{\mathcal{W}}(X) \, \rd X
= (-1)^m \int_{\mathcal{S}_{++}^d} f_{\alpha, I_d, \Theta}^{\mathcal{W}}(X) \, \overline{\nabla}_{[m]}\big(\omega_{\varepsilon,R} \, \phi_T\big)(X) \, \rd X.
\end{equation}
Expanding $\overline{\nabla}_{[m]}(\omega_{\varepsilon,R} \, \phi_T)$ by Leibniz' rule, we have
\[
\overline{\nabla}_{[m]}(\omega_{\varepsilon,R} \, \phi_T)
= \omega_{\varepsilon,R}\overline{\nabla}_{[m]}\phi_T + \sum_{\varnothing\neq A\subseteq [m]} \big(\overline{\nabla}_{A} \, \omega_{\varepsilon,R}\big)\big(\overline{\nabla}_{A^c}\phi_T\big),
\]
where $A^c = [m]\setminus A$. Since
\[
\overline{\nabla}_{[m]}\phi_T(X) = (-1)^m \left(\prod_{\ell=1}^m T_{i_{\ell}j_{\ell}}\right)\phi_T(X), \qquad X\in \mathcal{S}_{++}^d,
\]
we obtain
\[
\int_{\mathcal{S}_{++}^d} \omega_{\varepsilon,R}(X) \, \phi_T(X) \, \overline{\nabla}_{[m]} f_{\alpha, I_d, \Theta}^{\mathcal{W}}(X) \, \rd X
= \left(\prod_{\ell=1}^m T_{i_{\ell}j_{\ell}}\right)\int_{\mathcal{S}_{++}^d} \omega_{\varepsilon,R}(X) \, \phi_T(X) \, f_{\alpha, I_d, \Theta}^{\mathcal{W}}(X) \, \rd X + E_{\varepsilon,R},
\]
where the error term $E_{\varepsilon,R}$ satisfies
\[
|E_{\varepsilon,R}| \leq C_T \sum_{\varnothing\neq A\subseteq [m]} \int_{\mathcal{S}_{++}^d} f_{\alpha, I_d, \Theta}^{\mathcal{W}}(X) \, |\overline{\nabla}_{A} \, \omega_{\varepsilon,R}(X)| \, \rd X,
\]
for a suitable constant $C_T\in (0, \infty)$.

Each term in $\overline{\nabla}_{A} \, \omega_{\varepsilon,R}$ contains either a derivative of $\chi(|X|/\varepsilon)$, hence is supported on $\{\varepsilon\leq |X|\leq 2\varepsilon\}$, or a derivative of $\zeta(\tr(X)/R)$, hence is supported on $\{R\leq \tr(X)\leq 2R\}$, or both. Since derivatives of $|X|$ are polynomials and derivatives of $\tr(X)$ are constants, there exist constants $C_{A}, M_{A} > 0$ such that
\[
|\overline{\nabla}_{A} \, \omega_{\varepsilon,R}(X)|
\leq C_{A}(1 + \tr(X))^{M_{A}} \left[\sum_{q = 1}^{|A|} |X|^{-q} \ind_{\{\varepsilon\leq |X|\leq 2\varepsilon\}} + \ind_{\{R\leq \tr(X)\leq 2R\}}\right].
\]
Here we used that every derivative landing on $\chi(|X|/\varepsilon)$ produces a factor $\varepsilon^{-1}$, and on the support of such a derivative we have $\varepsilon\leq |X|\leq 2\varepsilon$, so $\varepsilon^{-q}\leq 2^q |X|^{-q}$. Also, from the expression of the noncentral Wishart density in \eqref{eq:expression.density.NCW}, and the bound \eqref{eq:Bessel.bound} on the matrix Bessel function,
\[
f_{\alpha, I_d, \Theta}^{\mathcal{W}}(X) \leq C_{\Theta}(1 + \tr(X))^{M_{\Theta}} |X|^{\beta} \etr(-X/4), \qquad X\in \mathcal{S}_{++}^d,
\]
for some constants $C_{\Theta}\in (0, \infty)$ and $M_{\Theta}\in \N$. Thus, the boundary-shell contribution to $E_{\varepsilon,R}$ is dominated by a constant multiple of
\[
(1 + \tr(X))^M |X|^{\beta-q} \etr(-X/4), \qquad 1\leq q\leq m,
\]
for some constant $M\in \N$, which is integrable because $\beta-m > -1$. The outer-shell contribution to $E_{\varepsilon,R}$ is dominated by
\[
(1 + \tr(X))^M |X|^{\beta} \etr(-X/4),
\]
which is also integrable. Since the corresponding shell indicators converge pointwise to $0$ as $\varepsilon\downarrow 0$ and $R\uparrow\infty$, dominated convergence and \eqref{eq:integrability} yield
\[
E_{\varepsilon,R}\to 0.
\]

Moreover, $\overline{\nabla}_{[m]} f_{\alpha, I_d, \Theta}^{\mathcal{W}},f_{\alpha, I_d, \Theta}^{\mathcal{W}}\in L^1(\mathcal{S}_{++}^d)$ by \eqref{eq:integrability} and \eqref{eq:bounds.derivative.noncentral.Wishart}, so dominated convergence gives
\[
\begin{aligned}
\int_{\mathcal{S}_{++}^d} \omega_{\varepsilon,R}(X) \, \phi_T(X) \, \overline{\nabla}_{[m]} f_{\alpha, I_d, \Theta}^{\mathcal{W}}(X) \, \rd X
&\to \int_{\mathcal{S}_{++}^d} \phi_T(X) \, \overline{\nabla}_{[m]} f_{\alpha, I_d, \Theta}^{\mathcal{W}}(X) \, \rd X, \\
\int_{\mathcal{S}_{++}^d} \omega_{\varepsilon,R}(X) \, \phi_T(X) \, f_{\alpha, I_d, \Theta}^{\mathcal{W}}(X) \, \rd X
&\to \int_{\mathcal{S}_{++}^d} \phi_T(X) \, f_{\alpha, I_d, \Theta}^{\mathcal{W}}(X) \, \rd X.
\end{aligned}
\]
Letting $\varepsilon\downarrow 0$ and $R\uparrow\infty$ in \eqref{eq:IBP.localization} completes the proof of \eqref{eq:IBP} and the lemma.
\end{proof}

\begin{lemma}\label{lem:Laplace.derivative.Wishart}
Let $\alpha > 3d-3$ be given, and let $\mathrm{adj}(\cdot)$ denote the adjugate matrix operator. For any $k\in \N$, $(i_1, j_1), \ldots, (i_k, j_k)\in [d]^2$, and $\Theta,X\in \mathcal{S}_{++}^d$, we have
\[
\left(\prod_{\ell=1}^k \nabla_{\Theta, i_{\ell} j_{\ell}}\right) f_{\alpha, I_d, \Theta}^{\mathcal{W}}(X)
= (-1)^k \left(\prod_{\ell=1}^k \Big\{\nabla_{\!X} \, \mathrm{adj}(I_d + 2\nabla_{\!X})\Big\}_{i_{\ell} j_{\ell}}\right) f_{\alpha + 2k, I_d, \Theta}^{\mathcal{W}}(X),
\]
where $\nabla_{\Theta}$ acts on the noncentrality parameter $\Theta$ and $\nabla_{\!X}$ acts on the state variable $X$.
\end{lemma}

\begin{proof}[Proof of Lemma~\ref{lem:Laplace.derivative.Wishart}]
By Remark~\ref{rem:Wishart.mgf} with $\Sigma = I_d$,
\begin{equation}\label{eq:Laplace}
\mathcal{L}[f_{\alpha, I_d, \Theta}^{\mathcal{W}}](T)
= \frac{\etr\{-T(I_d + 2T)^{-1}\Theta\}}{|I_d + 2T|^{\alpha/2}}, \qquad T\in \mathcal{S}_{+}^d.
\end{equation}
The interchange of $\nabla_{\Theta, ij}$ with the Laplace integral is justified locally uniformly in $\Theta\in \mathcal{S}_{++}^d$ by the same Bessel bounds used in the proof of Lemma~\ref{lem:Laplace.polynomial.calculus}; for suitable constants $C,M > 0$,
\[
\big|\nabla_{\Theta, ij} f_{\alpha, I_d, \Theta}^{\mathcal{W}}(X)\big|
\leq C(1 + \tr(X))^M |X|^{(\alpha-d-1)/2}\etr(-X/4),
\]
which is integrable under the present assumption on $\alpha$. Differentiating \eqref{eq:Laplace} under the integral sign with respect to the $(i,j)$ component of $\Theta$ gives
\[
\begin{aligned}
\mathcal{L}[\nabla_{\Theta, ij} f_{\alpha, I_d, \Theta}^{\mathcal{W}}](T)
&= \nabla_{\Theta, ij} \, \mathcal{L}[f_{\alpha, I_d, \Theta}^{\mathcal{W}}](T) \\
&= -\big(T(I_d + 2T)^{-1}\big)_{ij} \, \mathcal{L}[f_{\alpha, I_d, \Theta}^{\mathcal{W}}](T), \qquad i,j\in [d].
\end{aligned}
\]
We can write
\[
T(I_d + 2T)^{-1} = \frac{T \, \mathrm{adj}(I_d + 2T)}{|I_d + 2T|}.
\]
Moreover, by Remark~\ref{rem:Wishart.mgf} again,
\[
|I_d + 2T|^{-1} \, \mathcal{L}[f_{\alpha, I_d, \Theta}^{\mathcal{W}}](T) = \mathcal{L}[f_{\alpha + 2, I_d, \Theta}^{\mathcal{W}}](T).
\]
Hence, for each $(i,j)\in \smash{[d]^2}$,
\[
\mathcal{L}[\nabla_{\Theta, ij} f_{\alpha, I_d, \Theta}^{\mathcal{W}}](T)
= -\big(T \, \mathrm{adj}(I_d + 2T)\big)_{ij} \, \mathcal{L}[f_{\alpha + 2, I_d, \Theta}^{\mathcal{W}}](T).
\]
Every entry of the matrix $U \, \mathrm{adj}(I_d + 2U)$ is a polynomial of degree at most $d$ in the entries of~$U$. Since $\alpha > 3d-3$, we have $\alpha + 2 > d-1 + 2d$, so Lemma~\ref{lem:Laplace.polynomial.calculus} applies to $\smash{f_{\alpha + 2, I_d, \Theta}^{\mathcal{W}}}$ with $m = d$:
\[
\mathcal{L}\left[-\big\{\nabla_{\!X} \, \mathrm{adj}(I_d + 2\nabla_{\!X})\big\}_{ij} \, f_{\alpha + 2, I_d, \Theta}^{\mathcal{W}}\right](T)
= - \big(T \, \mathrm{adj}(I_d + 2T)\big)_{ij} \, \mathcal{L}[f_{\alpha + 2, I_d, \Theta}^{\mathcal{W}}](T).
\]
Thus the two sides in the next display have the same Laplace transform for every $T\in \mathcal{S}_{+}^d$. By uniqueness of Laplace transforms on $\mathcal{S}_{+}^d$, they are equal for a.e.\ $X\in \mathcal{S}_{++}^d$; since both sides are continuous on $\mathcal{S}_{++}^d$, the equality holds for every $X\in \mathcal{S}_{++}^d$:
\begin{equation}\label{eq:Wishart.parameter.derivative.identity.k.1}
\nabla_{\Theta, ij} \, f_{\alpha, I_d, \Theta}^{\mathcal{W}}(X)
= -\big\{\nabla_{\!X} \, \mathrm{adj}(I_d + 2\nabla_{\!X})\big\}_{ij} \, f_{\alpha + 2, I_d, \Theta}^{\mathcal{W}}(X), \qquad X\in \mathcal{S}_{++}^d.
\end{equation}
The claim follows by iterating \eqref{eq:Wishart.parameter.derivative.identity.k.1} $k$ times, using that the families of operators $\nabla_{\Theta,ij}$ and $\{\nabla_{\!X} \, \mathrm{adj}(I_d + 2\nabla_{\!X})\}_{ij}$ commute among themselves and with each other. Since the shape parameter $\alpha$ increases by $2$ at each step, the baseline assumption $\alpha > 3d-3$ guarantees that the same argument applies throughout the iteration.
\end{proof}

Lemma~\ref{lem:Luk.Wishart.Lemma.2.5} is a scale-normalized Wishart analog of Lemma~2.5 of \citet{Luk1994PhD}.

\begin{lemma}\label{lem:Luk.Wishart.Lemma.2.5}
Let $\alpha > 3d-3$, $t\geq 0$, and $k\in \N$. Suppose that $h\in C^{kd}(\mathcal{S}_{++}^d)$, and that for every differential monomial $\overline{\nabla}$ of order at most $kd$ in the entries of $\nabla_{\!X}$, including the order-zero monomial, there exist constants $C_{\overline{\nabla}}, N_{\overline{\nabla}}\geq 0$ such that
\begin{equation}\label{eq:derivatives.h.assumption}
|\overline{\nabla} h(X)| \leq C_{\overline{\nabla}} (1 + \|X\|_F^{N_{\overline{\nabla}}}), \qquad X\in \mathcal{S}_{++}^d.
\end{equation}
Then, for every $(i_1, j_1), \ldots, (i_k, j_k)\in [d]^2$ and $\Lambda\in \mathcal{S}_{++}^d$,
\begin{equation}\label{eq:claim}
\left(\prod_{\ell=1}^k \nabla_{\Lambda, i_{\ell} j_{\ell}}\right) (\mathcal{Q}_t h)(\Lambda)
= e^{-2kt} \, \EE\left[\left(\prod_{\ell=1}^k \mathcal{D}_{i_{\ell} j_{\ell}}\right) h(\mathfrak{Y}_{k,t})\right],
\end{equation}
where $\mathcal{D} = \nabla_{\!X} \, \mathrm{adj}(I_d - 2\nabla_{\!X})$, $\mathfrak{Y}_{k,t}\sim \mathcal{W}_d(\alpha + 2k, I_d, e^{-2t}\Lambda)$ and is defined in \eqref{eq:Wishart.Q.operator.def}.
\end{lemma}

\begin{proof}[Proof of Lemma~\ref{lem:Luk.Wishart.Lemma.2.5}]
For a temporary shape parameter $\gamma > 3d-3$, define
\[
\mathcal{Q}_t^{(\gamma)} g(\Lambda) \leqdef \int_{\mathcal{S}_{++}^d} g(X) f_{\gamma,I_d,e^{-2t}\Lambda}^{\mathcal{W}}(X) \, \rd X, \qquad \Lambda\in \mathcal{S}_{++}^d,
\]
so that $\mathcal{Q}_t = \mathcal{Q}_t^{(\alpha)}$. We first prove the one-step identity
\begin{equation}\label{eq:Luk.Wishart.one.step.identity}
\nabla_{\Lambda,ij}(\mathcal{Q}_t^{(\gamma)}g)(\Lambda) = e^{-2t} \, \mathcal{Q}_t^{(\gamma + 2)}(\mathcal{D}_{ij} \, g)(\Lambda),
\end{equation}
whenever $g\in C^d(\mathcal{S}_{++}^d)$ and all differential monomials of order at most $d$ in the entries of $\nabla_{\!X}$, including the order-zero monomial, applied to $g$ have polynomial growth.

Fix $\Lambda\in \mathcal{S}_{++}^d$ and choose a compact neighborhood $K\subseteq \mathcal{S}_{++}^d$ of $e^{-2t}\Lambda$. By Lemma~\ref{lem:Laplace.derivative.Wishart} with $k = 1$, for each $\Theta\in K$ the quantity $\nabla_{\Theta,ij} f_{\gamma,I_d, \Theta}^{\mathcal{W}}$ can be written as a constant-coefficient $X$-differential operator of order at most $d$ applied to $f_{\gamma + 2,I_d, \Theta}^{\mathcal{W}}$. The estimates established in the proof of Lemma~\ref{lem:Laplace.polynomial.calculus} are uniform for $\Theta\in K$, so after multiplication by $g(X)$ the integrands below are dominated by a single $L^1(\mathcal{S}_{++}^d)$ function of $X$. Hence differentiation under the integral sign is justified by dominated convergence.

By the chain rule and differentiating under the integral sign, we obtain
\[
\nabla_{\Lambda,ij}(\mathcal{Q}_t^{(\gamma)}g)(\Lambda)
= e^{-2t} \int_{\mathcal{S}_{++}^d} g(X) \left[\nabla_{\Theta,ij} f_{\gamma,I_d, \Theta}^{\mathcal{W}}(X)\right]_{\Theta = e^{-2t}\Lambda} \, \rd X.
\]
Applying Lemma~\ref{lem:Laplace.derivative.Wishart}, which is permissible since $\gamma > 3d-3$, gives
\[
\nabla_{\Lambda,ij}(\mathcal{Q}_t^{(\gamma)}g)(\Lambda)
= - e^{-2t} \int_{\mathcal{S}_{++}^d} g(X) \mathcal{A}_{ij} f_{\gamma + 2,I_d,e^{-2t}\Lambda}^{\mathcal{W}}(X) \, \rd X,
\]
where $\mathcal{A}_{ij} \leqdef \{\nabla_{\!X}\mathrm{adj}(I_d + 2\nabla_{\!X})\}_{ij}$.

The operator $\mathcal{A}_{ij}$ is a constant-coefficient differential operator of order at most $d$. Given that $\gamma > 3d-3$, we have $(\gamma + 2-d-1)/2 > d-1$. Exactly as in the proof of Lemma~\ref{lem:Laplace.polynomial.calculus}, this implies that derivatives of $\smash{f_{\gamma + 2,I_d,e^{-2t}\Lambda}^{\mathcal{W}}}$ of order at most $d-1$ vanish on $\partial\mathcal{S}_{+}^d$, while derivatives of order at most $d$ are locally integrable there. Together with the polynomial growth assumption on the derivatives of $g$ and the bounds in \eqref{eq:bounds.derivative.noncentral.Wishart} for the derivatives of $f_{\gamma + 2,I_d,e^{-2t}\Lambda}^{\mathcal{W}}$, this justifies the integration by parts below.

Since the entries of $\nabla_{\!X}$ are constant-coefficient commuting differential operators, the formal adjoint of $\mathcal{A}_{ij}$ with respect to Lebesgue measure is obtained by replacing $\nabla_{\!X}$ with $-\nabla_{\!X}$, namely
\[
\mathcal{A}_{ij}^*
= \{(-\nabla_{\!X})\mathrm{adj}(I_d - 2\nabla_{\!X})\}_{ij}
= - \mathcal{D}_{ij}.
\]
Therefore,
\[
\begin{aligned}
\nabla_{\Lambda,ij}(\mathcal{Q}_t^{(\gamma)}g)(\Lambda)
&= - e^{-2t} \int_{\mathcal{S}_{++}^d} \{\mathcal{A}_{ij}^* \, g(X)\} f_{\gamma + 2,I_d,e^{-2t}\Lambda}^{\mathcal{W}}(X) \, \rd X \\
&= e^{-2t} \int_{\mathcal{S}_{++}^d} \{\mathcal{D}_{ij} \, g(X)\} f_{\gamma + 2,I_d,e^{-2t}\Lambda}^{\mathcal{W}}(X) \, \rd X,
\end{aligned}
\]
which proves \eqref{eq:Luk.Wishart.one.step.identity}.

We now iterate \eqref{eq:Luk.Wishart.one.step.identity}. Set
\[
g_0 \leqdef h, \qquad g_r \leqdef \left(\prod_{\ell=1}^r \mathcal{D}_{i_{\ell} j_{\ell}}\right)h, \qquad \gamma_r \leqdef \alpha + 2r, \qquad r\in \{0, \ldots,k\}.
\]
Since each $\mathcal{D}_{ij}$ is a constant-coefficient differential operator of order at most $d$, the function $g_r$ belongs to $C^{(k-r)d}(\mathcal{S}_{++}^d)$, and all its differential monomials of order at most $(k-r)d$, including the order-zero monomial, have polynomial growth. Moreover, $\gamma_r > 3d-3$ for every $r\in \{0, \ldots,k\}$. Thus the one-step identity \eqref{eq:Luk.Wishart.one.step.identity} can be applied successively to the pairs $(\gamma_r,g_r)$ for $r\in \{0, \ldots,k-1\}$.

We prove by induction on $r$ that
\begin{equation}\label{eq:proof.induction}
\left(\prod_{\ell=1}^r \nabla_{\Lambda,i_{\ell} j_{\ell}}\right)(\mathcal{Q}_t^{(\alpha)}h)(\Lambda)
= e^{-2rt} \, \mathcal{Q}_t^{(\alpha + 2r)}g_r(\Lambda), \qquad r\in \{0, \ldots,k\}.
\end{equation}
The claim is trivial for $r = 0$. If it holds for some $r\in \{0, \ldots,k-1\}$, then applying \eqref{eq:Luk.Wishart.one.step.identity} with $\gamma = \alpha + 2r$ and $g = g_r$ gives
\[
\begin{aligned}
\left(\prod_{\ell=1}^{r + 1} \nabla_{\Lambda,i_{\ell} j_{\ell}}\right)(\mathcal{Q}_t^{(\alpha)}h)(\Lambda)
&= e^{-2rt} \, \nabla_{\Lambda,i_{r + 1}j_{r + 1}}\mathcal{Q}_t^{(\alpha + 2r)}g_r(\Lambda) \\[-3mm]
&= e^{-2(r + 1)t} \, \mathcal{Q}_t^{(\alpha + 2r + 2)}(\mathcal{D}_{i_{r + 1}j_{r + 1}}g_r)(\Lambda) \\
&= e^{-2(r + 1)t} \, \mathcal{Q}_t^{(\alpha + 2(r + 1))}g_{r + 1}(\Lambda),
\end{aligned}
\]
where we used the fact that constant-coefficient differential operators commute. This completes the induction and proves \eqref{eq:proof.induction}.

Taking $r = k$ yields
\[
\left(\prod_{\ell=1}^k \nabla_{\Lambda,i_{\ell} j_{\ell}}\right)(\mathcal{Q}_t h)(\Lambda)
= e^{-2kt} \int_{\mathcal{S}_{++}^d} \left\{\left(\prod_{\ell=1}^k \mathcal{D}_{i_{\ell} j_{\ell}}\right) h(X)\right\} f_{\alpha + 2k,I_d,e^{-2t}\Lambda}^{\mathcal{W}}(X) \, \rd X,
\]
which is the claimed expression \eqref{eq:claim}.
\end{proof}

\begin{remark}\label{rem:eq.claim.expectation.exists}
The expectation on the right-hand side of \eqref{eq:claim} is finite under the weaker condition $\alpha > d-1$. Indeed, set
\[
G(X) \leqdef \left(\prod_{\ell=1}^k \mathcal{D}_{i_{\ell} j_{\ell}}\right)h(X).
\]
Each entry of $\mathcal{D}$ is a constant-coefficient differential operator of order at most $d$, hence $G$ is a finite linear combination of differential monomials of order at most $kd$ in the entries of $\nabla_{\!X}$. By \eqref{eq:derivatives.h.assumption}, there exist constants $C,N < \infty$ such that
\[
|G(X)| \leq C(1 + \|X\|_F^N), \qquad X\in \mathcal{S}_{++}^d.
\]
Let $\beta \leqdef \alpha + 2k$ and $\Theta \leqdef e^{-2t}\Lambda$. If $\alpha > d-1$, then $\beta > d-1$, so $\mathfrak{Y}_{k,t}\sim \mathcal{W}_d(\beta,I_d,\Theta)$ is well defined. The moment generating form of the noncentral Wishart transform, obtained from Remark~\ref{rem:Wishart.mgf} on its natural domain, gives, for every $0 < s < 1/2$,
\[
\EE[\exp\{s \, \tr(\mathfrak{Y}_{k,t})\}]
= (1 - 2s)^{-\beta d/2}\exp\left\{\frac{s}{1 - 2s}\tr(\Theta)\right\} < \infty.
\]
Since $\mathfrak{Y}_{k,t}\in \mathcal{S}_{+}^d$ almost surely, $\|\mathfrak{Y}_{k,t}\|_F \leq \tr(\mathfrak{Y}_{k,t})$, and therefore $\EE[\|\mathfrak{Y}_{k,t}\|_F^N] < \infty$. Consequently, $\EE[|G(\mathfrak{Y}_{k,t})|] < \infty$. Thus the condition $\alpha > 3d-3$ in Lemma~\ref{lem:Luk.Wishart.Lemma.2.5} is not an integrability condition for the expectation in \eqref{eq:claim}; it is used in the present proof of the transfer identity.
\end{remark}

\section{Proofs of the main results}\label{sec:proofs.main.results}

\subsection{Proof of Proposition~\ref{prop:generator}}

For any $f\in C^2(\mathcal{S}_{++}^d)$, the extended generator of $(\mathfrak{W}_t)_{t\geq 0}$ arises from It\^o's formula as
\[
\mathcal{A}^{\mathcal{W}}f(S)
= \underbrace{2 \, \tr\left\{(\alpha \Sigma - S) \nabla f(S)\right\}}_{\text{drift term}} + \underbrace{\frac{1}{2} \sum_{i,j,k, \ell = 1}^d \frac{( \, \rd \mathfrak{W}_t^{(\mathrm{diff})})_{ij} \bullet ( \, \rd \mathfrak{W}_t^{(\mathrm{diff})})_{k\ell}}{\rd t} \nabla_{ij} \nabla_{k\ell} f(S)}_{\text{diffusion term}},
\]
where $\bullet$ denotes the It\^o product, and
\[
 \, \rd \mathfrak{W}_t^{(\mathrm{diff})}
 \leqdef \sqrt{2} \, ( \, \rd \mathfrak{X}_t + \, \rd \mathfrak{X}_t^{\top}), \qquad \, \rd \mathfrak{X}_t
 \leqdef S^{1/2} \, \rd \mathfrak{B}_t \Sigma^{1/2}.
\]
Since $( \, \rd \mathfrak{X}_t)_{ij} = \sum_{p, q = 1}^d (S^{1/2})_{ip} ( \, \rd \mathfrak{B}_t)_{pq} (\Sigma^{1/2})_{qj}$, the identity $( \, \rd \mathfrak{B}_t)_{pq} \bullet ( \, \rd \mathfrak{B}_t)_{p'q'} = \delta_{pp'} \delta_{qq'} \, \rd t$ yields
\[
\begin{aligned}
( \, \rd \mathfrak{X}_t)_{ij} \bullet ( \, \rd \mathfrak{X}_t)_{k\ell}
&= \sum_{p, q, p', q' = 1}^d (S^{1/2})_{ip} (S^{1/2})_{kp'} (\Sigma^{1/2})_{qj} (\Sigma^{1/2})_{q'\ell} \, \{( \, \rd \mathfrak{B}_t)_{pq} \bullet ( \, \rd \mathfrak{B}_t)_{p'q'}\} \\
&= \sum_{p, q = 1}^d (S^{1/2})_{ip} (S^{1/2})_{kp} (\Sigma^{1/2})_{qj} (\Sigma^{1/2})_{q\ell} \, \rd t \\
&= S_{ik} \Sigma_{j\ell} \, \rd t.
\end{aligned}
\]
It follows that the quadratic covariation between the $(i,j)$ and $(k, \ell)$ entries of $\smash{\rd \mathfrak{W}_t^{(\mathrm{diff})}}$ is
\[
\begin{aligned}
( \, \rd \mathfrak{W}_t^{(\mathrm{diff})})_{ij} \bullet ( \, \rd \mathfrak{W}_t^{(\mathrm{diff})})_{k\ell}
&= 2 \, \{( \, \rd \mathfrak{X}_t)_{ij} + ( \, \rd \mathfrak{X}_t)_{ji}\} \bullet \{( \, \rd \mathfrak{X}_t)_{k\ell} + ( \, \rd \mathfrak{X}_t)_{\ell k}\} \\
&= 2 \, \{S_{ik} \Sigma_{j\ell} + S_{i\ell} \Sigma_{jk} + S_{jk} \Sigma_{i\ell} + S_{j\ell} \Sigma_{ik}\} \, \rd t.
\end{aligned}
\]
Using the symmetry of $\nabla$, the diffusion part of $\mathcal{A}^{\mathcal{W}}f(S)$ becomes
\[
\begin{aligned}
&\sum_{i,j,k, \ell = 1}^d \{S_{ik} \Sigma_{j\ell} + S_{i\ell} \Sigma_{jk} + S_{jk} \Sigma_{i\ell} + S_{j\ell} \Sigma_{ik}\} \nabla_{ij} \nabla_{k\ell} f(S) \\
&\qquad = \left\{\sum_{i, k = 1}^d S_{ik} \sum_{j, \ell = 1}^d \nabla_{ij} \Sigma_{j\ell} \nabla_{\ell k} + \sum_{i, \ell = 1}^d S_{i\ell} \sum_{j, k = 1}^d \nabla_{ij} \Sigma_{jk} \nabla_{k\ell} \right. \\
&\hspace{25mm}\left. + \sum_{j, k = 1}^d S_{jk} \sum_{i, \ell = 1}^d \nabla_{ji} \Sigma_{i\ell} \nabla_{\ell k} + \sum_{j, \ell = 1}^d S_{j\ell} \sum_{i, k = 1}^d \nabla_{ji} \Sigma_{ik} \nabla_{k\ell}\right\} f(S) \\
&\qquad = \left\{\sum_{i, k = 1}^d S_{ik} (\nabla \Sigma \nabla)_{ik} + \sum_{i, \ell = 1}^d S_{i\ell} (\nabla \Sigma \nabla)_{i\ell} + \sum_{j, k = 1}^d S_{jk} (\nabla \Sigma \nabla)_{jk} + \sum_{j, \ell = 1}^d S_{j\ell} (\nabla \Sigma \nabla)_{j\ell}\right\} f(S) \\
&\qquad = 4 \, \tr\{S \nabla \Sigma \nabla f(S)\}.
\end{aligned}
\]
Consequently, the extended generator of $(\mathfrak{W}_t)_{t\geq 0}$ is
\[
\mathcal{A}^{\mathcal{W}} f(S) = 2 \, \tr\{(\alpha \Sigma - S) \nabla f(S)\} + 4 \, \tr\{S \nabla \Sigma \nabla f(S)\},
\]
as claimed.

\subsection{Proof of Proposition~\ref{prop:extension.Luk.1994.Lemma.2.4}}
\setcounter{pstep}{0}

Fix $T,W\in \mathcal{S}_{+}^d$. The proof is divided into four steps. In Step~\ref{step:1}, we identify the $\mathcal{S}_{+}^d$-valued affine extension of the Wishart process as an affine diffusion and write down the corresponding Riccati equations. In Step~\ref{step:2}, we solve these equations explicitly and identify the conditional transition law. In Step~\ref{step:3}, we pass to the limiting law and obtain the semigroup representation. In Step~\ref{step:4}, we prove the invariance of $\mathcal{W}_d(\alpha, \Sigma)$ by comparing Laplace transforms.

\step{Affine transform formula and Riccati equations}\label{step:1}
In the notation of \citet[Theorems~2.4 and~2.6]{CuchieroFilipovicMayerhoferTeichmann2011}, with the diffusion matrix parameter denoted here by $a$ to avoid confusion with the shape parameter $\alpha$, take
\[
a = 2\Sigma, \qquad b = 2\alpha\Sigma, \qquad B(S) = -2S, \qquad c = 0, \qquad \gamma = 0_{d\times d}, \qquad m(\cdot) = \mu(\cdot) = 0.
\]
Given that $\alpha > d-1$, we have $b - (d - 1)a = 2(\alpha - d + 1)\Sigma\in \mathcal{S}_{++}^d$, so the constant drift admissibility condition in their Definition~2.3 is satisfied. The linear drift $B(S) = -2S$ is of the form $HS + SH^{\top}$ with $H = -I_d$, and the remaining admissibility conditions in their Definition~2.3 are immediate. Thus Theorem~2.4 of \citet{CuchieroFilipovicMayerhoferTeichmann2011} yields an affine process on $\mathcal{S}_{+}^d$ with the affine transform formula below. Since $c = 0$, $\gamma = 0_{d\times d}$ and $m(\cdot) = \mu(\cdot) = 0$, the process is conservative. Moreover, by Theorem~2.6 of \citet{CuchieroFilipovicMayerhoferTeichmann2011}, it is an affine diffusion whose semimartingale representation is exactly \eqref{eq:Wishart.process}. Its affine transform formula gives
\begin{equation}\label{eq:affine.transform.Wishart}
\EE\big[\etr(-T\mathfrak{W}_t)\mid \mathfrak{W}_0 = W\big] = \exp\{-\varphi_t(T)-\tr(\Psi_t(T)W)\},
\end{equation}
where $\varphi_t(T)$ and $\Psi_t(T)$ solve the generalized Riccati equations
\begin{equation}\label{eq:Riccati.Wishart.semigroup}
\frac{\rd}{\rd t}\varphi_t(T) = 2\alpha \, \tr(\Sigma\Psi_t(T)), \qquad \varphi_0(T) = 0,
\end{equation}
and
\begin{equation}\label{eq:Riccati.Wishart.semigroup.Psi}
\frac{\rd}{\rd t}\Psi_t(T) = -2\Psi_t(T) - 4\Psi_t(T)\Sigma\Psi_t(T), \qquad \Psi_0(T) = T.
\end{equation}

\step{Solution of the Riccati equations and transition law}\label{step:2}
Recall that $\Sigma_t = (1 - e^{-2t})\Sigma$, and set
\[
R_t \leqdef I_d + 2T\Sigma_t, \qquad \varphi_t(T) \leqdef \frac{\alpha}{2}\log(|R_t|), \qquad \Psi_t(T) \leqdef e^{-2t}R_t^{-1}T.
\]
Note that $|R_t| = |I_d + 2T^{1/2}\Sigma_t T^{1/2}| > 0$, so $\varphi_t(T)$ is well defined. First,
\[
\Psi_t(T)
= e^{-2t}(I_d + 2 T \Sigma_t)^{-1}T
= e^{-2t}T^{1/2}(I_d + 2T^{1/2} \Sigma_t T^{1/2})^{-1}T^{1/2}\in \mathcal{S}_{+}^d.
\]
Since $\frac{\rd}{\rd t}R_t = 4e^{-2t}T\Sigma$, direct differentiation gives
\[
\frac{\rd}{\rd t}\varphi_t(T)
= \frac{\alpha}{2}\tr\left(R_t^{-1}\frac{\rd}{\rd t}R_t\right)
= 2\alpha e^{-2t}\tr(R_t^{-1}T\Sigma)
= 2\alpha \, \tr(\Sigma\Psi_t(T)),
\]
and
\[
\frac{\rd}{\rd t}\Psi_t(T)
= -2e^{-2t}R_t^{-1}T - e^{-2t}R_t^{-1}\left(\frac{\rd}{\rd t}R_t\right)R_t^{-1}T
= -2\Psi_t(T) - 4\Psi_t(T)\Sigma\Psi_t(T).
\]
Also, we have $R_0 = I_d$, $\varphi_0(T) = 0$ and $\Psi_0(T) = T$. Therefore, the displayed functions solve \eqref{eq:Riccati.Wishart.semigroup} and \eqref{eq:Riccati.Wishart.semigroup.Psi}. Substituting them into \eqref{eq:affine.transform.Wishart} yields
\begin{equation}\label{eq:conditional.Laplace.Wishart}
\EE\big[\etr(-T\mathfrak{W}_t)\mid \mathfrak{W}_0 = W\big]
= \frac{\etr\{-e^{-2t}(I_d + 2T\Sigma_t)^{-1}TW\}}{|I_d + 2T\Sigma_t|^{\alpha/2}}.
\end{equation}
Let $\Theta_t \leqdef e^{-2t}\Sigma_t^{-1}W$. Since $\Theta_t\Sigma_t^{-1} = e^{-2t}\Sigma_t^{-1}W\Sigma_t^{-1}\in \mathcal{S}_{+}^d$ and $\alpha > d-1$, the noncentral Wishart law $\mathcal{W}_d(\alpha, \Sigma_t, \Theta_t)$ is well defined. Its Laplace transform is
\[
\frac{\etr\{-T\Sigma_t(I_d + 2T\Sigma_t)^{-1}\Theta_t\}}{|I_d + 2T\Sigma_t|^{\alpha/2}}, \qquad T\in \mathcal{S}_{+}^d.
\]
Since $T\Sigma_t$ commutes with $(I_d + 2T\Sigma_t)^{-1}$, we have
\[
T\Sigma_t(I_d + 2T\Sigma_t)^{-1}\Theta_t
= e^{-2t}(I_d + 2T\Sigma_t)^{-1}TW.
\]
Thus the Laplace transform of $\mathcal{W}_d(\alpha, \Sigma_t, \Theta_t)$ agrees with \eqref{eq:conditional.Laplace.Wishart}. By uniqueness of Laplace transforms on $\mathcal{S}_{+}^d$,
\[
\mathfrak{W}_t \mid \{\mathfrak{W}_0 = W\} \sim \mathcal{W}_d(\alpha, \Sigma_t, e^{-2t}\Sigma_t^{-1}W), \qquad t > 0,
\]
which is \eqref{eq:W.t.conditional}. Since $\alpha > d-1$ and $\Sigma_t\in \mathcal{S}_{++}^d$, the density formula \eqref{eq:noncentral.Wishart.density} also gives
\[
P_t(W, \partial\mathcal{S}_{+}^d) = 0, \qquad t > 0, \qquad W\in \mathcal{S}_{+}^d.
\]

\step{Limiting law and semigroup representation}\label{step:3}
Letting $t\to\infty$ in \eqref{eq:conditional.Laplace.Wishart}, and using $\Sigma_t\to\Sigma$ and $e^{-2t}(I_d + 2T\Sigma_t)^{-1}T\to 0_{d\times d}$, we obtain
\[
\lim_{t\to\infty}\EE\big[\etr(-T\mathfrak{W}_t)\mid \mathfrak{W}_0 = W\big]
= |I_d + 2T\Sigma|^{-\alpha/2}, \qquad T\in \mathcal{S}_{+}^d.
\]
The right-hand side is the Laplace transform of $\mathcal{W}_d(\alpha, \Sigma)$. Hence, by L\'evy's continuity theorem for Laplace transforms on $\mathcal{S}_{+}^d$,
\[
\mathfrak{W}_t \mid \{\mathfrak{W}_0 = W\} \stackrel{\mathrm{law}}{\longrightarrow} \mathfrak{W}_{\infty}\sim \mathcal{W}_d(\alpha, \Sigma), \qquad t\to\infty,
\]
which proves \eqref{eq:W.infty}. Then \eqref{eq:rep} follows immediately because $\Sigma_t^{1/2} \mathfrak{S}_t \Sigma_t^{1/2}\sim \mathcal{W}_d(\alpha, \Sigma_t, e^{-2t}\Sigma_t^{-1}W)$ by a rescaling of the Laplace transform; see, e.g., Lemma~2.1 of \citet{GenestMacKayOuimet2026}.

\step{Invariance}\label{step:4}
To prove \eqref{eq:invariance}, it suffices to compare Laplace transforms. For any $T\in \mathcal{S}_{+}^d$, by \eqref{eq:conditional.Laplace.Wishart},
\[
\begin{aligned}
\int_{\mathcal{S}_{+}^d} \etr(-TY) \, (\xi \mathcal{P}_t^{\mathcal{W}})(\rd Y)
&= \int_{\mathcal{S}_{+}^d} \EE\big[\etr(-T\mathfrak{W}_t)\mid \mathfrak{W}_0 = W\big] \, \xi(\rd W) \\
&= \frac{1}{|I_d + 2T\Sigma_t|^{\alpha/2}} \int_{\mathcal{S}_{+}^d} \etr\{-U_t(T)W\} \, \xi(\rd W),
\end{aligned}
\]
where $U_t(T) \leqdef e^{-2t}(I_d + 2T\Sigma_t)^{-1}T = \Psi_t(T)\in \mathcal{S}_{+}^d$. Applying the Wishart Laplace transform to $\xi$ gives
\[
\int_{\mathcal{S}_{+}^d} \etr\{-U_t(T)W\} \, \xi(\rd W)
= \frac{1}{|I_d + 2U_t(T)\Sigma|^{\alpha/2}}.
\]
Since $\Sigma_t = (1-e^{-2t})\Sigma$, the matrices $T\Sigma_t$ and $T\Sigma$ commute. Therefore,
\[
\begin{aligned}
|I_d + 2U_t(T)\Sigma|
&= \left|I_d + 2e^{-2t}(I_d + 2T\Sigma_t)^{-1}T\Sigma\right| \\[-2mm]
&= \left|(I_d + 2T\Sigma_t + 2e^{-2t}T\Sigma)(I_d + 2T\Sigma_t)^{-1}\right|
= \frac{|I_d + 2T\Sigma|}{|I_d + 2T\Sigma_t|}.
\end{aligned}
\]
Substituting this into the previous display yields
\[
\int_{\mathcal{S}_{+}^d} \etr(-TY) \, (\xi \mathcal{P}_t^{\mathcal{W}})(\rd Y)
= \frac{1}{|I_d + 2T\Sigma_t|^{\alpha/2}}\left(\frac{|I_d + 2T\Sigma_t|}{|I_d + 2T\Sigma|}\right)^{\alpha/2}
= |I_d + 2T\Sigma|^{-\alpha/2}.
\]
This is the Laplace transform of $\xi$. By uniqueness of Laplace transforms on $\mathcal{S}_{+}^d$, we get $\xi \mathcal{P}_t^{\mathcal{W}} = \xi$, which is \eqref{eq:invariance}. This concludes the proof.

\subsection{Proof of Corollary~\ref{cor:Stein.Wishart.generator.step.1}}

Let $(\mathfrak{W}_t^{\infty})_{t\geq 0}$ be a stationary Wishart process with invariant distribution $\mathcal{W}_d(\alpha, \Sigma)$, and write $\mathfrak{W}_{\infty}\sim \mathcal{W}_d(\alpha, \Sigma)$ for its one-time marginal distribution. For any given $f\in C_{\mathcal{A}^{\mathcal{W}}}^2(\mathcal{S}_{++}^d)$, define the local martingale $\smash{M_t^f \leqdef f(\mathfrak{W}_t^{\infty}) - f(\mathfrak{W}_0^{\infty}) - \int_0^t \mathcal{A}^{\mathcal{W}}f(\mathfrak{W}_s^{\infty}) \, \rd s, ~t\geq 0}$. The integrability assumptions in \eqref{eq:C.2.A.W} and stationarity imply that $\smash{(M_t^f)_{t\geq 0}}$ is in fact a martingale. Hence, $\EE[M_t^f] = 0$. By stationarity, $\mathfrak{W}_s^{\infty}\sim \mathfrak{W}_{\infty}$ for every $s\geq 0$, and by the invariance of the Wishart law in Proposition~\ref{prop:extension.Luk.1994.Lemma.2.4}, we have $\EE[f(\mathfrak{W}_t^{\infty})] = \EE[f(\mathfrak{W}_0^{\infty})] = \EE[f(\mathfrak{W}_{\infty})]$ and $\EE[\mathcal{A}^{\mathcal{W}}f(\mathfrak{W}_s^{\infty})] = \EE[\mathcal{A}^{\mathcal{W}}f(\mathfrak{W}_{\infty})]$. Therefore, $\EE[\mathcal{A}^{\mathcal{W}}f(\mathfrak{W}_{\infty})] = 0$, which proves the forward implication.

We now prove the converse implication. Assume that the Stein identities $\EE[\mathcal{A}^{\mathcal{W}} f(\mathfrak{W})] = 0$ hold for the $\mathcal{S}_{++}^d$-valued random matrix $\mathfrak{W}\sim \mathcal{W}_d(\alpha, \Sigma)$ and all $\smash{f\in C_{\mathcal{A}^{\mathcal{W}}}^2(\mathcal{S}_{++}^d)}$. Fix $T\in \mathcal{S}_{++}^d$, and set $f_T(S) \leqdef \etr(-TS), ~S\in \mathcal{S}_{++}^d$, and
\[
L(T) \leqdef \EE[\etr(-T\mathfrak{W})].
\]
We have $f_T\in C_{\mathcal{A}^{\mathcal{W}}}^2(\mathcal{S}_{++}^d)$ because of the exponential decay. Also, we have $\nabla f_T(S) = -T f_T(S)$ and $\nabla \Sigma \nabla f_T(S) = T\Sigma T f_T(S)$. Hence, by Proposition~\ref{prop:generator} and the assumption $\EE[\mathcal{A}^{\mathcal{W}}f_T(\mathfrak{W})] = 0$,
\begin{equation}\label{eq:Stein.key}
\EE[\tr\{\mathfrak{W}T\}\etr(-T\mathfrak{W})] + 2 \, \EE[\tr\{\mathfrak{W}T\Sigma T\}\etr(-T\mathfrak{W})] = \alpha \, \tr(\Sigma T)L(T).
\end{equation}
For $H\in \mathcal{S}^d$, dominated convergence gives $D L(T)[H] = -\EE[\tr\{H\mathfrak{W}\}\etr(-T\mathfrak{W})]$. Consequently, by \eqref{eq:Stein.key},
\begin{equation}\label{eq:D.L}
D L(T)[-(T + 2T\Sigma T)] = \alpha \, \tr(\Sigma T)L(T), \qquad T\in \mathcal{S}_{++}^d.
\end{equation}
Now define
\[
T_t \leqdef \left(e^tT^{-1} + 2(e^t - 1)\Sigma\right)^{-1}, \qquad t\geq 0.
\]
Then $T_t\in \mathcal{S}_{++}^d$, $T_0 = T$, and $T_t\to 0_{d\times d}$ as $t\to\infty$. Since $T_t^{-1} = e^tT^{-1} + 2(e^t - 1)\Sigma$, differentiating $T_tT_t^{-1} = I_d$ gives $\smash{(\tfrac{\rd}{\rd t}T_t) T_t^{-1} + T_t (\tfrac{\rd}{\rd t}T_t^{-1}) = 0_{d\times d}}$, and thus
\[
\frac{\rd}{\rd t}T_t = - T_t \left(\frac{\rd}{\rd t}T_t^{-1}\right) T_t = -T_t(T_t^{-1} + 2\Sigma)T_t = - (T_t + 2T_t\Sigma T_t).
\]
It follows from \eqref{eq:D.L} that
\[
\frac{\rd}{\rd t}\log L(T_t) = \frac{D L(T_t)[\frac{\rd}{\rd t}T_t]}{L(T_t)} = \alpha \, \tr(\Sigma T_t).
\]
By bounded convergence, $L(T_t)\to 1$ as $t\to\infty$. Therefore,
\[
\log L(T) = -\alpha \int_0^{\infty} \tr(\Sigma T_t) \, \rd t.
\]
Using the change of variable $u = 1 - e^{-t}$ ($\rd t = e^t \rd u$), we obtain
\[
\begin{aligned}
\int_0^{\infty} \tr(\Sigma T_t) \, \rd t
&= \int_0^1 \tr\{\Sigma(T^{-1} + 2u\Sigma)^{-1}\} \, \rd u
= \frac{1}{2}\int_0^1 \tr\{(T^{-1} + 2u\Sigma)^{-1}(2\Sigma)\} \, \rd u \\
&= \frac{1}{2}\left[\log|T^{-1} + 2u\Sigma|\right]_{u = 0}^{u = 1}
= \frac{1}{2}\log\left(\frac{|T^{-1} + 2\Sigma|}{|T^{-1}|}\right)
= \frac{1}{2}\log|I_d + 2T\Sigma|,
\end{aligned}
\]
where the third equality follows from Jacobi's formula, $\smash{\tfrac{\rd}{\rd u} \log|A(u)| = \tr\{A(u)^{-1} A'(u)\}}$, with $A(u) = T^{-1} + 2u\Sigma$. Thus
\[
L(T) = |I_d + 2T\Sigma|^{-\alpha/2}, \qquad T\in \mathcal{S}_{++}^d.
\]
By bounded convergence, this extends to every $T\in \mathcal{S}_{+}^d$. By Remark~\ref{rem:Wishart.mgf} and the uniqueness of Laplace transforms on $\mathcal{S}_{+}^d$, we get $\smash{\mathfrak{W} \stackrel{\mathrm{law}}{=} \mathfrak{W}_{\infty}\sim \mathcal{W}_d(\alpha, \Sigma)}$. This concludes the proof.

\subsection{Proof of Lemma~\ref{lem:contraction}}

Fix $t > 0$ and $W\in \mathcal{S}_{++}^d$. Note that $\alpha > 3d-3$ by assumption, so $\alpha > d-1$. Therefore, by Proposition~\ref{prop:extension.Luk.1994.Lemma.2.4}, we have
\[
(\mathcal{P}_t^{\mathcal{W}} h)(W) = \EE\big[h(\Sigma_t^{1/2}\mathfrak{S}_t\Sigma_t^{1/2})\big] = (\mathcal{Q}_t h_t)(\Lambda_t),
\]
where the operator $\mathcal{Q}_t$ is defined in \eqref{eq:Wishart.Q.operator.def}, $\mathfrak{S}_t\sim \mathcal{W}_d(\alpha, I_d, e^{-2t}\Lambda_t)$, and $\Lambda_t = \Sigma_t^{-1/2}W\Sigma_t^{-1/2}$. Also, if $\mathfrak{V}\sim \mathcal{W}_d(\alpha, I_d)$, then $(\mathcal{Q}_t h_t)(0) = \EE[h(\Sigma_t^{1/2}\mathfrak{V} \Sigma_t^{1/2})]$ and $\EE[h(\mathfrak{W}_{\infty})] = \EE[h(\Sigma^{1/2} \mathfrak{V} \Sigma^{1/2})]$. Hence,
\begin{equation}\label{eq:semigroup.decomposition.Wishart}
\begin{aligned}
\big|(\mathcal{P}_t^{\mathcal{W}} h)(W)-\EE[h(\mathfrak{W}_{\infty})]\big|
&\leq \big|(\mathcal{Q}_t h_t)(\Lambda_t)-(\mathcal{Q}_t h_t)(0)\big| \\
&\qquad + \big|\EE[h(\Sigma_t^{1/2}\mathfrak{V} \Sigma_t^{1/2})]-\EE[h(\Sigma^{1/2} \mathfrak{V} \Sigma^{1/2})]\big|.
\end{aligned}
\end{equation}

We begin by bounding the first term on the right-hand side of \eqref{eq:semigroup.decomposition.Wishart}. The map $X\mapsto \Sigma_t^{1/2}X\Sigma_t^{1/2}$ is linear and maps $\mathcal{S}_{++}^d$ into $\mathcal{S}_{++}^d$, so $h_t\in C^d(\mathcal{S}_{++}^d)$. By the chain rule, each differential monomial of order at most $d$ in the entries of $\nabla_{\!X}$ applied to $h_t$ is a finite linear combination of differential monomials of the same order applied to $h$ at $\Sigma_t^{1/2}X\Sigma_t^{1/2}$, with coefficients depending only on $\Sigma_t$. Since $\|\Sigma_t^{1/2}X\Sigma_t^{1/2}\|_F \leq \|\Sigma_t\|_2\|X\|_F$, the assumption \eqref{eq:contraction.asump.h.t} implies that $h_t$ satisfies the regularity and growth assumptions required in Lemma~\ref{lem:Luk.Wishart.Lemma.2.5}. Let $\Lambda\in \mathcal{S}_{++}^d$ and $U\in \mathcal{S}^d$. By the standard identification between Fr\'echet derivatives and directional derivatives,
\[
D(\mathcal{Q}_t h_t)(\Lambda)[U]
= \langle U, \nabla_{\Lambda}\rangle_F (\mathcal{Q}_t h_t)(\Lambda)
= \sum_{i,j = 1}^d U_{ij} \, \nabla_{\Lambda, ij}(\mathcal{Q}_t h_t)(\Lambda).
\]
By applying Lemma~\ref{lem:Luk.Wishart.Lemma.2.5} with $k = 1$, we obtain
\[
D(\mathcal{Q}_t h_t)(\Lambda)[U] = e^{-2t} \, \EE\big[\langle U, \mathcal{D}\rangle_F \, h_t(\mathfrak{Y}_{1,t})\big],
\]
where $\mathfrak{Y}_{1,t}\sim \mathcal{W}_d(\alpha + 2, I_d, e^{-2t}\Lambda)$. Therefore, from the definition of $M_t^{\mathcal{D}}(h)$ in \eqref{eq:def.M.t.D.h}, it follows that
\[
|D(\mathcal{Q}_t h_t)(\Lambda)[U]| \leq e^{-2t} \, M_t^{\mathcal{D}}(h) \, \|U\|_F.
\]
Now apply the fundamental theorem of calculus along the line segment $s\mapsto s\Lambda_t$, $0\leq s\leq 1$, to obtain
\[
(\mathcal{Q}_t h_t)(\Lambda_t)-(\mathcal{Q}_t h_t)(0) = \int_0^1 D(\mathcal{Q}_t h_t)(s\Lambda_t)[\Lambda_t] \, \rd s,
\]
and hence
\begin{equation}\label{eq:first.term.bound.Wishart}
\begin{aligned}
\big|(\mathcal{Q}_t h_t)(\Lambda_t)-(\mathcal{Q}_t h_t)(0)\big|
&\leq e^{-2t} \, M_t^{\mathcal{D}}(h) \, \|\Lambda_t\|_F \\
&= \frac{e^{-2t}}{1 - e^{-2t}} \, M_t^{\mathcal{D}}(h) \, \|\Sigma^{-1/2}W\Sigma^{-1/2}\|_F.
\end{aligned}
\end{equation}

We now bound the second term on the right-hand side of \eqref{eq:semigroup.decomposition.Wishart}. Given that $\Sigma_t^{1/2}\mathfrak{V} \Sigma_t^{1/2} = (1 - e^{-2t}) \, \Sigma^{1/2} \mathfrak{V} \Sigma^{1/2}$, we have, by the Lipschitz property of $h$,
\[
\begin{aligned}
\big|\EE[h(\Sigma_t^{1/2}\mathfrak{V} \Sigma_t^{1/2})]-\EE[h(\Sigma^{1/2} \mathfrak{V} \Sigma^{1/2})]\big|
&\leq [h]_1 \, \EE\big[\|\Sigma_t^{1/2}\mathfrak{V} \Sigma_t^{1/2}-\Sigma^{1/2} \mathfrak{V} \Sigma^{1/2}\|_F\big] \\
&= e^{-2t} [h]_1 \, \EE\big[\|\Sigma^{1/2} \mathfrak{V} \Sigma^{1/2}\|_F\big] \\
&\leq e^{-2t} [h]_1 \, \|\Sigma^{1/2}\|_2^2 \, \EE[\|\mathfrak{V}\|_F] \\
&= e^{-2t} [h]_1 \, \|\Sigma\|_2 \, \EE[\|\mathfrak{V}\|_F],
\end{aligned}
\]
where $[h]_1$ denotes the minimum Lipschitz constant of $h$ with respect to the Frobenius norm. Combining this bound with \eqref{eq:semigroup.decomposition.Wishart} and \eqref{eq:first.term.bound.Wishart} proves \eqref{eq:Wishart.Holder.contraction}. This concludes the proof.

\subsection{Proof of Theorem~\ref{thm:Stein.solutions.Wishart}}

We adapt the standard Markov semigroup argument; the key input is the semigroup decay bound in Lemma~\ref{lem:contraction}.

Let $\widehat{h}(W) \leqdef \EE[h(\mathfrak{W}_{\infty})]-h(W)$ for $W\in \mathcal{S}_{++}^d$. Then we have $\|\widehat{h}\|_{\infty} \leq 2\|h\|_{\infty}$ and $[\widehat{h}]_1 = [h]_1$. Also, for $t > 0$, $\smash{\widehat{h}_t(X) \leqdef \widehat{h}(\Sigma_t^{1/2}X\Sigma_t^{1/2})
= \EE[h(\mathfrak{W}_{\infty})] - h_t(X)}$, and therefore, since $\mathcal{D}$ annihilates constants,
\[
M_{1, \Sigma}^{\mathcal{D}}(\widehat{h}) = M_{1, \Sigma}^{\mathcal{D}}(h).
\]

We first verify that the assumptions of Lemma~\ref{lem:contraction} are available uniformly for $t\geq 1$. Since $h\in C_b^d(\mathcal{S}_{++}^d)$, the regularity and growth assumption \eqref{eq:contraction.asump.h.t} holds with $N_{\overline{\nabla}} = 0$. Fix a differential monomial $\overline{\nabla}$ of order $m\leq d$ in the entries of $\nabla_{\!X}$. By the chain rule, $\overline{\nabla}h_t(X)$ is a finite linear combination of differential monomials of order $m$ applied to $h$ at $\smash{\Sigma_t^{1/2}X\Sigma_t^{1/2}}$, with coefficients depending only on $\Sigma_t$. Since $\|\Sigma_t^{1/2}\|_2^2 = \|\Sigma_t\|_2 = (1-e^{-2t})\|\Sigma\|_2\leq \|\Sigma\|_2$ for $t\geq 1$, and since $h\in C_b^d(\mathcal{S}_{++}^d)$, there exists a finite constant $C_{\overline{\nabla}, \Sigma}$ such that
\[
|\overline{\nabla}h_t(X)|\leq C_{\overline{\nabla}, \Sigma}, \qquad X\in \mathcal{S}_{++}^d, \qquad t\geq 1.
\]
Moreover, each entry of $\mathcal{D} = \nabla_{\!X}\mathrm{adj}(I_d-2\nabla_{\!X})$ is a finite linear combination, with constants depending only on $d$, of differential monomials of orders at most $d$ in the entries of $\nabla_{\!X}$. Hence
\[
M_{1, \Sigma}^{\mathcal{D}}(h) < \infty.
\]
The same conclusions hold with $h$ replaced by $\widehat{h}$.

The function defined in \eqref{eq:fh.def.Wishart.beta} can thus be rewritten as
\[
f_h(W) = \int_0^{\infty}(\mathcal{P}^{\mathcal{W}}_t\widehat{h})(W) \, \rd t.
\]

Fix $W\in \mathcal{S}_{++}^d$. For $t\in (0, 1]$, we simply have
\[
\big|(\mathcal{P}^{\mathcal{W}}_t\widehat{h})(W)\big|
\leq \|\widehat{h}\|_{\infty}
\leq 2\|h\|_{\infty}.
\]
For $t\geq 1$, Lemma~\ref{lem:contraction} applied to $\widehat{h}$ yields
\[
\begin{aligned}
\big|(\mathcal{P}^{\mathcal{W}}_t\widehat{h})(W)\big|
&= \big|\EE[h(\mathfrak{W}_{\infty})]-(\mathcal{P}^{\mathcal{W}}_t h)(W)\big| \\
&\leq e^{-2t} \left\{\frac{M_{1, \Sigma}^{\mathcal{D}}(h)}{1 - e^{-2}} \, \|\Sigma^{-1/2}W\Sigma^{-1/2}\|_F + [h]_1 \, \|\Sigma\|_2 \, \EE[\|\mathfrak{V}\|_F]\right\},
\end{aligned}
\]
where $\mathfrak{V}\sim \mathcal{W}_d(\alpha, I_d)$. Hence,
\[
\begin{aligned}
|f_h(W)|
&\leq \int_0^1 \big|(\mathcal{P}^{\mathcal{W}}_t\widehat{h})(W)\big| \, \rd t + \int_1^{\infty} \big|(\mathcal{P}^{\mathcal{W}}_t\widehat{h})(W)\big| \, \rd t \\
&\leq 2\|h\|_{\infty} + \left(\int_1^{\infty} e^{-2t} \, \rd t\right) \left\{\frac{M_{1, \Sigma}^{\mathcal{D}}(h)}{1 - e^{-2}} \, \|\Sigma^{-1/2}W\Sigma^{-1/2}\|_F + [h]_1 \, \|\Sigma\|_2 \, \EE[\|\mathfrak{V}\|_F]\right\} \\
&\leq 2\|h\|_{\infty} + \frac{1}{2} \left\{\frac{M_{1, \Sigma}^{\mathcal{D}}(h)}{1 - e^{-2}} \, \|\Sigma^{-1/2}W\Sigma^{-1/2}\|_F + [h]_1 \, \|\Sigma\|_2 \, \EE[\|\mathfrak{V}\|_F]\right\},
\end{aligned}
\]
which proves \eqref{eq:fh.bound.Wishart.beta}. In particular, $f_h$ is well defined pointwise.

We now verify that $f_h$ belongs to the domain of $\mathcal{A}^{\mathcal{W}}$ and satisfies the Wishart Stein equation \eqref{eq:Stein.equation.Wishart}. Fix $s > 0$ and $W\in \mathcal{S}_{++}^d$. By the Markov property, the semigroup property, and Fubini's theorem, we have
\begin{equation}\label{eq:semigroup.Fubini}
(\mathcal{P}_s^{\mathcal{W}} f_h)(W)
= \EE\left[\int_0^{\infty}(\mathcal{P}_t^{\mathcal{W}}\widehat{h})(\mathfrak{W}_s) \, \rd t \mid \mathfrak{W}_0 = W\right]
= \int_0^{\infty}(\mathcal{P}_{t + s}^{\mathcal{W}}\widehat{h})(W) \, \rd t,
\end{equation}
provided that
\[
\int_0^{\infty}\EE\big[|(\mathcal{P}_t^{\mathcal{W}}\widehat{h})(\mathfrak{W}_s)| \mid \mathfrak{W}_0 = W\big] \, \rd t < \infty.
\]
This integrability is immediate on $(0, 1]$ from the bound $|(\mathcal{P}_t^{\mathcal{W}}\widehat{h})(\mathfrak{W}_s)|\leq 2\|h\|_{\infty}$. On $[1, \infty)$, since $\mathfrak{W}_s\in \mathcal{S}_{++}^d$ almost surely by Proposition~\ref{prop:extension.Luk.1994.Lemma.2.4}, Lemma~\ref{lem:contraction} gives
\[
\begin{aligned}
&\EE\big[|(\mathcal{P}_t^{\mathcal{W}}\widehat{h})(\mathfrak{W}_s)| \mid \mathfrak{W}_0 = W\big] \\
&\qquad\leq e^{-2t} \left\{\frac{M_{1, \Sigma}^{\mathcal{D}}(h)}{1 - e^{-2}} \, \EE\big[\|\Sigma^{-1/2}\mathfrak{W}_s\Sigma^{-1/2}\|_F \mid \mathfrak{W}_0 = W\big]
 + [h]_1 \, \|\Sigma\|_2 \, \EE[\|\mathfrak{V}\|_F]\right\}.
\end{aligned}
\]
It therefore remains to check that, for $\mathfrak{W}_s\mid\{\mathfrak{W}_0 = W\}\sim \mathcal{W}_d(\alpha, \Sigma_s, e^{-2s} \Sigma_s^{-1}W)$, we have
\[
\EE\big[\|\Sigma^{-1/2}\mathfrak{W}_s\Sigma^{-1/2}\|_F \mid \mathfrak{W}_0 = W\big] < \infty.
\]
Since $\Sigma^{-1/2}\mathfrak{W}_s\Sigma^{-1/2}\in \mathcal{S}_{+}^d$, we have
$\|\Sigma^{-1/2}\mathfrak{W}_s\Sigma^{-1/2}\|_F\leq
\tr(\Sigma^{-1/2}\mathfrak{W}_s\Sigma^{-1/2})$, and thus
\[
\begin{aligned}
\EE\big[\|\Sigma^{-1/2}\mathfrak{W}_s\Sigma^{-1/2}\|_F \mid \mathfrak{W}_0 = W\big]
&\leq \tr\left(\Sigma^{-1/2} \, \EE[\mathfrak{W}_s\mid \mathfrak{W}_0 = W] \, \Sigma^{-1/2}\right) \\
&= \alpha \, \tr(\Sigma^{-1/2} \Sigma_s\Sigma^{-1/2}) + e^{-2s}\tr(\Sigma^{-1/2}W\Sigma^{-1/2}) \\
&= \alpha d(1 - e^{-2s}) + e^{-2s}\tr(\Sigma^{-1}W) < \infty.
\end{aligned}
\]
Therefore the application of Fubini's theorem is justified in \eqref{eq:semigroup.Fubini}.

Changing variables $u \leqdef t + s$ gives
\[
(\mathcal{P}_s^{\mathcal{W}} f_h)(W)
= \int_s^{\infty}(\mathcal{P}_u^{\mathcal{W}}\widehat{h})(W) \, \rd u
= f_h(W)-\int_0^s(\mathcal{P}_u^{\mathcal{W}}\widehat{h})(W) \, \rd u.
\]
Therefore,
\begin{equation}\label{eq:continuity.Wishart}
\frac{(\mathcal{P}_s^{\mathcal{W}} f_h)(W) - f_h(W)}{s}
= -\frac{1}{s}\int_0^s(\mathcal{P}_u^{\mathcal{W}}\widehat{h})(W) \, \rd u.
\end{equation}
Since $\widehat{h}$ is bounded and continuous, and since $\mathfrak{W}_u\to W$ almost surely as $u\downarrow 0$
by continuity of the Wishart process paths, dominated convergence yields
\[
\lim_{u\downarrow 0}(\mathcal{P}_u^{\mathcal{W}}\widehat{h})(W) = \widehat{h}(W).
\]
Letting $s\downarrow 0$ in \eqref{eq:continuity.Wishart} therefore gives
\[
\mathcal{A}^{\mathcal{W}} f_h(W)
= \lim_{s\downarrow 0}-\frac{1}{s}\int_0^s(\mathcal{P}_u^{\mathcal{W}}\widehat{h})(W) \, \rd u
= -\widehat{h}(W)
= h(W)-\EE[h(\mathfrak{W}_{\infty})],
\]
which is precisely \eqref{eq:Stein.equation.Wishart}. This completes the proof.

\subsection{Proof of Theorem~\ref{thm:smoothness.estimates.wishart}}

For $t > 0$, write $\Lambda_t(W) = \Sigma_t^{-1/2} W \Sigma_t^{-1/2}$ and $h_t(X) = h(\Sigma_t^{1/2} X \Sigma_t^{1/2})$. By Proposition~\ref{prop:extension.Luk.1994.Lemma.2.4},
\[
(\mathcal{P}_t^{\mathcal{W}} h)(W) = (\mathcal{Q}_t h_t)(\Lambda_t(W)), \qquad W\in \mathcal{S}_{++}^d.
\]

We first check that Lemma~\ref{lem:Luk.Wishart.Lemma.2.5} can be applied with $k = m$ to the family $(h_t)_{t > 0}$. Fix a differential monomial $\overline{\nabla}$ of order $q\leq md$ in the entries of $\nabla_{\!X}$. By the chain rule, $\overline{\nabla}h_t(X)$ is a finite linear combination of differential monomials of order $q$ applied to $h$ at $\smash{\Sigma_t^{1/2}X\Sigma_t^{1/2}}$, with coefficients depending only on $\Sigma_t$. Since $\|\Sigma_t\|_2 = (1-e^{-2t})\|\Sigma\|_2\leq \|\Sigma\|_2$ for $t > 0$, and since $h\in C_b^{md}(\mathcal{S}_{++}^d)$, there exists a finite constant $C_{\overline{\nabla}, \Sigma}$ such that
\[
|\overline{\nabla}h_t(X)|\leq C_{\overline{\nabla}, \Sigma}, \qquad X\in \mathcal{S}_{++}^d, \qquad t > 0.
\]
Thus the growth condition \eqref{eq:derivatives.h.assumption} holds for $h_t$ with $N_{\overline{\nabla}} = 0$. Since $\alpha > 3d-3$, Lemma~\ref{lem:Luk.Wishart.Lemma.2.5} is applicable with $k = m$.

We begin with a formula for the directional derivatives of $(\mathcal{P}_t^{\mathcal{W}} h)(W)$. Let $U_1, \ldots,U_m\in \mathcal{S}^d$, and let $\mathfrak{Y}_{m,t,W}\sim \mathcal{W}_d(\alpha + 2m, I_d, e^{-2t}\Lambda_t(W))$. Since $\Lambda_t$ is linear,
\[
D^m(\mathcal{P}_t^{\mathcal{W}} h)(W)[U_1, \ldots,U_m]
= D^m(\mathcal{Q}_t h_t)(\Lambda_t(W))[\Sigma_t^{-1/2} U_1 \Sigma_t^{-1/2}, \ldots, \Sigma_t^{-1/2} U_m \Sigma_t^{-1/2}].
\]
By the directional form of Lemma~\ref{lem:Luk.Wishart.Lemma.2.5},
\[
D^m(\mathcal{Q}_t h_t)(\Lambda)[V_1, \ldots,V_m]
= e^{-2mt} \, \EE\left[\left(\prod_{\ell=1}^m \langle V_{\ell}, \mathcal{D}\rangle_F\right) h_t(\mathfrak{Y}_{m,t}^{\Lambda})\right],
\]
where $\mathfrak{Y}_{m,t}^{\Lambda}\sim \mathcal{W}_d(\alpha + 2m, I_d, e^{-2t}\Lambda)$ and $\mathcal{D} = \nabla_{\!X} \, \mathrm{adj}(I_d-2\nabla_{\!X})$.

Fix $U\in \mathcal{S}^d$ and set $Y = \Sigma_t^{1/2} X \Sigma_t^{1/2}$. Since $\Sigma_t^{1/2}$ is constant and symmetric, the chain rule gives $\smash{\nabla_{\!X} = \Sigma_t^{1/2} \nabla_{\!Y} \Sigma_t^{1/2}}$ when both sides act on $\smash{h_t(X) = h(\Sigma_t^{1/2} X \Sigma_t^{1/2})}$. Hence
\[
\mathcal{D} h_t(X) = (\Sigma_t^{1/2}\nabla_{\!Y} \Sigma_t^{1/2})\mathrm{adj}(I_d - 2 \Sigma_t^{1/2}\nabla_{\!Y} \Sigma_t^{1/2}) \, h(Y).
\]
Note that
\[
\begin{aligned}
\mathrm{adj}(I_d - 2 \Sigma_t^{1/2}\nabla_{\!Y} \Sigma_t^{1/2})
&= \mathrm{adj}(\Sigma_t^{-1/2}(I_d - 2\Sigma_t\nabla_{\!Y})\Sigma_t^{1/2}) \\
&= \mathrm{adj}(\Sigma_t^{1/2}) \mathrm{adj}(I_d - 2\Sigma_t\nabla_{\!Y}) \mathrm{adj}(\Sigma_t^{-1/2}) \\
&= \Sigma_t^{-1/2} \mathrm{adj}(I_d - 2\Sigma_t\nabla_{\!Y}) \Sigma_t^{1/2},
\end{aligned}
\]
where the scalar determinant factors cancel in the last equality. Consequently,
\[
\mathcal{D} h_t(X)
= \Sigma_t^{1/2} \Big\{\nabla_{\!Y} \mathrm{adj}(I_d - 2\Sigma_t\nabla_{\!Y}) h(Y)\Big\} \Sigma_t^{1/2}
= \Sigma_t^{1/2} (\mathcal{D}_{t, \Sigma} h)(Y) \Sigma_t^{1/2}.
\]
Taking Frobenius inner products,
\[
\langle \Sigma_t^{-1/2} U \Sigma_t^{-1/2}, \mathcal{D}\rangle_F h_t(X)
= \langle \Sigma_t^{-1/2} U \Sigma_t^{-1/2}, \mathcal{D} h_t(X)\rangle_F
= \langle U,(\mathcal{D}_{t, \Sigma} h)(Y)\rangle_F
= \langle U, \mathcal{D}_{t, \Sigma}\rangle_F h(Y).
\]
Since all involved operators have constant coefficients and commute, iteration yields
\[
\left(\prod_{\ell=1}^m \langle \Sigma_t^{-1/2} U_{\ell} \Sigma_t^{-1/2}, \mathcal{D}\rangle_F\right) h_t(X)
= \left(\prod_{\ell=1}^m \langle U_{\ell}, \mathcal{D}_{t, \Sigma}\rangle_F\right) h(\Sigma_t^{1/2} X \Sigma_t^{1/2}).
\]
If we now define $\mathfrak{Z}_{m,t,W} \leqdef \Sigma_t^{1/2}\mathfrak{Y}_{m,t,W} \Sigma_t^{1/2}$, then
\begin{equation}\label{eq:general.Wishart.directional.derivative}
D^m(\mathcal{P}_t^{\mathcal{W}} h)(W)[U_1, \ldots,U_m]
= e^{-2mt} \, \EE\left[\left(\prod_{\ell=1}^m \langle U_{\ell}, \mathcal{D}_{t, \Sigma}\rangle_F\right) h(\mathfrak{Z}_{m,t,W})\right].
\end{equation}

We now prove \eqref{eq:Wishart.Stein.bound.1}. Fix index pairs $(i_1,j_1), \ldots,(i_m,j_m)\in [d]^2$. For $1\leq i,j\leq d$, define
\[
H^{(ij)} \leqdef \frac{1}{2}(\bb{e}_i \bb{e}_j^{\top} + \bb{e}_j \bb{e}_i^{\top})\in \mathcal{S}^d.
\]
The matrix $\mathcal{D}_{t, \Sigma}$ is symmetric as a matrix of commuting differential operators. Indeed, with $A = \nabla_{\!W}$ and $B = 2\Sigma_t$, the identity $A\mathrm{adj}(I_d-BA) = \mathrm{adj}(I_d-AB)A$ gives
\[
\mathcal{D}_{t, \Sigma}
= \nabla_{\!W}\mathrm{adj}(I_d-2\Sigma_t\nabla_{\!W})
= \mathrm{adj}(I_d-2\nabla_{\!W}\Sigma_t)\nabla_{\!W},
\]
and taking transposes gives $\mathcal{D}_{t, \Sigma}^{\top} = \mathcal{D}_{t, \Sigma}$. By the symmetric-gradient convention,
\[
\langle H^{(ij)}, \nabla_{\!W}\rangle_F = \nabla_{\!W,ij} \qquad \text{and} \qquad \langle H^{(ij)}, \mathcal{D}_{t, \Sigma}\rangle_F = (\mathcal{D}_{t, \Sigma})_{ij}.
\]
Applying \eqref{eq:general.Wishart.directional.derivative} with $U_{\ell} = H^{(i_{\ell} j_{\ell})}$, we obtain
\[
\left(\prod_{\ell=1}^m \nabla_{\!W,i_{\ell} j_{\ell}}\right)(\mathcal{P}_t^{\mathcal{W}} h)(W)
= e^{-2mt} \, \EE\left[\left(\prod_{\ell=1}^m (\mathcal{D}_{t, \Sigma})_{i_{\ell} j_{\ell}}\right) h(\mathfrak{Z}_{m,t,W})\right].
\]
Hence
\[
\left|\left(\prod_{\ell=1}^m \nabla_{\!W,i_{\ell} j_{\ell}}\right)(\mathcal{P}_t^{\mathcal{W}} h)(W)\right|
\leq e^{-2mt} \, \sup_{s > 0} \left\|\left(\prod_{\ell=1}^m (\mathcal{D}_{s, \Sigma})_{i_{\ell} j_{\ell}}\right) h\right\|_{\infty}.
\]
Since $\EE[h(\mathfrak{W}_{\infty})]$ is constant in $W$, its derivatives vanish. The right-hand side is integrable in $t$, so dominated convergence allows differentiation under the integral sign in \eqref{eq:fh.def.Wishart.beta}, giving
\[
\left(\prod_{\ell=1}^m \nabla_{\!W,i_{\ell} j_{\ell}}\right) f_h(W)
= -\int_0^{\infty} \left(\prod_{\ell=1}^m \nabla_{\!W,i_{\ell} j_{\ell}}\right) (\mathcal{P}_t^{\mathcal{W}} h)(W) \, \rd t.
\]
Therefore,
\[
\begin{aligned}
\left|\left(\prod_{\ell=1}^m \nabla_{\!W,i_{\ell} j_{\ell}}\right) f_h(W)\right|
&\leq \int_0^{\infty} e^{-2mt} \, \rd t \, \sup_{s > 0} \left\|\left(\prod_{\ell=1}^m (\mathcal{D}_{s, \Sigma})_{i_{\ell} j_{\ell}}\right) h\right\|_{\infty} \\
&= \frac{1}{2m} \sup_{s > 0} \left\|\left(\prod_{\ell=1}^m (\mathcal{D}_{s, \Sigma})_{i_{\ell} j_{\ell}}\right) h\right\|_{\infty},
\end{aligned}
\]
which is exactly \eqref{eq:Wishart.Stein.bound.1}.

Next, we prove \eqref{eq:Wishart.Stein.bound.4}. Assume that $\mathcal{M}_m^{\mathcal{D}, \Sigma}(h) < \infty$. For unit vectors $U_1, \ldots,U_m\in \mathcal{S}^d$, \eqref{eq:general.Wishart.directional.derivative} gives
\[
\big|D^m(\mathcal{P}_t^{\mathcal{W}} h)(W)[U_1, \ldots,U_m]\big|
\leq e^{-2mt} \, \mathcal{M}_m^{\mathcal{D}, \Sigma}(h).
\]
Again, the right-hand side is integrable in $t$, so dominated convergence yields
\[
D^m f_h(W)[U_1, \ldots,U_m]
= -\int_0^{\infty} D^m(\mathcal{P}_t^{\mathcal{W}} h)(W)[U_1, \ldots,U_m] \, \rd t.
\]
Hence, whenever $\|U_1\|_F = \cdots = \|U_m\|_F = 1$,
\[
\big|D^m f_h(W)[U_1, \ldots,U_m]\big|
\leq \int_0^{\infty} e^{-2mt} \, \rd t \, \mathcal{M}_m^{\mathcal{D}, \Sigma}(h)
= \frac{1}{2m} \, \mathcal{M}_m^{\mathcal{D}, \Sigma}(h).
\]
Taking the suprema over $W\in \mathcal{S}_{++}^d$ and over $\|U_1\|_F = \cdots = \|U_m\|_F = 1$ yields \eqref{eq:Wishart.Stein.bound.4}. This completes the proof.

\section{Proofs of the application results}\label{sec:proofs.application.results}

\subsection{Proof of Lemma~\ref{lem:matrix.normal.Wishart.operator.connection}}\label{sec:proof.matrix.normal.Wishart.operator.connection}

The set $\mathcal{R}_{\nu,d}$ is open because it is the inverse image of $\mathcal{S}_{++}^d$ under the continuous map $X\mapsto X^{\top}X$. Its complement has Lebesgue measure zero because it is contained in the zero set of any fixed $d\times d$ minor, viewed as a nonzero polynomial on $\R^{\nu\times d}$. Let $X\in \mathcal{R}_{\nu,d}$, and set $S \leqdef X^{\top}X$. By the chain rule, $\nabla_{\!X} g_f(X) = 2 X \nabla f(S)$, so the drift part of \eqref{eq:MN.OU.generator.mean.zero.application} is
\[
-\tr\{X^{\top}\nabla_{\!X}g_f(X)\} = -2 \, \tr\{X^{\top}X\nabla f(S)\} = -2 \, \tr\{S\nabla f(S)\}.
\]

It remains to compute the diffusion part of \eqref{eq:MN.OU.generator.mean.zero.application}. In coordinates, $\nabla_{\!X} g_f(X) = 2 X \nabla f(S)$ gives
\[
\nabla_{\!X,ar}g_f(X) = 2\sum_{u = 1}^d X_{au}\nabla_{ur}f(S).
\]
Differentiating once more with respect to $X_{as}$ yields
\[
\nabla_{\!X,as}\nabla_{\!X,ar}g_f(X) = 2\nabla_{sr}f(S) + 2\sum_{u = 1}^d X_{au}\nabla_{\!X,as}\nabla_{ur}f(S).
\]
Since $S_{pq} = \sum_{b = 1}^{\nu} X_{bp}X_{bq}$, we have $\nabla_{\!X,as}S_{pq} = \delta_{ps}X_{aq} + \delta_{qs}X_{ap}$. Using the symmetric-gradient convention in the $S$ variable, it follows that
\[
\nabla_{\!X,as}\nabla_{ur}f(S) = \sum_{p,q = 1}^d(\delta_{ps}X_{aq} + \delta_{qs}X_{ap})\nabla_{pq}\nabla_{ur}f(S) = 2\sum_{v = 1}^d X_{av}\nabla_{sv}\nabla_{ur}f(S).
\]
Consequently,
\[
\nabla_{\!X,as}\nabla_{\!X,ar}g_f(X) = 2\nabla_{sr}f(S) + 4\sum_{u,v = 1}^d X_{au}X_{av}\nabla_{sv}\nabla_{ur}f(S).
\]
Substituting this identity into the second term of \eqref{eq:MN.OU.generator.mean.zero.application}, we obtain
\[
\begin{aligned}
\tr\{\Sigma\nabla_{\!X}^{\top}\nabla_{\!X}g_f(X)\}
&= \sum_{a = 1}^{\nu}\sum_{r,s = 1}^d \Sigma_{rs}\nabla_{\!X,ar}\nabla_{\!X,as}g_f(X) \\
&= 2 \nu \, \tr\{\Sigma\nabla f(S)\} + 4 \sum_{u,v = 1}^d S_{uv} \sum_{r,s = 1}^d \Sigma_{rs}\nabla_{sv}\nabla_{ur}f(S) \\
&= 2 \nu \, \tr\{\Sigma\nabla f(S)\} + 4 \, \tr\{S\nabla\Sigma\nabla f(S)\}.
\end{aligned}
\]
Combining the drift and diffusion terms gives, for $X\in \mathcal{R}_{\nu,d}$,
\[
\mathcal{A}^{\mathrm{OU}}g_f(X) = 2 \, \tr\{(\nu\Sigma - X^{\top}X)\nabla f(X^{\top}X)\} + 4 \, \tr\{X^{\top}X \nabla \Sigma \nabla f(X^{\top}X)\},
\]
which is \eqref{eq:MN.Wishart.pullback.identity}.

\subsection{Proof of Proposition~\ref{prop:Wishart.MANOVA}}

Let $\mathfrak{N}\sim \mathcal{N}_{\nu\times d}(0_{\nu\times d},I_{\nu}\otimes\Sigma)$ and note that $\mathfrak{W} \stackrel{\mathrm{law}}{ = } \mathfrak{N}^{\top}\mathfrak{N}\sim \mathcal{W}_d(\nu, \Sigma)$. Let $f_h$ denote the solution of the Wishart Stein equation associated with $h$. Throughout the proof, we write $\mathcal{I} \leqdef [\nu]\times [d]$ and, for $i\in[n]$, $A,B\in \mathcal{I}$ and $\theta\in[0,1]$, set $Z_{i,A}\leqdef(\mathfrak{Z}_i)_A$, $Z_A\leqdef Z_{1,A}$, $\sigma_{A,B}\leqdef \EE[Z_AZ_B]$, and
\[
\mathfrak{X}_{-i}^{(n)} \leqdef \mathfrak{X}^{(n)}-\frac{1}{\sqrt{n}}\mathfrak{Z}_i, \qquad \mathfrak{X}_{i, \theta}^{(n)} \leqdef \mathfrak{X}_{-i}^{(n)}+\frac{\theta}{\sqrt{n}}\mathfrak{Z}_i.
\]
For $\varepsilon>0$, set
\[
g_{h, \varepsilon}(X) \leqdef f_h(X^{\top}X+\varepsilon I_d), \qquad S_{\varepsilon}(X) \leqdef X^{\top}X+\varepsilon I_d, \qquad X\in \R^{\nu\times d}.
\]
Since $S_{\varepsilon}(\mathfrak{X}^{(n)})\in \mathcal{S}_{++}^d$ almost surely, the Wishart Stein equation and \eqref{eq:MN.Wishart.regularized.identity} give
\[
\begin{aligned}
\EE[h(S_{\varepsilon}(\mathfrak{X}^{(n)}))] - \EE[h(\mathfrak{W})]
&= \EE\left[\left(\mathcal{A}^{\mathcal{W}}_{\nu, \Sigma}f_h\right)(S_{\varepsilon}(\mathfrak{X}^{(n)}))\right] \\
&= \EE[\mathcal{A}^{\mathrm{OU}}g_{h, \varepsilon}(\mathfrak{X}^{(n)})] - 2\varepsilon \, \EE[\tr\{\nabla f_h(S_{\varepsilon}(\mathfrak{X}^{(n)}))\}] \\
&\qquad + 4\varepsilon \, \EE[\tr\{I_d \nabla \Sigma \nabla f_h(S_{\varepsilon}(\mathfrak{X}^{(n)}))\}].
\end{aligned}
\]
Moreover, by the assumption $\smash{\sum_{m=1}^5 \mathcal{M}_m^{\mathcal{D}, \Sigma}(h) < \infty}$ and \eqref{eq:Wishart.Stein.bound.4} in Theorem~\ref{thm:smoothness.estimates.wishart}, applied with $m\in\{1,2\}$, we have $\mathcal{M}_1(f_h)<\infty$ and $\mathcal{M}_2(f_h)<\infty$. Therefore, there is a finite constant $C_h$, independent of $n$ and $\varepsilon\in(0,1]$, such that
\begin{equation}\label{eq:Wishart.MANOVA.regularized.comparison}
\left|\EE[h(S_{\varepsilon}(\mathfrak{X}^{(n)}))] - \EE[h(\mathfrak{W})] - \EE[\mathcal{A}^{\mathrm{OU}}g_{h, \varepsilon}(\mathfrak{X}^{(n)})]\right| \leq C_h\varepsilon.
\end{equation}

Fix $\varepsilon\in(0,1]$. Using the coordinate form of the matrix normal Ornstein--Uhlenbeck operator, we have
\begin{equation}\label{eq:Wishart.MANOVA.normal.identity}
\EE[\mathcal{A}^{\mathrm{OU}}g_{h, \varepsilon}(\mathfrak{X}^{(n)})]
= \EE\left[\sum_{A,B\in \mathcal{I}}\sigma_{A,B}\partial_A\partial_B g_{h, \varepsilon}(\mathfrak{X}^{(n)})\right] - \EE\left[\sum_{A\in \mathcal{I}}(\mathfrak{X}^{(n)})_A\partial_A g_{h, \varepsilon}(\mathfrak{X}^{(n)})\right].
\end{equation}

We next record derivative bounds for $g_{h, \varepsilon}$ that are uniform in $\varepsilon$. We have
\[
D(X^{\top}X+\varepsilon I_d)[U] = X^{\top} U + U^{\top} X, \qquad D^2(X^{\top}X+\varepsilon I_d)[U,V] = U^{\top} V + V^{\top} U,
\]
and $D^k(X^{\top}X+\varepsilon I_d) = 0$ for $k\geq 3$. Hence, if $\|U_1\|_F = \dots = \|U_k\|_F = 1$, then
\[
\|D(X^{\top}X+\varepsilon I_d)[U_j]\|_F \leq 2 \|X\|_F, \qquad \|D^2(X^{\top}X+\varepsilon I_d)[U_i,U_j]\|_F \leq 2.
\]
By Fa\`a di Bruno's formula, $D^kg_{h, \varepsilon}(X)$ is a sum over partitions of $\{1, \ldots,k\}$ into blocks of size one and two. If $\ell$ is the number of two-blocks, then the number of such partitions is $k!/(2^{\ell}\ell!(k-2\ell)!)$, and the corresponding term is bounded by $2^{k-\ell}\|X\|_F^{k-2\ell}\mathcal{M}_{k-\ell}(f_h)$. Consequently, for
\begin{equation}\label{eq:def.G.k.h}
\mathcal{G}_{k,h}(x) \leqdef \sum_{\ell=0}^{\lfloor k/2\rfloor} \frac{2^{k-2\ell}k!}{\ell!(k-2\ell)!} x^{k-2\ell} \mathcal{M}_{k-\ell}(f_h),
\end{equation}
we have
\begin{equation}\label{eq:gh.derivative.general.bound}
\sup_{\substack{U_1, \ldots,U_k\in \R^{\nu\times d}\\ \|U_1\|_F = \cdots = \|U_k\|_F = 1}} |D^kg_{h, \varepsilon}(X)[U_1, \ldots,U_k]| \leq \mathcal{G}_{k,h}(\|X\|_F).
\end{equation}

We now expand the two terms in \eqref{eq:Wishart.MANOVA.normal.identity}. Since $\mathfrak{X}_{-i}^{(n)}$ and $\mathfrak{Z}_i$ are independent, $\EE[Z_{i,A}]=0$, and $\EE[Z_{i,A}Z_{i,B}]=\sigma_{A,B}$, a Taylor expansion of $\partial_A g_{h, \varepsilon}$ around $\smash{\mathfrak{X}_{-i}^{(n)}}$ gives
\[
\begin{aligned}
\sum_{A\in \mathcal{I}}\EE[(\mathfrak{X}^{(n)})_A\partial_A g_{h, \varepsilon}(\mathfrak{X}^{(n)})]
&= \frac{1}{\sqrt{n}}\sum_{i=1}^n\sum_{A\in \mathcal{I}}\EE\left[Z_{i,A}\partial_A g_{h, \varepsilon}\left(\mathfrak{X}_{-i}^{(n)} + \frac{1}{\sqrt{n}}\mathfrak{Z}_i\right)\right] \\
&= \frac{1}{n}\sum_{i=1}^n\sum_{A,B\in \mathcal{I}}\sigma_{A,B}\EE[\partial_A\partial_B g_{h, \varepsilon}(\mathfrak{X}_{-i}^{(n)})] \\
&\qquad + \frac{1}{2 n^{3/2}} \sum_{i=1}^n \sum_{A,B,C\in \mathcal{I}}\EE[Z_{A}Z_{B}Z_{C}] \EE[\partial_A\partial_B\partial_C \, g_{h, \varepsilon}(\mathfrak{X}_{-i}^{(n)})] + R_{1, \varepsilon},
\end{aligned}
\]
where, by \eqref{eq:gh.derivative.general.bound},
\[
|R_{1, \varepsilon}| \leq \frac{1}{6n}\sum_{A,B,C,D\in \mathcal{I}}
\EE\left[|Z_{A}Z_{B}Z_{C}Z_{D}| \, \mathcal{G}_{4,h}(\|\mathfrak{X}_{1,\theta_1}^{(n)}\|_F)\right],
\]
for some $\theta_1\in[0,1]$. Similarly, a Taylor expansion of $\partial_A\partial_B g_{h, \varepsilon}$ around $\mathfrak{X}_{-i}^{(n)}$ gives
\[
\sum_{A,B\in \mathcal{I}}\sigma_{A,B}\EE[\partial_A\partial_B g_{h, \varepsilon}(\mathfrak{X}^{(n)})]
= \frac{1}{n}\sum_{i=1}^n\sum_{A,B\in \mathcal{I}}\sigma_{A,B}\EE[\partial_A\partial_B g_{h, \varepsilon}(\mathfrak{X}_{-i}^{(n)})] + R_{2, \varepsilon},
\]
where, by \eqref{eq:gh.derivative.general.bound},
\[
|R_{2, \varepsilon}| \leq \frac{1}{2n}\sum_{A,B,C,D\in \mathcal{I}}|\sigma_{A,B}|\EE\left[|Z_{C}Z_{D}| \, \mathcal{G}_{4,h}(\|\mathfrak{X}_{1,\theta_2}^{(n)}\|_F)\right],
\]
for some $\theta_2\in[0,1]$. Also, a Taylor expansion of $\partial_A\partial_B\partial_C \, g_{h, \varepsilon}$ around $\mathfrak{X}^{(n)}$ gives
\[
\begin{aligned}
&\frac{1}{2 n^{3/2}} \sum_{i=1}^n \sum_{A,B,C\in \mathcal{I}}\EE[Z_{A}Z_{B}Z_{C}] \EE[\partial_A\partial_B\partial_C \, g_{h, \varepsilon}(\mathfrak{X}_{-i}^{(n)})] \\
&\qquad= \frac{1}{2\sqrt{n}} \sum_{A,B,C\in \mathcal{I}}\EE[Z_{A}Z_{B}Z_{C}] \EE[\partial_A\partial_B\partial_C \, g_{h, \varepsilon}(\mathfrak{X}^{(n)})] + R_{3, \varepsilon},
\end{aligned}
\]
where, by \eqref{eq:gh.derivative.general.bound} and the assumption that $\mathfrak{Z}_1,\ldots,\mathfrak{Z}_n$ are identically distributed,
\[
|R_{3, \varepsilon}| \leq \frac{1}{2n}\sum_{A,B,C,D\in \mathcal{I}}|\EE[Z_{A}Z_{B}Z_{C}]|\EE\left[|Z_{D}| \, \mathcal{G}_{4,h}(\|\mathfrak{X}_{1,\theta_3}^{(n)}\|_F)\right],
\]
for some $\theta_3\in[0,1]$. Combining the preceding displays yields
\begin{equation}\label{eq:Wishart.MANOVA.after.first.expansion}
\begin{aligned}
|\EE[\mathcal{A}^{\mathrm{OU}}g_{h, \varepsilon}(\mathfrak{X}^{(n)})]|
&\leq |R_{1, \varepsilon}| + |R_{2, \varepsilon}| + |R_{3, \varepsilon}| \\
&\qquad+ \frac{1}{2\sqrt{n}} \sum_{A,B,C\in \mathcal{I}} |\EE[Z_{A}Z_{B}Z_{C}]| |\EE[\partial_A\partial_B\partial_C \, g_{h, \varepsilon}(\mathfrak{X}^{(n)})]|.
\end{aligned}
\end{equation}

It remains to bound the third-derivative term using symmetry considerations. Fix $A,B,C\in \mathcal{I}$. The function $g_{h, \varepsilon}$ is even because $g_{h, \varepsilon}(-Y)=g_{h, \varepsilon}(Y)$. Hence $\partial_A\partial_B\partial_C \, g_{h, \varepsilon}$ is odd. Since $\mathfrak{N}$ is centered matrix normal,
\[
\EE[\partial_A\partial_B\partial_C \, g_{h, \varepsilon}(\mathfrak{N})] = 0.
\]
By \eqref{eq:gh.derivative.general.bound}, the function $Y\mapsto \partial_A\partial_B\partial_C \, g_{h, \varepsilon}(Y)$ has polynomial growth. We use Theorem~3.4 of \citet{GauntOuimetRichards2026} through smooth compactly supported truncations: the theorem is first applied to the truncated functions and then the truncation is removed by dominated convergence, using \eqref{eq:gh.derivative.general.bound} and the finiteness of Gaussian moments. This yields a solution $\Psi_{A,B,C}^{(\varepsilon)}$ of the matrix normal Stein equation
\[
\sum_{D,E\in \mathcal{I}}\sigma_{D,E}\partial_D\partial_E\Psi_{A,B,C}^{(\varepsilon)}(Y) - \sum_{D\in \mathcal{I}}Y_D\partial_D\Psi_{A,B,C}^{(\varepsilon)}(Y) = \partial_A\partial_B\partial_C \, g_{h, \varepsilon}(Y), \qquad Y\in \R^{\nu\times d}.
\]
Equivalently, $\Psi_{A,B,C}^{(\varepsilon)}$ is the semigroup solution
\[
\Psi_{A,B,C}^{(\varepsilon)}(Y) = -\int_0^{\infty}\EE\left[\partial_A\partial_B\partial_C \, g_{h, \varepsilon}\left(e^{-t}Y + \sqrt{1-e^{-2t}}\, \mathfrak{N}'\right)\right]\rd t,
\]
where $\mathfrak{N}'$ is an independent copy of $\mathfrak{N}$. The same domination justifies the differentiations under the integral sign below.

Arguing as we did to derive the bound \eqref{eq:Wishart.MANOVA.after.first.expansion}, but expanding only up to third-order partial derivatives, we obtain the following:
\begin{equation}\label{eq:R4.R5.intermediate}
\left|\EE[\partial_A\partial_B\partial_C g_{h, \varepsilon}(\mathfrak{X}^{(n)})]\right| \leq |R_{4, \varepsilon}^{A,B,C}| + |R_{5, \varepsilon}^{A,B,C}|,
\end{equation}
where
\begin{align}
|R_{4, \varepsilon}^{A,B,C}| \leq \frac{1}{2\sqrt{n}}\sum_{D,E,F\in \mathcal{I}}\EE\left[|Z_{D}Z_{E}Z_{F}| \, |\partial_D\partial_E\partial_F\Psi_{A,B,C}^{(\varepsilon)}(\mathfrak{X}_{1, \theta_4}^{(n)})|\right], \label{eq:R4.intermediate} \\
|R_{5, \varepsilon}^{A,B,C}| \leq \frac{1}{\sqrt{n}}\sum_{D,E,F\in \mathcal{I}}|\sigma_{D,E}|\EE\left[|Z_{F}| \, |\partial_D\partial_E\partial_F\Psi_{A,B,C}^{(\varepsilon)}(\mathfrak{X}_{1, \theta_5}^{(n)})|\right], \label{eq:R5.intermediate}
\end{align}
for some $\theta_4,\theta_5\in[0,1]$.

Differentiating twice under the integral sign and then using Gaussian integration by parts to take the third derivative (using a similar calculation to the one used in the proof of inequality (3.14) of \cite{GauntOuimetRichards2026}) gives
\[
\begin{aligned}
&\partial_D\partial_E\partial_F\Psi_{A,B,C}^{(\varepsilon)}(Y) \\
&\qquad= -\int_0^{\infty} \frac{e^{-3t}}{\sqrt{1-e^{-2t}}} \sum_{Q\in \mathcal{I}} (I_{\nu} \otimes \Sigma^{-1})_{F,Q} \EE[(\mathfrak{N}')_Q \partial_A\partial_B\partial_C\partial_D\partial_E \, g_{h, \varepsilon}(e^{-t}Y + \sqrt{1-e^{-2t}}\, \mathfrak{N}')]\rd t.
\end{aligned}
\]
Note that $\int_0^{\infty} e^{-3t} (1-e^{-2t})^{-1/2} \, \rd t = \pi/4$. Also, since $e^{-t}\leq 1$ and $\sqrt{1-e^{-2t}} \leq 1$, we have
\[
\|e^{-t}Y + \sqrt{1-e^{-2t}}\, \mathfrak{N}'\|_F^p \leq (\|Y\|_F+\|\mathfrak{N}'\|_F)^p \leq (1+\|Y\|_F)^p (1+\|\mathfrak{N}'\|_F)^p, \qquad p \geq 0.
\]
Consequently, we have the bound
\[
|\partial_D\partial_E\partial_F\Psi_{A,B,C}^{(\varepsilon)}(Y)| \leq M \, \mathcal{G}_{5,h}(1+\|Y\|_F),
\]
where $\|\Sigma^{-1}\|_{\infty} = \|(I_{\nu} \otimes \Sigma^{-1})\|_{\infty} \equiv \max_{F\in \mathcal{I}} \sum_{Q\in \mathcal{I}} |(I_{\nu} \otimes \Sigma^{-1})_{F,Q}|$, and
\[
M \leqdef \frac{\pi}{4} \|\Sigma^{-1}\|_{\infty} \max_{\substack{p\in\{1,3,5\} \\ Q\in \mathcal{I}}} \EE\left[|(\mathfrak{N}')_Q|(1+\|\mathfrak{N}'\|_F)^p\right].
\]
It follows from \eqref{eq:R4.intermediate} and \eqref{eq:R5.intermediate} that
\[
\begin{aligned}
|R_{4, \varepsilon}^{A,B,C}| + |R_{5, \varepsilon}^{A,B,C}|
&\leq \frac{M}{\sqrt{n}}\sum_{D,E,F\in \mathcal{I}} \left\{\frac{1}{2} \EE\left[|Z_{D}Z_{E}Z_{F}| \, \mathcal{G}_{5,h}(1 + \|\mathfrak{X}_{1,\theta_4}^{(n)}\|_F) \right]\right. \\
&\hspace{40mm}\left.+ |\sigma_{D,E}| \EE\left[|Z_F| \, \mathcal{G}_{5,h}(1 + \|\mathfrak{X}_{1,\theta_5}^{(n)}\|_F)\right]\right\}.
\end{aligned}
\]

We shall use the following explicit upper bound on $M$. Let $\mathfrak{G}\in \R^{\nu\times d}$ have iid standard normal entries. Since $\smash{\mathfrak{N}' \stackrel{\mathrm{law}}{=} \mathfrak{G}\Sigma^{1/2}}$, we have $\|\mathfrak{N}'\|_F \leq \|\Sigma^{1/2}\|_2\|\mathfrak{G}\|_F = \|\Sigma\|_2^{1/2}\|\mathfrak{G}\|_F$. Also, for every $Q\in \mathcal{I}$ and every $p\in\{1,3,5\}$,
\[
|(\mathfrak{N}')_Q|(1+\|\mathfrak{N}'\|_F)^p \leq \|\mathfrak{N}'\|_F(1+\|\mathfrak{N}'\|_F)^p \leq (1 + \|\mathfrak{N}'\|_F)^6 \leq 2^5 \{1 + \|\mathfrak{N}'\|_F^6\}.
\]
Since $\|\mathfrak{G}\|_F^2\sim \chi^2_{\nu d}$, it follows that $\EE[\|\mathfrak{G}\|_F^6] = \nu d(\nu d+2)(\nu d+4)$, and thus
\[
M \leq 8\pi \|\Sigma^{-1}\|_{\infty} \left\{1 + \|\Sigma\|_2^3 \nu d(\nu d+2)(\nu d+4)\right\} \equiv K.
\]

Combining \eqref{eq:Wishart.MANOVA.after.first.expansion} and \eqref{eq:R4.R5.intermediate}, we have the following bound:
\[
|\EE[\mathcal{A}^{\mathrm{OU}}g_{h, \varepsilon}(\mathfrak{X}^{(n)})]|
\leq |R_{1, \varepsilon}| + |R_{2, \varepsilon}| + |R_{3, \varepsilon}| + \frac{1}{2\sqrt{n}} \sum_{A,B,C\in \mathcal{I}} |\EE[Z_{A}Z_{B}Z_{C}]| (|R_{4, \varepsilon}^{A,B,C}| + |R_{5, \varepsilon}^{A,B,C}|).
\]
We now find bounds on $\smash{|R_{1, \varepsilon}|}$, $\smash{|R_{2, \varepsilon}|}$, $\smash{|R_{3, \varepsilon}|}$, $\smash{|R_{4, \varepsilon}^{A,B,C}|}$, and $\smash{|R_{5, \varepsilon}^{A,B,C}|}$ that do not involve $\theta_1,\ldots,\theta_5$ or moments of $\mathfrak{X}_{-1}^{(n)}$. For $q\in\{0, \ldots,8\}$, set
\[
\mu_q \leqdef \EE[\|\mathfrak{Z}_1\|_F^q], \qquad
\Lambda_q \leqdef
\begin{cases}
1, & q=0, \\
\mu_2^{1/2}, & q=1, \\
2^q(q-1)^{q/2}\mu_q, & 2\leq q\leq8.
\end{cases}
\]
We claim that, for $0\leq q\leq8$,
\begin{equation}\label{eq:leave.one.out.moment.only.Z.explicit}
\EE[\|\mathfrak{X}_{-1}^{(n)}\|_F^q]\leq \Lambda_q.
\end{equation}
The case $q=0$ is immediate. For $q=1$, by the Cauchy-Schwarz inequality and the independence and centering of $\mathfrak{Z}_2, \ldots, \mathfrak{Z}_n$,
\[
\EE[\|\mathfrak{X}_{-1}^{(n)}\|_F] \leq \left(\EE[\|\mathfrak{X}_{-1}^{(n)}\|_F^2]\right)^{1/2} = \left(\frac{1}{n}\sum_{i=2}^n \EE[\|\mathfrak{Z}_i\|_F^2]\right)^{1/2} = \left(\frac{n-1}{n}\mu_2\right)^{1/2} \leq \mu_2^{1/2}.
\]
Now let $2\leq q\leq8$. By symmetrization and Khintchine's inequality, for independent Rademacher random variables $\varepsilon_2, \ldots, \varepsilon_n$, independent of $\mathfrak{Z}_2, \ldots, \mathfrak{Z}_n$, we have
\[
\begin{aligned}
\EE[\|\mathfrak{X}_{-1}^{(n)}\|_F^q]
&= \frac{1}{n^{q/2}}\EE\left[\left\|\sum_{i=2}^n \mathfrak{Z}_i\right\|_F^q\right]
\leq \frac{2^q}{n^{q/2}}\EE\left[\left\|\sum_{i=2}^n \varepsilon_i\mathfrak{Z}_i\right\|_F^q\right] \\
&\leq \frac{2^q(q-1)^{q/2}}{n^{q/2}}\EE\left[\left(\sum_{i=2}^n \|\mathfrak{Z}_i\|_F^2\right)^{q/2}\right]
\leq 2^q(q-1)^{q/2}\mu_q,
\end{aligned}
\]
where we applied Jensen's inequality to obtain the last bound. This proves \eqref{eq:leave.one.out.moment.only.Z.explicit}.

For $r\in \N$ and $q\in \N_0$ with $r+q\leq8$, define
\[
\mathcal{L}_{q,r} \leqdef
\begin{cases}
\mu_r, & q=0, \\
2^{q-1}\{\Lambda_q\mu_r+\mu_{r+q}\}, & q\geq 1,
\end{cases} \qquad
\mathcal{L}_{q,r}^{+} \leqdef
\begin{cases}
\mu_r, & q=0, \\
2^{q-1}\left\{2^{q-1}(1+\Lambda_q)\mu_r+\mu_{r+q}\right\}, & q\geq 1.
\end{cases}
\]
Indeed, since $\mathfrak{X}_{1, \theta}^{(n)}=\mathfrak{X}_{-1}^{(n)}+\theta n^{-1/2}\mathfrak{Z}_1$, the independence of $\mathfrak{X}_{-1}^{(n)}$ and $\mathfrak{Z}_1$ gives, for $q\geq1$ and $A_1, \ldots,A_r\in \mathcal{I}$,
\[
\begin{aligned}
\EE\left[\left(\prod_{\ell=1}^r |Z_{A_{\ell}}|\right)\|\mathfrak{X}_{1, \theta}^{(n)}\|_F^q\right]
&\leq 2^{q-1}\EE\left[\left(\prod_{\ell=1}^r |Z_{A_{\ell}}|\right)\left\{\|\mathfrak{X}_{-1}^{(n)}\|_F^q+\frac{\|\mathfrak{Z}_1\|_F^q}{n^{q/2}}\right\}\right] \\
&\leq 2^{q-1}\left\{\Lambda_q\EE\left[\prod_{\ell=1}^r |Z_{A_{\ell}}|\right]+\EE\left[\left(\prod_{\ell=1}^r |Z_{A_{\ell}}|\right)\|\mathfrak{Z}_1\|_F^q\right]\right\} \\
&\leq 2^{q-1}\{\Lambda_q\mu_r+\mu_{r+q}\}.
\end{aligned}
\]
For $q=0$, we simply use $\EE[\prod_{\ell=1}^r |Z_{A_{\ell}}|]\leq \mu_r$. Thus, for $\theta\in [0,1]$,
\begin{equation}\label{eq:theta.free.L.bound.Z.only.explicit}
\EE\left[\left(\prod_{\ell=1}^r |Z_{A_{\ell}}|\right)\|\mathfrak{X}_{1, \theta}^{(n)}\|_F^q\right]\leq \mathcal{L}_{q,r}, \qquad r+q\leq8.
\end{equation}
Similarly, since $1+\|\mathfrak{X}_{1, \theta}^{(n)}\|_F \leq 1+\|\mathfrak{X}_{-1}^{(n)}\|_F+\|\mathfrak{Z}_1\|_F/\sqrt{n}$, we have, for $q\geq1$,
\[
\begin{aligned}
\EE\left[\left(\prod_{\ell=1}^r |Z_{A_{\ell}}|\right)(1+\|\mathfrak{X}_{1, \theta}^{(n)}\|_F)^q\right]
&\leq 2^{q-1}\EE\left[\left(\prod_{\ell=1}^r |Z_{A_{\ell}}|\right)\left\{(1+\|\mathfrak{X}_{-1}^{(n)}\|_F)^q+\frac{\|\mathfrak{Z}_1\|_F^q}{n^{q/2}}\right\}\right] \\
&\leq 2^{q-1}\left\{2^{q-1}(1+\Lambda_q)\EE\left[\prod_{\ell=1}^r |Z_{A_{\ell}}|\right]+\EE\left[\left(\prod_{\ell=1}^r |Z_{A_{\ell}}|\right)\|\mathfrak{Z}_1\|_F^q\right]\right\} \\
&\leq 2^{q-1}\left\{2^{q-1}(1+\Lambda_q)\mu_r+\mu_{r+q}\right\}.
\end{aligned}
\]
Consequently, for $\theta\in [0,1]$,
\begin{equation}\label{eq:theta.free.L.plus.bound.Z.only.explicit}
\EE\left[\left(\prod_{\ell=1}^r |Z_{A_{\ell}}|\right)(1+\|\mathfrak{X}_{1, \theta}^{(n)}\|_F)^q\right]\leq \mathcal{L}_{q,r}^{+}, \qquad q\geq0, \qquad r+q\leq8.
\end{equation}

For $r\in\{1,2,3,4\}$, define
\[
\mathcal{U}_{4,h}^{[r]} \leqdef 12\mathcal{M}_2(f_h)\mathcal{L}_{0,r}+48\mathcal{M}_3(f_h)\mathcal{L}_{2,r}+16\mathcal{M}_4(f_h)\mathcal{L}_{4,r},
\]
whenever the right-hand side is defined, and, for $r\in\{1,3\}$, define
\[
\mathcal{U}_{5,h}^{[r]} \leqdef 120\mathcal{M}_3(f_h)\mathcal{L}_{1,r}^{+}+160\mathcal{M}_4(f_h)\mathcal{L}_{3,r}^{+}+32\mathcal{M}_5(f_h)\mathcal{L}_{5,r}^{+}.
\]
By \eqref{eq:def.G.k.h},
\[
\begin{aligned}
\mathcal{G}_{4,h}(x) &= 12\mathcal{M}_2(f_h)+48x^2\mathcal{M}_3(f_h)+16x^4\mathcal{M}_4(f_h), \\
\mathcal{G}_{5,h}(x) &= 120x\mathcal{M}_3(f_h)+160x^3\mathcal{M}_4(f_h)+32x^5\mathcal{M}_5(f_h).
\end{aligned}
\]
Therefore, the bounds \eqref{eq:theta.free.L.bound.Z.only.explicit} and \eqref{eq:theta.free.L.plus.bound.Z.only.explicit} imply that, for $\theta\in [0,1]$,
\begin{equation}\label{eq:theta.free.G4.bound.Z.only.explicit}
\EE\left[\left(\prod_{\ell=1}^r |Z_{A_{\ell}}|\right)\mathcal{G}_{4,h}(\|\mathfrak{X}_{1, \theta}^{(n)}\|_F)\right]\leq \mathcal{U}_{4,h}^{[r]},
\end{equation}
and
\begin{equation}\label{eq:theta.free.G5.bound.Z.only.explicit}
\EE\left[\left(\prod_{\ell=1}^r |Z_{A_{\ell}}|\right)\mathcal{G}_{5,h}(1+\|\mathfrak{X}_{1, \theta}^{(n)}\|_F)\right]\leq \mathcal{U}_{5,h}^{[r]}.
\end{equation}

Set $\omega_8 \leqdef 1\vee \max_{A\in \mathcal{I}}\EE[|Z_A|^8]$. We have
\[
\mu_8
= \EE[\|\mathfrak{Z}_1\|_F^8]
= \EE\left[\left(\sum_{A\in \mathcal{I}}Z_A^2\right)^4\right]
\leq (\nu d)^3\sum_{A\in \mathcal{I}}\EE[|Z_A|^8]
\leq (\nu d)^4\omega_8.
\]
Hence, for $0\leq q\leq 8$, we find that $\mu_q\leq \mu_8^{q/8}\leq (\nu d)^{q/2}\omega_8$. More generally, if $\rho\in\{0,1,2\}$ and $\rho+q_1+\cdots+q_s\leq8$, then
\begin{equation}\label{eq:mu.products.omega8.bound}
\mu_2^{\rho/2}\prod_{j=1}^s\mu_{q_j}\leq (\nu d)^{(\rho+q_1+\cdots+q_s)/2}\omega_8.
\end{equation}
Using the definitions of $\Lambda_q$, $\mathcal{L}_{q,r}$ and $\mathcal{L}_{q,r}^{+}$, the bound \eqref{eq:mu.products.omega8.bound} yields
\[
\mathcal{L}_{0,r}\leq (\nu d)^{r/2}\omega_8, \qquad \mathcal{L}_{2,r}\leq 10(\nu d)^{(r+2)/2}\omega_8, \qquad \mathcal{L}_{4,r}\leq 1160(\nu d)^{(r+4)/2}\omega_8,
\]
\[
\mathcal{L}_{1,r}^{+}\leq 3(\nu d)^{(r+1)/2}\omega_8, \qquad \mathcal{L}_{3,r}^{+}\leq (20+256\sqrt{2})(\nu d)^{(r+3)/2}\omega_8, \qquad \mathcal{L}_{5,r}^{+}\leq 262416(\nu d)^{(r+5)/2}\omega_8.
\]
Similarly,
\[
\mu_2\mathcal{L}_{0,2}\leq (\nu d)^2\omega_8, \qquad \mu_2\mathcal{L}_{2,2}\leq 10(\nu d)^3\omega_8, \qquad \mu_2\mathcal{L}_{4,2}\leq 1160(\nu d)^4\omega_8,
\]
\[
\mu_2\mathcal{L}_{1,1}^{+}\leq 3(\nu d)^2\omega_8, \qquad \mu_2\mathcal{L}_{3,1}^{+}\leq (20+256\sqrt{2})(\nu d)^3\omega_8, \qquad \mu_2\mathcal{L}_{5,1}^{+}\leq 262416(\nu d)^4\omega_8.
\]
Consequently,
\[
\mathcal{U}_{4,h}^{[4]}\leq \omega_8\left\{12(\nu d)^2\mathcal{M}_2(f_h)+480(\nu d)^3\mathcal{M}_3(f_h)+18560(\nu d)^4\mathcal{M}_4(f_h)\right\},
\]
\[
\mu_2\mathcal{U}_{4,h}^{[2]}\leq \omega_8\left\{12(\nu d)^2\mathcal{M}_2(f_h)+480(\nu d)^3\mathcal{M}_3(f_h)+18560(\nu d)^4\mathcal{M}_4(f_h)\right\},
\]
\[
\mathcal{U}_{4,h}^{[1]}\leq \omega_8\left\{12(\nu d)^{1/2}\mathcal{M}_2(f_h)+480(\nu d)^{3/2}\mathcal{M}_3(f_h)+18560(\nu d)^{5/2}\mathcal{M}_4(f_h)\right\},
\]
and
\[
\mathcal{U}_{5,h}^{[3]}\leq \omega_8\left\{360(\nu d)^2\mathcal{M}_3(f_h)+(3200+40960\sqrt{2})(\nu d)^3\mathcal{M}_4(f_h)+8397312(\nu d)^4\mathcal{M}_5(f_h)\right\},
\]
\[
\mu_2 \, \mathcal{U}_{5,h}^{[1]}\leq \omega_8\left\{360(\nu d)^2\mathcal{M}_3(f_h)+(3200+40960\sqrt{2})(\nu d)^3\mathcal{M}_4(f_h)+8397312(\nu d)^4\mathcal{M}_5(f_h)\right\}.
\]
Moreover, $|\sigma_{A,B}|=|\EE[Z_AZ_B]|\leq \EE[|Z_AZ_B|]\leq \mu_2$, so that
\[
\sum_{A,B\in \mathcal{I}}|\sigma_{A,B}| \leq (\nu d)^2\mu_2.
\]
Using \eqref{eq:theta.free.G4.bound.Z.only.explicit} in the bounds for $R_{1, \varepsilon},R_{2, \varepsilon},R_{3, \varepsilon}$, and using \eqref{eq:theta.free.G5.bound.Z.only.explicit} in the bound for $R_{4, \varepsilon}^{A,B,C}+R_{5, \varepsilon}^{A,B,C}$, gives
\[
|R_{1, \varepsilon}|+|R_{2, \varepsilon}|+|R_{3, \varepsilon}| \leq \frac{1}{n}\left\{\frac{(\nu d)^4}{6}\mathcal{U}_{4,h}^{[4]}+\frac{(\nu d)^4\mu_2}{2}\mathcal{U}_{4,h}^{[2]}+\frac{\nu d}{2}\left(\sum_{A,B,C\in \mathcal{I}}|\EE[Z_AZ_BZ_C]|\right)\mathcal{U}_{4,h}^{[1]}\right\},
\]
and, for every fixed $A,B,C\in \mathcal{I}$,
\[
|R_{4, \varepsilon}^{A,B,C}|+|R_{5, \varepsilon}^{A,B,C}| \leq \frac{K}{\sqrt{n}}\left\{\frac{(\nu d)^3}{2}\mathcal{U}_{5,h}^{[3]}+(\nu d)^3\mu_2\mathcal{U}_{5,h}^{[1]}\right\}.
\]
It follows from \eqref{eq:Wishart.MANOVA.after.first.expansion} and \eqref{eq:R4.R5.intermediate} that
\[
\begin{aligned}
|\EE[\mathcal{A}^{\mathrm{OU}}g_{h, \varepsilon}(\mathfrak{X}^{(n)})]|
&\leq \frac{1}{n}\left\{\frac{(\nu d)^4}{6}\mathcal{U}_{4,h}^{[4]}+\frac{(\nu d)^4\mu_2}{2}\mathcal{U}_{4,h}^{[2]}+\frac{\nu d}{2}\left(\sum_{A,B,C\in \mathcal{I}}|\EE[Z_AZ_BZ_C]|\right)\mathcal{U}_{4,h}^{[1]}\right\} \\
&\qquad + \frac{K}{2n}\left(\sum_{A,B,C\in \mathcal{I}}|\EE[Z_AZ_BZ_C]|\right)\left\{\frac{(\nu d)^3}{2}\mathcal{U}_{5,h}^{[3]}+(\nu d)^3\mu_2\mathcal{U}_{5,h}^{[1]}\right\}.
\end{aligned}
\]
Since $\nu d\geq1$ and $K\leq1+K$, the preceding display yields
\[
|\EE[\mathcal{A}^{\mathrm{OU}}g_{h, \varepsilon}(\mathfrak{X}^{(n)})]| \leq \frac{(\nu d)^8(1+K) \omega_8}{n} \left\{\beta_2\mathcal{M}_2(f_h)+\beta_3\mathcal{M}_3(f_h)+\beta_4\mathcal{M}_4(f_h)+\beta_5\mathcal{M}_5(f_h)\right\},
\]
where
\[
\beta_j \leqdef u_j+v_j\sum_{A,B,C\in \mathcal{I}}|\EE[Z_AZ_BZ_C]|, \qquad j\in\{2,3,4,5\},
\]
with
\[
(u_2, u_3, u_4, u_5) \leqdef \left(8, 320, \frac{37120}{3}, 0\right), \qquad (v_2, v_3, v_4, v_5) \leqdef \left(6, 510, 11680+30720\sqrt{2}, 6297984\right).
\]
Finally, we can use \eqref{eq:Wishart.Stein.bound.4} in Theorem~\ref{thm:smoothness.estimates.wishart} to obtain $\smash{\mathcal{M}_m(f_h) \leq \frac{1}{2m} \, \mathcal{M}_m^{\mathcal{D}, \Sigma}(h)}$ for $m\in \{2,3,4,5\}$. Thus, uniformly in $\varepsilon\in(0,1]$,
\[
\begin{aligned}
&|\EE[\mathcal{A}^{\mathrm{OU}}g_{h, \varepsilon}(\mathfrak{X}^{(n)})]| \\
&\qquad\leq \frac{(\nu d)^8(1+K) \omega_8}{n} \left\{\alpha_2\mathcal{M}_2^{\mathcal{D}, \Sigma}(h)+\alpha_3\mathcal{M}_3^{\mathcal{D}, \Sigma}(h)+\alpha_4\mathcal{M}_4^{\mathcal{D}, \Sigma}(h)+\alpha_5\mathcal{M}_5^{\mathcal{D}, \Sigma}(h)\right\},
\end{aligned}
\]
where in obtaining $\alpha_2,\ldots,\alpha_5$ the values of $r_2,\ldots,r_5$ and $s_2,\ldots s_5$ are obtained by rounding numbers up to the nearest integer. Since $h$ is identified with its bounded continuous extension to $\mathcal{S}_{+}^d$ and $S_{\varepsilon}(\mathfrak{X}^{(n)})=\mathfrak{W}_{\nu}^{(n)}+\varepsilon I_d\to \smash{\mathfrak{W}_{\nu}^{(n)}}$ almost surely, dominated convergence gives
\[
\lim_{\varepsilon\downarrow0}\EE[h(S_{\varepsilon}(\mathfrak{X}^{(n)}))] = \EE[h(\mathfrak{W}_{\nu}^{(n)})].
\]
Letting $\varepsilon\downarrow0$ in \eqref{eq:Wishart.MANOVA.regularized.comparison} and using the preceding uniform bound gives the conclusion.

\subsection{Proof of Proposition~\ref{prop:Wishart.Satterthwaite}}

Let $\alpha_{\bullet} \leqdef \sum_{j=1}^N \alpha_j$. By assumption, the $\alpha_j$ are positive integers, so we may take independent random vectors $\smash{\{\bb{X}_{j,r}:j\in [N], \, r\in [\alpha_j]\}}$ such that $\bb{X}_{j,r}\sim \mathcal{N}_d(\bb{0}_d, \Sigma_j)$. By the standard Gaussian representation of the Wishart distribution, we may write $\mathfrak{G}_j = \sum_{r = 1}^{\alpha_j}\bb{X}_{j,r}\bb{X}_{j,r}^{\top}$ for $j\in [N]$, and therefore
\begin{equation}\label{eq:T.decomp.Satterthwaite.heterogeneous}
\mathfrak{T} = \sum_{j=1}^N \sum_{r = 1}^{\alpha_j}\bb{X}_{j,r}\bb{X}_{j,r}^{\top}.
\end{equation}
Since $\alpha_{\bullet}\geq d + 1$, the Gaussian vectors appearing in \eqref{eq:T.decomp.Satterthwaite.heterogeneous} span $\R^d$ almost surely, and hence $\mathfrak{T}\in \mathcal{S}_{++}^d$ almost surely. Let $f_h$ denote the solution of the Stein equation associated with the target law $\smash{\mathfrak{W}_{\nu}\sim \mathcal{W}_d(\nu, \widetilde{\Sigma}_{\nu})}$. Since $\nu > 3d-3$, Theorem~\ref{thm:smoothness.estimates.wishart} applies to this target distribution. For each $(j,r)$, define the leave-one-vector-out matrix
\[
\mathfrak{T}_{j,r} \leqdef \mathfrak{T}-\bb{X}_{j,r}\bb{X}_{j,r}^{\top}.
\]
We assumed $\alpha_{\bullet}\geq d + 1$, so the matrix $\mathfrak{T}_{j,r}$ is a sum of at least $d$ independent full-covariance Gaussian rank-one matrices. These Gaussian vectors span $\R^d$ almost surely, so $\mathfrak{T}_{j,r}\in \mathcal{S}_{++}^d$ almost surely. Consequently, for every $\bb{x}\in \R^d$, the matrix $\mathfrak{T}_{j,r} + \bb{x}\bb{x}^{\top}$ belongs to $\mathcal{S}_{++}^d$ almost surely, and all derivatives of $f_h$ appearing below are evaluated inside their domain.

By the Stein equation \eqref{eq:Stein.equation.Wishart} and the expression of the Wishart extended generator in \eqref{eq:Wishart.process.generator}, with $(\alpha, \Sigma)$ replaced by $\smash{(\nu, \widetilde{\Sigma}_{\nu})}$, we have
\begin{equation}\label{eq:Stein.identity.Satterthwaite.heterogeneous}
\begin{aligned}
\EE[h(\mathfrak{T})]-\EE[h(\mathfrak{W}_{\nu})]
&= 2 \, \EE[\tr\{(\nu\widetilde{\Sigma}_{\nu}-\mathfrak{T})\nabla f_h(\mathfrak{T})\}] + 4 \, \EE[\tr\{\mathfrak{T}\nabla\widetilde{\Sigma}_{\nu}\nabla f_h(\mathfrak{T})\}] \\
&= 2 \, \EE[\tr\{(\overline{\Sigma}-\mathfrak{T})\nabla f_h(\mathfrak{T})\}] + 4 \, \EE[\tr\{\mathfrak{T}\nabla\widetilde{\Sigma}_{\nu}\nabla f_h(\mathfrak{T})\}],
\end{aligned}
\end{equation}
since $\nu\widetilde{\Sigma}_{\nu} = \overline{\Sigma}$. Fix $(j,r)$, define
\[
F_{j,r}(\bb{x}) \leqdef \nabla f_h(\mathfrak{T}_{j,r} + \bb{x}\bb{x}^{\top})\bb{x}, \qquad \bb{x}\in \R^d,
\]
and introduce the $\Sigma_j$-divergence
\[
\mathrm{div}_{\Sigma_j}F_{j,r}(\bb{x}) \leqdef \tr\{\Sigma_j D F_{j,r}(\bb{x})\} = \sum_{a,b = 1}^d(\Sigma_j)_{ab} \, \bb{e}_a^{\top}D F_{j,r}(\bb{x})[\bb{e}_b].
\]
Also, for $\bb{x}\in \R^d$ and $a\in [d]$, let
\[
H_{\bb{x}}^{(a)} \leqdef \bb{x}\bb{e}_a^{\top} + \bb{e}_a\bb{x}^{\top}\in \mathcal{S}^d.
\]
Writing $S_{\bb{x}} \leqdef \mathfrak{T}_{j,r} + \bb{x}\bb{x}^{\top}$, the chain rule gives
\[
D F_{j,r}(\bb{x})[\bb{e}_b] = D(\nabla f_h)(S_{\bb{x}})[H_{\bb{x}}^{(b)}]\bb{x} + \nabla f_h(S_{\bb{x}})\bb{e}_b.
\]
Therefore,
\begin{align}
\mathrm{div}_{\Sigma_j}F_{j,r}(\bb{x})
&= \tr\{\Sigma_j\nabla f_h(S_{\bb{x}})\} + \sum_{a,b = 1}^d(\Sigma_j)_{ab} \, \bb{e}_a^{\top}D(\nabla f_h)(S_{\bb{x}})[H_{\bb{x}}^{(b)}]\bb{x} \notag \\
&= \tr\{\Sigma_j\nabla f_h(S_{\bb{x}})\} + \frac{1}{2}\sum_{a,b = 1}^d(\Sigma_j)_{ab} \, D^2 f_h(S_{\bb{x}})[H_{\bb{x}}^{(b)},H_{\bb{x}}^{(a)}] \notag \\
&= \tr\{\Sigma_j\nabla f_h(S_{\bb{x}})\} + \frac{1}{2}\sum_{a,b = 1}^d(\Sigma_j)_{ab}\sum_{i, \ell = 1}^d x_i x_{\ell} D^2 f_h(S_{\bb{x}})[\bb{e}_i\bb{e}_b^{\top}+\bb{e}_b\bb{e}_i^{\top}, \bb{e}_{\ell}\bb{e}_a^{\top}+\bb{e}_a\bb{e}_{\ell}^{\top}] \notag \\
&= \tr\{\Sigma_j\nabla f_h(S_{\bb{x}})\} + 2\sum_{i, \ell = 1}^d x_i x_{\ell}\sum_{a,b = 1}^d \nabla_{ib}(\Sigma_j)_{ba}\nabla_{a\ell} f_h(S_{\bb{x}}) \notag \\
&= \tr\{\Sigma_j\nabla f_h(S_{\bb{x}})\} + 2 \, \bb{x}^{\top}\nabla\Sigma_j\nabla f_h(S_{\bb{x}})\bb{x}. \label{eq:div.display.Satterthwaite.heterogeneous}
\end{align}
Moreover, by the assumption $\smash{\mathcal{M}_1^{\mathcal{D}, \widetilde{\Sigma}_{\nu}}(h) + \mathcal{M}_2^{\mathcal{D}, \widetilde{\Sigma}_{\nu}}(h) < \infty}$ and Theorem~\ref{thm:smoothness.estimates.wishart}~$(ii)$, applied with $m = 1,2$ and target scale matrix $\smash{\widetilde{\Sigma}_{\nu}}$, $f_h$ has bounded first and second derivatives. Thus, conditionally on $\mathfrak{T}_{j,r}$, the map $F_{j,r}$ is $C^1$ with at most linear growth and its derivative has at most quadratic growth. Gaussian integration by parts below is therefore justified. Since $\bb{X}_{j,r}$ is centered Gaussian with covariance matrix $\Sigma_j$ and is independent of $\mathfrak{T}_{j,r}$, Gaussian integration by parts yields
\[
\EE[\bb{X}_{j,r}^{\top}F_{j,r}(\bb{X}_{j,r})] = \EE[\mathrm{div}_{\Sigma_j}F_{j,r}(\bb{X}_{j,r})].
\]
Using \eqref{eq:div.display.Satterthwaite.heterogeneous} and the identity $\mathfrak{T} = \mathfrak{T}_{j,r} + \bb{X}_{j,r}\bb{X}_{j,r}^{\top}$, we deduce
\begin{equation}\label{eq:ibp.Satterthwaite.heterogeneous}
\EE[\bb{X}_{j,r}^{\top}\nabla f_h(\mathfrak{T})\bb{X}_{j,r}] = \EE[\tr\{\Sigma_j\nabla f_h(\mathfrak{T})\}] + 2 \, \EE[\bb{X}_{j,r}^{\top}\nabla\Sigma_j\nabla f_h(\mathfrak{T})\bb{X}_{j,r}].
\end{equation}
Summing \eqref{eq:ibp.Satterthwaite.heterogeneous} over $j\in [N]$ and $r\in [\alpha_j]$, we get from \eqref{eq:T.decomp.Satterthwaite.heterogeneous} and the definition $\overline{\Sigma} = \sum_{j=1}^N \alpha_j\Sigma_j$ that
\[
\begin{aligned}
\EE[\tr\{\mathfrak{T}\nabla f_h(\mathfrak{T})\}]
&= \sum_{j=1}^N \sum_{r = 1}^{\alpha_j}\EE[\bb{X}_{j,r}^{\top}\nabla f_h(\mathfrak{T})\bb{X}_{j,r}] \\
&= \EE[\tr\{\overline{\Sigma}\nabla f_h(\mathfrak{T})\}] + 2\sum_{j=1}^N \sum_{r = 1}^{\alpha_j}\EE[\bb{X}_{j,r}^{\top}\nabla\Sigma_j\nabla f_h(\mathfrak{T})\bb{X}_{j,r}].
\end{aligned}
\]
Substituting this identity into \eqref{eq:Stein.identity.Satterthwaite.heterogeneous} and using again \eqref{eq:T.decomp.Satterthwaite.heterogeneous}, we find
\begin{align}
\big|\EE[h(\mathfrak{T})]-\EE[h(\mathfrak{W}_{\nu})]\big|
&= 4 \, \Bigg|\sum_{j=1}^N \sum_{r = 1}^{\alpha_j}\EE[\bb{X}_{j,r}^{\top}\nabla\widetilde{\Sigma}_{\nu}\nabla f_h(\mathfrak{T})\bb{X}_{j,r}] - \sum_{j=1}^N \sum_{r = 1}^{\alpha_j}\EE[\bb{X}_{j,r}^{\top}\nabla\Sigma_j\nabla f_h(\mathfrak{T})\bb{X}_{j,r}]\Bigg| \notag \\
&\leq 4\sum_{j=1}^N \sum_{r = 1}^{\alpha_j}\EE\big[|\bb{X}_{j,r}^{\top}\nabla(\widetilde{\Sigma}_{\nu}-\Sigma_j)\nabla f_h(\mathfrak{T})\bb{X}_{j,r}|\big], \label{eq:before.bound.Satterthwaite.heterogeneous}
\end{align}
where in obtaining the inequality we applied the triangle inequality. It remains to bound the quadratic form in \eqref{eq:before.bound.Satterthwaite.heterogeneous}. Let $B\in \mathcal{S}^d$. From the calculation in \eqref{eq:div.display.Satterthwaite.heterogeneous}, with $\Sigma_j$ replaced by $B$, and the definition of $\mathcal{M}_2(f_h)$ in \eqref{eq:M.m}, we have
\[
\big|\bb{x}^{\top}\nabla B\nabla f_h(S_{\bb{x}})\bb{x}\big| = \frac{1}{4} \left|\sum_{a,b = 1}^dB_{ab} \, D^2 f_h(S_{\bb{x}})[H_{\bb{x}}^{(b)},H_{\bb{x}}^{(a)}]\right| \leq \frac{\mathcal{M}_2(f_h)}{4}\sum_{a,b = 1}^d|B_{ab}| \, \|H_{\bb{x}}^{(b)}\|_F \, \|H_{\bb{x}}^{(a)}\|_F.
\]
By the Cauchy-Schwarz inequality,
\[
\sum_{a,b = 1}^d|B_{ab}| \, \|H_{\bb{x}}^{(b)}\|_F \, \|H_{\bb{x}}^{(a)}\|_F \leq \|B\|_F\left(\sum_{a,b = 1}^d\|H_{\bb{x}}^{(b)}\|_F^2\|H_{\bb{x}}^{(a)}\|_F^2\right)^{1/2} = \|B\|_F\sum_{a = 1}^d\|H_{\bb{x}}^{(a)}\|_F^2.
\]
Now, $\|H_{\bb{x}}^{(a)}\|_F^2 = \|\bb{x}\bb{e}_a^{\top} + \bb{e}_a\bb{x}^{\top}\|_F^2 = 2\|\bb{x}\|_2^2 + 2x_a^2$, so that $\sum_{a = 1}^d\|H_{\bb{x}}^{(a)}\|_F^2 = 2(d + 1)\|\bb{x}\|_2^2$. Therefore,
\begin{equation}\label{eq:quadratic.form.bound.Satterthwaite.heterogeneous}
\big|\bb{x}^{\top}\nabla B\nabla f_h(S_{\bb{x}})\bb{x}\big| \leq \frac{d + 1}{2} \, \mathcal{M}_2(f_h) \, \|B\|_F \, \|\bb{x}\|_2^2.
\end{equation}
Applying \eqref{eq:quadratic.form.bound.Satterthwaite.heterogeneous} with $B = \widetilde{\Sigma}_{\nu}-\Sigma_j$ and $\bb{x} = \bb{X}_{j,r}$, and taking expectations, yields
\[
\EE\big[|\bb{X}_{j,r}^{\top}\nabla(\widetilde{\Sigma}_{\nu}-\Sigma_j)\nabla f_h(\mathfrak{T})\bb{X}_{j,r}|\big] \leq \frac{d + 1}{2} \, \|\widetilde{\Sigma}_{\nu}-\Sigma_j\|_F \, \mathcal{M}_2(f_h) \, \tr(\Sigma_j).
\]
Combining this with \eqref{eq:before.bound.Satterthwaite.heterogeneous}, we obtain
\[
\big|\EE[h(\mathfrak{T})]-\EE[h(\mathfrak{W}_{\nu})]\big| \leq 2(d + 1)\mathcal{M}_2(f_h)\sum_{j=1}^N \alpha_j\tr(\Sigma_j)\|\Sigma_j-\widetilde{\Sigma}_{\nu}\|_F.
\]
Finally, by Theorem~\ref{thm:smoothness.estimates.wishart}~$(ii)$ with $m = 2$ and target scale matrix $\widetilde{\Sigma}_{\nu}$, $\smash{\mathcal{M}_2(f_h) \leq \frac{1}{4}\mathcal{M}_2^{\mathcal{D}, \widetilde{\Sigma}_{\nu}}(h)}$. Substituting this into the previous display gives \eqref{eq:heterogeneous.Satterthwaite.bound.nu}. This completes the proof.

\subsection{Proof of Proposition~\ref{prop:Wishart.deBruijn}}

Fix $g\in C_b^2(\mathcal{S}_{++}^d)$ satisfying $m_g \leq g(X)\leq M_g$ for all $X\in \mathcal{S}_{++}^d$. For $t > 0$, let $\Sigma_t \leqdef (1 - e^{-2t})\Sigma$ and $\smash{p_t(W, Y) \leqdef f_{\alpha, \Sigma_t, e^{-2t}\Sigma_t^{-1}W}^{\mathcal{W}}(Y)}$, so that, by the distributional relation \eqref{eq:W.t.conditional} and Eq.~\eqref{semiformula} of Proposition~\ref{prop:extension.Luk.1994.Lemma.2.4},
\[
(\mathcal{P}_t^{\mathcal{W}} h)(W) = \int_{\mathcal{S}_{++}^d} h(Y) \, p_t(W, Y) \, \rd Y.
\]
Also set
\[
f_{\xi}(X) \leqdef \frac{\xi(\rd X)}{\rd X} = f_{\alpha, \Sigma, 0_{d\times d}}^{\mathcal{W}}(X), \qquad X\in \mathcal{S}_{++}^d.
\]
By \eqref{eq:noncentral.Wishart.density}, for every $W, Y\in \mathcal{S}_{++}^d$,
\[
\begin{aligned}
p_t(W, Y) f_{\xi}(W)
&= \frac{|W|^{\alpha/2 - (d + 1)/2}|Y|^{\alpha/2 - (d + 1)/2}}{|2\Sigma|^{\alpha/2}|2\Sigma_t|^{\alpha/2}\Gamma_d(\alpha/2)^2} \\
&\qquad \times \etr\left\{-\frac{1}{2}\Sigma_t^{-1}Y - \frac{e^{-2t}}{2}\Sigma_t^{-1}W - \frac{1}{2}\Sigma^{-1}W\right\} {}_0F_1\left(\frac{\alpha}{2}; \frac{e^{-2t}}{4}\Sigma_t^{-1}W\Sigma_t^{-1}Y\right).
\end{aligned}
\]
Since $\Sigma_t^{-1} = (1 - e^{-2t})^{-1}\Sigma^{-1}$, we have $(e^{-2t}/2)\Sigma_t^{-1} + (1/2) \Sigma^{-1} = (1/2) \Sigma_t^{-1}$. Moreover, we have ${}_0F_1(\nu; AB) = {}_0F_1(\nu; BA)$ for square matrices $A$ and $B$ of the same size. Hence
\[
p_t(W, Y)f_{\xi}(W) = p_t(Y, W)f_{\xi}(Y), \qquad W, Y\in \mathcal{S}_{++}^d.
\]
It follows that the Wishart semigroup is symmetric in $L^2(\xi)$. Indeed, for bounded measurable functions $\varphi, \psi:\mathcal{S}_{++}^d\to \R$,
\[
\begin{aligned}
\int_{\mathcal{S}_{++}^d} (\mathcal{P}_t^{\mathcal{W}}\varphi)(W) \psi(W) \, \xi(\rd W)
&= \int_{\mathcal{S}_{++}^d}\int_{\mathcal{S}_{++}^d} \psi(W)\varphi(Y) \, p_t(W, Y) \, \rd Y \, \xi(\rd W) \\
&= \int_{\mathcal{S}_{++}^d}\int_{\mathcal{S}_{++}^d} \psi(W)\varphi(Y) \, p_t(Y, W) \, \rd W \, \xi(\rd Y) \\
&= \int_{\mathcal{S}_{++}^d} \varphi(Y)(\mathcal{P}_t^{\mathcal{W}}\psi)(Y) \, \xi(\rd Y).
\end{aligned}
\]
Therefore, since $\mu_t = \mu \mathcal{P}_t^{\mathcal{W}}$ and $\mu(\rd W) = g(W) \, \xi(\rd W)$,
\[
\int_{\mathcal{S}_{++}^d} \varphi(Y) \, \mu_t(\rd Y)
= \int_{\mathcal{S}_{++}^d} (\mathcal{P}_t^{\mathcal{W}}\varphi)(W) \, g(W) \, \xi(\rd W)
= \int_{\mathcal{S}_{++}^d} \varphi(Y)(\mathcal{P}_t^{\mathcal{W}}g)(Y) \, \xi(\rd Y).
\]
This proves the first claim of the proposition, namely
\begin{equation}\label{eq:Wishart.g.t}
g_t = \mathcal{P}_t^{\mathcal{W}} g.
\end{equation}
For $h\in C_b^2(\mathcal{S}_{++}^d)$, the Kolmogorov backward equation for the Wishart semigroup gives
\begin{equation}\label{eq:Kolmogorov.backward}
\partial_t \mathcal{P}_t^{\mathcal{W}} h = \mathcal{A}^{\mathcal{W}}\mathcal{P}_t^{\mathcal{W}} h, \qquad t > 0.
\end{equation}
Applying \eqref{eq:Kolmogorov.backward} with $h = g$ and using \eqref{eq:Wishart.g.t}, we obtain
\begin{equation}\label{eq:Wishart.forward.g.t}
\partial_t g_t = \mathcal{A}^{\mathcal{W}}g_t.
\end{equation}

Next, let $\varphi, \psi\in C_b^2(\mathcal{S}_{++}^d)$. By Proposition~\ref{prop:extension.Luk.1994.Lemma.2.4}, the measure $\xi$ is invariant for $(\mathcal{P}_t^{\mathcal{W}})_{t\geq 0}$, and therefore
\[
\int_{\mathcal{S}_{++}^d} (\mathcal{P}_t^{\mathcal{W}}\varphi)(X) \, \xi(\rd X) = \int_{\mathcal{S}_{++}^d} \varphi(X) \, \xi(\rd X), \qquad t\geq 0.
\]
Using this invariance and Dynkin's formula gives $\smash{\int_{\mathcal{S}_{++}^d} \mathcal{A}^{\mathcal{W}}\varphi(X) \, \xi(\rd X) = 0}$. Combining this with \eqref{eq:Wishart.product.rule} and the symmetry of $\mathcal{P}_t^{\mathcal{W}}$, hence of $\mathcal{A}^{\mathcal{W}}$, we obtain
\[
\begin{aligned}
0
&= \int_{\mathcal{S}_{++}^d} \mathcal{A}^{\mathcal{W}}(\varphi\psi)(X) \, \xi(\rd X) \\
&= \int_{\mathcal{S}_{++}^d} \varphi(X)\mathcal{A}^{\mathcal{W}}\psi(X) \, \xi(\rd X) + \int_{\mathcal{S}_{++}^d} \psi(X)\mathcal{A}^{\mathcal{W}}\varphi(X) \, \xi(\rd X) \\
&\qquad + 8 \int_{\mathcal{S}_{++}^d} \tr\{X \nabla \varphi(X)\Sigma \nabla \psi(X)\} \, \xi(\rd X) \\
&= 2 \int_{\mathcal{S}_{++}^d} \psi(X)\mathcal{A}^{\mathcal{W}}\varphi(X) \, \xi(\rd X) + 8 \int_{\mathcal{S}_{++}^d} \tr\{X \nabla \varphi(X)\Sigma \nabla \psi(X)\} \, \xi(\rd X).
\end{aligned}
\]
Thus,
\begin{equation}\label{eq:Wishart.integration.by.parts}
\int_{\mathcal{S}_{++}^d} \psi(X)\mathcal{A}^{\mathcal{W}}\varphi(X) \, \xi(\rd X) = -4 \int_{\mathcal{S}_{++}^d} \tr\{X \nabla \varphi(X)\Sigma \nabla \psi(X)\} \, \xi(\rd X).
\end{equation}

The boundedness assumption now gives the required domination. Since $\mathcal{P}_t^{\mathcal{W}}$ is Markov, \eqref{eq:Wishart.g.t} implies
\[
m_g \leq g_t(X)\leq M_g, \qquad X\in \mathcal{S}_{++}^d, \quad t\geq 0.
\]
In particular, $\log(g_t)$ is bounded uniformly in $X$ and $t$. Moreover, by the regularity of the Wishart semigroup on $C_b^2(\mathcal{S}_{++}^d)$, for every compact interval $I\subset (0, \infty)$, the functions $g_t$, $\nabla g_t$ and the second-order derivatives of $g_t$ are bounded uniformly over $t\in I$. Hence, from the generator formula \eqref{eq:Wishart.process.generator}, there exists a constant $C_I\in (0, \infty)$ such that
\begin{equation}\label{eq:dom}
|\mathcal{A}^{\mathcal{W}}g_t(X)| \leq C_I(1 + \tr(X)), \qquad X\in \mathcal{S}_{++}^d, \quad t\in I.
\end{equation}
Since $\xi = \mathcal{W}_d(\alpha, \Sigma)$ has finite first moment, the right-hand side of \eqref{eq:dom} is integrable with respect to $\xi$. Thus dominated convergence justifies the entropy differentiation below. The lower bound on $g_t$ also implies that $\log(g_t)\in C_b^2(\mathcal{S}_{++}^d)$ whenever $g_t\in C_b^2(\mathcal{S}_{++}^d)$, so the integration by parts identity \eqref{eq:Wishart.integration.by.parts} applies with $\varphi = g_t$ and $\psi = \log(g_t)$.

We can now differentiate the relative entropy. Using \eqref{eq:Wishart.forward.g.t} and the domination above (this is the role of Assumption~\hyperref[ass:Wishart.deBruijn.differentiation]{\textup{(A)}} in Remark~\ref{rem:Wishart.deBruijn.assumptions}), we have
\[
\begin{aligned}
\frac{\rd}{\rd t} \mathrm{Ent}(\mu_t \, \| \, \xi)
&= \frac{\rd}{\rd t} \int_{\mathcal{S}_{++}^d} g_t(X)\log(g_t(X)) \, \xi(\rd X) \\
&= \int_{\mathcal{S}_{++}^d} (1 + \log(g_t(X))) \, \partial_t g_t(X) \, \xi(\rd X) \\
&= \int_{\mathcal{S}_{++}^d} (1 + \log(g_t(X))) \, \mathcal{A}^{\mathcal{W}}g_t(X) \, \xi(\rd X).
\end{aligned}
\]
Since $\int_{\mathcal{S}_{++}^d} \mathcal{A}^{\mathcal{W}}g_t(X) \, \xi(\rd X) = 0$ (this is the first part of Assumption~\hyperref[ass:Wishart.deBruijn.integration.by.parts]{\textup{(B)}} in Remark~\ref{rem:Wishart.deBruijn.assumptions}), this reduces to
\[
\frac{\rd}{\rd t}\mathrm{Ent}(\mu_t \, \| \, \xi) = \int_{\mathcal{S}_{++}^d} \log(g_t(X)) \, \mathcal{A}^{\mathcal{W}}g_t(X) \, \xi(\rd X).
\]
Applying \eqref{eq:Wishart.integration.by.parts} with $\varphi = g_t$ and $\psi = \log(g_t)$ (this is the second part of Assumption~\hyperref[ass:Wishart.deBruijn.integration.by.parts]{\textup{(B)}} in Remark~\ref{rem:Wishart.deBruijn.assumptions}) gives
\[
\begin{aligned}
\frac{\rd}{\rd t}\mathrm{Ent}(\mu_t \, \| \, \xi)
&= -4 \int_{\mathcal{S}_{++}^d} \tr\{X \nabla g_t(X)\Sigma \nabla \log(g_t)(X)\} \, \xi(\rd X) \\
&= -4 \int_{\mathcal{S}_{++}^d} \tr\{X \nabla \log(g_t)(X)\Sigma \nabla \log(g_t)(X)\}g_t(X) \, \xi(\rd X) \\
&= - J_{\mathcal{W}, \Sigma}(\mu_t \mid \xi),
\end{aligned}
\]
which proves \eqref{eq:Wishart.deBruijn}. Then, the identity \eqref{eq:Wishart.deBruijn.tau} follows from the chain rule and the change of variable $\tau = e^{-2t}$. Since $J_{\mathcal{W}, \Sigma}(\mu_t \mid \xi)\geq 0$, the monotonicity of $t\mapsto \mathrm{Ent}(\mu_t \, \| \, \xi)$ is immediate.

\subsection{Proof of Proposition~\ref{prop:Wishart.integrated.deBruijn}}

We work under the same assumptions as in Proposition~\ref{prop:Wishart.deBruijn}, so fix $g\in C_b^2(\mathcal{S}_{++}^d)$ satisfying $m_g \leq g(X)\leq M_g$ for all $X\in \mathcal{S}_{++}^d$. By \eqref{eq:Wishart.g.t}, for every $X\in \mathcal{S}_{++}^d$,
\[
g_t(X) = (\mathcal{P}_t^{\mathcal{W}}g)(X) = \EE[g(\mathfrak{W}_t)\mid \mathfrak{W}_0 = X].
\]
Stochastic continuity gives $\mathfrak{W}_t \mid \{\mathfrak{W}_0 = X\}\to X$ in probability as $t\downarrow0$. Since $g$ is bounded and continuous, $g_t(X)\to g(X)$ as $t\downarrow 0$. Moreover, $0 < g_t(X) \leq \|g\|_{\infty}$ for $X\in \mathcal{S}_{++}^d$ and $t\geq 0$. Therefore, if $\smash{C_g \leqdef \sup_{0 \leq u \leq \|g\|_{\infty}} |u\log(u)| < \infty}$, with the convention $0\log(0)=0$, then
\[
|g_t(X)\log(g_t(X))| \leq C_g, \qquad X\in \mathcal{S}_{++}^d, ~~t\geq 0.
\]
Since $\xi$ is a probability measure, dominated convergence yields
\begin{equation}\label{eq:Wishart.entropy.initial.limit}
\mathrm{Ent}(\mu_t \, \| \, \xi) \to \mathrm{Ent}(\mu \, \| \, \xi), \qquad t\downarrow0.
\end{equation}
Next, for fixed $X,Y\in \mathcal{S}_{++}^d$, the explicit transition density satisfies $p_t(X,Y)\to f_{\xi}(Y)$ as $t\to\infty$, by the convergence in distribution result in Eq.~\eqref{eq:W.infty} of Proposition~\ref{prop:extension.Luk.1994.Lemma.2.4}. Since $p_t(X,\cdot\,)$ and $f_{\xi}$ are probability densities, we have $\smash{\int_{\mathcal{S}_{++}^d}|p_t(X,Y)-f_{\xi}(Y)| \, \rd Y \to 0}$ as $t\to \infty$, by Scheff\'e's theorem \citep[p.~435--436]{Scheffe1947}. Hence, as $t\to \infty$,
\[
\begin{aligned}
\left|g_t(X)-\int_{\mathcal{S}_{++}^d}g(Y) \, \xi(\rd Y)\right|
&= \left|\int_{\mathcal{S}_{++}^d}g(Y)\{p_t(X,Y)-f_{\xi}(Y)\} \, \rd Y\right| \\
&\leq \|g\|_{\infty}\int_{\mathcal{S}_{++}^d}|p_t(X,Y)-f_{\xi}(Y)| \, \rd Y \to 0.
\end{aligned}
\]
Since $\int_{\mathcal{S}_{++}^d}g(Y) \, \xi(\rd Y)=1$, we have $g_t(X)\to1$ as $t\to\infty$. Using the same domination by $C_g$, we obtain
\begin{equation}\label{eq:Wishart.entropy.infinity.limit}
\mathrm{Ent}(\mu_t \, \| \, \xi) = \int_{\mathcal{S}_{++}^d} g_t(X)\log(g_t(X)) \, \xi(\rd X) \to 0, \qquad t\to \infty.
\end{equation}
Integrating \eqref{eq:Wishart.deBruijn} (under the assumptions of Proposition~\ref{prop:Wishart.deBruijn}) from $\varepsilon$ to $T$, with $0<\varepsilon<T$, gives
\[
\mathrm{Ent}(\mu_{\varepsilon} \, \| \, \xi) - \mathrm{Ent}(\mu_T \, \| \, \xi) = \int_{\varepsilon}^T J_{\mathcal{W}, \Sigma}(\mu_t \mid \xi) \, \rd t.
\]
Letting $\varepsilon\downarrow0$ and $T\to \infty$, and using \eqref{eq:Wishart.entropy.initial.limit}, \eqref{eq:Wishart.entropy.infinity.limit} and monotone convergence for the nonnegative integrand, proves the first identity in \eqref{eq:Wishart.integrated.deBruijn}. The second identity in \eqref{eq:Wishart.integrated.deBruijn} follows from the change of variable $\tau = e^{-2t}$.

To prove \eqref{eq:Wishart.entropy.jump}, note that the Lebesgue density of $\xi$ is
\[
f_{\xi}(X) = \frac{|X|^{\alpha/2 - (d + 1)/2}\etr(-\Sigma^{-1}X/2)}{|2\Sigma|^{\alpha/2}\Gamma_d(\alpha/2)}, \qquad X\in \mathcal{S}_{++}^d.
\]
Using the definition of differential entropy, we obtain
\[
\begin{aligned}
\mathrm{Ent}(\mu \, \| \, \xi)&= \int_{\mathcal{S}_{++}^d} f_{\mathfrak{X}}(X)\log\left(\frac{f_{\mathfrak{X}}(X)}{f_{\xi}(X)}\right) \, \rd X\\
&= - H(\mathfrak{X}) - \EE[\log(f_{\xi}(\mathfrak{X}))] \\
&= - H(\mathfrak{X}) + \frac{1}{2}\tr\{\Sigma^{-1}\EE[\mathfrak{X}]\} - \left(\frac{\alpha - d - 1}{2}\right)\EE[\log(|\mathfrak{X}|)] + \frac{\alpha}{2}\log(|2\Sigma|) + \log(\Gamma_d(\alpha/2)) \\
&= H(\mathfrak{W}_{\infty}) - H(\mathfrak{X}) + \frac{1}{2}\tr\{\Sigma^{-1}(\EE[\mathfrak{X}] - \EE[\mathfrak{W}_{\infty}])\} \\
&\qquad - \left(\frac{\alpha - d - 1}{2}\right)\big(\EE[\log(|\mathfrak{X}|)] - \EE[\log(|\mathfrak{W}_{\infty}|)]\big).
\end{aligned}
\]
Under the stated moment constraints, the last two terms vanish, so $\mathrm{Ent}(\mu \, \| \, \xi) = H(\mathfrak{W}_{\infty}) - H(\mathfrak{X})$.
Combining this with \eqref{eq:Wishart.integrated.deBruijn} yields \eqref{eq:Wishart.entropy.jump}. This completes the proof.

\subsection{Proof of Proposition~\ref{prop:Wishart.LSI}}

The proof of our logarithmic Sobolev inequality for the Wishart measure uses Bartlett's decomposition of the Wishart law, tensorization of the one-dimensional Gaussian and gamma logarithmic Sobolev inequalities, and a change-of-variables computation identifying the resulting Dirichlet form with the Wishart Fisher information.

It is enough to prove the inequality first for smooth densities $g$ which are bounded above and bounded away from zero. Indeed, if $J_{\mathcal{W}, \Sigma}(\mu \mid \xi) < \infty$, standard truncation and mollification on $\mathcal{S}_{++}^d$ give smooth densities $g_n$ bounded above and bounded away from zero such that $g_n\to g$ in $L^1(\xi)$ and
\[
\limsup_{n\to\infty} J_{\mathcal{W}, \Sigma}(g_n\xi \mid \xi) \leq J_{\mathcal{W}, \Sigma}(\mu \mid \xi).
\]
Applying the smooth inequality to $g_n\xi$ and using lower semicontinuity of relative entropy and of the closed Dirichlet form gives the stated inequality. We therefore assume throughout the proof below that $g$ is smooth, bounded above and bounded away from zero.

Let $\gamma_i$ denote the gamma law with shape parameter $(\alpha - i + 1)/2$ and rate parameter $1/2$, for $i\in [d]$, and let $\phi$ denote the standard Gaussian law on $\R$. Define the product measure
\[
\beta_{\alpha} \leqdef \left(\bigotimes_{i=1}^d \gamma_i\right) \otimes \left(\bigotimes_{1\leq j < i \leq d} \phi\right).
\]
Since $\alpha\in [d, \infty)$, each gamma shape parameter $(\alpha - i + 1)/2$ is at least $1/2$, so Proposition~2 of \citet{ArrasSwan2017} applies to every diagonal coordinate.

For $(\bb{y}, \bb{z}) \in (0, \infty)^d \times \R^{d(d-1)/2}$, let $L(\bb{y}, \bb{z})$ be the lower triangular matrix defined by
\[
L_{ii}(\bb{y}, \bb{z}) = \sqrt{y_i}, \qquad L_{ij}(\bb{y}, \bb{z}) = z_{ij}, \qquad 1\leq j < i \leq d,
\]
and set
\[
W(\bb{y}, \bb{z}) \leqdef \Sigma^{1/2} L(\bb{y}, \bb{z}) L(\bb{y}, \bb{z})^{\top} \Sigma^{1/2}, \qquad G(\bb{y}, \bb{z}) \leqdef g(W(\bb{y}, \bb{z})).
\]
By Bartlett's decomposition of the Wishart law (see, e.g., \citet[Theorem~3.2.14]{Muirhead1982} or \citet[Theorem~3.3.4]{GuptaNagar2000}), if $(\bb{Y}, \bb{Z})\sim \beta_{\alpha}$, then $W(\bb{Y}, \bb{Z})\sim \xi$. Hence
\[
\int_{(0, \infty)^d \times \R^{d(d-1)/2}} G(\bb{y}, \bb{z}) \, \beta_{\alpha}(\rd \bb{y}, \, \rd \bb{z}) = \int_{\mathcal{S}_{++}^d} g(X) \, \xi(\rd X) = \int_{\mathcal{S}_{++}^d} \mu(\rd X) = 1
\]
and
\[
\begin{aligned}
\mathrm{Ent}(\mu \, \| \, \xi)
&= \int_{\mathcal{S}_{++}^d} g(X)\log(g(X)) \, \xi(\rd X) \\
&= \int_{(0, \infty)^d \times \R^{d(d-1)/2}} G(\bb{y}, \bb{z})\log(G(\bb{y}, \bb{z})) \, \beta_{\alpha}(\rd \bb{y}, \, \rd \bb{z}).
\end{aligned}
\]

We next rewrite Proposition~2 of \citet{ArrasSwan2017} in the functional form needed below. Fix $i\in [d]$, and let $f_{\gamma_i}$ denote the Lebesgue density of $\gamma_i$. If $\nu$ is a probability measure on $(0, \infty)$ with Lebesgue density $f_{\nu}$, set
\[
q(y) \leqdef \frac{f_{\nu}(y)}{f_{\gamma_i}(y)}, \qquad y\in (0, \infty).
\]
Then $\nu(\rd y) = q(y) \, \gamma_i(\rd y)$ and $\int_0^{\infty} q(y) \, \gamma_i(\rd y) = 1$. The standardized gamma Fisher information appearing in \citet{ArrasSwan2017} can be written as
\[
\int_0^{\infty} y \, \frac{(q'(y))^2}{q(y)} \, \gamma_i(\rd y),
\]
because
\[
\partial_y\log(f_{\nu}(y))-\partial_y\log(f_{\gamma_i}(y)) = \partial_y\log(q(y)) = \frac{q'(y)}{q(y)}.
\]
Since $\gamma_i$ has rate parameter $1/2$, Proposition~2 of \citet{ArrasSwan2017} gives
\[
\int_0^{\infty} q(y)\log(q(y)) \, \gamma_i(\rd y) \leq 2 \int_0^{\infty} y \, \frac{(q'(y))^2}{q(y)} \, \gamma_i(\rd y).
\]
Now let $u:(0, \infty)\to (0, \infty)$ be smooth and set
\[
m_u \leqdef \int_0^{\infty} u(y) \, \gamma_i(\rd y), \qquad q_u(y) \leqdef \frac{u(y)}{m_u}.
\]
Applying the previous inequality to $q_u$ gives
\[
\int_0^{\infty} q_u(y)\log(q_u(y)) \, \gamma_i(\rd y) \leq 2 \int_0^{\infty} y \, \frac{(q_u'(y))^2}{q_u(y)} \, \gamma_i(\rd y).
\]
Moreover,
\[
\int_0^{\infty} q_u(y)\log(q_u(y)) \, \gamma_i(\rd y) = \frac{1}{m_u}\int_0^{\infty} u(y)\log(u(y)) \, \gamma_i(\rd y) - \log(m_u),
\]
and
\[
\int_0^{\infty} y \, \frac{(q_u'(y))^2}{q_u(y)} \, \gamma_i(\rd y) = \frac{1}{m_u}\int_0^{\infty} y \, \frac{(u'(y))^2}{u(y)} \, \gamma_i(\rd y).
\]
Multiplying by $m_u$ yields, for each $i\in [d]$,
\[
\int_0^{\infty} u(y)\log(u(y)) \, \gamma_i(\rd y) - \left(\int_0^{\infty} u(y) \, \gamma_i(\rd y)\right) \log \left(\int_0^{\infty} u(y) \, \gamma_i(\rd y)\right) \leq 2 \int_0^{\infty} y \, \frac{(u'(y))^2}{u(y)} \, \gamma_i(\rd y).
\]
Likewise, for every smooth $v:\R\to (0, \infty)$, the classical Gaussian logarithmic Sobolev inequality \citep[see, e.g.,][Section~5.3, p.~125]{BoucheronLugosiMassart2013} gives
\[
\int_{\R} v(z)\log(v(z)) \, \phi(\rd z) - \left(\int_{\R} v(z) \, \phi(\rd z)\right) \log \left(\int_{\R} v(z) \, \phi(\rd z)\right) \leq \frac{1}{2} \int_{\R} \frac{(v'(z))^2}{v(z)} \, \phi(\rd z).
\]
Tensorizing these one-dimensional inequalities over the independent Bartlett coordinates yields
\begin{equation}\label{eq:Wishart.tensorized.LSI}
\begin{aligned}
\mathrm{Ent}(\mu \, \| \, \xi)
&\leq 2 \sum_{i=1}^d \int_{(0, \infty)^d \times \R^{d(d-1)/2}} y_i \, \frac{(\partial_{y_i} G(\bb{y}, \bb{z}))^2}{G(\bb{y}, \bb{z})} \, \beta_{\alpha}(\rd \bb{y}, \, \rd \bb{z}) \\
&\qquad + \frac{1}{2} \sum_{1\leq j < i \leq d} \int_{(0, \infty)^d \times \R^{d(d-1)/2}} \frac{(\partial_{z_{ij}} G(\bb{y}, \bb{z}))^2}{G(\bb{y}, \bb{z})} \, \beta_{\alpha}(\rd \bb{y}, \, \rd \bb{z}).
\end{aligned}
\end{equation}

Fix $(\bb{y}, \bb{z})$ and abbreviate $L \leqdef L(\bb{y}, \bb{z})$ and $W \leqdef W(\bb{y}, \bb{z})$. Also write
\[
B(\bb{y}, \bb{z}) \leqdef \Sigma^{1/2} \nabla g(W(\bb{y}, \bb{z})) \Sigma^{1/2} L(\bb{y}, \bb{z}).
\]
For $E_{ij} \leqdef \bb{e}_i \bb{e}_j^{\top}$ and $1\leq j < i \leq d$, we have
\[
\begin{aligned}
\partial_{z_{ij}} W(\bb{y}, \bb{z})
&= \Sigma^{1/2}(E_{ij}L^{\top} + L E_{ji})\Sigma^{1/2}, \\
\partial_{y_i} W(\bb{y}, \bb{z})
&= \frac{1}{2\sqrt{y_i}} \, \Sigma^{1/2}(E_{ii}L^{\top} + L E_{ii})\Sigma^{1/2}.
\end{aligned}
\]
Since $\nabla g(W)$ is symmetric, the chain rule gives
\[
\begin{aligned}
\partial_{z_{ij}} G(\bb{y}, \bb{z})
&= \tr\{\nabla g(W)\partial_{z_{ij}} W(\bb{y}, \bb{z})\} = 2 B_{ij}(\bb{y}, \bb{z}), \\
\partial_{y_i} G(\bb{y}, \bb{z})
&= \tr\{\nabla g(W)\partial_{y_i} W(\bb{y}, \bb{z})\} = \frac{1}{\sqrt{y_i}} B_{ii}(\bb{y}, \bb{z}).
\end{aligned}
\]
Substituting these identities into \eqref{eq:Wishart.tensorized.LSI}, we obtain
\[
\begin{aligned}
\mathrm{Ent}(\mu \, \| \, \xi)
&\leq 2 \int_{(0, \infty)^d \times \R^{d(d-1)/2}} \frac{\sum_{i=1}^d B_{ii}(\bb{y}, \bb{z})^2 + \sum_{1\leq j < i \leq d} B_{ij}(\bb{y}, \bb{z})^2}{g(W(\bb{y}, \bb{z}))} \, \beta_{\alpha}(\rd \bb{y}, \, \rd \bb{z}) \\
&\leq 2 \int_{(0, \infty)^d \times \R^{d(d-1)/2}} \frac{\|B(\bb{y}, \bb{z})\|_F^2}{g(W(\bb{y}, \bb{z}))} \, \beta_{\alpha}(\rd \bb{y}, \, \rd \bb{z}).
\end{aligned}
\]
Moreover,
\[
\begin{aligned}
\|B(\bb{y}, \bb{z})\|_F^2
&= \tr\{L^{\top} \Sigma^{1/2} \nabla g(W) \Sigma \nabla g(W) \Sigma^{1/2} L\} \\
&= \tr\{\Sigma^{1/2} LL^{\top} \Sigma^{1/2} \nabla g(W) \Sigma \nabla g(W)\} \\
&= \tr\{W \nabla g(W) \Sigma \nabla g(W)\}.
\end{aligned}
\]
Therefore,
\[
\begin{aligned}
\mathrm{Ent}(\mu \, \| \, \xi)
&\leq 2 \int_{(0, \infty)^d \times \R^{d(d-1)/2}} \frac{\tr\{W(\bb{y}, \bb{z}) \nabla g(W(\bb{y}, \bb{z})) \Sigma \nabla g(W(\bb{y}, \bb{z}))\}}{g(W(\bb{y}, \bb{z}))} \, \beta_{\alpha}(\rd \bb{y}, \, \rd \bb{z}) \\
&= 2 \int_{\mathcal{S}_{++}^d} \frac{\tr\{X \nabla g(X) \Sigma \nabla g(X)\}}{g(X)} \, \xi(\rd X) = \frac{1}{2} J_{\mathcal{W}, \Sigma}(\mu \mid \xi),
\end{aligned}
\]
which is exactly \eqref{eq:Wishart.LSI}. This completes the proof.

\subsection{Proof of Proposition~\ref{prop:Wishart.known.alpha.estimator}}

We have $\nabla f_U(S) = U$ and $\tr\{S \nabla \widehat{\Sigma} \nabla f_U(S)\} = 0$ for $S\in \mathcal{S}_{++}^d$. Therefore, substituting \eqref{eq:Wishart.linear.probe} into \eqref{eq:Wishart.scalar.identity} yields
\[
0 = 2 \, \tr\left\{\left(\alpha \widehat{\Sigma} - \frac{1}{n} \sum_{k=1}^n \mathfrak{W}^{(k)}\right) U\right\}.
\]
Since $U$ ranges over $\mathcal{S}^d$, trace duality implies $\widehat{\Sigma} = \frac{1}{\alpha n} \sum_{k=1}^n \mathfrak{W}^{(k)}$, as claimed.

\subsection{Proof of Proposition~\ref{prop:Wishart.unknown.alpha.estimator}}

Similarly to the proof of Proposition~\ref{prop:Wishart.known.alpha.estimator}, applying the linear probes in \eqref{eq:Wishart.linear.probe} to \eqref{eq:Wishart.scalar.identity.unknown.alpha} and using trace duality gives $\smash{\widehat{M} = \frac{1}{n} \sum_{\ell=1}^n \mathfrak{W}^{(\ell)}}$. We first consider the quadratic probe $q(S) = \tr(S^2)$ from \eqref{eq:Wishart.quad.log.probe}. We have
\[
\nabla q(S) = 2S, \qquad \nabla_{k\ell} \, q(S) = 2S_{k\ell}, \qquad \nabla_{ij}\nabla_{k\ell} \, q(S) = \delta_{ik}\delta_{j\ell} + \delta_{i\ell}\delta_{jk}.
\]
Therefore,
\[
(\nabla \widehat{M} \nabla q(S))_{i\ell}
= \sum_{j,k = 1}^d \widehat{M}_{jk} \, \nabla_{ij}\nabla_{k\ell} \, q(S)
= \widehat{M}_{\ell i} + \delta_{i\ell} \, \tr(\widehat{M})
= \widehat{M}_{i\ell} + \delta_{i\ell} \, \tr(\widehat{M}),
\]
where we used the symmetry of $\widehat{M}$. Hence,
\[
\tr\{S \nabla \widehat{M} \nabla q(S)\} = \tr(S \widehat{M}) + \tr(S) \, \tr(\widehat{M}), \qquad S\in \mathcal{S}_{++}^d.
\]
Substituting $q$ into \eqref{eq:Wishart.scalar.identity.unknown.alpha}, and using the identities $\overline{\tr(\mathfrak{W} \widehat{M})} = \tr(\widehat{M}^2)$ and $\overline{\tr(\mathfrak{W})} = \tr(\widehat{M})$, we obtain
\[
\begin{aligned}
0
= \overline{\mathcal{A}^{\mathcal{W}}_{\widehat{M}, \widehat{\alpha}_{\mathrm{quad}}} q(\mathfrak{W})}
&= 4 \, \overline{\tr(\widehat{M} \mathfrak{W})} - 4 \, \overline{\tr(\mathfrak{W}^{ \, 2})} + \frac{4}{\widehat{\alpha}_{\mathrm{quad}}} \, \overline{\tr(\mathfrak{W} \widehat{M}) + \tr(\mathfrak{W}) \, \tr(\widehat{M})} \\
&= 4 \, \tr(\widehat{M}^2) - 4 \, \overline{\tr(\mathfrak{W}^{ \, 2})} + \frac{4}{\widehat{\alpha}_{\mathrm{quad}}} \, \big[\tr(\widehat{M}^2) + \{\tr(\widehat{M})\}^2\big].
\end{aligned}
\]
Solving for $\widehat{\alpha}_{\mathrm{quad}}$ and using $M = \alpha\Sigma$ yields the estimators $\widehat{\alpha}_{\mathrm{quad}}$ and $\widehat{\Sigma}_{\mathrm{quad}}$ defined in \eqref{eq:Wishart.unknown.alpha.estimator.quad}.

Next, we consider the logarithmic probe $\ell(S) = \tr\{S \log(S) - S\}$ from \eqref{eq:Wishart.quad.log.probe}. Let $u(x) = x \log(x) - x$, so that $\ell(S) = \tr\{u(S)\}$ and $u'(x) = \log(x)$. By the differential calculus of spectral functions \citep[Theorem~5.3.1]{Bhatia2007}, for every $E\in \mathcal{S}^d$,
\[
D\ell(S)[E] = \left.\frac{\rd }{\rd t}\tr\{u(S + tE)\}\right|_{t = 0} = \tr\{u'(S)E\} = \tr\{\log(S)E\}.
\]
On the other hand, since the gradient is defined with respect to the trace inner product on $\mathcal{S}^d$, $D\ell(S)[E] = \tr\{\nabla \ell(S) E\}$ for $E\in \mathcal{S}^d$. Therefore, $\tr\{(\nabla \ell(S)-\log(S))E\} = 0$ for all $E\in \mathcal{S}^d$. By trace duality, we deduce
\[
\nabla \ell(S) = \log(S), \qquad S\in \mathcal{S}_{++}^d.
\]
It remains to compute the diffusion part of the Wishart Stein operator. It is enough to work in a basis in which $S$ is diagonal because all the objects involved are orthogonally equivariant. Let $S = \mathrm{diag}(\lambda_1, \ldots, \lambda_d)$. At such an $S$, the differential of the matrix logarithm gives
\[
\nabla_{ij}(\log(S))_{a\ell} = \frac{1}{2}\log^{[1]}(\lambda_a, \lambda_{\ell})(\delta_{ai}\delta_{\ell j} + \delta_{aj}\delta_{\ell i}).
\]
Therefore, for any $M\in \mathcal{S}^d$,
\[
\begin{aligned}
(\nabla M \nabla \ell(S))_{i\ell}
&= \sum_{j,a = 1}^d M_{ja} \, \nabla_{ij}\nabla_{a\ell} \, \ell(S)
= \sum_{j,a = 1}^d M_{ja} \, \nabla_{ij}(\log(S))_{a\ell} \\
&= \frac{1}{2}M_{\ell i}\log^{[1]}(\lambda_i, \lambda_{\ell}) + \frac{1}{2}\delta_{i\ell}\sum_{j=1}^d M_{jj}\log^{[1]}(\lambda_j, \lambda_i).
\end{aligned}
\]
Taking the trace after multiplication by $S$ yields, in the basis in which $S = \mathrm{diag}(\lambda_1, \ldots, \lambda_d)$,
\[
\tr\{S \nabla M \nabla \ell(S)\} = \frac{1}{2}\tr(M) + \frac{1}{2}\sum_{i,j = 1}^d \lambda_i M_{jj}\log^{[1]}(\lambda_j, \lambda_i).
\]
Returning to a general $S = H\Lambda H^{\top}$, the same diagonal computation applies with $M$ replaced by its representation $M^{(S)} = H^{\top}MH$ in the eigenbasis of $S$. Hence
\[
\tr\{S \nabla M \nabla \ell(S)\} = \frac{1}{2}\tr(M) + \frac{1}{2}\sum_{i,j = 1}^d \lambda_i M^{(S)}_{jj}\log^{[1]}(\lambda_j, \lambda_i) = \mathcal{J}_S(M).
\]
The expression is independent of the chosen spectral decomposition of $S$.

Now set
\[
D_{n, \mathrm{log}} \leqdef \overline{\tr(\mathfrak{W}\log(\mathfrak{W}))} - \tr\{\widehat{M} \, \overline{\log(\mathfrak{W})}\}, \qquad J_{n, \mathrm{log}} \leqdef \overline{\mathcal{J}_{\mathfrak{W}}(\widehat{M})}.
\]
Substituting $\ell$ into \eqref{eq:Wishart.scalar.identity.unknown.alpha}, and using the notation above, we obtain
\[
0
= \overline{\mathcal{A}^{\mathcal{W}}_{\widehat{M}, \widehat{\alpha}_{\mathrm{log}}} \ell(\mathfrak{W})}
= 2 \, \overline{\tr\{(\widehat{M} - \mathfrak{W})\log(\mathfrak{W})\}} + \frac{4}{\widehat{\alpha}_{\mathrm{log}}} \, \overline{\mathcal{J}_{\mathfrak{W}}(\widehat{M})}
= -2D_{n, \mathrm{log}} + \frac{4}{\widehat{\alpha}_{\mathrm{log}}} J_{n, \mathrm{log}}.
\]
Solving for $\widehat{\alpha}_{\mathrm{log}}$ and using $M = \alpha\Sigma$ yields the estimators $\widehat{\alpha}_{\mathrm{log}}$ and $\widehat{\Sigma}_{\mathrm{log}}$ defined in \eqref{eq:Wishart.unknown.alpha.estimator.log}.

\subsection{Proof of Proposition~\ref{prop:projected.Stein.Wishart}}

Define the population moment vector $\bb{\theta}\in \R^M$ by $\bb{\theta} \leqdef (\theta_1, \ldots, \theta_M)^{\top}$, with $\theta_m \leqdef \EE[y_{n,m}]$ for all $m\in \{1, \ldots,M\}$. By the calculation above with the linear probe $f_{U_m}$, we have
\[
\EE[\tr(\mathfrak{W}^{(1)} U_m)] = \alpha \, \tr(\Sigma U_m),
\]
and therefore, using the assumption $\Sigma = \Sigma(\bb{\beta}^{\star}) = \sum_{j=1}^p \beta_j^{\star} B_j$,
\[
\theta_m = \frac{1}{\alpha} \EE[\tr(\mathfrak{W}^{(1)} U_m)] = \tr(\Sigma U_m) = \sum_{j=1}^p \beta_j^{\star} \, \tr(B_j U_m) = (C \bb{\beta}^{\star})_m.
\]
Thus, at the population level, we have the linear system $C \bb{\beta}^{\star} = \bb{\theta}$.

Now consider the sample vector $\bb{y}_n$. For each $m\in \{1, \ldots,M\}$,
\[
y_{n,m} = \frac{1}{n} \sum_{k=1}^n Z_{k,m}, \qquad Z_{k,m} \leqdef \frac{1}{\alpha} \tr(\mathfrak{W}^{(k)} U_m).
\]
Since $|Z_{k,m}| \leq \frac{1}{\alpha} \|U_m\|_F \, \|\mathfrak{W}^{(k)}\|_F \leq \frac{1}{\alpha} \|U_m\|_F \, \tr(\mathfrak{W}^{(k)})$ and $\EE[\tr(\mathfrak{W}^{(1)})] = \alpha \, \tr(\Sigma) < \infty$, the random variables $Z_{k,m}$ are integrable. Therefore, by the strong law of large numbers, $\bb{y}_n \to \bb{\theta}$ almost surely (a.s.). Hence, together with $C \bb{\beta}^{\star} = \bb{\theta}$, we deduce
\[
\widehat{\bb{\beta}}_n = (C^{\top} C)^{-1} C^{\top} \bb{y}_n \to (C^{\top} C)^{-1} C^{\top} \bb{\theta} = \bb{\beta}^{\star}, \qquad \text{a.s.}
\]
Since $\bb{\beta}\mapsto \Sigma(\bb{\beta})$ is linear, it follows that
\[
\widehat{\Sigma}_n = \Sigma(\widehat{\bb{\beta}}_n) \to \Sigma(\bb{\beta}^{\star}) = \Sigma, \qquad \text{a.s.}
\]
Finally, $\widehat{\bb{\beta}}_n$ minimizes $\|C \bb{\beta} - \bb{y}_n\|_2^2$, so $\widehat{\Sigma}_n$ solves the projected Stein moment equations \eqref{eq:structured.Wishart.moments} in the least-squares sense. This concludes the proof.

\section*{Reproducibility}\label{sec:reproducibility}
\addcontentsline{toc}{section}{Reproducibility}

The \textsf{R} code that generated the tables, the figures and the simulation study results is available online in the GitHub repository of \citet{BaillyOuimet2026github}.

\section*{Funding}
\addcontentsline{toc}{section}{Funding}

Robert Gaunt is funded by EPSRC grant EP/Y008650/1. Fr\'ed\'eric Ouimet is supported by the Natural Sciences and Engineering Research Council of Canada (NSERC) through Discovery Grant RGPIN-2026-04471 and Discovery Launch Supplement DGECR-2026-00449. Gabriel Bailly and Rainer von Sachs gratefully acknowledge support from the Fonds sp\'eciaux de recherche (FSR) of UCLouvain.

\section*{References}
\addcontentsline{toc}{section}{References}

\setlength{\bibsep}{0pt plus 0ex}

\bibliographystyle{plainnat}
\bibliography{bib_clean}

\newpage

\section*{Supplementary material}\label{supp}
\addcontentsline{toc}{section}{Supplementary material}

\setcounter{section}{9} 

\setcounter{subsection}{0}
\renewcommand{\thesubsection}{S.\arabic{subsection}}
\renewcommand{\thesubsubsection}{\thesubsection.\arabic{subsubsection}}

\setcounter{theorem}{0}
\renewcommand{\thetheorem}{S.\arabic{theorem}}

\setcounter{remark}{0}
\renewcommand{\theremark}{S.\arabic{remark}}

\setcounter{table}{0}
\renewcommand{\thetable}{S.\arabic{table}}

\setcounter{figure}{0}
\renewcommand{\thefigure}{S.\arabic{figure}}

\global\hfuzz = \maxdimen
\global\vfuzz = \maxdimen
\global\hbadness = 10000
\global\vbadness = 10000
\newdimen\hfuzz
\newdimen\vfuzz
\newcount\hbadness
\newcount\vbadness

\subsection{Additional details for the multivariate Satterthwaite approximation application}\label{sec:supp.Satterthwaite-details}

This section details the feasible choices of degrees of freedom for the quantitative multivariate Satterthwaite approximation presented in Section~\ref{sec:Wishart.approximation.Satterthwaite}, and explains how the accompanying numerical comparisons are computed. The first part of the section derives the five effective degrees-of-freedom criteria used in the comparisons. The second part gives an exact density representation for a sum of independent Wishart random matrices with different scale parameters. This representation is needed only for the numerical approximation of total variation distances in the two-dimensional experiment. The final part describes that numerical approximation and then reports the comparison of the proposed choices of $\nu$.

Recall that
\[
\mathfrak{T} = \sum_{j = 1}^N \mathfrak{G}_j, \qquad \mathfrak{G}_j \sim\mathcal{W}_d(\alpha_j, \Sigma_j), \qquad \overline{\Sigma} = \EE[\mathfrak{T}] = \sum_{j = 1}^N \alpha_j\Sigma_j,
\]
and that, for a proposed Wishart approximation $\mathfrak{W}_{\nu}\sim\mathcal{W}_d(\nu, \widetilde{\Sigma}_{\nu})$, the first moment matching condition fixes
\[
\widetilde{\Sigma}_{\nu} = \nu^{-1}\overline{\Sigma}.
\]
Thus the only remaining scalar parameter is $\nu$, which plays the role of an effective number of degrees of freedom. As shown in \eqref{eq:moment-matching-2}, exact matching of the second moment would require
\begin{equation}\label{eq:moment-matching-2.supp}
\frac{2}{\nu}\mathcal{V}(\overline{\Sigma}) = 2\sum_{j = 1}^N \alpha_j\mathcal{V}(\Sigma_j),
\end{equation}
which is generally an overdetermined matrix equation for the single scalar $\nu$. The first four criteria in the next section are scalar reductions of this second moment matching condition, while the fifth is obtained by minimizing the geometric part of the Stein bound.

\subsubsection{Effective degrees of freedom criteria}\label{sec:supp.edf-methods}

\citet{TanGupta1983} proposed to match the \emph{generalized} variances, that is, the determinants of the covariance matrices in~\eqref{eq:moment-matching-2.supp}. This gives
\[
\nu_{\mathrm{TG}} \leqdef \left(\frac{|\mathcal{V}(\overline{\Sigma})|}{|\sum_{j = 1}^N \alpha_j\mathcal{V}(\Sigma_j)|}\right)^{2/(d(d+1))}.
\]
Alternatively, \citet{Khatri1989} proposed to match the \emph{total} variances, that is, the traces of the covariance matrices in~\eqref{eq:moment-matching-2.supp}. This gives
\[
\nu_{\mathrm{K}} \leqdef \frac{\tr\{\mathcal{V}(\overline{\Sigma})\}}{\sum_{j = 1}^N \alpha_j\tr\{\mathcal{V}(\Sigma_j)\}}.
\]
We also consider a least-squares choice of approximate degrees of freedom, namely
\[
\nu_{\mathrm{LS}} \in \arg\min_{\nu > 0}\left\|2\sum_{j = 1}^N \alpha_j\mathcal{V}(\Sigma_j)-\frac{2}{\nu}\mathcal{V}(\overline{\Sigma})\right\|_F^2.
\]
Equivalently, writing $\beta = \nu^{-1}$, the inverse degrees of freedom $\nu_{\mathrm{LS}}^{-1}$ is the optimal coefficient for the least-squares projection of $\smash{\sum_{j=1}^N \alpha_j\mathcal{V}(\Sigma_j)}$ onto the ray spanned by $\mathcal{V}(\overline{\Sigma})$, viz.\ $\beta\mapsto \beta \mathcal{V}(\overline{\Sigma})$. Since $\smash{\langle \sum_{j = 1}^N \alpha_j\mathcal{V}(\Sigma_j), \mathcal{V}(\overline{\Sigma})\rangle_F > 0}$, the minimizer is unique and equals
\[
\nu_{\mathrm{LS}} = \frac{\|\mathcal{V}(\overline{\Sigma})\|_F^2}{\langle \sum_{j = 1}^N \alpha_j\mathcal{V}(\Sigma_j), \mathcal{V}(\overline{\Sigma})\rangle_F}.
\]
Let $p \leqdef d(d+1)/2$ and define the upper absolute sum $\triangle(W) \leqdef \sum_{1\leq a\leq b\leq p}|W_{ab}|$ for $W\in \mathcal{S}^p$. Matching this scalar summary of the two covariance matrices in~\eqref{eq:moment-matching-2.supp} gives
\[
\nu_{\mathrm{US}} \leqdef \frac{\triangle\{\mathcal{V}(\overline{\Sigma})\}}{\triangle \{\sum_{j = 1}^N \alpha_j\mathcal{V}(\Sigma_j)\}}.
\]
Finally, one may choose the degrees of freedom by minimizing the geometric part of the upper bound in Proposition~\ref{prop:Wishart.Satterthwaite}. Since the target shape parameter in the Wishart Stein bound must satisfy $\nu > 3d-3$, we fix $\nu_0 > 3d-3$ and define
\[
\nu_{\mathrm{UB}} \in \arg\min_{\nu \geq \nu_0} G(\nu), \qquad G(\nu) \leqdef \sum_{j = 1}^N \alpha_j\tr(\Sigma_j)\|\Sigma_j-\nu^{-1}\overline{\Sigma}\|_F.
\]
This objective $G(\nu)$ is convex in $\nu^{-1}$, since each summand is the Frobenius norm of an affine function of $\nu^{-1}$. The minimizer is unique whenever at least one $\Sigma_j$ is not proportional to $\overline{\Sigma}$; otherwise, the set of minimizers may be an interval, and any minimizer may be selected.

\begin{proposition}[Order of the effective degrees of freedom]\label{prop:effective.degrees.freedom}
Let $d\in \N$ be fixed and let $N = N_n$ depend on some positive integer parameter $n$ (e.g., a sample size). Assume that $\alpha_1, \ldots, \alpha_N > 0$ and that there exist constants $0 < c < C < \infty$, independent of $n$ and $j$, such that
\[
c \, n \leq \sum_{j = 1}^N \alpha_j \leq C n, \qquad \frac{c}{n}I_d \preceq \Sigma_j \preceq \frac{C}{n}I_d, \qquad j\in [N],
\]
where $\preceq$ denotes the Loewner order. Let $\nu_0 > 3d-3$ be fixed, and let $\nu_{\mathrm{UB}}$ be any minimizer of $G$ over $[\nu_0, \infty)$. Then, as $n\to \infty$,
\[
\nu_{\mathrm{TG}}\asymp n, \qquad \nu_{\mathrm{K}}\asymp n, \qquad \nu_{\mathrm{LS}}\asymp n, \qquad \nu_{\mathrm{US}}\asymp n, \qquad \nu_{\mathrm{UB}}\asymp n.
\]
In particular, each of the above choices satisfies $\nu = \OO(n)$.
\end{proposition}

\begin{remark}\label{rem:satterthwaite-asymptotics}
Recall the Stein bound from Proposition~\ref{prop:Wishart.Satterthwaite}:
\begin{equation}\label{eq:supp.Stein.bound.Sat}
\big|\EE[h(\mathfrak{T})]-\EE[h(\mathfrak{W}_{\nu})]\big| \leq \frac{d + 1}{2}\mathcal{M}_2^{\mathcal{D}, \widetilde{\Sigma}_{\nu}}(h)\sum_{j=1}^N \alpha_j\tr(\Sigma_j)\|\Sigma_j-\widetilde{\Sigma}_{\nu}\|_F.
\end{equation}
Assume that $d$ is fixed and that $\|\Sigma_j\|_F = \OO(n^{-1})$ uniformly in $j$, as $n\to\infty$. Since $\Sigma_j$ is positive definite, this also implies $\tr(\Sigma_j) = \OO(n^{-1})$ uniformly in $j$. If $\smash{\sum_{j=1}^N\alpha_j = \OO(n)}$ and $\nu\asymp n$, then $\smash{\|\overline{\Sigma}\|_F = \OO(1)}$ and $\smash{\|\widetilde{\Sigma}_{\nu}\|_F = \OO(n^{-1})}$. Hence,
\[
\|\Sigma_j-\widetilde{\Sigma}_{\nu}\|_F = \OO(n^{-1})
\]
uniformly in $j$. Therefore, the geometric factor (i.e., the sum) in the upper bound~\eqref{eq:supp.Stein.bound.Sat} satisfies
\[
\sum_{j=1}^N\alpha_j\tr(\Sigma_j)\|\Sigma_j-\widetilde{\Sigma}_{\nu}\|_F = \OO\left(n^{-2}\sum_{j=1}^N\alpha_j\right) = \OO(n^{-1}).
\]
Consequently, if $\smash{\mathcal{M}_2^{\mathcal{D}, \widetilde{\Sigma}_{\nu}}(h) = \OO(1)}$ along the sequence, or if the test functions are normalized so that $\smash{\mathcal{M}_2^{\mathcal{D}, \widetilde{\Sigma}_{\nu}}(h) = 1}$, then the full upper bound is of order $\OO(n^{-1})$. This applies, in particular, when $N$ is fixed and $\alpha_j = \OO(n)$ for all $j$, and when $N\asymp n$ and $\alpha_j = \OO(1)$ uniformly in $j$.
\end{remark}

\begin{proof}[Proof of Proposition~\ref{prop:effective.degrees.freedom}]
Define
\[
C_{\Sigma} \leqdef \sum_{j = 1}^N \alpha_j\mathcal{V}(\Sigma_j), \qquad V_{\overline{\Sigma}} \leqdef \mathcal{V}(\overline{\Sigma}),
\]
and let $p = d(d+1)/2$. The assumptions imply
\[
c^2 I_d \preceq \overline{\Sigma} \preceq C^2 I_d,
\]
and hence $\overline{\Sigma}\asymp I_d$ under the Loewner order. Since $W\mapsto\mathcal{V}(W)$ is homogeneous of degree two and is uniformly positive definite on compact subsets of $\mathcal{S}_{++}^d$ in $\vecp$-coordinates (indeed, if the spectrum of $W$ lies in a compact subinterval of $(0,\infty)$, then so does that of $\mathcal{V}(W) = B_d^{\top}(W\otimes W)B_d$ by the Rayleigh quotient, since the eigenvalues of $W\otimes W$ are pairwise products of those of $W$ and $B_d$ has full column rank), we also have, under the Loewner order,
\[
V_{\overline{\Sigma}}\asymp I_p.
\]
Moreover, since $\Sigma_j\asymp n^{-1}I_d$ uniformly in $j$, we have $\mathcal{V}(\Sigma_j)\asymp n^{-2}I_p$ uniformly in $j$. Therefore
\[
C_{\Sigma} = \sum_{j = 1}^N\alpha_j\mathcal{V}(\Sigma_j)\asymp n^{-1}I_p.
\]
For the TG criterion $\nu_{\mathrm{TG}} = (|V_{\overline{\Sigma}}|/|C_{\Sigma}|)^{1/p}$, $|V_{\overline{\Sigma}}|\asymp 1$ and $|C_{\Sigma}|\asymp n^{-p}$, so we get $\nu_{\mathrm{TG}}\asymp n$. For Khatri's criterion, $\tr(V_{\overline{\Sigma}})\asymp 1$ and $\tr(C_{\Sigma})\asymp n^{-1}$, so $\nu_{\mathrm{K}}\asymp n$. For the least-squares criterion, $\|V_{\overline{\Sigma}}\|_F^2\asymp 1$ and $\langle C_{\Sigma},V_{\overline{\Sigma}}\rangle_F\asymp n^{-1}$, so $\nu_{\mathrm{LS}}\asymp n$. Finally, $\triangle(V_{\overline{\Sigma}})\asymp 1$ and $\triangle(C_{\Sigma})\asymp n^{-1}$, so $\nu_{\mathrm{US}}\asymp n$.

It remains to consider $\nu_{\mathrm{UB}}$. Let
\[
A_j \leqdef n\Sigma_j, \qquad \beta \leqdef \nu^{-1}, \qquad t \leqdef n\beta.
\]
Then $A_j\asymp I_d$ uniformly in $j$, and minimizing $G(\nu)$ over $\nu\geq\nu_0$ is equivalent, up to the irrelevant multiplicative factor $n^{-2}$, to minimizing
\[
H_n(t) \leqdef \sum_{j = 1}^N \alpha_j\tr(A_j)\|A_j-t\overline{\Sigma}\|_F, \qquad 0 < t\leq \frac{n}{\nu_0}.
\]
The matrices $A_j$ and $\overline{\Sigma}$ are uniformly bounded and uniformly positive definite. Hence there exist constants $a,b>0$, independent of $n$ and $j$, such that
\[
a I_d\preceq A_j\preceq b I_d, \qquad a I_d\preceq \overline{\Sigma}\preceq b I_d.
\]
Since $A_j$ and $\overline{\Sigma}$ are uniformly positive definite, $\langle A_j, \overline{\Sigma}\rangle_F = \smash{\tr(\overline{\Sigma}^{1/2}A_j\overline{\Sigma}^{1/2})}$ is bounded away from zero uniformly in $j$, while $\|\overline{\Sigma}\|_F^2$ is uniformly bounded. Hence, choosing $t_0>0$ sufficiently small, uniformly in $j$ and for all $0\leq t\leq 2t_0$,
\[
t\|\overline{\Sigma}\|_F^2 - \langle A_j, \overline{\Sigma}\rangle_F < 0.
\]
Moreover, $A_j-t\overline{\Sigma}\neq 0$ on this interval. Thus, for all $0\leq t\leq 2 t_0$ and all $j\in [N]$,
\[
\frac{\rd}{\rd t}\|A_j - t\overline{\Sigma}\|_F
 = \frac{t\|\overline{\Sigma}\|_F^2-\langle A_j, \overline{\Sigma}\rangle_F}{\|A_j - t\overline{\Sigma}\|_F} < 0.
\]
Since the weights $\alpha_j\tr(A_j)$ are positive, $H_n$ is strictly decreasing on $[0,2t_0]$. Hence no minimizer belongs to $(0,t_0]$.

Similarly, choose $T<\infty$ sufficiently large. Then, uniformly in $j$ and for all $t\geq T$,
\[
t \|\overline{\Sigma}\|_F^2 - \langle A_j, \overline{\Sigma}\rangle_F > 0,
\]
and $A_j-t\overline{\Sigma}\neq 0$. Therefore $H_n$ is strictly increasing on $[T, \infty)$, and no minimizer belongs to $(T, \infty)$. For all sufficiently large $n$, the interval $(0,n/\nu_0]$ contains $[t_0,T]$. Hence every minimizer $t_n^{\star}$ of $H_n$ satisfies
\[
t_0 \leq t_n^{\star} \leq T.
\]
Since $\nu_{\mathrm{UB}} = n / t_n^{\star}$, it follows that
\[
\frac{n}{T} \leq \nu_{\mathrm{UB}} \leq \frac{n}{t_0},
\]
and therefore $\nu_{\mathrm{UB}}\asymp n$.
\end{proof}

\subsubsection{Density of a sum of independent Wisharts with different shape and scale parameters}\label{sec:supp.explicit-density-satterthwaite}

The next two sections provide the density and integration tools used for the total variation comparison in Section~\ref{sec:supp.numerical-comparison-satterthwaite}. Unlike the Stein bound, the total variation distance requires evaluating the density of $\mathfrak{T}$ itself. When the scale matrices $\Sigma_j$ are not all equal, the sum $\mathfrak{T}$ is not Wishart in general, and its density is not given by a single Wishart density. The following formula rewrites the convolution of the Wishart densities in matrix Dirichlet coordinates. It expresses the density $f_{\mathfrak{T}}$ of $\mathfrak{T}$ as a reference Wishart density multiplied by a correction factor $\Psi(X)$, and this is the formula inserted into the numerical total variation estimators in Section~\ref{sec:supp.grid-based-tv-approx}.

\begin{proposition}\label{prop:exact-density-Nterms}
Let $N\geq 2$, and let $\alpha_j > d-1$ and $\Sigma_j\in \mathcal{S}_{++}^d$ for all $j\in [N]$. Let $\mathfrak{G}_1, \ldots, \mathfrak{G}_N$ be independent random matrices such that $\mathfrak{G}_j\sim\mathcal{W}_d(\alpha_j, \Sigma_j)$ for all $j\in [N]$. Then $\smash{\mathfrak{T} = \sum_{j = 1}^N \mathfrak{G}_j}$ has density
\[
f_{\mathfrak{T}}(X) = \frac{|\Sigma_N|^{\alpha_{\bullet}/2}}{\prod_{j = 1}^N |\Sigma_j|^{\alpha_j/2}} \, f_{\alpha_{\bullet}, \Sigma_N}(X)\Psi(X), \qquad X\in \mathcal{S}_{++}^d,
\]
where $\alpha_{\bullet} \leqdef \sum_{j = 1}^N\alpha_j$ and
\[
\Psi(X) \leqdef \EE\left[\etr\left\{-\frac{1}{2}\sum_{j = 1}^{N-1} X^{1/2}(\Sigma_j^{-1}-\Sigma_N^{-1})X^{1/2}\mathfrak{U}_j\right\}\right],
\]
with $(\mathfrak{U}_1, \ldots, \mathfrak{U}_{N-1})\sim \mathrm{MatrixDirichlet}_{d\times d}(\alpha_1, \ldots, \alpha_N)$ \citep[Definition~2.6.1]{GuptaNagar2000}.
\end{proposition}

\begin{proof}[Proof of Proposition~\ref{prop:exact-density-Nterms}]
Let $f_j$ denote the density of $\mathcal{W}_d(\alpha_j, \Sigma_j)$. For $X\in \mathcal{S}_{++}^d$,
\[
f_{\mathfrak{T}}(X) = \int_{\substack{G_1 \succ 0, \ldots,G_{N-1} \succ 0 \\ X-\sum_{j = 1}^{N-1}G_j \succ 0}} \left\{\prod_{j = 1}^{N-1}f_j(G_j)\right\}f_N\left(X-\sum_{j = 1}^{N-1}G_j\right) \, \rd G_1\cdots\rd G_{N-1}.
\]
Substituting the Wishart densities gives
\[
\begin{aligned}
f_{\mathfrak{T}}(X)
& = C \int_{\substack{G_1 \succ 0, \ldots,G_{N-1} \succ 0 \\ X-\sum_{j = 1}^{N-1}G_j \succ 0}} \left\{\prod_{j = 1}^{N-1}|G_j|^{(\alpha_j-d-1)/2}\right\}\left|X-\sum_{j = 1}^{N-1}G_j\right|^{(\alpha_N-d-1)/2} \\
&\qquad \times \etr\left(-\frac{1}{2}\left\{\sum_{j = 1}^{N-1}\Sigma_j^{-1}G_j+\Sigma_N^{-1}\left(X-\sum_{j = 1}^{N-1}G_j\right)\right\}\right) \, \rd G_1\cdots\rd G_{N-1},
\end{aligned}
\]
where
\[
C \leqdef \left(2^{\alpha_{\bullet} d/2}\prod_{j = 1}^N\{|\Sigma_j|^{\alpha_j/2}\Gamma_d(\alpha_j/2)\}\right)^{-1}.
\]
Make the change of variables $G_j = X^{1/2}U_jX^{1/2}$ for $j = 1, \ldots,N-1$. The integration domain becomes
\[
\triangle_{d,N-1} \leqdef \left\{(U_1, \ldots,U_{N-1}): U_j \succ 0, \, j = 1, \ldots,N-1, \, I_d-\sum_{j = 1}^{N-1}U_j \succ 0\right\},
\]
and on calculating the Jacobian, we obtain $\rd G_1\cdots\rd G_{N-1} = |X|^{(N-1)(d+1)/2}\rd U_1\cdots\rd U_{N-1}$. Hence, using the cyclic symmetry of the trace,
\[
\begin{aligned}
f_{\mathfrak{T}}(X)
& = C \, |X|^{(\alpha_{\bullet}-d-1)/2}\etr\left(-\frac{1}{2}\Sigma_N^{-1}X\right) \\
&\hspace{10mm} \times \int_{\triangle_{d,N-1}} \left\{\prod_{j = 1}^{N-1}|U_j|^{(\alpha_j-d-1)/2}\right\}\left|I_d-\sum_{j = 1}^{N-1}U_j\right|^{(\alpha_N-d-1)/2} \\
&\hspace{30mm} \times \etr\left\{-\frac{1}{2}\sum_{j = 1}^{N-1} X^{1/2}(\Sigma_j^{-1}-\Sigma_N^{-1})X^{1/2}U_j\right\} \, \rd U_1\cdots\rd U_{N-1}.
\end{aligned}
\]
The density of $(\mathfrak{U}_1, \ldots, \mathfrak{U}_{N-1})\sim \mathrm{MatrixDirichlet}_{d\times d}(\alpha_1, \ldots, \alpha_N)$ is
\[
f_{(\mathfrak{U}_1, \ldots, \mathfrak{U}_{N-1})}(U_1, \ldots,U_{N-1}) = \frac{\Gamma_d(\alpha_{\bullet}/2)}{\prod_{j = 1}^N\Gamma_d(\alpha_j/2)} \left\{\prod_{j = 1}^{N-1}|U_j|^{(\alpha_j-d-1)/2}\right\}\left|I_d-\sum_{j = 1}^{N-1}U_j\right|^{(\alpha_N-d-1)/2}.
\]
It follows that
\[
f_{\mathfrak{T}}(X) = \left(2^{\alpha_{\bullet} d/2}\left\{\prod_{j = 1}^N|\Sigma_j|^{\alpha_j/2}\right\}\Gamma_d(\alpha_{\bullet}/2)\right)^{-1}|X|^{(\alpha_{\bullet}-d-1)/2}\etr\left(-\frac{1}{2}\Sigma_N^{-1}X\right)\Psi(X).
\]
Since
\[
f_{\alpha_{\bullet}, \Sigma_N}(X) = \left(2^{\alpha_{\bullet} d/2}|\Sigma_N|^{\alpha_{\bullet}/2}\Gamma_d(\alpha_{\bullet}/2)\right)^{-1}|X|^{(\alpha_{\bullet}-d-1)/2}\etr\left(-\frac{1}{2}\Sigma_N^{-1}X\right),
\]
the claimed expression follows.
\end{proof}

\subsubsection{Numerical approximation of the total variation distance}\label{sec:supp.grid-based-tv-approx}

We now describe the numerical approximation of the total variation distance used to compare the feasible choices of $\nu$ in Section~\ref{sec:supp.numerical-comparison-satterthwaite}. The purpose of this section is computational: Proposition~\ref{prop:exact-density-Nterms} reduces the evaluation of $f_{\mathfrak{T}}(X)$ to the evaluation of the scalar factor $\Psi(X)$, and the total variation estimators below repeatedly use this pointwise density evaluation. In the implementation used for the $d = 2$, $N = 3$ experiment reported in Section~\ref{sec:supp.numerical-comparison-satterthwaite}, $\Psi(X)$ is evaluated deterministically by applying Weyl's integration formula to the matrix Dirichlet integral, rather than by simulating from the matrix Dirichlet distribution. This avoids introducing an additional inner Monte Carlo error at every value of $X$.

Let $d = 2$ and set
\[
\beta_j \leqdef \frac{\alpha_j-3}{2}, \qquad j = 1,2,3, \qquad \beta_{23} \leqdef \beta_2+\beta_3+\frac{3}{2}.
\]
For fixed $X\in \mathcal{S}_{++}^2$, define
\[
A_j(X) \leqdef X^{1/2}(\Sigma_j^{-1}-\Sigma_3^{-1})X^{1/2}, \qquad j = 1,2.
\]
In the matrix Dirichlet integral, we first use the change of variables
\[
U_2 = (I_2-U_1)^{1/2}V(I_2-U_1)^{1/2}, \qquad 0 \prec V \prec I_2.
\]
The Jacobian contributes the factor $|I_2-U_1|^{3/2}$, so that
\[
\begin{aligned}
\Psi(X)
& = c_{\alpha} \int_{0 \prec U_1 \prec I_2}\int_{0 \prec V \prec I_2} |U_1|^{\beta_1}|I_2-U_1|^{\beta_{23}}|V|^{\beta_2}|I_2-V|^{\beta_3} \\
&\qquad \times \etr\left\{-\frac{1}{2}A_1(X)U_1-\frac{1}{2}A_2(X)(I_2-U_1)^{1/2}V(I_2-U_1)^{1/2}\right\} \, \rd V\rd U_1,
\end{aligned}
\]
where
\[
c_{\alpha} \leqdef \frac{\Gamma_2(\alpha_{\bullet}/2)}{\Gamma_2(\alpha_1/2)\Gamma_2(\alpha_2/2)\Gamma_2(\alpha_3/2)}.
\]
For a symmetric $2\times 2$ matrix $M$ with $0 \prec M \prec I_2$, we write the ordered eigendecomposition
\[
M = R_{\theta}
\begin{pmatrix}
\lambda_1 & 0 \\
0 & \lambda_2
\end{pmatrix}
R_{\theta}^{\top}, \qquad
0 < \lambda_2 < \lambda_1 < 1, \qquad 0 \leq \theta < \pi,
\]
where
\[
R_{\theta} \leqdef
\begin{pmatrix}
\cos\theta & -\sin\theta \\
\sin\theta & \cos\theta
\end{pmatrix}.
\]
Weyl's formula gives
\[
\rd M = (\lambda_1-\lambda_2) \, \rd\lambda_1\rd\lambda_2\rd\theta.
\]
We apply this formula once to $U_1$ and once to $V$. For $U_1$, write $\lambda_1 = r$ and $\lambda_2 = rq$, with $r,q\in(0,1)$. For $V$, write $\mu_1 = s$ and $\mu_2 = st$, with $s,t\in(0,1)$. Define
\[
U(r,q, \theta) \leqdef R_{\theta}
\begin{pmatrix}
r & 0 \\
0 & rq
\end{pmatrix}
R_{\theta}^{\top}, \qquad
V(s,t, \phi) \leqdef R_{\phi}
\begin{pmatrix}
s & 0 \\
0 & st
\end{pmatrix}
R_{\phi}^{\top},
\]
and
\[
B(r,q, \theta) \leqdef \{I_2-U(r,q, \theta)\}^{1/2}.
\]
Then the integral becomes
\[
\begin{aligned}
\Psi(X)
&= c_{\alpha}\int_{[0,1]^4}\int_{[0, \pi]^2} \etr\left\{-\frac{1}{2}A_1(X)U(r,q, \theta)-\frac{1}{2}A_2(X)B(r,q, \theta)V(s,t, \phi)B(r,q, \theta)\right\} \\
&\quad \times r^{2\beta_1+2}q^{\beta_1}(1-q)(1-r)^{\beta_{23}}(1-rq)^{\beta_{23}}s^{2\beta_2+2}t^{\beta_2}(1-t)(1-s)^{\beta_3}(1-st)^{\beta_3} \, \rd\phi\rd\theta\rd t\rd s\rd q\rd r.
\end{aligned}
\]
This is the formula implemented numerically for the density evaluations in the two-dimensional total variation calculation. The variables $r,q,s,t$ are integrated using Gauss-Jacobi quadrature rules adapted to the endpoint weights, and the angular variables $\theta, \phi$ are integrated using Gauss-Legendre quadrature on $[0, \pi]$. The cross factors $(1-rq)^{\beta_{23}}$ and $(1-st)^{\beta_3}$, together with the exponential factor, are evaluated at the quadrature nodes. The implementation evaluates the sum on the log scale when needed, and checks the accuracy of the quadrature by setting the exponential factor equal to one, in which case the integral should be close to one after multiplication by $c_{\alpha}$.

The resulting approximation of $\Psi(X)$ is inserted into the density formula from Proposition~\ref{prop:exact-density-Nterms} with $N = 3$. Thus, in the numerical calculations, $f_{\mathfrak{T}}(X)$ is evaluated as
\[
\widehat{f}_{\mathfrak{T}}(X) = \frac{|\Sigma_3|^{\alpha_{\bullet}/2}}{\prod_{j = 1}^3 |\Sigma_j|^{\alpha_j/2}} \, f_{\alpha_{\bullet}, \Sigma_3}(X)\widehat{\Psi}(X),
\]
where $\widehat{\Psi}(X)$ denotes the Weyl quadrature approximation defined by the preceding integral.

We compare the law of $\mathfrak{T}$ with that of $\mathfrak{W}_{\nu}\sim\mathcal{W}_2(\nu, \widetilde{\Sigma}_{\nu})$, where $\widetilde{\Sigma}_{\nu} = \nu^{-1}\overline{\Sigma}$ and $\nu$ is chosen by one of the five effective degrees-of-freedom criteria. The total variation distance is
\[
d_{\mathrm{TV}}(\mathfrak{T}, \mathfrak{W}_{\nu}) = \frac{1}{2}\int_{\mathcal{S}_{++}^2}|f_{\mathfrak{T}}(X)-f_{\nu, \widetilde{\Sigma}_{\nu}}(X)| \, \rd X,
\]
which we can approximate using a Bartlett unit-cube estimator, motivated by the identity
\[
d_{\mathrm{TV}}(\mathfrak{T}, \mathfrak{W}_{\nu}) = \frac{1}{2}\EE_{\mathfrak{W}_{\nu}}\left[\left|\frac{f_{\mathfrak{T}}(\mathfrak{W}_{\nu})}{f_{\nu, \widetilde{\Sigma}_{\nu}}(\mathfrak{W}_{\nu})}-1\right|\right] = \frac{1}{2} \EE_{\mathfrak{W}_{\nu}}\Bigl[\bigl|\exp\{\ell_{\nu}(\mathfrak{W}_{\nu})\}-1\bigr|\Bigr],
\]
where
\[
\ell_{\nu}(X) = \log f_{\mathfrak{T}}(X)-\log f_{\nu, \widetilde{\Sigma}_{\nu}}(X)
\]
is the log-likelihood ratio. The expectation with respect to $\mathfrak{W}_{\nu}$ is then mapped to an integral over the unit cube by means of the Bartlett representation of Wishart random matrices \citep[see e.g.][Section~3.2.4]{Muirhead1982}. Let $\smash{L_{\nu}L_{\nu}^{\top} = \widetilde{\Sigma}_{\nu}}$ and $u = (u_1,u_2,u_3)\in(0,1)^3$. Define
\[
C_{\nu}(u) \leqdef
\begin{pmatrix}
\sqrt{F_{\chi^2_{\nu}}^{-1}(u_1)} & 0 \\
\Phi^{-1}(u_3) & \sqrt{F_{\chi^2_{\nu-1}}^{-1}(u_2)}
\end{pmatrix},
\]
where $\smash{F_{\chi^2_{\nu}}^{-1}}$ is the quantile function of the chi-square distribution with $\nu$ degrees of freedom and $\Phi^{-1}$ is the standard normal quantile function. We then set
\[
X_{\nu}(u) \leqdef L_{\nu}C_{\nu}(u)C_{\nu}(u)^{\top}L_{\nu}^{\top}.
\]
If $\bb{\upsilon} = (\upsilon_1, \upsilon_2, \upsilon_3)$ is a random vector uniformly distributed on $(0,1)^3$, then
\[
F_{\chi^2_{\nu}}^{-1}(\upsilon_1) \sim \chi^2_{\nu}, \qquad F_{\chi^2_{\nu-1}}^{-1}(\upsilon_2) \sim \chi^2_{\nu-1}, \qquad \Phi^{-1}(\upsilon_3)\sim \mathcal{N}(0,1).
\]
Hence $C_{\nu}(\bb{\upsilon})$ is precisely the lower-triangular Bartlett factor for a two-dimensional Wishart random matrix with $\nu$ degrees of freedom. Consequently,
\[
X_{\nu}(\bb{\upsilon}) = L_{\nu} C_{\nu}(\bb{\upsilon})C_{\nu}(\bb{\upsilon})^{\top}L_{\nu}^{\top} \sim \mathcal{W}_2(\nu, \widetilde{\Sigma}_{\nu}).
\]
Therefore, for any integrable $h$ on $\mathcal{S}^2_{++}$,
\begin{equation}\label{eq:bartlett-integration-formula}
\EE_{\mathfrak{W}_{\nu}}[h(\mathfrak{W}_{\nu})]
 = \int_{\mathcal{S}^2_{++}} h(X) \, f_{\nu, \widetilde{\Sigma}_{\nu}}(X) \, \rd X =
\int_{(0,1)^3} h(X_{\nu}(\bb{\upsilon})) \, \rd \bb{\upsilon}.
\end{equation}
Indeed, if $Y = L_{\nu}^{-1}XL_{\nu}^{-\top}$, then the inverse Bartlett coordinates are obtained from the Cholesky factor of $Y$:
\[
u_1 = F_{\chi^2_{\nu}}(Y_{11}), \qquad
u_2 = F_{\chi^2_{\nu-1}}\left(Y_{22}-\frac{Y_{12}^2}{Y_{11}}\right), \qquad
u_3 = \Phi\left(\frac{Y_{12}}{\sqrt{Y_{11}}}\right),
\]
and~\eqref{eq:bartlett-integration-formula} follows from observing that the Jacobian of this transformation is precisely the Wishart density. Thus, we can rewrite the total variation distance as
\[
d_{\mathrm{TV}}(\mathfrak{T}, \mathfrak{W}_{\nu}) = \frac{1}{2} \int_{(0,1)^3} \bigl|\exp\{\ell_{\nu}(X_{\nu}(\bb{\upsilon}))\}-1\bigr| \, \rd \bb{\upsilon}.
\]

In practice, we will use the Weyl quadrature approximation of the log-likelihood ratio,
\[
\begin{aligned}
\widehat{\ell}_{\nu}(X)
&\leqdef \log \widehat{f}_{\mathfrak{T}}(X)-\log f_{\nu, \widetilde{\Sigma}_{\nu}}(X) \\
& = \log\widehat{\Psi}(X)+(\nu-\alpha_{\bullet})\log 2+\log\Gamma_2(\nu/2)-\log\Gamma_2(\alpha_{\bullet}/2)+\frac{\nu}{2}\log|\widetilde{\Sigma}_{\nu}| \\
&\qquad -\frac{1}{2}\sum_{j = 1}^3\alpha_j\log|\Sigma_j|+\frac{\alpha_{\bullet}-\nu}{2}\log|X|-\frac{1}{2}\tr\{(\Sigma_3^{-1}-\widetilde{\Sigma}_{\nu}^{-1})X\}.
\end{aligned}
\]
Thus, for Sobol points $u_1, \ldots,u_M\in(0,1)^3$, the Bartlett unit-cube approximation to the total variation distance is
\begin{equation}\label{eq:tv-bartlett-approx}
\widehat{\mathrm{TV}}(\nu) \leqdef \frac{1}{2M}\sum_{m = 1}^M \left|\exp\{\widehat{\ell}_{\nu}(X_{\nu}(u_m))\} - 1\right|.
\end{equation}
In the actual numerical implementation, the unit-cube coordinates are clipped by a very small numerical tolerance before applying the quantile functions. The computations over the points $u_m$ are independent and are therefore parallelized by splitting the point set into batches.

\subsubsection{Numerical comparisons}\label{sec:supp.numerical-comparison-satterthwaite}

The purpose of the numerical study is to determine whether the geometric part of the Stein bound gives a reliable ordering of the scalar Satterthwaite reductions from the point of view of total variation distance. We therefore do not use the numerical experiment only to compare the values of the different effective degrees of freedom: rather, for each configuration of the summands, we compare two orderings on the set
\[
\mathcal{M} = \{\mathrm{TG}, \mathrm{K}, \mathrm{LS}, \mathrm{US}\}.
\]
The first ordering is induced by the estimated total variation distance, while the second ordering is induced by the geometric part of the bound. The upper-bound choice $\nu_{\mathrm{UB}}$ is not included in this ranking comparison: although it is useful as a diagnostic reference, it is obtained by directly minimizing the geometric term and is not a scalar reduction of the covariance-matching problem in the same sense as $\mathrm{TG}$, $\mathrm{K}$, $\mathrm{LS}$ and $\mathrm{US}$. This distinction is also consistent with the one-dimensional situation, where the minimizer of the geometric upper-bound criterion may have a worse total variation performance than the classical Satterthwaite degree of freedom.

\paragraph{Comparisons setup}
Throughout the experiment, we consider independent random matrices $\mathfrak{G}_j\sim \mathcal{W}_2(\alpha_j, \Sigma_j)$, $j = 1,2,3$, and their sum $\mathfrak{T} = \mathfrak{G}_1+\mathfrak{G}_2+\mathfrak{G}_3$. For each $k\in\mathcal{M}$, we write $\nu_k$ for the corresponding scalar reduction and
\[
\mathfrak{W}_k \sim \mathcal{W}_2(\nu_k, \overline{\Sigma}/\nu_k), \qquad \overline{\Sigma} = \sum_{j = 1}^3\alpha_j\Sigma_j.
\]
The geometric quantity associated with $k$ is denoted by
\[
B(\nu_k) = \sum_{j = 1}^3\alpha_j\tr(\Sigma_j) \left\|\Sigma_j-\nu_k^{-1}\overline{\Sigma}\right\|_F,
\]
whereas the empirical total variation estimate is denoted by $\widehat{\mathrm{TV}}(\nu_k)$, as defined in~\eqref{eq:tv-bartlett-approx}. All reported computations use the Bartlett-QMC estimator with $2500$ Sobol points, and $8$ quadrature nodes per coordinate for the Weyl quadrature. Throughout, we will denote
\[
k_{\mathrm{TV}}\in\arg\min_{k\in\mathcal{M}} \widehat{\mathrm{TV}}(\nu_k), \qquad
k_B\in \arg\min_{k\in\mathcal{M}} B(\nu_k).
\]

\paragraph{Description of the scenarios}
The configurations are organized into three ranking regimes. The first is the diagonal common-eigenspace regime. Here $\Sigma_i = \Lambda_i$, where
\[
\Lambda_i =
\begin{pmatrix}
\lambda_{i,1} & 0\\
0 & \lambda_{i,2}
\end{pmatrix}, \qquad
\lambda_{i,1}, \lambda_{i,2}>0.
\]
The ratios $\lambda_{i,1}/\lambda_{i,2}$ are not all equal, so the scale matrices are not merely proportional. This regime is the cleanest non-trivial test case, since the matrices commute and the discrepancy is driven by eigenvalue heterogeneity rather than by competing eigenspaces. The eighteen base configurations are listed in Table~\ref{tab:satt-diagonal-scenarios}. This diagonal framework will be presented alongside the second regime: that is, the common-rotation regime where $\Sigma_i = R_{\theta}\Lambda_iR_{\theta}^{\top}$ with the same rotation matrix $R_{\theta}$ for all $i$. This is still a common-eigenspace regime, but it is no longer diagonal in the canonical coordinates unless $R_{\theta} = I_2$. Each of the eighteen base configurations is combined with the following five common rotation settings:
\[
\theta \in \left\{0^{\circ}, 15^{\circ}, 30^{\circ}, 45^{\circ}, 60^{\circ}\right\}.
\]
The third regime is the heterogeneous-rotation regime, where $\Sigma_i = R_{\theta_i}\Lambda_iR_{\theta_i}^{\top}$ and the rotations $R_{\theta_i}$ are not all equal. This is the genuine non-commuting stress test. The rotation profiles used here are angle vectors $(\theta_1, \theta_2, \theta_3)$. Each of the eighteen base configurations is combined with the following five heterogeneous rotation profiles:
\[
(0^{\circ},10^{\circ},-10^{\circ}), \qquad (0^{\circ},20^{\circ},-20^{\circ}), \qquad (-30^{\circ},0^{\circ},30^{\circ}), \qquad (0^{\circ},35^{\circ},70^{\circ}), \qquad (-45^{\circ},15^{\circ},55^{\circ}).
\]

\begin{table}[H]
\centering
\renewcommand{\arraystretch}{1.0} 
\small
\begin{tabular}{lc c c c c}
\toprule
Profile & Scenario & $(\alpha_1, \alpha_2, \alpha_3)$ & $\Lambda_1$ & $\Lambda_2$ & $\Lambda_3$ \\
\midrule
mild crossing & 01 & $(22,22,22)$ & $(1.300,0.900)$ & $(1.050,1.050)$ & $(0.850,1.300)$ \\
&02 & $(14,22,30)$ & $(1.300,0.900)$ & $(1.050,1.050)$ & $(0.850,1.300)$ \\
&03 & $(30,22,14)$ & $(1.300,0.900)$ & $(1.050,1.050)$ & $(0.850,1.300)$ \\[0.3cm]
same-axis anisotropy & 04 & $(22,22,22)$ & $(1.500,0.850)$ & $(1.300,0.950)$ & $(1.100,1.050)$ \\
&05 & $(14,22,30)$ & $(1.500,0.850)$ & $(1.300,0.950)$ & $(1.100,1.050)$ \\
&06 & $(30,22,14)$ & $(1.500,0.850)$ & $(1.300,0.950)$ & $(1.100,1.050)$ \\[0.3cm]
opposite-axis anisotropy & 07 & $(22,22,22)$ & $(1.600,0.750)$ & $(1.050,1.050)$ & $(0.750,1.600)$ \\
&08 & $(14,22,30)$ & $(1.600,0.750)$ & $(1.050,1.050)$ & $(0.750,1.600)$ \\
&09 & $(30,22,14)$ & $(1.600,0.750)$ & $(1.050,1.050)$ & $(0.750,1.600)$ \\[0.3cm]
one eccentric summand & 10 & $(22,22,22)$ & $(1.750,0.700)$ & $(0.950,1.200)$ & $(0.950,1.200)$ \\
&11 & $(14,22,30)$ & $(1.750,0.700)$ & $(0.950,1.200)$ & $(0.950,1.200)$ \\
&12 & $(30,22,14)$ & $(1.750,0.700)$ & $(0.950,1.200)$ & $(0.950,1.200)$ \\[0.3cm]
approx. equal determinants & 13 & $(22,22,22)$ & $(1.600,0.625)$ & $(1.250,0.800)$ & $(0.900,1.111)$ \\
&14 & $(14,22,30)$ & $(1.600,0.625)$ & $(1.250,0.800)$ & $(0.900,1.111)$ \\
&15 & $(30,22,14)$ & $(1.600,0.625)$ & $(1.250,0.800)$ & $(0.900,1.111)$ \\[0.3cm]
mostly scalar scale contrast & 16 & $(22,22,22)$ & $(1.450,1.250)$ & $(0.950,1.050)$ & $(0.750,0.950)$ \\
&17 & $(14,22,30)$ & $(1.450,1.250)$ & $(0.950,1.050)$ & $(0.750,0.950)$ \\
&18 & $(30,22,14)$ & $(1.450,1.250)$ & $(0.950,1.050)$ & $(0.750,0.950)$ \\
\bottomrule
\end{tabular}
\caption{The eighteen base scenarios used in the diagonal common-eigenspace regime. The pair in column $\Lambda_i$ denotes $(\lambda_{i,1}, \lambda_{i,2})$.}
\label{tab:satt-diagonal-scenarios}
\normalsize
\end{table}

\paragraph{Agreement and rank correlation}
We first report, for each regime, the agreement proportion: that is, the proportion of scenarios where $k_{\mathrm{TV}} = k_{B}$. In the diagonal regime, 94\% of the scenarios agree on the best choice, and the criteria are practically indistinguishable when agreement does not hold. In the common-rotation regime, this proportion decreases to 58\%; in the heterogeneous-rotation regime, it is 70\%. We also report Kendall's rank correlation $\tau$ for each scenario, and display the numerical values in Figure~\ref{fig:kendall-tau}.

\begin{figure}[!ht]
\centering
\includegraphics[width = 0.49\linewidth, trim={0.5cm 1cm 0.5cm 0.5cm}, clip]{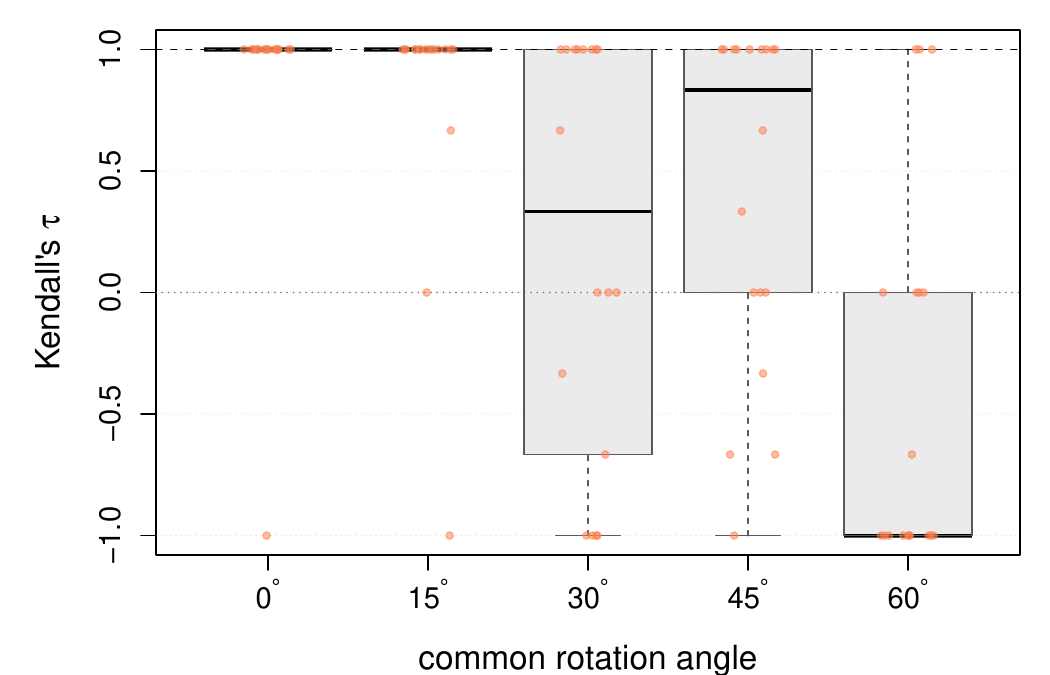}
\includegraphics[width = 0.49\linewidth, trim={0.5cm 1cm 0.5cm 0.5cm}, clip]{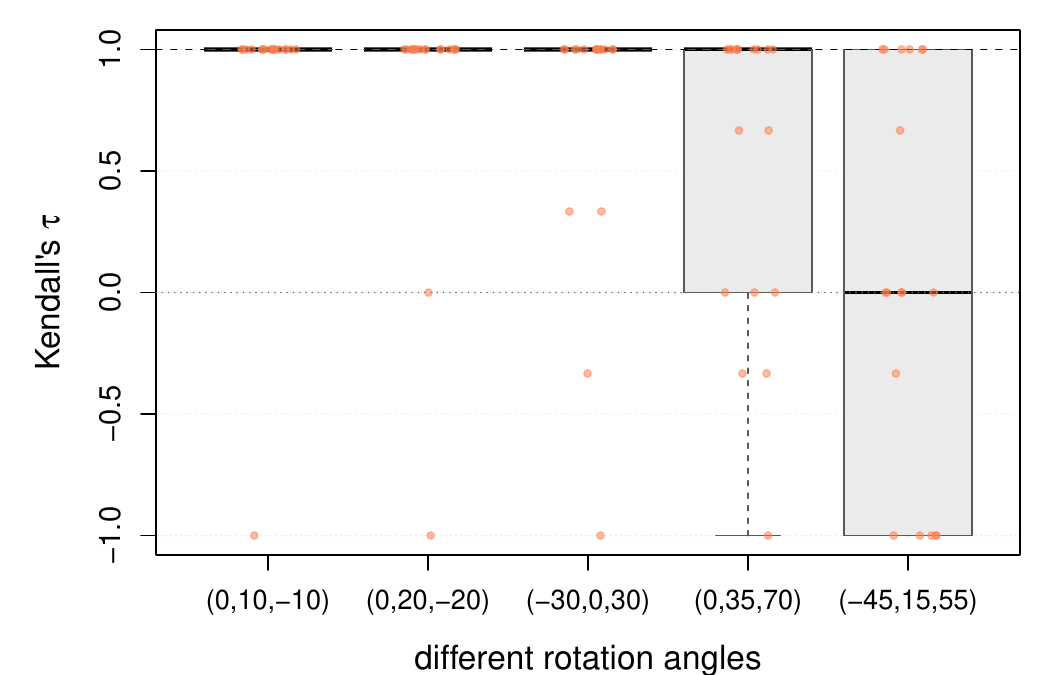}
\caption{Dispersion of Kendall's rank correlation, in the considered regimes.}
\label{fig:kendall-tau}
\end{figure}

We observe that the ordering of a majority of scenarios, namely 66\%, is perfectly described by the geometric criterion: in these configurations, the rank correlation is exactly 1. These configurations cover every scenario where $k_{\mathrm{TV}} = k_{B}$. On the other hand, a substantial number of exact minimizers change after a common rotation, as indicated by the left panel of Figure~\ref{fig:kendall-tau} covering the second regime. This behavior is consistent with the coordinate sensitivity of some scalar reductions, especially those involving entrywise or unscaled half-vectorized summaries, that is, the LS and US criteria. Finally, it appears that the third regime is less sensitive to this phenomenon.

\paragraph{Regret analysis}
We also report the relative total variation loss of choosing the geometric choice $k_B$ over $k_{\mathrm{TV}}$,
\[
\mathrm{Loss}(k_B)
\leqdef
\frac{\widehat{\mathrm{TV}}(\nu_{k_B})}{\widehat{\mathrm{TV}}(\nu_{k_{\mathrm{TV}}})} - 1.
\]
This quantity is equal to 0 when the geometric criterion agrees with the minimizer of the total variation. Values close to 0 indicate that, even if the exact minimizer differs, the discrepancy is practically negligible in total variation. The numerical values are displayed in Figure~\ref{fig:rel-TV}; most losses are below 5\%, although a few outlying configurations have substantially larger losses.

\begin{figure}[!ht]
\centering
\includegraphics[width = 0.49\linewidth, trim={0.5cm 1cm 0.5cm 0.5cm}, clip]{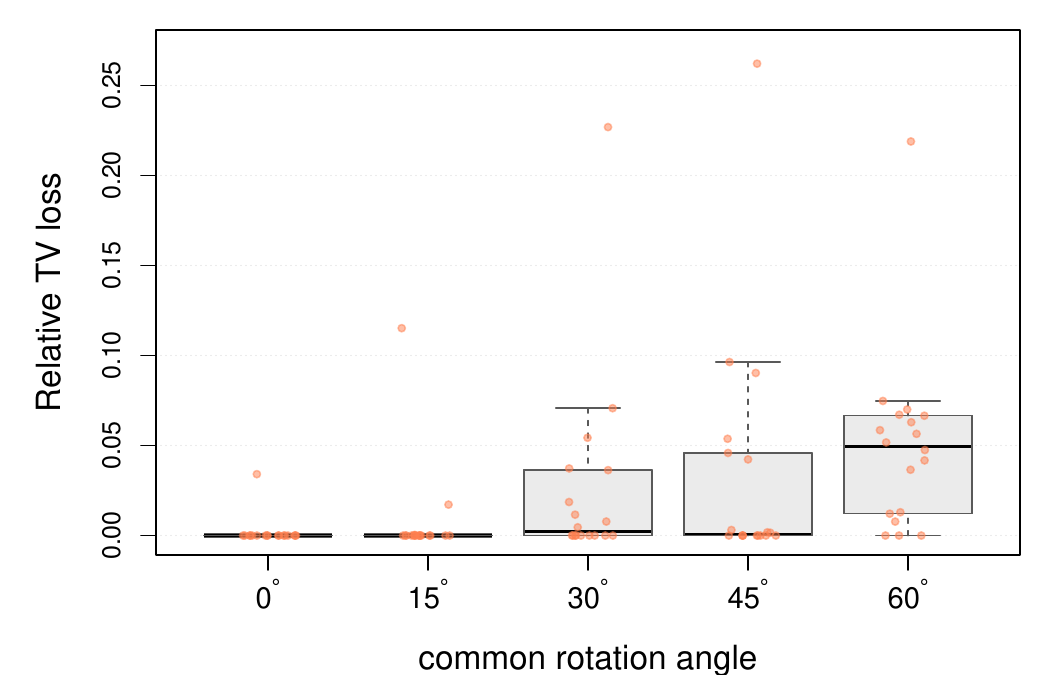}
\includegraphics[width = 0.49\linewidth, trim={0.5cm 1cm 0.5cm 0.5cm}, clip]{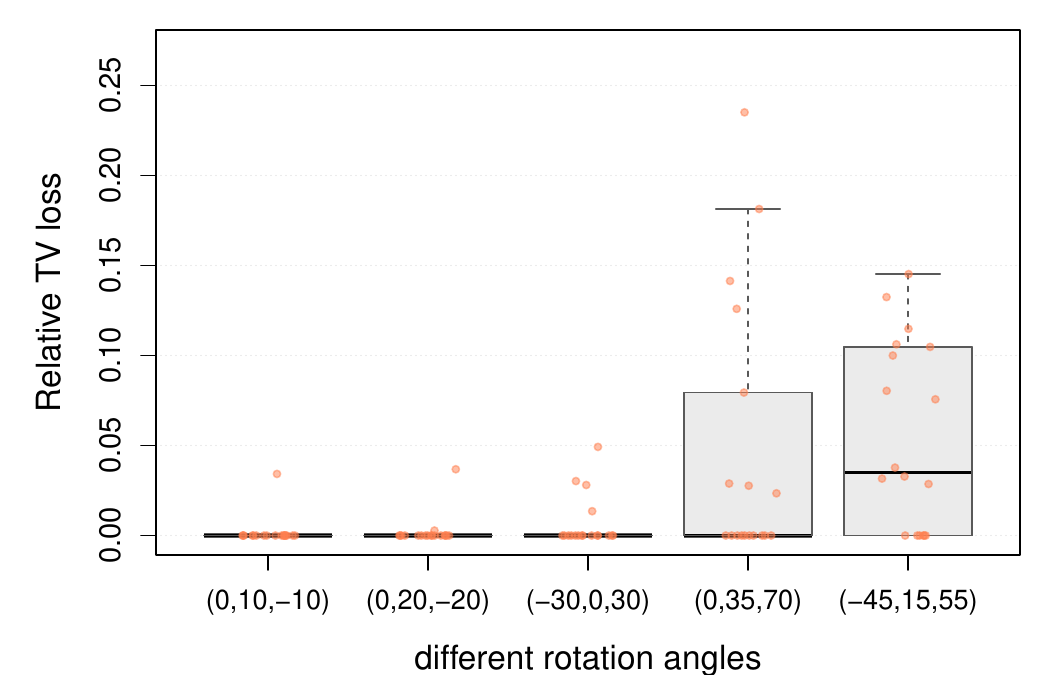}
\caption{Regret of choosing the minimizer of the geometric criterion, in the considered regimes.}
\label{fig:rel-TV}
\end{figure}

\smallskip
\subsection{Additional details for Stein's method-of-moments application}

\subsubsection{Numerical comparison of the log-MOM, quad-MOM, and MLE estimators}\label{sec:supp.log.MOM.vs.quad.MOM.vs.MLE}

This section details the numerical experiment mentioned in Remark~\ref{rem:numerical.illustration} and reported in Tables~\ref{tab:wishart-smom-mle-sigma-fro} and~\ref{tab:wishart-smom-mle-winner-counts}. In this experiment, the values of $\alpha_0$ and $\Sigma_0$ were selected randomly before the Monte Carlo replications. More precisely, ten values of $\alpha_0$ were generated independently from the uniform distribution on $[3,30]$. Ten positive definite matrices $\Sigma_0$ were generated independently as follows. With $d = 3$, eigenvalues were drawn according to
\[
\lambda_1, \lambda_2, \lambda_3 \stackrel{\mathrm{iid}}{\sim} \operatorname{LogUnif}(0.35,3.50),
\]
and an orthogonal matrix $Q$ was generated by applying the QR construction to a matrix with independent standard normal entries. We then set
\[
\Sigma_0 = Q \, \mathrm{diag}(\lambda_1, \ldots, \lambda_d) Q^{\top}.
\]
The generated values of $\alpha_0$ were crossed with the generated values of $\Sigma_0$, giving one hundred configurations in total. For each configuration, we simulated $500$ independent samples from $\mathcal{W}_3(\alpha_0, \Sigma_0)$, with sample sizes $n\in \{10,10^2,10^3,10^4\}$. Within each replication, the samples were nested across $n$. The corresponding values of $\alpha_0$ and $\Sigma_0$ are displayed in the first column of Table~\ref{tab:wishart-smom-mle-sigma-fro}, rounded to one decimal place for $\alpha_0$ and two decimal places for the entries of $\Sigma_0$.

For each replication, sample size, and parameter configuration, we computed three pairs of estimators: $\smash{(\widehat{\alpha}_{\mathrm{log}}, \widehat{\Sigma}_{\mathrm{log}})}$, $\smash{(\widehat{\alpha}_{\mathrm{quad}}, \widehat{\Sigma}_{\mathrm{quad}})}$, and $\smash{(\widehat{\alpha}_{\mathrm{MLE}}, \widehat{\Sigma}_{\mathrm{MLE}})}$. The first two are Stein's method-of-moments estimators in Proposition~\ref{prop:Wishart.unknown.alpha.estimator}, while the last one is the maximum likelihood estimator (MLE). More precisely, writing
\[
\widehat{M} = \frac{1}{n} \sum_{k=1}^n \mathfrak{W}^{(k)},
\]
the profile likelihood maximizer in the scale parameter, for fixed $\alpha$, is $\widehat{\Sigma}(\alpha) = \widehat{M}/\alpha$. The MLE of the shape parameter is therefore obtained by solving the scalar equation
\[
\overline{\log|\mathfrak{W}|} - \log|\widehat{M}| + d\log(\alpha) - d\log(2) - \psi_d(\alpha/2) = 0, \qquad \alpha > d-1,
\]
where
\[
\psi_d(a) \leqdef \sum_{j=1}^d \psi\left(a + \frac{1-j}{2}\right).
\]
The scale matrix is then set to $\widehat{\Sigma}_{\mathrm{MLE}} = \widehat{M}/\widehat{\alpha}_{\mathrm{MLE}}$.

The denominators in \eqref{eq:Wishart.unknown.alpha.estimator.quad} and \eqref{eq:Wishart.unknown.alpha.estimator.log} were checked against numerical zero. The log-MOM and quad-MOM estimates were marked invalid if the relevant denominator was numerically zero, if the estimated shape parameter did not satisfy $\widehat{\alpha}>d-1$, or if the scale estimate failed the numerical positive-definiteness check. The MLE was marked invalid if the profile likelihood equation could not be solved reliably or if the resulting estimate failed the same parameter-space checks. Table~\ref{tab:wishart-smom-mle-sigma-fro} reports, for each configuration and sample size, the median and interquartile range over the valid Monte Carlo replications of the relative Frobenius error $\|\widehat{\Sigma}-\Sigma_0\|_F/\|\Sigma_0\|_F$. Within each configuration and sample size, the smallest median error and the smallest interquartile range among the three estimators are shown in bold. Table~\ref{tab:wishart-smom-mle-winner-counts} gives the corresponding counts, by sample size, of how often each estimator attains the smallest median error or the smallest interquartile range. The MLE is included as an efficiency benchmark. The logarithmic and quadratic Stein's method-of-moments estimators are simpler to compute, since they are given explicitly by \eqref{eq:Wishart.unknown.alpha.estimator.quad} and \eqref{eq:Wishart.unknown.alpha.estimator.log}, and do not require solving the likelihood equation. This experiment is intended to check whether the improvement of the logarithmic Stein's method-of-moments estimator over the quadratic, classical moment estimator persists across randomly generated parameter configurations.

\bigskip
\begingroup
\footnotesize
\input{wishart_smom_mle_sigma_fro_table.tex}
\endgroup

\subsubsection{Numerical comparison of the projected and naive Stein MOM estimators}\label{sec:supp.projected.vs.naive}

This section details the numerical experiment mentioned in Remark~\ref{rem:numerical.illustration:prop:projected.Stein.Wishart}. The results are reported in Tables~\ref{tab:wishart-structured-smom-sigma-fro}, \ref{tab:wishart-structured-smom-winner-counts}, and~\ref{tab:wishart-structured-smom-ratios}. In this experiment, the values of $\alpha_0$ and $\Sigma_0$ were selected randomly before the Monte Carlo replications, with $\Sigma_0$ constrained to belong to the two-dimensional compound-symmetry subspace
\[
\Sigma_0 = \Sigma(\bb{\beta}^{\star}) = \beta_1^{\star}B_1 + \beta_2^{\star}B_2,
\]
where $B_1 = I_{10}$ and $B_2 = \bb{1}_{10}\bb{1}_{10}^{\top}-I_{10}$. Ten values of $\alpha_0$ were generated independently from the uniform distribution on $[12,80]$. Ten structured matrices $\Sigma_0$ were generated independently by drawing the compound-symmetry eigenvalues
\[
\lambda_{\perp}\sim \operatorname{LogUnif}(0.50,3.00), \qquad \lambda_{\parallel}\sim \operatorname{LogUnif}(0.50,15.00),
\]
and then setting, with $d = 10$,
\[
\beta_2^{\star} = \frac{\lambda_{\parallel}-\lambda_{\perp}}{d}, \qquad \beta_1^{\star} = \lambda_{\perp} + \beta_2^{\star}.
\]
The generated values of $\alpha_0$ were crossed with the generated values of $\bb{\beta}^{\star}$, giving one hundred structured configurations in total. For each structured configuration, we simulated $500$ independent samples from $\mathcal{W}_{10}(\alpha_0, \Sigma_0)$, with sample sizes $n\in \{10,10^2,10^3,10^4\}$. Within each replication, the samples were nested across $n$. The corresponding values of $\alpha_0$ and $\bb{\beta}^{\star}$ are displayed in the first columns of Tables~\ref{tab:wishart-structured-smom-sigma-fro} and~\ref{tab:wishart-structured-smom-ratios}, rounded to one decimal place for $\alpha_0$ and two decimal places for the coordinates of $\bb{\beta}^{\star}$.

The probe matrices $\{U_m\}_{m = 1}^{M}$ were chosen as an orthonormal basis of $\mathcal{S}^{10}$ for the Frobenius inner product, and the templates were the matrices $B_1$ and $B_2$ above. For each replication, sample size, and structured configuration, we computed the projected Stein's method-of-moments estimator $\widehat{\Sigma}_n$ from Proposition~\ref{prop:projected.Stein.Wishart} and the naive (unstructured) Stein's method-of-moments estimator
\[
\widetilde{\Sigma}_n \leqdef \frac{1}{\alpha_0 n} \sum_{k=1}^n \mathfrak{W}^{(k)}.
\]
Writing
\[
\widehat{M} = \frac{1}{n} \sum_{k=1}^n \mathfrak{W}^{(k)},
\]
the unstructured MLE is $\widehat{M}/\alpha_0$, because $\alpha_0$ is known. It is therefore identical to $\widetilde{\Sigma}_n$ and is not included as a separate competitor. The reported projected and naive scale estimates were marked invalid if the corresponding scale estimate failed the numerical positive-definiteness check. No shape parameter is estimated in this experiment, since $\alpha_0$ is known.

Table~\ref{tab:wishart-structured-smom-sigma-fro} reports, for each structured configuration and sample size, the median and interquartile range over the valid Monte Carlo replications of the relative Frobenius error $\|\widehat{\Sigma}-\Sigma_0\|_F/\|\Sigma_0\|_F$. Within each structured configuration and sample size, the smaller median error and the smaller interquartile range among the two estimators are shown in bold. Table~\ref{tab:wishart-structured-smom-winner-counts} gives the corresponding counts, by sample size, of how often each estimator attains the smaller median error or the smaller interquartile range. Table~\ref{tab:wishart-structured-smom-ratios} reports the median and interquartile range of the ratio $\|\widehat{\Sigma}_{n}-\Sigma_0\|_F/\|\widetilde{\Sigma}_{n}-\Sigma_0\|_F$, corresponding to the entries labeled $\mathrm{proj}/\mathrm{naive}$. Values below one favor the projected Stein's method-of-moments estimator. This experiment is intended to check whether the gain from projecting the Stein moment equations onto the structured subspace persists across randomly generated parameter configurations, while comparing the projected estimator with the usual unstructured moment benchmark.

\bigskip
\begingroup
\footnotesize
\input{wishart_structured_smom_sigma_fro_table.tex}
\endgroup

\begingroup
\footnotesize
\input{wishart_structured_smom_ratio_table.tex}
\endgroup

\end{document}

%% file: wishart_smom_mle_winner_count_table.tex
\renewcommand{\arraystretch}{1.0}
\setlength{\tabcolsep}{2.5pt}
\begin{longtable}{@{}lrrrrrrrr@{}}
\caption{Counts, over the 100 parameter configurations, of how often each estimator attains the smallest median and the smallest interquartile range in Table~\ref{tab:wishart-smom-mle-sigma-fro}, respectively. }\label{tab:wishart-smom-mle-winner-counts}\\
\hline
Estimator & \multicolumn{2}{c}{$n=10$} & \multicolumn{2}{c}{$n=100$} & \multicolumn{2}{c}{$n=1000$} & \multicolumn{2}{c}{$n=10000$} \\
 & Median & IQR & Median & IQR & Median & IQR & Median & IQR \\
\hline
\endfirsthead
\hline
Estimator & \multicolumn{2}{c}{$n=10$} & \multicolumn{2}{c}{$n=100$} & \multicolumn{2}{c}{$n=1000$} & \multicolumn{2}{c}{$n=10000$} \\
 & Median & IQR & Median & IQR & Median & IQR & Median & IQR \\
\hline
\endhead
\hline
\multicolumn{9}{r}{Continued on next page}\\
\hline
\endfoot
\hline
\endlastfoot
log-MOM & 15 & 10 & 14 & 19 & 13 & 18 & 16 & 11 \\
quad-MOM & 0 & 0 & 1 & 1 & 0 & 0 & 0 & 1 \\
MLE & 85 & 90 & 85 & 80 & 87 & 82 & 84 & 88 \\
\end{longtable}

%% file: wishart_structured_smom_winner_count_table.tex
\renewcommand{\arraystretch}{1.0}
\setlength{\tabcolsep}{2.5pt}
\begin{longtable}{@{}lrrrrrrrr@{}}
\caption{Winner counts for the structured Wishart simulation. For each sample size, the table counts among the 100 structured configurations how often each estimator attains the smallest median error and the smallest interquartile range.}\label{tab:wishart-structured-smom-winner-counts}\\
\hline
Estimator & \multicolumn{2}{c}{$n=10$} & \multicolumn{2}{c}{$n=100$} & \multicolumn{2}{c}{$n=1000$} & \multicolumn{2}{c}{$n=10000$} \\
 & Median & IQR & Median & IQR & Median & IQR & Median & IQR \\
\hline
\endfirsthead
\hline
Estimator & \multicolumn{2}{c}{$n=10$} & \multicolumn{2}{c}{$n=100$} & \multicolumn{2}{c}{$n=1000$} & \multicolumn{2}{c}{$n=10000$} \\
 & Median & IQR & Median & IQR & Median & IQR & Median & IQR \\
\hline
\endhead
\hline
\multicolumn{9}{r}{Continued on next page}\\
\hline
\endfoot
\hline
\endlastfoot
proj-SMOM & 100 & 9 & 100 & 9 & 100 & 10 & 100 & 9 \\
naive-SMOM & 0 & 91 & 0 & 91 & 0 & 90 & 0 & 91 \\
\hline
\end{longtable}

%% file: wishart_smom_mle_sigma_fro_table.tex
\renewcommand{\arraystretch}{0.95}
\setlength{\tabcolsep}{3pt}
\begin{longtable}{@{}p{0.21\textwidth}lrrrrrrrr@{}}
\caption{Monte Carlo comparison of the logarithmic Stein method-of-moments estimator, the quadratic moment estimator, and the maximum likelihood estimator. Entries are the median and interquartile range, over $B=500$ replications, of $\smash{\|\widehat{\Sigma}-\Sigma_0\|_F/\|\Sigma_0\|_F}$. The first column gives the value of $\alpha_0$ and the explicit matrix $\Sigma_0$. Within each $(\alpha_0,\Sigma_0)$ configuration and sample size, the smallest median and the smallest interquartile range are shown in bold.}\label{tab:wishart-smom-mle-sigma-fro}\\
\hline
\multicolumn{1}{c}{$\alpha_0, \Sigma_0$} & Estimator & \multicolumn{2}{c}{$n=10$} & \multicolumn{2}{c}{$n=100$} & \multicolumn{2}{c}{$n=1000$} & \multicolumn{2}{c}{$n=10000$} \\
 &  & Median & IQR & Median & IQR & Median & IQR & Median & IQR \\
\hline
\endfirsthead
\hline
\multicolumn{1}{c}{$\alpha_0, \Sigma_0$} & Estimator & \multicolumn{2}{c}{$n=10$} & \multicolumn{2}{c}{$n=100$} & \multicolumn{2}{c}{$n=1000$} & \multicolumn{2}{c}{$n=10000$} \\
 &  & Median & IQR & Median & IQR & Median & IQR & Median & IQR \\
\hline
\endhead
\hline
\multicolumn{10}{r}{Continued on next page}\\
\hline
\endfoot
\hline
\endlastfoot
\multirow{3}{0.21\textwidth}{\centering\scriptsize\begin{tabular}{@{}c@{ }c@{}}$10.8,$ & {\setlength{\arraycolsep}{2pt}$\left(\begin{matrix}1.31 & -0.06 & 0.11 \\ -0.06 & 1.13 & 0.40 \\ 0.11 & 0.40 & 0.97\end{matrix}\right)$}\end{tabular}} & log-MOM & \textbf{0.225} & 0.134 & 0.0697 & 0.0402 & 0.0215 & 0.0113 & 0.00681 & 0.00437 \\*
 & quad-MOM & 0.248 & 0.15 & 0.0757 & 0.0482 & 0.024 & 0.0136 & 0.00743 & 0.00515 \\*
 & MLE & 0.227 & \textbf{0.126} & \textbf{0.0675} & \textbf{0.0365} & \textbf{0.021} & \textbf{0.0112} & \textbf{0.00673} & \textbf{0.00407} \\
\hline
\multirow{3}{0.21\textwidth}{\centering\scriptsize\begin{tabular}{@{}c@{ }c@{}}$10.8,$ & {\setlength{\arraycolsep}{2pt}$\left(\begin{matrix}1.40 & -0.08 & 0.09 \\ -0.08 & 1.39 & 0.00 \\ 0.09 & 0.00 & 0.80\end{matrix}\right)$}\end{tabular}} & log-MOM & 0.222 & 0.117 & 0.0709 & 0.0373 & 0.0216 & \textbf{0.0118} & 0.00713 & 0.00415 \\*
 & quad-MOM & 0.238 & 0.133 & 0.0783 & 0.0456 & 0.0234 & 0.013 & 0.00746 & 0.00455 \\*
 & MLE & \textbf{0.217} & \textbf{0.109} & \textbf{0.07} & \textbf{0.0366} & \textbf{0.0212} & 0.0119 & \textbf{0.00704} & \textbf{0.00384} \\
\hline
\multirow{3}{0.21\textwidth}{\centering\scriptsize\begin{tabular}{@{}c@{ }c@{}}$10.8,$ & {\setlength{\arraycolsep}{2pt}$\left(\begin{matrix}0.41 & 0.07 & 0.03 \\ 0.07 & 0.61 & 0.11 \\ 0.03 & 0.11 & 0.43\end{matrix}\right)$}\end{tabular}} & log-MOM & 0.214 & 0.124 & 0.0713 & 0.0395 & 0.0215 & \textbf{0.0125} & 0.00694 & 0.0037 \\*
 & quad-MOM & 0.24 & 0.154 & 0.0778 & 0.0472 & 0.0228 & 0.0147 & 0.00772 & 0.00462 \\
\end{longtable}

%% file: wishart_structured_smom_sigma_fro_table.tex
\renewcommand{\arraystretch}{1}
\setlength{\tabcolsep}{2.5pt}
\begin{longtable}{@{}p{0.14\textwidth}lrrrrrrrr@{}}
\caption{Monte Carlo comparison of the projected Stein method-of-moments estimator and the (unstructured) Stein method-of-moments estimator. Entries are the median and interquartile range, over $B=500$ replications, of $\smash{\|\widehat{\Sigma}-\Sigma_0\|_F/\|\Sigma_0\|_F}$. Within each structured configuration and sample size, the smallest median and the smallest interquartile range are shown in bold.}\label{tab:wishart-structured-smom-sigma-fro}\\
\hline
$\alpha_0,\bb{\beta}^{\star}$ & Estimator & \multicolumn{2}{c}{$n=10$} & \multicolumn{2}{c}{$n=100$} & \multicolumn{2}{c}{$n=1000$} & \multicolumn{2}{c}{$n=10000$} \\
 &  & Median & IQR & Median & IQR & Median & IQR & Median & IQR \\
\hline
\endfirsthead
\hline
$\alpha_0,\bb{\beta}^{\star}$ & Estimator & \multicolumn{2}{c}{$n=10$} & \multicolumn{2}{c}{$n=100$} & \multicolumn{2}{c}{$n=1000$} & \multicolumn{2}{c}{$n=10000$} \\
 &  & Median & IQR & Median & IQR & Median & IQR & Median & IQR \\
\hline
\endhead
\hline
\multicolumn{10}{r}{Continued on next page}\\
\hline
\endfoot
\hline
\endlastfoot
\multirow{2}{0.14\textwidth}{\centering\scriptsize\begin{tabular}{@{}c@{}}$23.3$,$(0.71,0.18)$\end{tabular}} & proj-SMOM & \textbf{0.0562} & 0.0623 & \textbf{0.0185} & 0.0199 & \textbf{0.00536} & 0.00589 & \textbf{0.00175} & 0.00194 \\*
 & naive-SMOM & 0.171 & \textbf{0.0365} & 0.0551 & \textbf{0.012} & 0.0171 & \textbf{0.00316} & 0.00543 & \textbf{0.00101} \\
\hline
\multirow{2}{0.14\textwidth}{\centering\scriptsize\begin{tabular}{@{}c@{}}$23.3$,$(0.74,0.09)$\end{tabular}} & proj-SMOM & \textbf{0.0471} & 0.0409 & \textbf{0.0152} & 0.0137 & \textbf{0.00453} & 0.0041 & \textbf{0.00149} & 0.00136 \\*
 & naive-SMOM & 0.203 & \textbf{0.0293} & 0.0646 & \textbf{0.00964} & 0.0206 & \textbf{0.00297} & 0.00644 & \textbf{0.000905} \\
\hline
\multirow{2}{0.14\textwidth}{\centering\scriptsize\begin{tabular}{@{}c@{}}$23.3$,$(3.01,1.15)$\end{tabular}} & proj-SMOM & \textbf{0.0599} & 0.0652 & \textbf{0.0192} & 0.0229 & \textbf{0.00566} & 0.007 & \textbf{0.00165} & 0.00207 \\*
 & naive-SMOM & 0.139 & \textbf{0.0371} & 0.045 & \textbf{0.0127} & 0.0141 & \textbf{0.00379} & 0.00437 & \textbf{0.00114} \\
\hline
\multirow{2}{0.14\textwidth}{\centering\scriptsize\begin{tabular}{@{}c@{}}$23.3$,$(1.01,0.50)$\end{tabular}} & proj-SMOM & \textbf{0.0613} & 0.0681 & \textbf{0.0178} & 0.023 & \textbf{0.00584} & 0.0074 & \textbf{0.00205} & 0.00254 \\*
 & naive-SMOM & 0.118 & \textbf{0.0381} & 0.0371 & \textbf{0.0139} & 0.0118 & \textbf{0.00437} & 0.00381 & \textbf{0.00148} \\
\end{longtable}

%% file: wishart_structured_smom_ratio_table.tex
\renewcommand{\arraystretch}{1}
\setlength{\tabcolsep}{2.5pt}
\begin{longtable}{@{}p{0.14\textwidth}lrrrrrrrr@{}}
\caption{Monte Carlo ratios of Frobenius errors for the projected Stein method-of-moments estimator. Entries are the median and interquartile range, over $B=500$ replications, of $\smash{\|\widehat{\Sigma}_{\mathrm{proj}}-\Sigma_0\|_F/\|\widehat{\Sigma}_{\mathrm{comp}}-\Sigma_0\|_F}$, where the competitor is the (unstructured) Stein method-of-moments estimator. Values below one favour the projected Stein estimator.}\label{tab:wishart-structured-smom-ratios}\\
\hline
$\alpha_0,\bb{\beta}^{\star}$ & Ratio & \multicolumn{2}{c}{$n=10$} & \multicolumn{2}{c}{$n=100$} & \multicolumn{2}{c}{$n=1000$} & \multicolumn{2}{c}{$n=10000$} \\
 &  & Median & IQR & Median & IQR & Median & IQR & Median & IQR \\
\hline
\endfirsthead
\hline
$\alpha_0,\bb{\beta}^{\star}$ & Ratio & \multicolumn{2}{c}{$n=10$} & \multicolumn{2}{c}{$n=100$} & \multicolumn{2}{c}{$n=1000$} & \multicolumn{2}{c}{$n=10000$} \\
 &  & Median & IQR & Median & IQR & Median & IQR & Median & IQR \\
\hline
\endhead
\hline
\multicolumn{10}{r}{Continued on next page}\\
\hline
\endfoot
\hline
\endlastfoot
\multirow{1}{0.14\textwidth}{\centering\scriptsize\begin{tabular}{@{}c@{}}$23.3$,$(0.71,0.18)$\end{tabular}} & $\mathrm{proj}/\mathrm{naive}$ & 0.343 & 0.316 & 0.346 & 0.31 & 0.325 & 0.302 & 0.33 & 0.287 \\
\hline
\multirow{1}{0.14\textwidth}{\centering\scriptsize\begin{tabular}{@{}c@{}}$23.3$,$(0.74,0.09)$\end{tabular}} & $\mathrm{proj}/\mathrm{naive}$ & 0.239 & 0.197 & 0.236 & 0.208 & 0.22 & 0.196 & 0.241 & 0.199 \\
\hline
\multirow{1}{0.14\textwidth}{\centering\scriptsize\begin{tabular}{@{}c@{}}$23.3$,$(3.01,1.15)$\end{tabular}} & $\mathrm{proj}/\mathrm{naive}$ & 0.462 & 0.367 & 0.442 & 0.397 & 0.423 & 0.381 & 0.401 & 0.407 \\
\hline
\multirow{1}{0.14\textwidth}{\centering\scriptsize\begin{tabular}{@{}c@{}}$23.3$,$(1.01,0.50)$\end{tabular}} & $\mathrm{proj}/\mathrm{naive}$ & 0.551 & 0.409 & 0.515 & 0.432 & 0.523 & 0.424 & 0.556 & 0.462 \\
\hline
\multirow{1}{0.14\textwidth}{\centering\scriptsize\begin{tabular}{@{}c@{}}$23.3$,$(1.35,0.35)$\end{tabular}} & $\mathrm{proj}/\mathrm{naive}$ & 0.334 & 0.299 & 0.342 & 0.312 & 0.35 & 0.302 & 0.304 & 0.326 \\
\hline
\multirow{1}{0.14\textwidth}{\centering\scriptsize\begin{tabular}{@{}c@{}}$23.3$,$(1.22,0.07)$\end{tabular}} & $\mathrm{proj}/\mathrm{naive}$ & 0.202 & 0.152 & 0.195 & 0.153 & 0.195 & 0.171 &  0.2 & 0.171 \\
\hline
\multirow{1}{0.14\textwidth}{\centering\scriptsize\begin{tabular}{@{}c@{}}$23.3$,$(1.44,0.44)$\end{tabular}} & $\mathrm{proj}/\mathrm{naive}$ & 0.342 & 0.337 & 0.399 & 0.327 & 0.392 & 0.357 & 0.343 & 0.352 \\
\hline
\multirow{1}{0.14\textwidth}{\centering\scriptsize\begin{tabular}{@{}c@{}}$23.3$,$(1.24,0.22)$\end{tabular}} & $\mathrm{proj}/\mathrm{naive}$ & 0.271 & 0.221 & 0.27 & 0.245 & 0.269 & 0.241 & 0.275 & 0.248 \\
\end{longtable}

%% file: bib_clean.bib
@article {survey23,
    AUTHOR = {Anastasiou, A. and Barp, A. and Briol, F.-X. and Ebner, B. and Gaunt, R. E. and Ghaderinezhad, F. and Gorham, J. and Gretton, A. and Ley, C. and Liu, Q. and Mackey, L. and Oates, C. J. and Reinert, G. and Swan, Y.},
     TITLE = {{Stein}'s method meets computational statistics: a review of some recent developments},
   JOURNAL = {Statist. Sci.},
  FJOURNAL = {Statistical Science},
    VOLUME = {38},
    NUMBER = {1},
      YEAR = {2023},
     PAGES = {120--139},
  MRNUMBER = {4534646},
       DOI = {10.1214/22-STS863},
}

@book {ah19book,
    AUTHOR = {Arras, B. and Houdr{\'e}, C.},
     TITLE = {On {Stein}'s method for infinitely divisible laws with finite first moment},
    SERIES = {SpringerBriefs in Probability and Mathematical Statistics},
 PUBLISHER = {Springer, Cham},
      YEAR = {2019},
     PAGES = {xi+104},
      ISBN = {978-3-030-15016-7; 978-3-030-15017-4},
  MRNUMBER = {3931309},
       DOI = {10.1007/978-3-030-15017-4},
}

@article {ah19,
    AUTHOR = {Arras, B. and Houdr{\'e}, C.},
     TITLE = {On {Stein}'s method for multivariate self-decomposable laws},
   JOURNAL = {Electron. J. Probab.},
  FJOURNAL = {Electronic Journal of Probability},
    VOLUME = {24},
      YEAR = {2019},
     PAGES = {Paper No. 128, 63},
  MRNUMBER = {4029431},
       DOI = {10.1214/19-EJP378},
}

@article {ArrasSwan2017,
    AUTHOR = {Arras, B. and Swan, Y.},
     TITLE = {A stroll along the gamma},
   JOURNAL = {Stochastic Process. Appl.},
  FJOURNAL = {Stochastic Processes and their Applications},
    VOLUME = {127},
    NUMBER = {11},
      YEAR = {2017},
     PAGES = {3661--3688},
  MRNUMBER = {3707241},
       DOI = {10.1016/j.spa.2017.03.012},
}

@incollection {dlmf16,
    AUTHOR = {Askey, R. A. and Olde Daalhuis, A. B.},
     TITLE = {Generalized hypergeometric functions and {M}eijer {$G$}-function},
 BOOKTITLE = {{NIST} Handbook of Mathematical Functions},
    EDITOR = {Olver, Frank W. J. and Lozier, Daniel W. and Boisvert, Ronald F. and Clark, Charles W.},
 PUBLISHER = {Cambridge University Press},
   ADDRESS = {New York},
      YEAR = {2010},
     PAGES = {403--418},
   CHAPTER = {16},
  MRNUMBER = {2655356},
       URL = {https://dlmf.nist.gov/16},
      NOTE = {Online {DLMF} version: Release 1.2.6 of 2026-03-15},
}

@misc{BaillyOuimet2026github,
	AUTHOR = {Bailly, G. and Ouimet, F.},
	YEAR = {2026},
	TITLE = {Stein{W}ishart},
	NOTE = {The GitHub repository is publicly available online at \href{https://github.com/FredericOuimetMcGill/SteinWishart}{https://github.com/FredericOuimetMcGill/SteinWishart}}
}

@article {MR1035659,
    AUTHOR = {Barbour, A. D.},
     TITLE = {{Stein}'s method for diffusion approximations},
   JOURNAL = {Probab. Theory Related Fields},
  FJOURNAL = {Probability Theory and Related Fields},
    VOLUME = {84},
      YEAR = {1990},
    NUMBER = {3},
     PAGES = {297--322},
  MRNUMBER = {1035659},
MRREVIEWER = {Holst, L.},
       DOI = {10.1007/BF01197887},
}

@book {Bhatia2007,
    AUTHOR = {Bhatia, R.},
     TITLE = {Positive {D}efinite {M}atrices},
    SERIES = {Princeton Series in Applied Mathematics},
 PUBLISHER = {Princeton University Press, Princeton, NJ},
      YEAR = {2007},
     PAGES = {x+254},
      ISBN = {978-0-691-12918-1; 0-691-12918-5},
  MRNUMBER = {2284176},
MRREVIEWER = {Smith, R. L.},
}

@book {BoucheronLugosiMassart2013,
    AUTHOR = {Boucheron, S. and Lugosi, G. and Massart, P.},
     TITLE = {Concentration {I}nequalities: {A} {N}onasymptotic {T}heory of {I}ndependence},
 PUBLISHER = {Oxford University Press, Oxford},
      YEAR = {2013},
     PAGES = {x+481},
      ISBN = {978-0-19-953525-5},
  MRNUMBER = {3185193},
MRREVIEWER = {Ravi, S.},
       DOI = {10.1093/acprof:oso/9780199535255.001.0001},
}

@article {Bru1991,
    AUTHOR = {Bru, M.-F.},
     TITLE = {{W}ishart processes},
   JOURNAL = {J. Theoret. Probab.},
  FJOURNAL = {Journal of Theoretical Probability},
    VOLUME = {4},
    NUMBER = {4},
      YEAR = {1991},
     PAGES = {725--751},
  MRNUMBER = {1132135},
       DOI = {10.1007/BF01259552},
}

@article {cfr11,
    AUTHOR = {Chatterjee, S. and Fulman, J. and R{\"o}llin, A.},
     TITLE = {Exponential approximation by {Stein}'s method and spectral graph theory},
   JOURNAL = {ALEA, Lat. Am. J. Probab. Math. Stat.},
  FJOURNAL = {ALEA. Latin American Journal of Probability and Mathematical Statistics},
    VOLUME = {8},
      YEAR = {2011},
     PAGES = {197--223},
  MRNUMBER = {2802856},
       URL = {https://alea.impa.br/articles/v8/08-11.pdf},
}

@article {c75,
    AUTHOR = {Chen, L. H. Y.},
     TITLE = {Poisson approximation for dependent trials},
   JOURNAL = {Ann. Probab.},
  FJOURNAL = {The Annals of Probability},
    VOLUME = {3},
    NUMBER = {3},
      YEAR = {1975},
     PAGES = {534--545},
  MRNUMBER = {0428387},
       DOI = {10.1214/aop/1176996359},
}

@book {MR2732624,
    AUTHOR = {Chen, L. H. Y. and Goldstein, L. and Shao, Q.-M.},
     TITLE = {Normal {A}pproximation by {S}tein's {M}ethod},
    SERIES = {Probability and its Applications (New York)},
 PUBLISHER = {Springer, Heidelberg},
      YEAR = {2011},
     PAGES = {xii+405},
      ISBN = {978-3-642-15006-7},
  MRNUMBER = {2732624},
MRREVIEWER = {R\'{e}veillac, A.},
       DOI = {10.1007/978-3-642-15007-4},
}

@article {stable24,
    AUTHOR = {Chen, P. and Nourdin, I. and Xu, L. and Yang, X.},
     TITLE = {Multivariate stable approximation by {Stein}'s method},
   JOURNAL = {J. Theoret. Probab.},
  FJOURNAL = {Journal of Theoretical Probability},
    VOLUME = {37},
    NUMBER = {1},
      YEAR = {2024},
     PAGES = {446--488},
  MRNUMBER = {4716322},
       DOI = {10.1007/s10959-023-01244-x},
}

@article {CuchieroFilipovicMayerhoferTeichmann2011,
    AUTHOR = {Cuchiero, C. and Filipovi\'{c}, D. and Mayerhofer, E. and Teichmann, J.},
     TITLE = {Affine processes on positive semidefinite matrices},
   JOURNAL = {Ann. Appl. Probab.},
  FJOURNAL = {The Annals of Applied Probability},
    VOLUME = {21},
    NUMBER = {2},
      YEAR = {2011},
     PAGES = {397--463},
  MRNUMBER = {2807963},
MRREVIEWER = {Fahrner, I.},
       DOI = {10.1214/10-AAP710},
}

@article {tudor,
    AUTHOR = {Dhoyer, R. and Tudor, C. A.},
     TITLE = {Limit behavior in high-dimensional regime for the {Wishart} tensors in {Wiener} chaos},
   JOURNAL = {J. Theoret. Probab.},
  FJOURNAL = {Journal of Theoretical Probability},
    VOLUME = {37},
    NUMBER = {2},
      YEAR = {2024},
     PAGES = {1445--1468},
  MRNUMBER = {4751298},
       DOI = {10.1007/s10959-024-01328-2},
}

@article {dz91,
    AUTHOR = {Diaconis, P. and Zabell, S.},
     TITLE = {Closed form summation for classical distributions: Variations on a theme of de {Moivre}},
   JOURNAL = {Statist. Sci.},
  FJOURNAL = {Statistical Science},
    VOLUME = {6},
    NUMBER = {3},
      YEAR = {1991},
     PAGES = {284--302},
  MRNUMBER = {1144242},
       DOI = {10.1214/ss/1177011699},
}

@article {EbnerFischerGauntPickerSwan2025,
    AUTHOR = {Ebner, B. and Fischer, A. and Gaunt, R. E. and Picker, B. and Swan, Y.},
     TITLE = {{Stein}'s method of moments},
   JOURNAL = {Scand. J. Stat.},
  FJOURNAL = {Scandinavian Journal of Statistics. Theory and Applications},
    VOLUME = {52},
    NUMBER = {4},
      YEAR = {2025},
     PAGES = {1594--1624},
  MRNUMBER = {4986908},
       DOI = {10.1111/sjos.70003},
}

@article {ev15,
    AUTHOR = {Eden, R. and V{\'{\i}}quez, J.},
     TITLE = {{Nourdin--Peccati} analysis on {Wiener} and {Wiener--Poisson} space for general distributions},
   JOURNAL = {Stochastic Process. Appl.},
  FJOURNAL = {Stochastic Processes and their Applications},
    VOLUME = {125},
    NUMBER = {1},
      YEAR = {2015},
     PAGES = {182--216},
  MRNUMBER = {3274696},
       DOI = {10.1016/j.spa.2014.09.001},
}

@phdthesis {Gaunt2013PhD,
    AUTHOR = {Gaunt, R. E.},
     TITLE = {Rates of convergence of {Variance-Gamma} approximations via {Stein}'s method},
    SCHOOL = {University of Oxford},
      YEAR = {2013},
      TYPE = {PhD thesis},
     PAGES = {240 pp.},
      NOTE = {The Queen's College},
}

@article {g14,
    AUTHOR = {Gaunt, R. E.},
     TITLE = {{Variance-Gamma} approximation via {Stein}'s method},
   JOURNAL = {Electron. J. Probab.},
  FJOURNAL = {Electronic Journal of Probability},
    VOLUME = {19},
    NUMBER = {38},
      YEAR = {2014},
     PAGES = {33 pp.},
  MRNUMBER = {3194737},
       DOI = {10.1214/EJP.v19-3020},
}

@article {GauntOuimetRichards2026,
    AUTHOR = {Gaunt, R. E. and Ouimet, F. and Richards, D.},
     TITLE = {{Stein}'s method for the matrix normal distribution},
   JOURNAL = {ArXiv preprint},
  FJOURNAL = {ArXiv preprint},
      YEAR = {2026},
       DOI = {10.48550/arXiv.2601.11422},
}

@article {gpr17,
    AUTHOR = {Gaunt, R. E. and Pickett, A. M. and Reinert, G.},
     TITLE = {{Chi-square} approximation by {Stein}'s method with application to {Pearson}'s statistic},
   JOURNAL = {Ann. Appl. Probab.},
  FJOURNAL = {The Annals of Applied Probability},
    VOLUME = {27},
    NUMBER = {2},
      YEAR = {2017},
     PAGES = {720--756},
  MRNUMBER = {3655852},
       DOI = {10.1214/16-AAP1213},
}

@article {gr23,
    AUTHOR = {Gaunt, R. E. and Reinert, G.},
     TITLE = {Bounds for the chi-square approximation of {F}riedman's
              statistic by {S}tein's method},
   JOURNAL = {Bernoulli},
  FJOURNAL = {Bernoulli},
    VOLUME = {29},
    NUMBER = {3},
      YEAR = {2023},
     PAGES = {2008--2034},
       DOI = {10.3150/22-BEJ1530},
}

@article {GenestMacKayOuimet2026,
    AUTHOR = {Genest, C. and MacKay, A. and Ouimet, F.},
     TITLE = {On noncentral {W}ishart mixtures of noncentral {W}isharts and
              their use for testing random effects in factorial design
              models},
   JOURNAL = {J. Math. Anal. Appl.},
  FJOURNAL = {Journal of Mathematical Analysis and Applications},
    VOLUME = {554},
    NUMBER = {1},
      YEAR = {2026},
     PAGES = {Paper No. 129897, 11 pp.},
  MRNUMBER = {4936518},
       DOI = {10.1016/j.jmaa.2025.129897},
}

@article {g91,
    AUTHOR = {G{\"o}tze, F.},
     TITLE = {On the rate of convergence in the multivariate {CLT}},
   JOURNAL = {Ann. Probab.},
  FJOURNAL = {Annals of Probability},
    VOLUME = {19},
    NUMBER = {2},
      YEAR = {1991},
     PAGES = {724--739},
  MRNUMBER = {1106283},
       DOI = {10.1214/aop/1176990448},
}

@book {GuptaNagar2000,
    AUTHOR = {Gupta, A. K. and Nagar, D. K.},
     TITLE = {Matrix {V}ariate {D}istributions},
    SERIES = {Monographs and Surveys in Pure and Applied Mathematics},
    VOLUME = {104},
   EDITION = {First},
 PUBLISHER = {Chapman \& Hall/CRC},
   ADDRESS = {Boca Raton, FL},
      YEAR = {2000},
     PAGES = {384},
      ISBN = {1-58488-046-5},
  MRNUMBER = {1738933},
       DOI = {10.1201/9780203749289},
}

@article {Haff1979WishartIdentity,
    AUTHOR = {Haff, L. R.},
     TITLE = {An identity for the {W}ishart distribution with applications},
   JOURNAL = {J. Multivariate Anal.},
  FJOURNAL = {Journal of Multivariate Analysis},
    VOLUME = {9},
    NUMBER = {4},
      YEAR = {1979},
     PAGES = {531--544},
  MRNUMBER = {556910},
       DOI = {10.1016/0047-259X(79)90056-3},
}

@article {Herz1955,
    AUTHOR = {Herz, C. S.},
     TITLE = {{B}essel functions of matrix argument},
   JOURNAL = {Ann. of Math. (2)},
  FJOURNAL = {Annals of Mathematics. Second Series},
    VOLUME = {61},
    NUMBER = {3},
      YEAR = {1955},
     PAGES = {474--523},
  MRNUMBER = {69960},
       DOI = {10.2307/1969810},
}

@article {Khatri1966,
    AUTHOR = {Khatri, C. G.},
     TITLE = {On certain distribution problems based on positive definite
              quadratic functions in normal vectors},
   JOURNAL = {Ann. Math. Statist.},
  FJOURNAL = {The Annals of Mathematical Statistics},
    VOLUME = {37},
    NUMBER = {2},
      YEAR = {1966},
     PAGES = {468--479},
  MRNUMBER = {190965},
       DOI = {10.1214/aoms/1177699530},
}

@article {Khatri1971,
    AUTHOR = {Khatri, C. G.},
     TITLE = {Series representations of distributions of quadratic form in
              the normal vectors and generalised variance},
   JOURNAL = {J. Multivariate Anal.},
  FJOURNAL = {Journal of Multivariate Analysis},
    VOLUME = {1},
    NUMBER = {2},
      YEAR = {1971},
     PAGES = {199--214},
  MRNUMBER = {303637},
       DOI = {10.1016/0047-259X(71)90011-X},
}

@article {Khatri1989,
    AUTHOR = {Khatri, C. G.},
     TITLE = {Multivariate generalization of {$t'$}-statistic based on the mean square successive difference},
   JOURNAL = {Commun. Statist. Theory Methods},
  FJOURNAL = {Communications in Statistics: Theory and Methods},
    VOLUME = {18},
    NUMBER = {5},
      YEAR = {1989},
     PAGES = {1983--1992},
       DOI = {10.1080/03610928908830015},
}

@book {Krantz2001,
    AUTHOR = {Krantz, S. G.},
     TITLE = {Function {T}heory of {S}everal {C}omplex {V}ariables},
   EDITION = {Second},
      NOTE = {Reprint of the 1992 edition},
 PUBLISHER = {AMS Chelsea Publishing, Providence, RI},
      YEAR = {2001},
     PAGES = {xvi+564},
      ISBN = {0-8218-2724-3},
  MRNUMBER = {1846625},
       DOI = {10.1090/chel/340},
}

@article {lnp15,
    AUTHOR = {Ledoux, M. and Nourdin, I. and Peccati, G.},
     TITLE = {{Stein}'s method, logarithmic {Sobolev} and transport inequalities},
   JOURNAL = {Geom. Funct. Anal.},
  FJOURNAL = {Geometric and Functional Analysis. GAFA},
    VOLUME = {25},
    NUMBER = {1},
      YEAR = {2015},
     PAGES = {256--306},
  MRNUMBER = {3320893},
       DOI = {10.1007/s00039-015-0312-0},
}

@phdthesis {Luk1994PhD,
    AUTHOR = {Luk, H. M.},
     TITLE = {{Stein}'s Method for the {G}amma Distribution and Related {S}tatistical Applications},
    SCHOOL = {University of Southern California},
      YEAR = {1994},
  MRNUMBER = {2693204},
}

@article {mrrs23,
    AUTHOR = {Mijoule, G. and Rai\v{c}, M. and Reinert, G. and Swan, Y.},
     TITLE = {{Stein}'s density method for multivariate continuous distributions},
   JOURNAL = {Electron. J. Probab.},
  FJOURNAL = {Electronic Journal of Probability},
    VOLUME = {28},
      YEAR = {2023},
     PAGES = {Paper No. 59, 40},
  MRNUMBER = {4583066},
       DOI = {10.1214/22-EJP883},
}

@article {m22,
    AUTHOR = {Mikulincer, D.},
     TITLE = {A {CLT} in {Stein}'s distance for generalized {Wishart} matrices and higher-order tensors},
   JOURNAL = {Int. Math. Res. Not. IMRN},
  FJOURNAL = {International Mathematics Research Notices. IMRN},
    VOLUME = {2022},
    NUMBER = {10},
      YEAR = {2022},
     PAGES = {7839--7872},
  MRNUMBER = {4418721},
       DOI = {10.1093/imrn/rnaa336},
}

@article {Mortarino2005,
    AUTHOR = {Mortarino, C.},
     TITLE = {A decomposition for a stochastic matrix with an application to {MANOVA}},
   JOURNAL = {J. Multivariate Anal.},
  FJOURNAL = {Journal of Multivariate Analysis},
    VOLUME = {92},
    NUMBER = {1},
      YEAR = {2005},
     PAGES = {134--144},
       DOI = {10.1016/S0047-259X(03)00135-0},
}

@book {Muirhead1982,
    AUTHOR = {Muirhead, R. J.},
     TITLE = {Aspects of {M}ultivariate {S}tatistical {T}heory},
    SERIES = {Wiley Series in Probability and Mathematical Statistics},
 PUBLISHER = {John Wiley \& Sons, Inc.},
   ADDRESS = {New York},
      YEAR = {1982},
     PAGES = {xix+673},
      ISBN = {0-471-09442-0},
  MRNUMBER = {652932},
       DOI = {10.1002/9780470316559},
}

@article {NaikRao2001,
    AUTHOR = {Naik, D. N. and Rao, S. S.},
     TITLE = {Analysis of multivariate repeated measures data with a {K}ronecker product structured covariance matrix},
   JOURNAL = {J. Appl. Stat.},
  FJOURNAL = {Journal of Applied Statistics},
    VOLUME = {28},
    NUMBER = {1},
      YEAR = {2001},
     PAGES = {91--105},
  MRNUMBER = {1834425},
       DOI = {10.1080/02664760120011626},
}

@book {np12,
    AUTHOR = {Nourdin, I. and Peccati, G.},
     TITLE = {{Normal Approximations with Malliavin Calculus: From Stein's Method to Universality}},
    SERIES = {Cambridge Tracts in Mathematics},
    VOLUME = {192},
 PUBLISHER = {Cambridge University Press},
   ADDRESS = {Cambridge},
      YEAR = {2012},
     PAGES = {xiv+239},
      ISBN = {978-1-107-01777-1},
  MRNUMBER = {2962301},
       DOI = {10.1017/CBO9781139084659},
}

@article {nz22,
    AUTHOR = {Nourdin, I. and Zheng, G.},
     TITLE = {Asymptotic behavior of large {G}aussian correlated {W}ishart matrices},
   JOURNAL = {J. Theoret. Probab.},
  FJOURNAL = {Journal of Theoretical Probability},
    VOLUME = {35},
    NUMBER = {4},
      YEAR = {2022},
     PAGES = {2239--2268},
  MRNUMBER = {4509068},
       DOI = {10.1007/s10959-021-01133-1},
}

@article {pr11,
    AUTHOR = {Pek{\"o}z, E. A. and R{\"o}llin, A.},
     TITLE = {New rates for exponential approximation and the theorems of {R{\'e}nyi} and {Yaglom}},
   JOURNAL = {Ann. Probab.},
  FJOURNAL = {The Annals of Probability},
    VOLUME = {39},
    NUMBER = {2},
      YEAR = {2011},
     PAGES = {587--608},
  MRNUMBER = {2789507},
       DOI = {10.1214/10-AOP559},
}

@article {Pfaffel2012,
    AUTHOR = {Pfaffel, O.},
     TITLE = {{Wishart} processes},
   JOURNAL = {ArXiv preprint},
  FJOURNAL = {ArXiv preprint},
      YEAR = {2012},
       DOI = {10.48550/arXiv.1201.3256},
}

@article {raic04,
    AUTHOR = {Rai\v{c}, M.},
     TITLE = {A multivariate {CLT} for decomposable random vectors with finite second moments},
   JOURNAL = {J. Theoret. Probab.},
  FJOURNAL = {Journal of Theoretical Probability},
    VOLUME = {17},
    NUMBER = {3},
      YEAR = {2004},
     PAGES = {573--603},
  MRNUMBER = {2091552},
       DOI = {10.1023/B:JOTP.0000040290.44087.68},
}

@incollection {Richards2010,
    AUTHOR = {Richards, D. St. P.},
     TITLE = {Functions of matrix argument},
 BOOKTITLE = {{NIST} Handbook of Mathematical Functions},
    EDITOR = {Olver, F. W. J. and Lozier, D. W. and Boisvert, R. F. and Clark, C. W.},
   CHAPTER = {35},
     PAGES = {767--774},
 PUBLISHER = {Cambridge University Press},
   ADDRESS = {Cambridge},
      YEAR = {2010},
      ISBN = {978-0-521-14063-8},
  MRNUMBER = {2723248},
       URL = {https://dlmf.nist.gov/35},
}

@article {ross11,
    AUTHOR = {Ross, N.},
     TITLE = {Fundamentals of {S}tein's method},
   JOURNAL = {Probab. Surv.},
  FJOURNAL = {Probability Surveys},
    VOLUME = {8},
      YEAR = {2011},
     PAGES = {210--293},
  MRNUMBER = {2861132},
       DOI = {10.1214/11-PS182},
}

@article {Satterthwaite1941,
    AUTHOR = {Satterthwaite, F. E.},
     TITLE = {Synthesis of variance},
   JOURNAL = {Psychometrika},
  FJOURNAL = {Psychometrika},
    VOLUME = {6},
    NUMBER = {5},
      YEAR = {1941},
     PAGES = {309--316},
       DOI = {10.1007/BF02288586},
}

@article {Satterthwaite1946,
    AUTHOR = {Satterthwaite, F. E.},
     TITLE = {An approximate distribution of estimates of variance components},
   JOURNAL = {Biometrics Bull.},
  FJOURNAL = {Biometrics Bulletin},
    VOLUME = {2},
    NUMBER = {6},
      YEAR = {1946},
     PAGES = {110--114},
       DOI = {10.2307/3002019},
}

@article {SingullKoski2012,
    AUTHOR = {Singull, M. and Koski, T.},
     TITLE = {On the distribution of matrix quadratic forms},
   JOURNAL = {Comm. Statist. Theory Methods},
  FJOURNAL = {Communications in Statistics. Theory and Methods},
    VOLUME = {41},
    NUMBER = {18},
      YEAR = {2012},
     PAGES = {3403--3415},
       DOI = {10.1080/03610926.2011.563009},
}

@incollection {s72,
    AUTHOR = {Stein, C.},
     TITLE = {A bound for the error in the normal approximation to the distribution of a sum of dependent random variables},
 BOOKTITLE = {Proceedings of the Sixth Berkeley Symposium on Mathematical Statistics and Probability, Vol. {II}: Probability Theory},
     PAGES = {583--602},
 PUBLISHER = {University of California Press},
   ADDRESS = {Berkeley, CA},
      YEAR = {1972},
  MRNUMBER = {0402873},
}

@book {steinmonograph,
    AUTHOR = {Stein, C.},
     TITLE = {Approximate Computation of Expectations},
   FSERIES = {Institute of Mathematical Statistics Lecture Notes---Monograph Series},
    SERIES = {IMS Lecture Notes Monogr. Ser.},
    VOLUME = {7},
 PUBLISHER = {Institute of Mathematical Statistics},
   ADDRESS = {Hayward, CA},
      YEAR = {1986},
     PAGES = {iv+164},
      ISBN = {0-940600-08-0},
  MRNUMBER = {0882007},
       DOI = {10.1214/lnms/1215466568},
}

@article {Tan1980,
    AUTHOR = {Tan, W. Y.},
     TITLE = {On probability distributions from mixtures of multivariate densities},
   JOURNAL = {South African Statist. J.},
  FJOURNAL = {South African Statistical Journal},
    VOLUME = {14},
    NUMBER = {1},
      YEAR = {1980},
     PAGES = {47--59},
	   URL = {https://hdl.handle.net/10520/AJA0038271X_604}
}

@article {TanGupta1983,
    AUTHOR = {Tan, W. Y. and Gupta, R. P.},
     TITLE = {On approximating a linear combination of central {W}ishart matrices with positive coefficients},
   JOURNAL = {Comm. Statist. Theory Methods},
  FJOURNAL = {Communications in Statistics. Theory and Methods},
    VOLUME = {12},
    NUMBER = {22},
      YEAR = {1983},
     PAGES = {2589--2600},
       DOI = {10.1080/03610928308828625},
}

@article {x19,
    AUTHOR = {Xu, L.},
     TITLE = {Approximation of stable law in {W}asserstein-1 distance by {S}tein's method},
   JOURNAL = {Ann. Appl. Probab.},
  FJOURNAL = {Annals of Applied Probability},
    VOLUME = {29},
    NUMBER = {1},
      YEAR = {2019},
     PAGES = {458--504},
  MRNUMBER = {3910009},
       DOI = {10.1214/18-AAP1424},
}

@book {BakryGentilLedoux2014,
    AUTHOR = {Bakry, D. and Gentil, I. and Ledoux, M.},
     TITLE = {Analysis and {G}eometry of {M}arkov {D}iffusion {O}perators},
    SERIES = {Grundlehren der mathematischen Wissenschaften},
    VOLUME = {348},
 PUBLISHER = {Springer, Cham},
      YEAR = {2014},
     PAGES = {xx+552},
      ISBN = {978-3-319-00226-2; 978-3-319-00227-9},
  MRNUMBER = {3155209},
       DOI = {10.1007/978-3-319-00227-9},
}

@article {Gross1975,
    AUTHOR = {Gross, L.},
     TITLE = {Logarithmic {S}obolev inequalities},
   JOURNAL = {Amer. J. Math.},
  FJOURNAL = {American Journal of Mathematics},
    VOLUME = {97},
    NUMBER = {4},
      YEAR = {1975},
     PAGES = {1061--1083},
  MRNUMBER = {0420249},
       DOI = {10.2307/2373688},
}

@incollection {BakryEmery1985,
    AUTHOR = {Bakry, D. and {\'E}mery, M.},
     TITLE = {Diffusions hypercontractives},
 BOOKTITLE = {S{\'e}minaire de Probabilit{\'e}s, XIX, 1983/84},
    SERIES = {Lecture Notes in Math.},
    VOLUME = {1123},
 PUBLISHER = {Springer-Verlag},
   ADDRESS = {Berlin},
      YEAR = {1985},
     PAGES = {177--206},
  MRNUMBER = {0889476},
       DOI = {10.1007/BFb0075847},
       URL = {https://www.numdam.org/item/SPS_1985__19__177_0/},
}

@incollection {bakry96,
    AUTHOR = {Bakry, D.},
     TITLE = {Remarques sur les semigroupes de {Jacobi}},
 BOOKTITLE = {Hommage \`a P. A. Meyer et J. Neveu},
    SERIES = {Ast{\'e}risque},
    NUMBER = {236},
 PUBLISHER = {Soci{\'e}t{\'e} math{\'e}matique de France},
   ADDRESS = {Paris},
      YEAR = {1996},
     PAGES = {23--39},
  MRNUMBER = {1417973},
       URL = {https://www.numdam.org/item/AST_1996__236__23_0/},
}

@article {Scheffe1947,
    AUTHOR = {Scheff{\'e}, H.},
     TITLE = {A useful convergence theorem for probability distributions},
   JOURNAL = {Ann. Math. Statist.},
  FJOURNAL = {The Annals of Mathematical Statistics},
    VOLUME = {18},
    NUMBER = {3},
      YEAR = {1947},
     PAGES = {434--438},
  MRNUMBER = {0021585},
       DOI = {10.1214/aoms/1177730390},
}
